\documentclass[11pt]{amsart}

\usepackage{etoolbox}

\makeatletter
\let\old@tocline\@tocline
\let\section@tocline\@tocline
\newcommand{\subsection@dotsep}{4.5}
\newcommand{\subsubsection@dotsep}{4.5}
\patchcmd{\@tocline}
{\hfil}
{\nobreak
	\leaders\hbox{$\m@th
		\mkern \subsection@dotsep mu\hbox{.}\mkern \subsection@dotsep mu$}\hfill
	\nobreak}{}{}
\let\subsection@tocline\@tocline
\let\@tocline\old@tocline

\patchcmd{\@tocline}
{\hfil}
{\nobreak
	\leaders\hbox{$\m@th
		\mkern \subsubsection@dotsep mu\hbox{.}\mkern \subsubsection@dotsep mu$}\hfill
	\nobreak}{}{}
\let\subsubsection@tocline\@tocline
\let\@tocline\old@tocline

\let\old@l@subsection\l@subsection
\let\old@l@subsubsection\l@subsubsection

\def\@tocwriteb#1#2#3{%
	\begingroup
	\@xp\def\csname #2@tocline\endcsname##1##2##3##4##5##6{%
		\ifnum##1>\c@tocdepth
		\else \sbox\z@{##5\let\indentlabel\@tochangmeasure##6}\fi}%
	\csname l@#2\endcsname{#1{\csname#2name\endcsname}{\@secnumber}{}}%
	\endgroup
	\addcontentsline{toc}{#2}%
	{\protect#1{\csname#2name\endcsname}{\@secnumber}{#3}}}%

\newlength{\@tocsectionindent}
\newlength{\@tocsubsectionindent}
\newlength{\@tocsubsubsectionindent}
\newlength{\@tocsectionnumwidth}
\newlength{\@tocsubsectionnumwidth}
\newlength{\@tocsubsubsectionnumwidth}
\newcommand{\settocsectionnumwidth}[1]{\setlength{\@tocsectionnumwidth}{#1}}
\newcommand{\settocsubsectionnumwidth}[1]{\setlength{\@tocsubsectionnumwidth}{#1}}
\newcommand{\settocsubsubsectionnumwidth}[1]{\setlength{\@tocsubsubsectionnumwidth}{#1}}
\newcommand{\settocsectionindent}[1]{\setlength{\@tocsectionindent}{#1}}
\newcommand{\settocsubsectionindent}[1]{\setlength{\@tocsubsectionindent}{#1}}
\newcommand{\settocsubsubsectionindent}[1]{\setlength{\@tocsubsubsectionindent}{#1}}

\renewcommand{\l@section}{\section@tocline{1}{\@tocsectionvskip}{\@tocsectionindent}{}{\@tocsectionformat}}%
\renewcommand{\l@subsection}{\subsection@tocline{1}{\@tocsubsectionvskip}{\@tocsubsectionindent}{}{\@tocsubsectionformat}}%
\renewcommand{\l@subsubsection}{\subsubsection@tocline{1}{\@tocsubsubsectionvskip}{\@tocsubsubsectionindent}{}{\@tocsubsubsectionformat}}%
\newcommand{\@tocsectionformat}{}
\newcommand{\@tocsubsectionformat}{}
\newcommand{\@tocsubsubsectionformat}{}
\expandafter\def\csname toc@1format\endcsname{\@tocsectionformat}
\expandafter\def\csname toc@2format\endcsname{\@tocsubsectionformat}
\expandafter\def\csname toc@3format\endcsname{\@tocsubsubsectionformat}
\newcommand{\settocsectionformat}[1]{\renewcommand{\@tocsectionformat}{#1}}
\newcommand{\settocsubsectionformat}[1]{\renewcommand{\@tocsubsectionformat}{#1}}
\newcommand{\settocsubsubsectionformat}[1]{\renewcommand{\@tocsubsubsectionformat}{#1}}
\newlength{\@tocsectionvskip}
\newcommand{\settocsectionvskip}[1]{\setlength{\@tocsectionvskip}{#1}}
\newlength{\@tocsubsectionvskip}
\newcommand{\settocsubsectionvskip}[1]{\setlength{\@tocsubsectionvskip}{#1}}
\newlength{\@tocsubsubsectionvskip}
\newcommand{\settocsubsubsectionvskip}[1]{\setlength{\@tocsubsubsectionvskip}{#1}}

\patchcmd{\tocsection}{\indentlabel}{\makebox[\@tocsectionnumwidth][l]}{}{}
\patchcmd{\tocsubsection}{\indentlabel}{\makebox[\@tocsubsectionnumwidth][l]}{}{}
\patchcmd{\tocsubsubsection}{\indentlabel}{\makebox[\@tocsubsubsectionnumwidth][l]}{}{}

\newcommand{\@sectypepnumformat}{}
\renewcommand{\contentsline}[1]{%
	\expandafter\let\expandafter\@sectypepnumformat\csname @toc#1pnumformat\endcsname%
	\csname l@#1\endcsname}
\newcommand{\@tocsectionpnumformat}{}
\newcommand{\@tocsubsectionpnumformat}{}
\newcommand{\@tocsubsubsectionpnumformat}{}
\newcommand{\setsectionpnumformat}[1]{\renewcommand{\@tocsectionpnumformat}{#1}}
\newcommand{\setsubsectionpnumformat}[1]{\renewcommand{\@tocsubsectionpnumformat}{#1}}
\newcommand{\setsubsubsectionpnumformat}[1]{\renewcommand{\@tocsubsubsectionpnumformat}{#1}}
\renewcommand{\@tocpagenum}[1]{%
	\hfill {\mdseries\@sectypepnumformat #1}}

\let\oldappendix\appendix
\renewcommand{\appendix}{%
	\leavevmode\oldappendix%
	\addtocontents{toc}{%
		\protect\settowidth{\protect\@tocsectionnumwidth}{\protect\@tocsectionformat\sectionname\space}%
		\protect\addtolength{\protect\@tocsectionnumwidth}{2em}}%
}
\makeatother

\makeatletter
\settocsectionnumwidth{2em}
\settocsubsectionnumwidth{2.5em}
\settocsubsubsectionnumwidth{3em}
\settocsectionindent{1pc}%
\settocsubsectionindent{\dimexpr\@tocsectionindent+\@tocsectionnumwidth}%
\settocsubsubsectionindent{\dimexpr\@tocsubsectionindent+\@tocsubsectionnumwidth}%
\makeatother

\settocsectionvskip{10pt}
\settocsubsectionvskip{0pt}
\settocsubsubsectionvskip{0pt}

\settocsectionformat{\bfseries}
\settocsubsectionformat{\mdseries}
\settocsubsubsectionformat{\mdseries}
\setsectionpnumformat{\bfseries}
\setsubsectionpnumformat{\mdseries}
\setsubsubsectionpnumformat{\mdseries}

\let\oldtableofcontents\tableofcontents
\renewcommand{\tableofcontents}{%
	\vspace*{-\linespacing}
	\oldtableofcontents}

\usepackage[dvipsnames]{xcolor}
\usepackage[hyperfootnotes=false]{hyperref}
\usepackage{enumerate}
\usepackage{tikz}
\usepackage{tikz-cd}
\usepackage{hyperref}
\usepackage{amsthm,amsfonts,amsmath,amscd,amssymb}
\usepackage{xcolor,float}
\usepackage[all]{xy}
\usepackage{mathrsfs}
\usepackage{graphics,graphicx}
\usepackage{caption}

\let\oldmarginpar\marginpar
\renewcommand\marginpar[1]{\-\oldmarginpar[\raggedleft\footnotesize #1]%
	{\raggedright\footnotesize #1}}
\theoremstyle{plain}
\newtheorem{thm}{Theorem}[section]
\newtheorem{lemma}[thm]{Lemma}
\newtheorem{example}[thm]{Example}
\newtheorem{prop}[thm]{Proposition}
\newtheorem{principle}[thm]{Principle}
\newtheorem{cor}[thm]{Corollary}

\theoremstyle{definition}

\newtheorem{definition}[thm]{Definition}
\newtheorem{remark}[thm]{Remark}

\newtheorem{ex}[thm]{Example}
\theoremstyle{remark}

\numberwithin{equation}{section}

\renewcommand{\S}{\mathbb{S}}
\renewcommand{\P}{\mathbb{P}}
\renewcommand{\L}{\mathbb{L}}
\newcommand{\F}{\mathbb{F}}
\newcommand{\D}{\mathbb{D}}
\newcommand{\N}{\mathbb{N}}
\newcommand{\Z}{\mathbb{Z}}
\newcommand{\Q}{\mathbb{Q}}
\newcommand{\R}{\mathbb{R}}
\newcommand{\C}{\mathbb{C}}
\newcommand{\SA}{\mathcal{A}}
\renewcommand{\SS}{\mathscr{S}}

\newcommand{\SM}{\mathcal{M}}

\newcommand{\SC}{\mathcal{C}}
\newcommand{\SB}{\mathcal{B}}
\newcommand{\SE}{\mathscr{E}}
\newcommand{\SF}{\mathscr{F}}
\newcommand{\SG}{\mathscr{G}}

\newcommand{\SL}{\mathcal{L}}

\newcommand{\A}{\mathcal{A}}
\newcommand{\La}{\Lambda}
\newcommand{\la}{\lambda}

\renewcommand{\a}{\alpha}

\renewcommand{\d}{\delta}

\newcommand{\dd}{\partial}

\newcommand{\sse}{\subseteq}
\newcommand{\lr}{\longrightarrow}

\newcommand{\x}{\times}

\newcommand{\Diff}{\operatorname{Diff}}

\newcommand{\GL}{\operatorname{GL}}
\newcommand{\PGL}{\operatorname{PGL}}

\newcommand{\Sh}{\operatorname{Sh}}
\newcommand{\Aug}{\operatorname{Aug}}
\newcommand{\Spec}{\operatorname{Spec}}

\newcommand{\wt}{\widetilde}

\newcommand{\st}{\text{st}}

\def\Op{{\mathcal O}{\it p}}
\newcounter{daggerfootnote}

\newcommand{\bC}{\mathbb{C}}

\newcommand{\bF}{\mathbb{F}}
\newcommand{\bG}{\mathbb{G}}
\newcommand{\bH}{\mathscr{H}}

\newcommand{\bR}{\mathbb{R}}
\newcommand{\bZ}{\mathbb{Z}}
\newcommand{\bT}{\mathbb{T}}

\newcommand{\cA}{\mathcal{A}}
\newcommand{\cB}{\mathcal{B}}
\newcommand{\cC}{\mathcal{C}}

\newcommand{\cL}{\mathcal{L}}
\newcommand{\cM}{\mathcal{M}}

\newcommand{\cR}{\mathcal{R}}

\begin{document}
	
	\title{Legendrian Weaves\\\vspace{0.5cm} {\bf\footnotesize{-- N-graph Calculus, Flag Moduli and Applications --}}}
	
	\author{Roger Casals}
	\address{University of California Davis, Dept. of Mathematics, Shields Avenue, Davis, CA 95616, USA}
	\email{casals@math.ucdavis.edu}
	
	\author{Eric Zaslow}
	\address{Northwestern University, Department of Mathematics, Evanston, IL 60208-2730, USA}
	\email{zaslow@math.northwestern.edu}
	
\maketitle

\begin{abstract} We study a class of Legendrian surfaces in contact five-folds by encoding their wavefronts via planar combinatorial structures. We refer to these surfaces as Legendrian weaves, and to the combinatorial objects as $N$-graphs. First, we develop a diagrammatic calculus which encodes contact geometric operations on Legendrian surfaces as multi-colored planar combinatorics. Second, we present an algebro-geometric characterization for the moduli space of microlocal constructible sheaves associated to these Legendrian surfaces. Then we use these $N$-graphs and the flag moduli description of these Legendrian invariants for several new applications to contact and symplectic topology.

Applications include showing that any finite group can be realized as a subquotient of a 3-dimensional Lagrangian concordance monoid for a Legendrian surface in $(J^1\S^2,\xi_\st)$, a new construction of infinitely many exact Lagrangian fillings for Legendrian links in $(\S^3,\xi_\st)$, and performing $\F_q$-rational point counts that distinguish Legendrian surfaces in $(\R^5,\xi_\st)$. In addition, the manuscript develops the notion of Legendrian mutation, studying microlocal monodromies and their transformations. The appendix illustrates the connection between our $N$-graph calculus for Lagrangian cobordisms and Elias-Khovanov-Williamson's Soergel Calculus.\vspace{1cm}
\end{abstract}
\setcounter{tocdepth}{1}
\tableofcontents

\hspace{0.25cm} Legendrian fronts arise naturally in several areas: in topology, as Cerf diagrams of families of smooth functions; in differential equations, as Stokes diagrams of an irregular singularity; and in analysis, as wavefront sets of distributions, generalizing the original context of wavefronts in geometric optics. This article studies Legendrian surfaces through the combinatorics of their wavefronts and develops the theory of $N$-graphs, planar structures encoding front singularities. The moduli space of simple sheaves microsupported on the Legendrian surface becomes an incidence problem for flags of vector spaces, as dictated by the $N$-graph. We exploit the connections between the combinatorics of $N$-graphs, algebraic geometry and cluster algebras to obtain results in contact and symplectic topology.

\section{Introduction}\label{sec:intro}

Legendrian knots in contact 3--manifolds \cite{Etnyre05, Geiges08} are central to the study of 3--dimensional contact geometry \cite{Bennequin83, Eliashberg93,Gompf98}. The study of Legendrian knot invariants makes extensive use of their planar front projections, both in the context of Floer theory \cite{EliashbergGiventalHofer00,Chekanov02,Ng03} and microlocal analysis \cite{KashiwaraSchapira,GKS_Quantization,STZ_ConstrSheaves}. Higher-dimensional Legendrian submanifolds have proven equally instrumental in the study of higher-dimensional symplectic and contact topology, including the development of Legendrian Kirby Calculus \cite{Eliashberg90a,Gompf98,CasalsMurphyPresas} and Lagrangian skeleta \cite{RSTZ14,Nadler17,Starkston}.

In the case of 6--dimensional symplectic manifolds and their 5--dimensional contact boundaries \cite{CieliebakEliashberg12,CasalsMurphy}, spatial front projections for Legendrian surfaces are available \cite{ArnoldSing,ArnoldGivental01}. First, this article develops a multi-colored planar diagrammatic calculus for the manipulation of such Legendrian surfaces in 5--dimensional contact manifolds and their Lagrangian projections in 4-dimensional symplectic manifolds. This diagrammatic calculus is first used for the efficient computation of microlocal Legendrian isotopy invariants, as we prove and illustrate throughout the manuscript. Then we provide several new applications, including new results in higher-dimensional contact geometry and low-dimensional symplectic topology. We also expect that this concrete description will prove itself useful for further results, such as the computation of symplectic invariants of Weinstein manifolds \cite[Section 6.4]{GPS3} and homological mirror symmetry \cite{Nadler_LG1,TreumannZaslow}, see also \cite[Section 4.4]{CasalsMurphy} and Remark \ref{rmk:HMS}.

Finally, even if Legendrian weaves are a specific class of Legendrian surfaces, we can actually use them to prove new results, such as Theorems \ref{thm:symmetries_intro} and \ref{thm:ThurstonLinksIntro} below, and the diagrammatic calculus presented here is already being used successfully in a variety of recent developments \cite{CasalsHonghao,CasalsSkel,CGGS,CasalsNg,GSW,GSW2}. In studying this manuscript, we hope that the reader will find these Legendrian surfaces as useful and fascinating as we have.

%
%
%

\subsection{Summary of Contributions} Let $G$ be an $N$-graph\footnote{Informally, an $N$-graph $G\sse C$ is a collection of trivalent graphs on $C$ decorated with labels $i\in[1,N]$ such that graphs with successive labels can only intersect at hexavalent vertices, where the six radiating half-edges on the surface must interlace. See Definition \ref{def:Ngraph} for details, and note that a $2$-graph is simply an embedded trivalent graph.} drawn on a smooth surface $C$. The notion of an $N$-graph, combinatorial in nature, is first defined in Section \ref{sec:NGraphsLegWeaves}. In a nutshell, our main contributions are as follows:

\begin{itemize}
	\item[A.] {\bf Diagrammatic Calculus and Legendrian Weaves}. The construction of a Legendrian surface $\Lambda(G)$ in the five-dimensional jet space $(J^1C,\xi_\st)$ associated to the $N$-graph $G$, along with a description of Legendrian Surface Reidemeister moves in terms of combinatorial $N$-graphs moves. Likewise, we show that Legendrian surgeries and Legendrian mutations, which we introduce, can be reflected by the diagrammatics of $N$-graphs. This is part of a general calculus of multi-colored planar diagrams that, as we show, captures Legendrian surfaces and 3-dimensional Lagrangian cobordisms between them. The translation from five-dimensional contact topology to such planar diagrammatics allows us to study contact topology through combinatorics and graph theory. In fact, we use this combinatorial perspective to construct Lagrangian and Legendrian surfaces that prove new results in contact topology.\\ 
	
	\item[B.] {\bf The Microlocal Sheaf Theory of $N$-Graphs}. A Legendrian surface $\Lambda\sse(J^1C,\xi_\st)$ specifies a category of constructible sheaves on $C\times \bR$ with singular support constrained by $\Lambda$.  When $\Lambda = \Lambda(G)$ for an $N$-graph $G$, we show that the moduli stack of objects $\SM(G)$ has a combinatorial description in terms of flag varieties, which we introduce in Section \ref{sec:flag}.  This space solves an incidence moduli problem for flags of subspaces in an $N$-dimensional $k$-vector space $V$, with $k$ a field, as dictated by the $N$-graph $G$.  This stack is typically an algebraic variety and can be studied by algebraic geometric and representation-theoretic techniques.  Following \cite{GKS_Quantization,STZ_ConstrSheaves,TreumannZaslow}, this space is shown to be a Legendrian invariant for surfaces $\La(G)$ and can be used to distinguish Legendrian isotopy types. In addition, we explicitly give formulas for the microlocal monodromies along certain cycles of $H_1(\La(G),\Z)$ in terms of generalized cross-ratios of flags, and their transformation under Legendrian mutations.\\
	
	\item[C.] {\bf Applications of $N$-Graph Calculus}. First, in Section \ref{sec:app} we use the diagrammatics in (A) to study the flag moduli spaces $\SM(G)$ in (B), including their rational point counts over finite fields $\F_q$. This allows us to distinguish many Legendrian surfaces, up to Legendrian isotopy and, independently, show that for any finite group $\mathbb{G}$, there exists a Legendrian surface in $(\R^5,\xi)$ whose 3-dimensional Lagrangian concordance monoid has $\mathbb{G}$ as a subquotient. Second, Section \ref{sec:app2} explains how to apply $N$-graph calculus to systematically study Lagrangian fillings of Legendrian links in $(\S^3,\xi_\st)$. In particular, we use Legendrian mutations to give new families of Legendrian links which admit infinitely many Lagrangian fillings.\\
	
	Finally, given $N$-triangulations $(C,\tau)$ of the smooth surface $C$, we construct $N$-graphs $G(\tau)$ such that the Lagrangian projections of the Legendrian surfaces $\La(G(\tau))$ relate to the Goncharov-Kenyon conjugate surfaces \cite[Section 1.1.1]{GoncharovKenyon13} associated to an $N$-triangulation.\footnote{See also \cite[Section 2.1]{Goncharov_IdealWebs}, and \cite[Section 4.2]{STWZ} describes the conjugate surface as a Lagrangian.} In Section \ref{sec:constr} we provide the construction of $G(\tau)$. In Section \ref{sec:app3}, we provide an example of how Hitchin's non-Abelianization map is described from this viewpoint. This provides a context for the symplectic study of the cluster structures associated to moduli spaces of framed local systems of Fock-Goncharov \cite{FockGoncharovII,FockGoncharov_ModuliLocSys}, and certain classes of Gaiotto-Moore-Neitzke's spectral networks \cite{GMN_Wallcrossing,GMN_SpecNet13,GMN_Cluster,GMN_SpecNetSnakes14}. In particular, the microlocal sheaf theory of $\La(G(\tau))$ connects, through the moduli space $\SM(G(\tau))$ in (B), with their spaces of flag configurations \cite[Section 3]{Goncharov_IdealWebs}.\\

\end{itemize}

\subsection{Main Results}

We now elaborate upon these topics and state our results.

\begin{center}
	\begin{figure}[h!]
		\centering
		\includegraphics[scale=0.9]{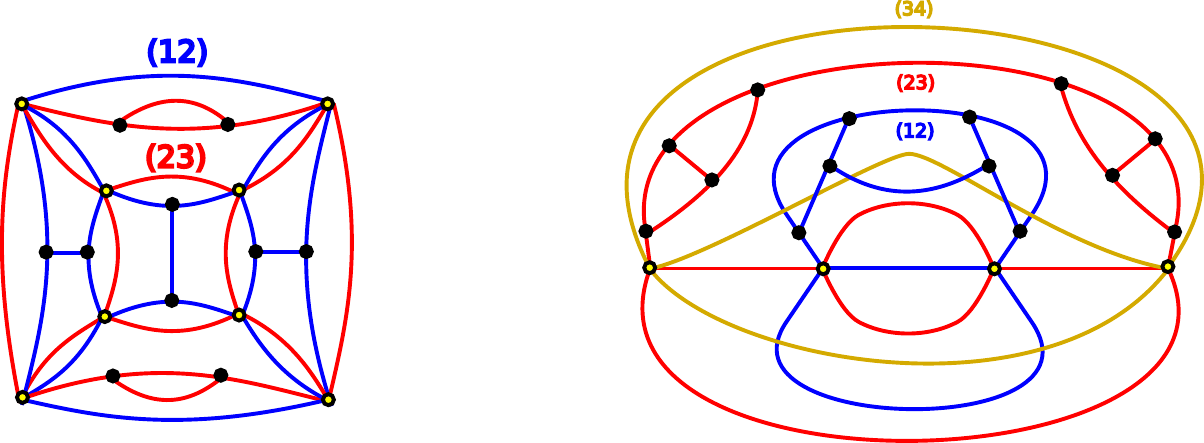}
		\caption{$3$-graph (left) and $4$-graph (right) on the 2-sphere $\S^2$. These correspond to Legendrian surfaces in the contact 5-space $(J^1\S^2,\xi_\st)$, respectively of genus 3 and 4. The geometric meaning of these two figures is explained in detail in Section \ref{sec:NGraphsLegWeaves}, in Subsections \ref{ssec:Ngraphs} and \ref{ssec:LegSing}.}
		\label{fig:NGraphExample}
	\end{figure}
\end{center}


{\bf Diagrammatic Calculus and Legendrian Weaves}.  Weinstein manifolds \cite{CieliebakEliashberg12,CieliebakEliashberg14,CasalsMurphy}, the symplectic counterpart of Stein manifolds, place Legendrian submanifolds at the forefront of higher-dimensional contact and symplectic topology.  In this manuscript, we define and study a new class of Legendrian surfaces $\La(G)$ in contact 5-manifolds, associated to an $N$-graph $G$, building on our previous works \cite{CasalsMurphy2,TreumannZaslow}.  Prior work on Legendrian surfaces \cite{Conormal2,Conormal1} has focused on the class of Legendrian tori $\La_K\sse (T^\infty\S^3,\xi_{\rm std})$ arising as the conormal torus of a smooth knot $K\sse\S^3$. The Legendrian surfaces $\La(G)$ we study provide a second infinite family of Legendrian submanifolds whose contact topology and sheaf invariants can be understood. Their geometry is governed by the combinatorial data of the $N$-graph $G$. Figure \ref{fig:NGraphExample} depicts two examples of $N$-graphs, representing Legendrian surfaces of genus 3 (left) and 4 (right).

We study three geometric operations for Legendrian surfaces in 5-dimensional contact manifolds. These are Legendrian isotopies \cite{ArnoldSing,CieliebakEliashberg12,Geiges08}, exact Lagrangian cobordisms \cite{Arnold76SurgeryI,BourgeoisSabloffTraynor15,EkholmHondaKalman16}, and Legendrian mutations, which we define in Section \ref{sec:moves}. Lagrangian cobordisms of indices 1 and 2 correspond to Legendrian 0- and 1-surgeries. We establish a correspondence between each of these three types of geometric operations and the combinatorics of $N$-graphs. In addition, we describe a combinatorial stabilization of an $N$-graph, which can be understood as a five-dimensional analogue of the Markov stabilization of a Legendrian braid \cite{Rolfsen76,PrasolovSossinsky}. Part of these results are summarized in the following two theorems (see Section \ref{sec:moves} for details), which are developed in the text:

\begin{thm}[Diagrammatics for Legendrian Weave Calculus I]\label{thm:DiagrammaticsI}
	Let $G$ be a local $N$-graph. The combinatorial moves in Figures \ref{fig:IntroReidemeister} and \ref{fig:IntroCuspMoves} are Legendrian isotopies for $\La(G)$.\hfill$\Box$
\end{thm}

\begin{center}
	\begin{figure}[h!]
		\centering
		\includegraphics[scale=0.8]{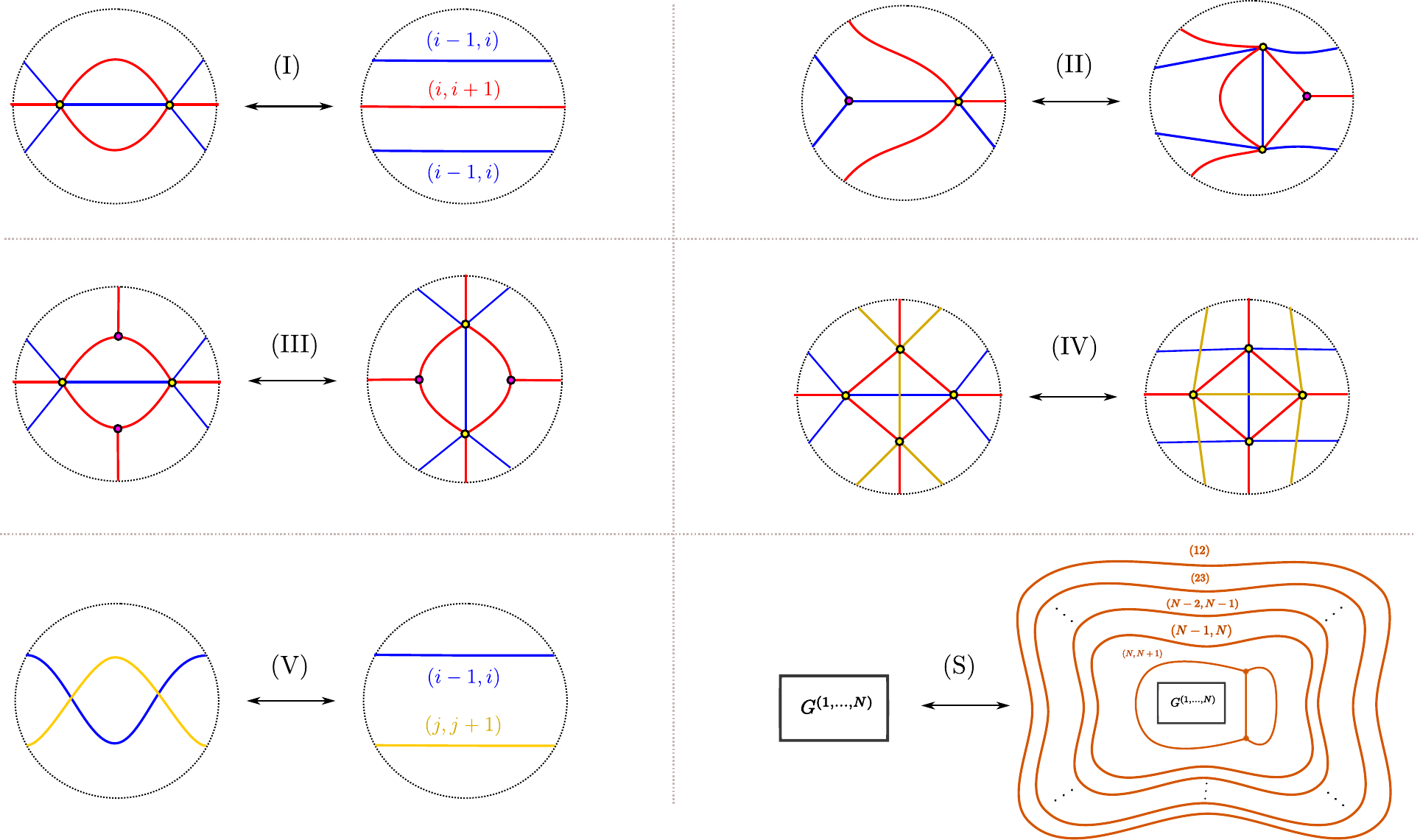}
		\caption{Combinatorial Moves for Legendrian Isotopies of Surfaces $\La(G)$. Moves I--V are local Legendrian isotopies in the 1-jet space $(J^1\R^2,\xi_\st)$. Move S in the lower right is local in  $(J^1\S^2,\xi_\st)$ after satelliting to the Legendrian unknot $\La_0\sse(\R^5,\xi_\st)$. See Section \ref{sec:NGraphsLegWeaves} for precise details on the geometric concepts represented by these pictures.}
		\label{fig:IntroReidemeister}
	\end{figure}
\end{center}

\begin{center}
	\begin{figure}[h!]
		\centering
		\includegraphics[scale=0.7]{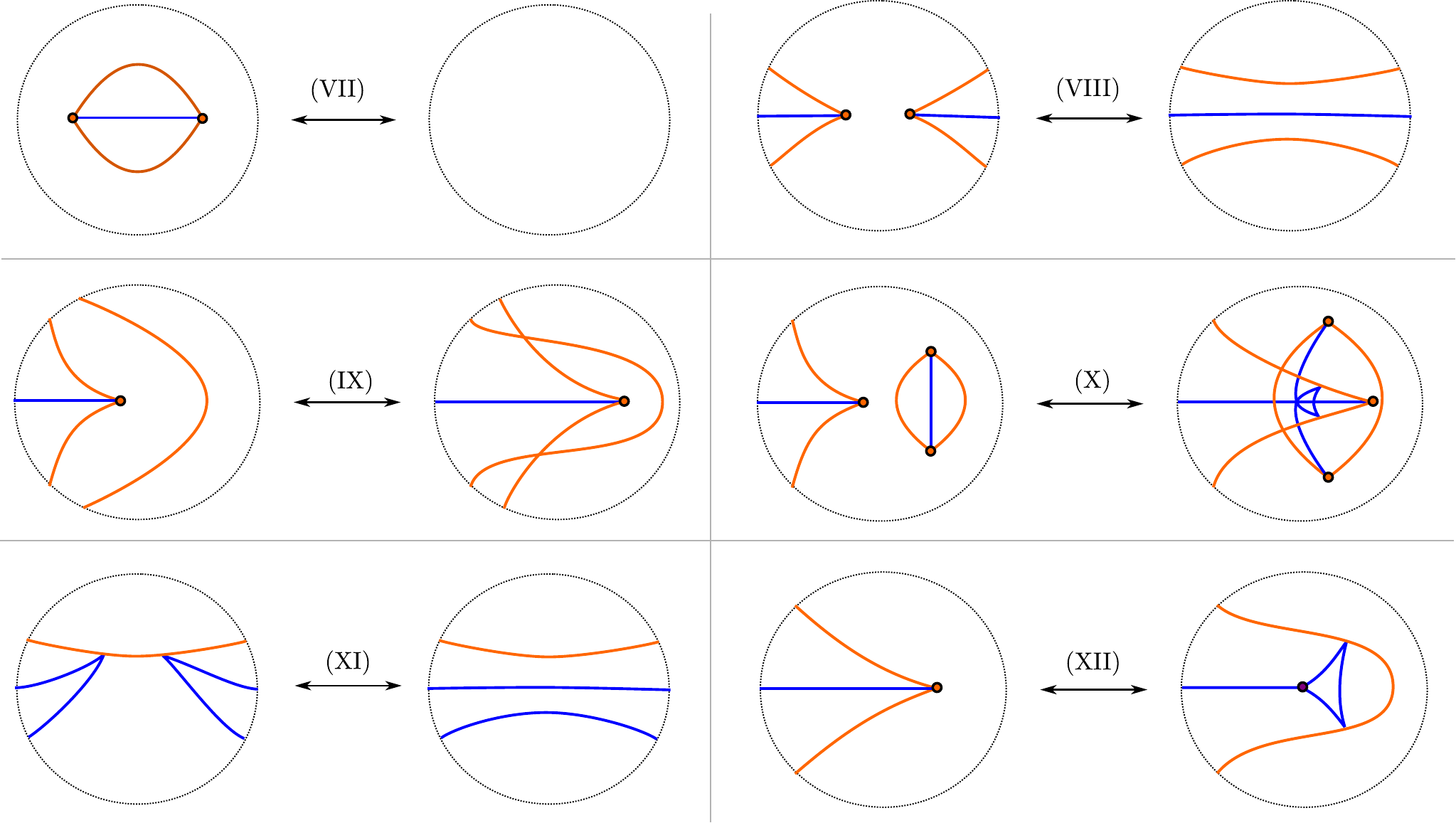}
		\caption{Combinatorial Moves for Legendrian Isotopies of Surfaces $\La(G)$. These are homotopies of spatial wavefronts involving $A_3$-swallowtail singularities. Section \ref{sec:NGraphsLegWeaves} explains the geometric meaning of these pictures.}
		\label{fig:IntroCuspMoves}
	\end{figure}
\end{center}

\begin{thm}[Diagrammatics for Legendrian Weave Calculus II]\label{thm:DiagrammaticsII}
	Let $G$ be a local $N$-graph. The combinatorial moves in Figure \ref{fig:IntroSurgeries} are Legendrian surgeries, of indices 0, 1 and 2, Legendrian mutations and connected sums with the standard and Clifford tori.\hfill$\Box$
\end{thm}

\begin{center}
	\begin{figure}[h!]
		\centering
		\includegraphics[scale=0.8]{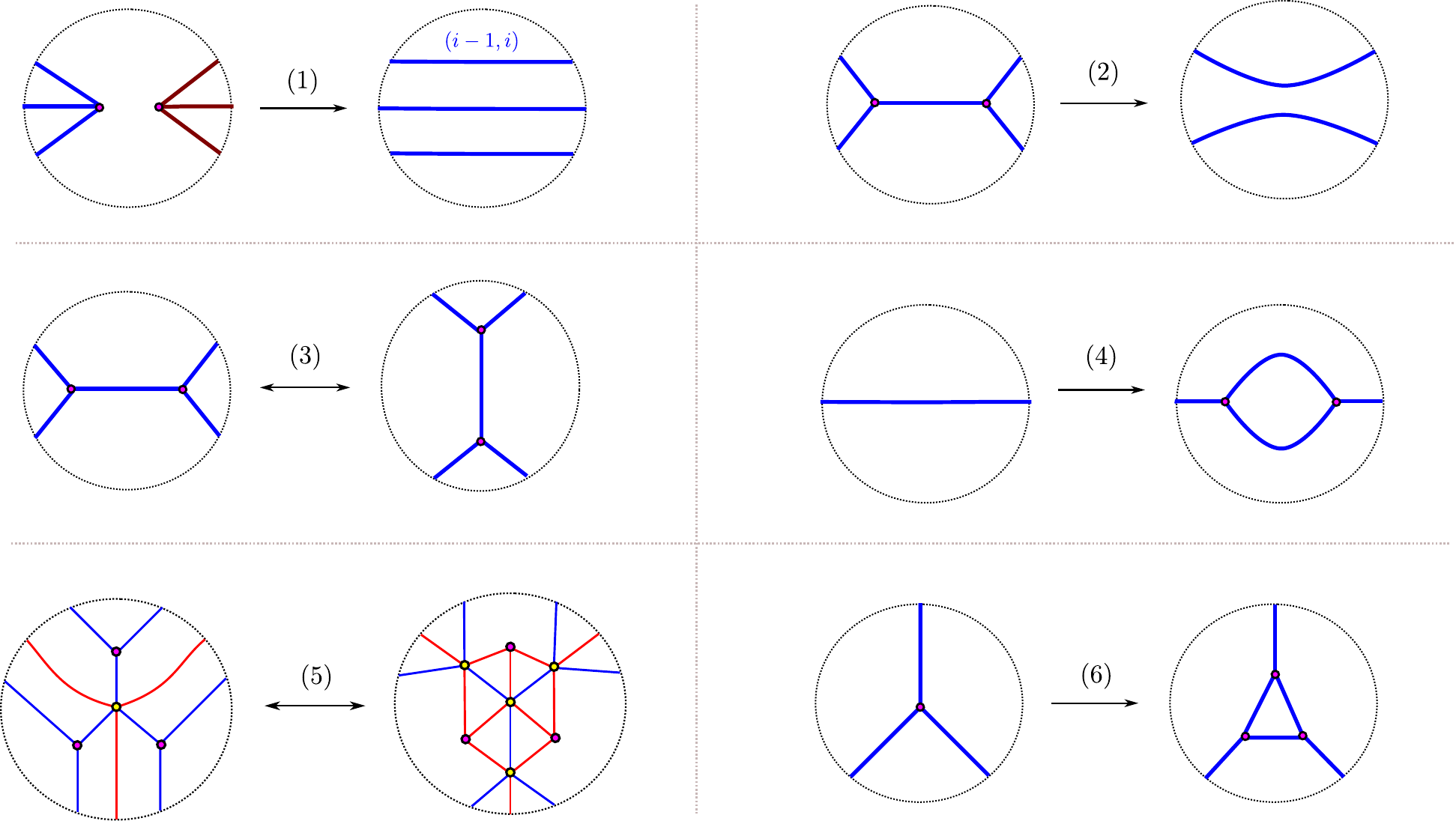}
		\caption{Table of Combinatorial Moves for Surfaces $\La(G)$ corresponding to Legendrian Surgeries, mutations and tori connected sums. See Section \ref{sec:NGraphsLegWeaves} for details.}
		\label{fig:IntroSurgeries}
	\end{figure}
\end{center}

Theorems \ref{thm:DiagrammaticsI} and \ref{thm:DiagrammaticsII} provide an efficient diagrammatic calculus to manipulate the Legendrian surfaces $\La(G)$ associated to $N$-graphs $G$. We refer to the Legendrian surfaces $\La(G)$ as {\it Legendrian weaves}, due to the resemble of their Legendrian fronts to a weaving pattern -- see Definition \ref{def:LegWeave2}. Theorems \ref{thm:DiagrammaticsI} and \ref{thm:DiagrammaticsII} are geometric in nature and are proven by manipulating Legendrian fronts for Legendrian surfaces in five dimensions. This is the content of Section \ref{sec:moves}, as part of our study of generic three-dimensional front singularities and their homotopies. In addition, Section \ref{sec:constr} provides several combinatorial constructions of Legendrian surfaces $\La\sse(\S^5,\xi_\st)$ which are used in our applications in Sections \ref{sec:app}, \ref{sec:app2} and \ref{sec:app3}.

\begin{remark}
The Legendrian weaves $\La(G)\sse(J^1C,\xi_\st)$ associated to an $N$-graph $G\sse C$ admit spatial wavefronts $\pi(\La(G))\sse C\times\R$ with front singularities solely of types\footnote{The $A_1^2$-singularity corresponds to a crossing, and the $A_1^3$-singularity is given by three planes intersecting transversely at a point. The $A_2$-singularity corresponds to a simple cusp, $A_3$-singularities are swallowtails, and $A_2A_1$-singularities are obtained by intersecting a cusp with a linear space.} $A_1^2$, $A_1^3$ and $D_4^-$, following V.I.~Arnol'd's notation \cite{Arnold76SurgeryI,ArnoldSing}. That said, their satellites $\iota(\La(G))\sse(Y,\xi_\st)$ typically acquire $A_2,A_2A_1$ and $A_3$ singularities. Satellites will be introduced and discussed in Section \ref{ssec:Satellite}; as an example, the satellite of $\La(G)$ along the standard Legendrian unknot $\La_0\sse(\R^5,\xi_\st)$ necessarily develops $A_2$-singularities. In addition, the standard 5-dimensional Legendrian Reidemeister surface moves include the creation of $A_3$ singularities, and the interaction of $A_2$ and $A_3$ singularities yield a $D_4^+$ singularity. These Legendrian singularities and 3-dimensional Reidemeister moves will also be discussed in Section \ref{sec:moves}.\hfill$\Box$
\end{remark}


{\bf The Microlocal Sheaf Theory of $N$-Graphs}. The relationship between sheaf theory and contact and symplectic geometry \cite{Sheaves3,Nadler_MicrolocalBrane,GKS_Quantization,Sheaves2} provides invariants of Lagrangian and Legendrian submanifolds up to Hamiltonian and contact isotopies \cite{STZ_ConstrSheaves,STWZ,CasalsHonghao}. These invariants are an alternative to the more analytical Floer-theoretic methods \cite{EkholmEtnyreSullivan05a,EkholmEtnyreNgSullivan13a,EkholmEtnyreNgSullivan13}, and have recently been shown to contain equivalent data \cite{GPS3,GPS1,GPS2}.  

Let $G$ be an $N$-graph on $C$, $\Lambda(G) \sse J^1(C)$ its Legendrian surface, and $\SC(G)$ the category of simple constructible sheaves on $C\times \bR$ microlocally supported along $\Lambda(G).$  In Section \ref{sec:flag}, we describe the moduli space of objects in $\SC(G)$ in terms of the combinatorics of $G$. Specifically, we define the {\it flag moduli space} $\SM(G)$ of an $N$-graph $G\sse C$, an algebraic stack -- often a variety -- as being described by explicit relations among elements in the flag variety $\GL(N,k)/B$, where $B$ is the Borel subgroup of upper triangular matrices. Already when $N=2,$ the number of rational $\bF_q$-points of $\SM(G)$, for a finite field $\bF_q$ is, up to a factor, the chromatic polynomial of the dual graph evaluated at $q+1 = |(GL(2,\bF_q)/B)(\bF_q)|$ \cite{TreumannZaslow}, and hence the moduli stack $\SM(G)$ geometrizes a familiar graph-theoretic construction.

For general $N$, this algebraic space $\SM(G)$ is the moduli space of an incidence problem between flags and their stabilizing monodromies. It has two particular virtues. First, $\SM(G)$ changes explicitly under certain combinatorial moves of the $N$-graph $G$ --- thus, each time we can simplify $G$ with our moves from Theorems \ref{thm:DiagrammaticsI} and \ref{thm:DiagrammaticsII}, we get closer to solving the moduli problem via purely diagrammatic techniques. Second, $\SM(G)$ is an invariant of the Legendrian isotopy class of $\Lambda(G) \sse J^1(C)$. In short, $\SM(G)$ is defined purely in terms of the combinatorics of the $N$-graph $G$, in a manner we understand, and we show it geometrically describes the following invariant:

\begin{thm}\label{thm:intro2}
	Let $C$ be a closed, smooth surface and $G\sse C$ an $N$-graph. The flag moduli space $\SM(G)$ is isomorphic to the moduli space of microlocal rank-one sheaves\footnote{Microlocal rank-one sheaves are also called microlocally simple or just simple \cite[Chapter 7]{KashiwaraSchapira}.} on $C\times\R$ microlocally supported along $\La(G)\sse(J^1C,\xi_\st)$.\hfill$\Box$
\end{thm}

After the work of Guillermou-Kashiwara-Schapira \cite{GKS_Quantization}, which constructs an equivalence of sheaf categories from a Legendrian isotopy, we conclude that
the algebraic isomorphism type of the moduli stack $\SM(G)$ is a Legendrian isotopy invariant of the Legendrian surface $\La(G)\sse (J^1C,\xi_\st)$.  In fact, it will remain a Legendrian isotopy invariant for certain satellites along $C\sse (\bR^5,\xi_\st)$, yielding a Legendrian invariant for $\La(G)\sse (\R^5,\xi_\st)$. Theorem \ref{thm:intro2}, proven in Section \ref{sec:flag}, is a generalization to $N\geq 2$ of \cite[Section 4]{TreumannZaslow} and the 2-dimensional surface analogue of the results in \cite{STZ_ConstrSheaves,STWZ}, where the computation of the moduli space of microlocal rank-1 sheaves for 1-dimensional Legendrian braid closures in $\R^3$ is expressed in algebraic combinatorial terms.\\

{\bf Applications of $N$-Graph Calculus}. Sections \ref{sec:app}, \ref{sec:app2} and \ref{sec:app3} exhibit a gallery of computations and uses of the flag moduli space $\SM(G)$, including the study of $\SM(G)$ as a complex variety and its finite $\F_q$-counts. For instance, our techniques readily prove the following sample result:

\begin{thm}[Flag Moduli for Ladder Graphs]\label{thm:intro1} Let $\SL_n\sse\S^2$ be the $(2n)$-runged ladder 3-graph of Figure \ref{fig:NLadderIntro}, and let $\F_q$ a finite field.  Then the flag moduli space $\SM(\SL_n)$ has orbifold point count
	
	$$|\SM(\SL_n)(\F_q)|=\frac{q^{2n-3}-q^{n-2}+q^{n-1}+q-1}{(q-1)^2}$$
	
In particular, the Legendrian 3-links of 2-spheres $\La(\SL_n)$ and $\La(\SL_m)$ are Legendrian isotopic if and only if $n=m$.\hfill$\Box$
\end{thm}

The infinitely many Legendrian surfaces $\La(\SL_n)$ in Theorem \ref{thm:intro1}, $n\in\N$, are pairwise smoothly isotopic. The distinct finite $\F_q$-counts of their flag moduli space $\SM(\La(\SL_n))(\F_q)$ give a direct proof that they are not Legendrian isotopic as Legendrian surfaces in $(\R^5,\xi_\st)$. Also, adding the ladder $3$-graphs in Theorem \ref{thm:intro1} into a face of an arbitrary $N$-graph $G$ typically changes the flag moduli space of $\SM(G)$ and thus produces another Legendrian surface, smoothly isotopic but {\it not} Legendrian isotopic to $\La(G)$.

In general, the computation of these Legendrian invariants translates into an incidence moduli problem, which can itself be simplified with our diagrammatic techniques, and then possibly solved with methods from algebraic geometry. In particular, we will understand the effect of combinatorial moves for $N$-graphs $G$ on the Legendrian invariants $\SM(\La(G))$. This will frequently allow for the computation
 of this moduli stack and distinguish Legendrian weaves up to Legendrian isotopy. This yields a wide range of results in the vein of Theorem \ref{thm:intro1}, as we will illustrate.  From this perspective, Legendrian weaves, which are in general surfaces of any genus, constitute an attractive complement to the family of knot conormals. \\

\begin{center}
	\begin{figure}[h!]
		\centering
		\includegraphics[scale=0.65]{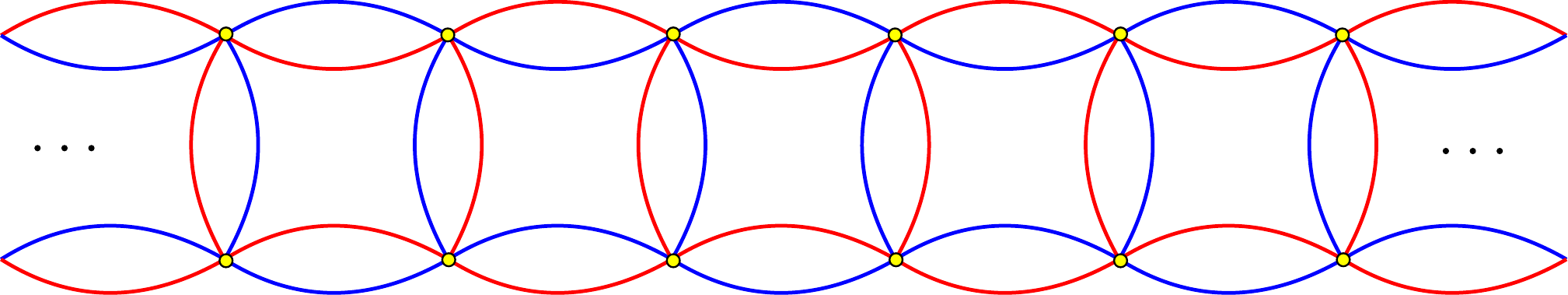}
		\caption{The bipartite Ladder 3-Graph $\SL_n$, where the right and left sides are identified after $2n$ rungs.}
		\label{fig:NLadderIntro}
	\end{figure}
\end{center}

We now illustrate a second application of our flag moduli stacks, detailed in Section \ref{sec:app}. Let $\La \subset (\S^5,\xi_\st)$ be an embedded Legendrian surface and let $\cL(\La)$ be the space of embedded Legendrian surfaces which are Legendrian isotopic $\La$, with base point $\La$.  Let $\L(\La)$ be the monoid of 3-dimensional exact Lagrangian concordances in the symplectization $(\S^5\times\R(t),e^t\la_\st)$, up to Hamiltonian isotopy, based at $\La$. 
The flag moduli spaces $\SM(G)$ will be used to show the following result:

\begin{thm}\label{thm:symmetries_intro} Let $\bG$ be an arbitrary finite group. Then there exists a Legendrian surface $\La_{\bG}\sse(\S^5,\xi_\st)$ such that
	\begin{itemize}
		\item[(i)] $\bG$ is a subquotient of the fundamental group $\pi_1(\cL(\La_\bG))$,
		\item[(ii)] $\bG$ is a subquotient of the 3-dimensional Lagrangian concordance monoid $\L(\La_\bG)$.
	\end{itemize}
	In fact, the latter is the image of the former via the graph map $\mbox{gr}:\pi_1(\cL(\La))\lr\L(\La)$.
\end{thm}

Theorem \ref{thm:symmetries_intro} essentially states that the study of the 3-dimensional Lagrangian concordance monoid can be as complicated as any finite group. The proof of Theorem \ref{thm:symmetries_intro} will exhibit the advantage of using combinatorial constructions on an $N$-graph $G$ to extract contact and symplectic information in $5$- and $6$-dimensions. Note that for 1-dimensional max-tb Legendrian torus links, T. K\'alm\'an provided finite cyclic subgroups of the 2-dimensional Lagrangian concordance monoid \cite{Kalman}, and J. Sabloff and M. Sullivan provided\footnote{The results of \cite{SabloffSullivan16} are stronger in higher-dimensions, but for Legendrian surfaces the only finite subgroups of the special orthogonal group $SO(2)$ must be cyclic -- see \cite[Remark 4.7]{SabloffSullivan16}.} finite cyclic subgroups of the 3-dimensional Lagrangian monoid for certain Legendrian surfaces \cite{SabloffSullivan16}. Sections \ref{sec:flag}, \ref{sec:app} and \ref{sec:app2} contain several computations and applications of the flag moduli spaces $\SM(G)$.

\begin{remark}
The Legendrian DGA of a Legendrian knot in $(\R^3,\xi_\st)$ can be computed algorithmically. The computation of Floer-theoretic invariants of general Legendrian submanifolds in arbitrary higher-dimensions represents a challenge \cite{Rizell_TwistedSurgery,EkholmEtnyreSullivan05b,EkholmEtnyreSullivan05a} --- see \cite{RutherfordSullivan1,RutherfordSullivan2} for progress in this direction. The class of Legendrian 2-tori arising as knot conormals is understood \cite{Ng10,Conormal2,Conormal1} and our results, in line with Theorem \ref{thm:intro1} and Theorem \ref{thm:symmetries_intro}, aim at achieving both a geometric and sheaf-theoretic understanding for the class of Legendrian weaves $\La(G)$.\hfill$\Box$
\end{remark}

For a third class of applications, consider an $N$-graph $G\sse\D^2$ with boundary. The Lagrangian projection of the Legendrian weave $\La(G)$ yields\footnote{A combinatorial criterion for embeddedness, which will be useful, is described in Lemma \ref{lem:free}.} an exact Lagrangian filling of a Legendrian link in $(\S^3,\xi_\st)$, associated to $\dd G$. In Section \ref{sec:app2} we will construct different $N$-graphs $G_1,G_2$ with $\dd G_1=\dd G_2$, and explain how microlocal monodromies can be used to show that the Lagrangian projections of the Legendrian weaves $\La(G_1)$ and $\La(G_2)$ are {\it not} Hamiltonian isotopic relative to their 1-dimensional Legendrian boundaries. In fact, $N$-graph calculus, in combination with Legendrian mutations, allows us to construct {\it infinitely} many distinct embedded Lagrangian fillings for certain Legendrian knots. The following family of Legendrian links is studied in detail in Subsection \ref{ssec:QuiverMutationsGeometric}:

\begin{thm}\label{thm:ThurstonLinksIntro} Let $\La_{s,t}=\La(\beta_{s,t})\sse(\S^3,\xi_\st)$ be the Legendrian link given by the standard satellite of the positive braid
	$$\beta_{s,t}=(\sigma_1^3\sigma_2)(\sigma_1^3\sigma_2^2)^s\sigma_1^3\sigma_2(\sigma_2^2\sigma_1^3)^t(\sigma_2\sigma_1^3)(\sigma_2^{t+1}\sigma_1^2\sigma_2^{s+2}),\qquad s,t\in\N,s,t\geq1.$$
	Then $\La_{s,t}\sse(\S^3,\xi_\st)$ admits infinitely many embedded exact Lagrangian fillings in $(\D^4,\la_\st)$ realized as $3$-graphs $G_{s,t}\sse\D^2$ and their Legendrian mutations.
\end{thm}

The 3-graphs representing the infinitely many Lagrangian fillings in Theorem \ref{thm:ThurstonLinksIntro} are diagrammatically interesting, with their complexity increasing as we geometrically realize the iterates in an infinite sequence of quiver mutations. For instance, Figure \ref{fig:ThurstonQuiverIntro} depicts an example of a Lagrangian filling associated to such a 3-graph, obtained after five mutations. Fortunately, the local mutations rules that we develop in Section \ref{ssec:legmutation_local} will allow us to control certain infinite sequences of $N$-graphs mutations and construct infinite sequences of pairwise distinct Lagrangian fillings.

\begin{center}
\begin{figure}[h!]
		\centering
		\includegraphics[scale=0.85]{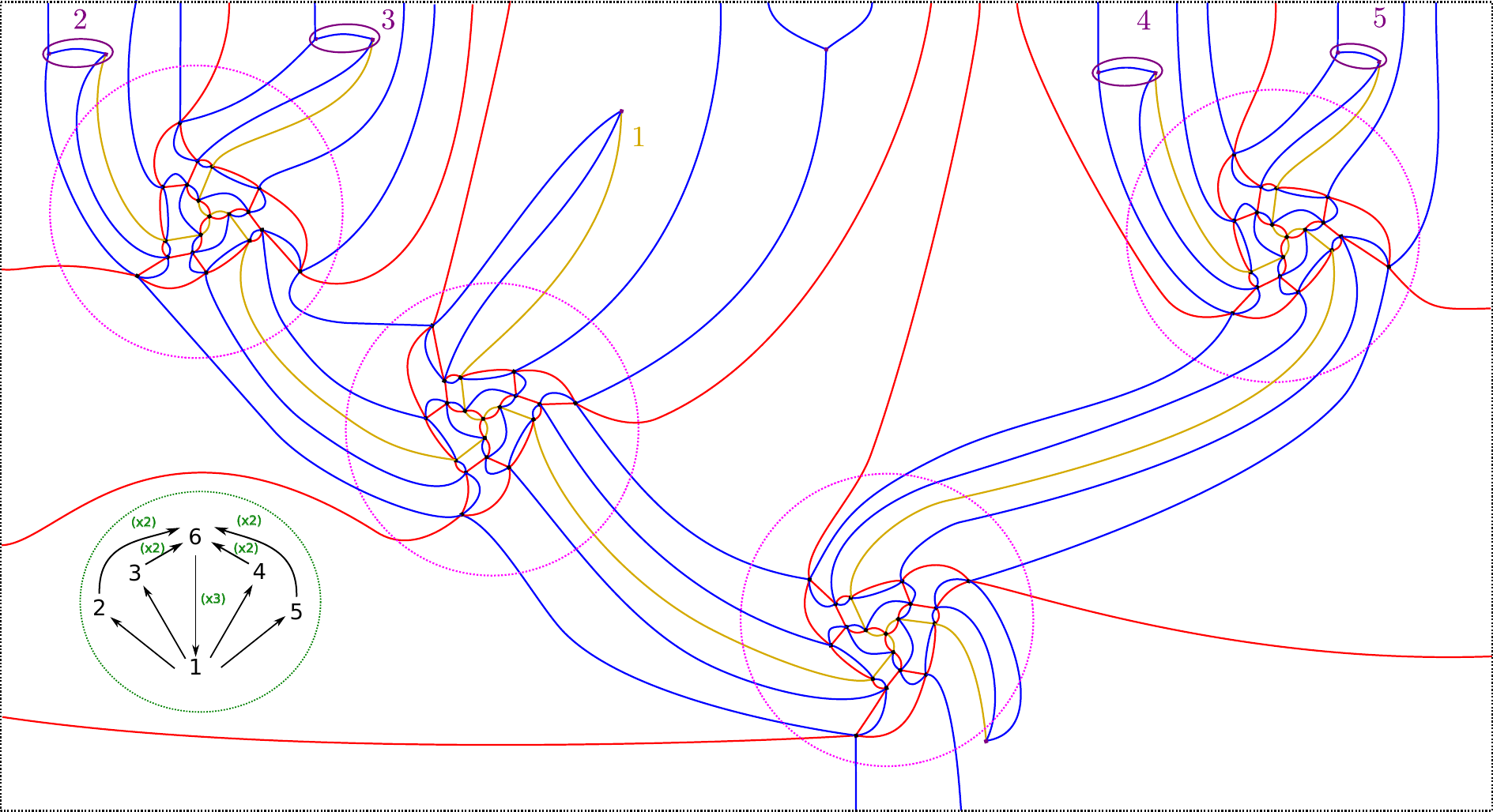}
		\caption{The 3-graph for one of the infinitely many Lagrangian fillings of the Legendrian link $\La_{1,1}\sse(\S^3,\xi_\st)$, as featured in Theorem \ref{thm:ThurstonLinksIntro}. Iterative 3-graph mutations will yield new 3-graphs $G\sse\D^2$ representing pairwise non-Hamiltonian isotopic Lagrangian fillings of $\La_{1,1}$.}
		\label{fig:ThurstonQuiverIntro}
	\end{figure}
\end{center}
	
Theorem \ref{thm:ThurstonLinksIntro} is an appropriate complement to the recent results \cite{CasalsHonghao}, as the construction of the infinitely many Lagrangian fillings in Theorem \ref{thm:ThurstonLinksIntro} is obtained directly by Legendrian mutations.\footnote{In contrast, the construction for torus links given by the first author in \cite{CasalsHonghao} uses Lagrangian concordances of infinite order. In that context, see also the upcoming work \cite{GSW} which will show that the square of the Donaldson-Thomas transformation \cite{GoncharovLinhui_DT} is a Lagrangian concordance, oftentimes of infinite order.} In more generality, Section \ref{sec:app2} develops the relation between the cluster algebra associated to the intersection quiver of a Lagrangian filling and the Legendrian mutations from Section \ref{ssec:legmutation}. In particular, $N$-graph calculus can serve as an effective tool to show that a given Legendrian link admits infinitely many Lagrangian fillings, in case the quiver is of infinite mutation type\footnote{This is generically the case.} and its vertices are represented by mutable 1-cycles in the $N$-graph $G$. In fact, any Legendrian link $\La(\beta)\sse(\S^3,\xi_\st)$ associated to a positive braid $\beta\in\mbox{Br}^+_N$ admits a Lagrangian filling -- oftentimes many -- given by an $N$-graph $G\sse\D^2$.\\


A final application of $N$-graph calculus for Legendrian weaves develops the connection of symplectic topology to V. Fock and A. Goncharov's cluster varieties of framed local systems \cite{FockGoncharov_ModuliLocSys} (see also \cite{Goncharov_IdealWebs,STWZ}), and should relate to the spectral networks of Gaiotto-Moore-Neitzke \cite{GMN_Wallcrossing,GMN_SpecNet13,GMN_SpecNetSnakes14}. For that, consider $N\in\N$ and $\tau$ an ideal $N$-triangulation of the smooth punctured surface $C$. In Section \ref{sec:constr}, we present a new construction that associates an $N$-graph $G(\tau)$ to an ideal $N$-triangulation $(C,\tau)$. In particular, each ideal $N$-triangulation $\tau$ yields a Legendrian surface $\La(G(\tau))\sse(J^1C,\xi_\st)$. In general, different $N$-triangulations lead to smoothly isotopic Legendrian surfaces which are {\it not} Legendrian isotopic, and they are distinguished by their flag moduli space $\SM(G(\tau))$. This also relies on the connection between microlocal monodromies and cluster algebras.
\begin{center}
	\begin{figure}[h!]
		\centering
		\includegraphics[scale=0.8]{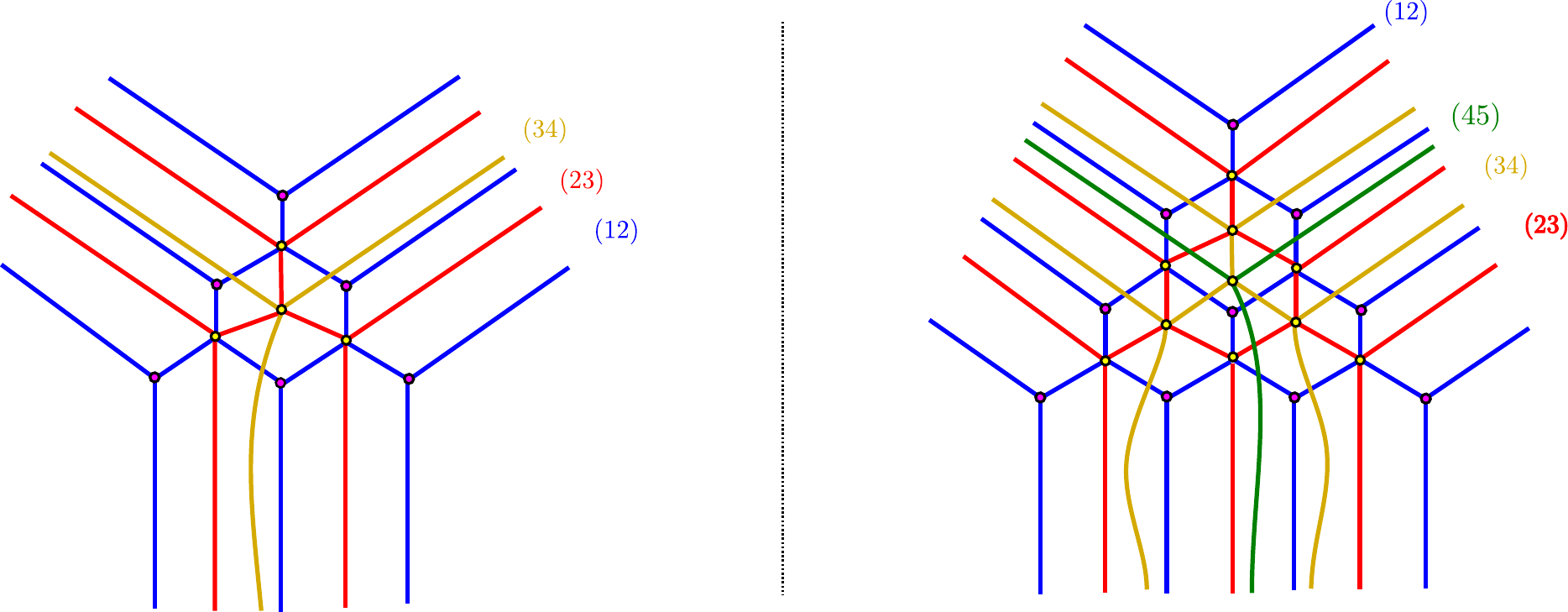}
		\caption{The Legendrian weave associated to a 4-triangle (left) and to a 5-triangle (right). The open Legendrian surface for the 4-triangle has genus one and two boundary components. The Legendrian surface for the 5-triangle has genus two and three boundary components.\hfill$\Box$}
		\label{fig:IntroNTriang}
	\end{figure}
\end{center}
The $N$-graph $G(\tau)$ and the Legendrian weave $\La(G(\tau))$ are both constructed with a local model on an $N$-triangle. Figure \ref{fig:IntroNTriang} depicts a Legendrian weave associated to the 4- and 5-graphs dual to 4- and 5-triangles. We will prove that their local flag moduli space is a complex torus by using Theorem \ref{thm:DiagrammaticsI} and the flag moduli space results from Section \ref{sec:flag}. The precise statement, proven in Section \ref{sec:app3}, reads as follows:

\begin{thm}\label{thm:FlagModuli_Ntriangle_Intro}
	Let $G(t_N)$ be the $N$-graph associated to an $N$-triangle $t_N$, and let $k$ a field. The flag moduli space of $G(t_N)$ is a ${N-1\choose 2}$-dimensional complex torus, i.e.
	$$\SM(t_N,G(t_N);k)\cong(k^*)^{{N-1\choose 2}}.$$
\end{thm}

The combinatorial number ${N-1\choose 2}$ appears geometrically as the rank of the first homology class of the Legendrian weave $\La(G(t_N))$. Now, the class of Legendrian weaves $\La(G(\tau))$ arising from ideal $N$-triangulations $\tau$ of punctured surfaces is of central interest in the study of moduli spaces of framed local systems for the Lie group $\GL(N,\C)$ \cite{FockGoncharov_ModuliLocSys}. Indeed, the Legendrian surface $\La(G(\tau))$ is a compactification of the Legendrian lift of the Goncharov-Kenyon Lagrangian conjugate surface $L_\tau\sse (T^*C,\la_\st)$, see \cite{Goncharov_IdealWebs,STWZ}. Thus, the non-Abelianization technique, expressing higher-rank local systems in $S$ in terms of rank-one local systems on $L_\tau$, can also be recovered by studying these Legendrian weaves $\La(\tau)$ -- see Section \ref{ssec:Nonabelianization} for an explicit computation. In particular, the set of Legendrian surfaces $\{\La(\tau)\}_\tau$ provides a symplectic geometric realization of the set of cluster charts in this moduli spaces of framed local systems. This parallels the work of \cite{STWZ} on conjugate surfaces. See Section \ref{sec:app3} for details.\\

{\bf Basic Notation and Color Code.} The germs of singularities of caustics and wavefronts are referred to according to the classical notation from the theory of singularities, following V.I. Arnol'd \cite{ArnoldSing}. Given a subset $X\sse Y$ of a smooth manifold $Y$, we denote by $\Op(X)$ an arbitrarily small but fixed open neighborhood of it, following M. Gromov \cite{Gromov86}.

Regarding colors, the two colors \color{blue} blue \color{black} and \color{red} red \color{black} are associated to edges with adjacent transpositions, i.e. edges with consecutive transpositions $(i-1,i),(i,i+1)$, for a choice $2\leq i\leq N-1$. The same holds for colors \color{red} red \color{black} and \textcolor{Dandelion}{yellow} \color{black} used together. The {\it three} colors \color{blue} blue\color{black}, \color{red} red \color{black} and \textcolor{Dandelion}{yellow} \color{black} together denote edges labeled by three consecutive transpositions $(i-1,i),(i,i+1)$ and $(i+1,i+2)$, respectively, for a choice $2\leq i\leq N-2$. In a diagram with the two colors \color{blue} blue \color{black} and \textcolor{Dandelion}{yellow}, without \color{red}red\color{black}, these two colors denote any edges with disjoint transpositions. The color \textcolor{orange}{orange} will exclusively be used to denote cusp edges, corresponding to edges of $A_2$-singularities. Finally, we use \textcolor{purple}{purple} dots (or black dots) for $D_4^-$ singularities, \textcolor{Dandelion}{yellow} dots for $A_1^3$ singularities and \textcolor{orange}{orange} dots for $A_3$-swallowtail singularities.\hfill$\Box$\\

{\bf Acknowledgements.} We thank Honghao Gao and Kevin Sackel for their thorough reading of the initial version of this manuscript, and Honghao Gao, Eugene Gorsky and Harold Williams for many valuable comments. We also thank the referees for their suggestions and comments. We are grateful to Dylan Thurston for providing key examples of quiver mutations, and to Ben Elias for discussions on Soergel calculus. We also thank J. Etnyre, O. Lazarev, I. Le, L. Ng, J. Sabloff, L. Traynor and D. Treumann for discussions, questions and interest in this work. R.~Casals is supported by the NSF grant DMS-1841913, a BBVA Research Fellowship and the Alfred P. Sloan Foundation. E.~Zaslow is supported by the NSF grant DMS-1708503.\hfill$\Box$\\



\section{$N$-graphs and Legendrian Weaves}\label{sec:NGraphsLegWeaves}

In this section we introduce the notion of an $N$-graph $G$ and construct the Legendrian surface $\La(G)$ associated to it. The interaction between the combinatorics of $G$ and the contact geometric invariants of $\La(G)$ is the starting focus of this article. The reader is referred to \cite{Graph1,Graph2} for introductory material on graph theory and to \cite{Etnyre05,Geiges08} for the basics of contact topology.


\subsection{$N$-graphs}\label{ssec:Ngraphs}
Let $C$ be a smooth surface and $N\in\N$ a natural number. An embedded graph $G\sse C$ is said to be trivalent if all its vertices have degree three.  Such a vertex is depicted on the left in Figure \ref{fig:VertexType}.

\begin{center}
	\begin{figure}[h!]
		\centering
		\includegraphics[scale=0.6]{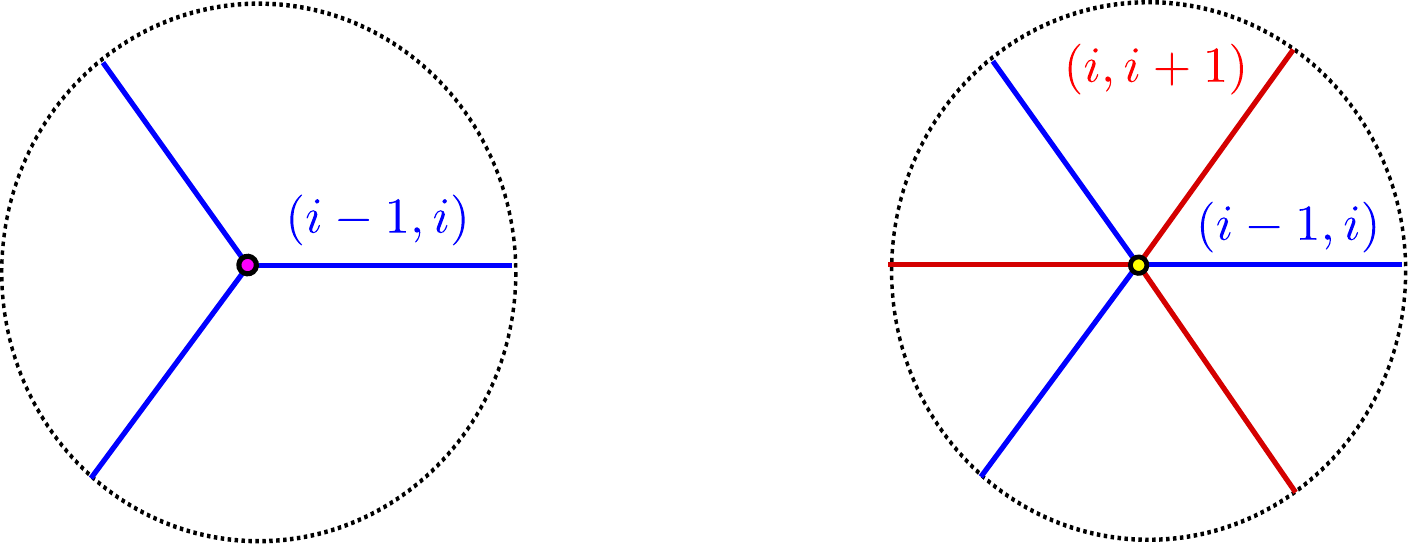}
		\caption{Trivalent vertex (left) and Hexagonal Point (right).}
		\label{fig:VertexType}
	\end{figure}
\end{center}

\begin{definition}
Let $J$ and $K$ be two trivalent graphs embedded in $C$, having an isolated intersection point at a common vertex $v\in J\cap K.$  The intersection $v$ is said to be {\it hexagonal} if the six half-edges in $C$ incident to $v$ interlace, i.e. alternately belong to $J$ and $K$.\hfill$\Box$
\end{definition}

The right diagram in Figure \ref{fig:VertexType} depicts a hexagonal vertex, where the graph $J$ is labeled $(i-1,i)$ in blue and $K$ is labeled $(i,i+1)$ in red. These hexagonal intersection points will be referred to as {\it hexagonal $(i,i+1)$-points}.

\begin{definition}\label{def:Ngraph}
An $N$-graph $G$ on a smooth surface $C$ is a set $G=\{G_i\}_{1\leq i\leq N-1}$ of $N-1$ embedded trivalent graphs $G_i\sse C$, possibly empty or disconnected, such that $G_i$ is allowed to intersect $G_{i+1}$ only at hexagonal points, $1\leq i\leq N-2.$\hfill$\Box$
\end{definition}

Two examples of $N$-graphs on the plane $C=\R^2$ are depicted in Figure \ref{fig:NGraphExample}. The (trivalent) vertices are depicted by {\it purple} or {\it black} dots and the hexagonal intersection points by {\it yellow} dots. Note that $G_i,G_j\sse C$ are allowed to intersect (anywhere) if $j\neq i,i\pm 1$, and they may intersect non-transversely.

\begin{remark}
	\label{rmk:bicolored}
	We can think of an $N$-graph as an immersed graph with colored edges, the color $i$ corresponding to the graph $G_i$, $1\leq i\leq N-1$.  Edges labeled by numbers differing by two or more may pass through one another (hence the immersed property, which is met generically), but not at a vertex. In particular, a 3-graph is a bicolored graph with monochromatic trivalent vertices and interlacing hexagonal vertices.\hfill$\Box$
\end{remark}

Consider $\tau(N):=\{(i,i+1)\in S_N:1\leq i\leq N-1\}\sse S_N$ the subset of simple transpositions and denote $\tau_i:=(i,i+1)$. We label the edges of an $N$-graph $G=\{G_i\}$ which belong to the graph $G_i$ with the transposition $\tau_i$, as we have done in Figure \ref{fig:NGraphExample}. These edges will also be referred to as $\tau_i$-edges, or $i$-edges. By definition, the trivalent vertices belonging to the graph $G_i$ have three incident $\tau_i$-edges. The hexagonal points in $G_i\cap G_{i+1}$ have six edges incident to it, alternately labeled with the transpositions $\tau_i$ and $\tau_{i+1}$ in $\tau(N)$. Figure \ref{fig:VertexType} depicts the local model for the trivalent vertices of the cubic graph $G_{i-1}$ and a hexagonal intersection point in $G_i\cap G_{i-1}$. Observe that a $2$-graph is, by definition, an embedded trivalent graph.

The study of $N$-graphs brings the combinatorial ingredients of the article, and we provide in Section \ref{sec:constr} several combinatorial constructions of $N$-graphs. For now, we introduce its geometric counterpart, the Legendrian surface associated to an $N$-graph.


\subsection{Singularities of wavefronts}\label{ssec:SingWave}

The Legendrian surface $\La(G)$ associated to an $N$-graph $G\sse C$ is an embedded Legendrian in the $1$-jet space $(J^1C,\xi_\st)$. The Legendrian surface $\La(G)$ is described by using germs of Legendrian wavefronts \cite[Section 3.1]{ArnoldSing} in the Darboux chart $(\R^5,\xi_\st)$, where the contact 4-distribution $\xi_\st$ is defined as
$$\xi_\st=\ker \a_\st,\mbox{ where }\a_\st:=dz-y_1dx_1-y_2dx_2,$$
and $(x_1,x_2,y_1,y_2,z)\in\R^5$ are Cartesian coordinates in $\R^5$. This is the local model for any contact 4-distribution in the neighborhood of a point \cite[Theorem 2.5.1]{Geiges08}. Since $\la_\st=y_1dx_1+y_2dx_2$ is the Liouville form of the cotangent bundle $(T^*\R^2,\omega_\st)$, this Darboux chart $(\R^5,\xi_\st)$ is contactomorphic to the 1-jet space $(J^1\R^2,\ker\{dz-\la_\st\})$.

The Legendrian fibration $\pi:\R^5\lr\R^3$, $\pi(x_1,x_2,y_1,y_2,z)=(x_1,x_2,z)$ allows us to assign a smoothly embedded Legendrian surface $\La(\Sigma)\sse\R^5$ in the domain of $\pi$ to certain singular surfaces $\Sigma\sse\R^3$ in its target. The coordinates $(y_1,y_2)$ of the Legendrian $\La(\Sigma)$ assigned to $\Sigma$ are
$$y_1=x_1\mbox{-slope of the tangent plane } T_{(x_1,x_2,z)}\Sigma,$$
$$y_2=x_2\mbox{-slope of the tangent plane } T_{(x_1,x_2,z)}\Sigma.$$
In a local parametrization $\sigma:\R^2\lr\R^3$ of $\Sigma$, $\sigma(u,v)=(u,v,z(u,v))$, this reads
$$y_1=\partial_uz(u,v),\quad y_2=\partial_vz(u,v).$$
This assignment is dictated by the vanishing of the contact 1-form $\a_\st$ along $\La=\La(\Sigma)$. The three-dimensional case is explained in detail in \cite[Section 3.2]{Geiges08}, the general case is discussed in \cite[Chapter 5]{ArnoldGivental01}, \cite[Section 3.2]{EkholmEtnyreSullivan05b} and \cite[Section 2]{CasalsMurphy}. The germs of singularities of $\Sigma$ that lift to an embedded Legendrian $\La$, and equivalently, the singularities of the map $\pi|_\La$, are restricted. These are known as singularities of fronts, or equivalently, Legendrian singularities \cite{ArnoldGivental01}. By definition, singular surfaces $\Sigma$ obtained as the image of an embedded Legendrian submanifold via a Legendrian mapping are referred to as (wave)fronts.

\begin{remark}
The classification of generic singularities of spatial fronts $\Sigma\sse\R^3$ is stated in \cite[Theorem 3.1.1]{ArnoldGivental01}, and that of generic singularities of a 1-parametric family of spatial fronts $\Sigma\sse\R^3$ is explained in \cite[Theorem 3.4.2]{ArnoldGivental01}.\hfill$\Box$
\end{remark}

The main spatial wavefronts $\Sigma$ that we use in the course of this article use three different germs of singularities of Legendrian fronts: $A_1^2,A_1^3$ and $D_4^-$, which we now describe. We emphasize that these are singularities of the wavefront projections only:  the corresponding local Legendrian surfaces are all smooth.

\subsubsection{The $A_1^2$ germ} This germ is obtained as a product of a 2-dimensional planar front times an interval. It is described by the germ of the singular surface
$$\Sigma(A_1^2)=\{(x_1,x_2,z)\in\R^3:\quad (x_1^2-z^2)=0\}$$
at the origin. This wavefront is informally called an $A_1^2$-\emph{crossing}, or a {\it crossing}, and the set of points $\{(x_1,x_2,z)\in\Sigma(A_1^2): x_1=0,z=0\}$ is referred to as an edge, or segment, of {\it $A_1^2$-crossings}. This spatial front is depicted on the left in Figure \ref{fig:SingWave}. Its Legendrian lift $\La(\Sigma(A_1^2))\sse(\R^5,\xi_\st)$ consists of two disjoint embedded Legendrian 2-disks.

\subsubsection{The $A_1^3$ germ} The wavefront $A_1^3$ is given by the germ at the origin of the singular surface
$$\Sigma(A_1^3)=\{(x_1,x_2,z)\in\R^3: (x_1-z)(x_1+z)(z-x_2)=0\}\sse\R^3,$$
This spatial front is depicted in the center of Figure \ref{fig:SingWave}. Considered as a germ, the origin is the $A_1^3$-wavefront singularity, and the codimension-1 singular strata consists of six half-lines of $A_1^2$ singularities. The Legendrian lift $\La(\Sigma(A_1^3))$ of the $A_1^3$ germ to $(\R^5,\xi_\st)$ consists of three disjoint embedded Legendrian 2-disks.

\subsubsection{The $D_4^-$ germ} The third germ $\Sigma(D_4^-)=\mbox{Im}(\delta_4^-)\sse\R^3$ of a Legendrian singularity that we use is given by the germ at the origin for the image of the map
$$\delta_4^-:\R^2\lr\R^3,\quad \delta_4^-(x,y)=\left(x^2-y^2,2xy,\frac{2}{3}(x^3 - 3xy^2)\right).$$
The $D_4^-$-singularity of the spatial wavefront $\mbox{Im}(\delta_4^-)$ is at $(0,0,0)\in\R^3$. The front $\mbox{Im}(\delta_4^-)$ itself also has three half-lines of $A_1^2$-crossings, intersecting at the origin. This is depicted in the right of Figure \ref{fig:SingWave}. The Legendrian lift $\La(\mbox{Im}(\delta_4^-))\sse(\R^5,\xi_\st)$ of the $D_4^-$ spatial front is an embedded Legendrian 2-disk.  We refer the reader to \cite{ArnoldSing,TreumannZaslow} for more descriptions --- see also Remark \ref{rmk:D4minus} below.
\begin{center}
\begin{figure}[h!]
\centering
  \includegraphics[scale=0.75]{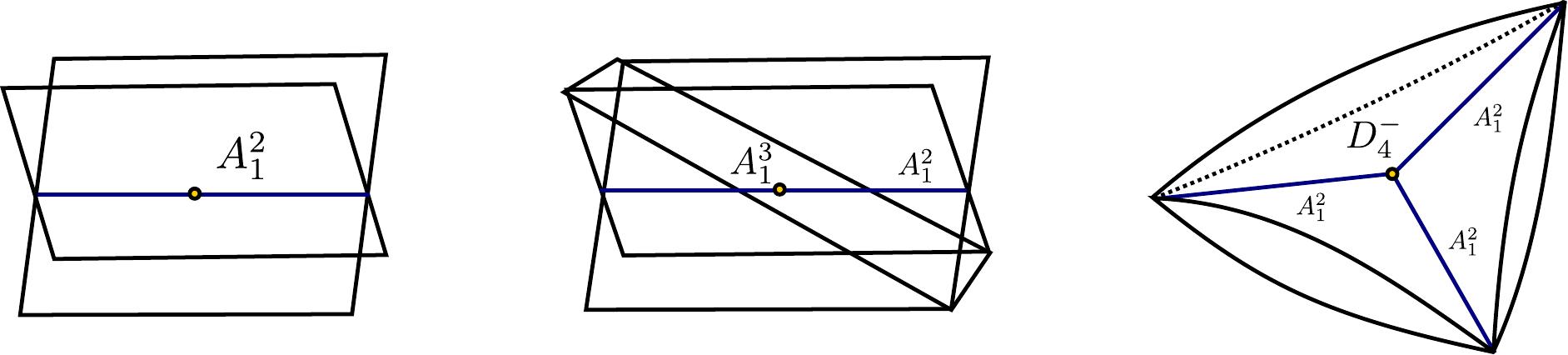}
  \caption{The $A_1^2$ spatial front (left), the germ of the $A_1^3$ Legendrian singularity (center) and the $D_4^-$ Legendrian wavefront (right).}
  \label{fig:SingWave}
\end{figure}
\end{center}
The connection of the above three Legendrian singularities with the Weyl groups, justifying their nomenclature, can be found in \cite[Section 3.3]{ArnoldGivental01}. It might be relevant to notice that $D_4^-$ is not the germ of a singularity for a {\it generic} Legendrian wavefront, but still a valid singularity for a given spatial wavefront. In addition, it is known that the singularity $D_4^-$ is generic in 1-parameter families of Legendrian fronts \cite[Section 3.3]{ArnoldSing}. As a result, most of the Legendrians we construct are non-generic, in their isotopy class, with respect to the fixed Legendrian projection.  This rigidification simplifes the analysis and
combinatorics.

\begin{remark}
\label{rmk:D4minus}
The $D_4^-$ Legendrian singularity has the property that its singular strata, excluding $A_1^2$ singularities, is a point, which lies in real codimension $2$. This is not the case for the majority of Legendrian surface singularities, such as the Legendrian $A_3$-swallowtail, cusp-edges $A_2A_1$ and the purse wavefront $D_4^+$, the former two even being generic. (These singularities feature in Section \ref{sec:moves}.) The geometric reason for this codimension-2 phenomenon is the existence of the {\it holomorphic} Legendrian surface singularity
$$t:\C\longrightarrow (J^1(\C,\C),\ker\{dw_1-w_2dw_3\}),\quad (w_1,w_2,w_3)=t(w)=\left(w^2,w,\frac{2}{3}w^3\right),$$
whose real part is the real Legendrian singularity $D_4^-$. This holomorphic map is the complexification of the real simple cusp singularities appearing in generic front projections of embedded Legendrian knots in a Darboux chart $(\R^3,\xi_\st)$.\hfill$\Box$
\end{remark}

We also use the $A_2,A_2A_1$ front singularities, geometrically represented by a simple cusp in $\R^2$ times an interval, and its intersection with a $2$-plane. These $A_2$-singularities do not directly arise from an $N$-graph $G\sse C$, but rather from satelliting the smooth surface $C$ to a Legendrian surface in a contact 5-manifold $(Y,\xi)$, typically $(\S^5,\xi_\st)$.


\subsection{Legendrian Weaves}\label{ssec:LegSing}

Let $G\sse C$ be an $N$-graph, as introduced in Subsection \ref{ssec:Ngraphs} above. The principle that associates a Legendrian $\La(G)$ to the $N$-graph $G$ is that $G$ dictates the configuration of $A_1^2$ singularities (crossings) of its Legendrian wavefront. This is possible because the singularities introduced in Subsection \ref{ssec:SingWave} are uniquely determined by their $A_1^2$ front singularities. Let us explain the construction in detail.

First, we choose the ambient contact manifold, where the embedded Legendrian surface $\La(G)$ belongs, to be the 1-jet space of the smooth surface $C$. That is,
$$\La(G)\sse (J^1C,\xi_\st)=(\{(x,z)\in T^*C\times \R\},\ker\{dz-\la_\st\}),$$
where $\la_\st\in\Omega^1(T^*C)$ is the Liouville form \cite[Section 1.4]{Geiges08}, and see \cite[Example 2]{ArnoldSing} and \cite[Example 2.5.11]{Geiges08} for details on the $1$-jet space. The local germs described in Subsection \ref{ssec:SingWave} above and the Legendrian front projection $\pi:(J^1C,\xi_\st)\lr C\times\R$ allow us to assign a Legendrian $\La(\Sigma)\sse (J^1C,\xi_\st)$ to a spatial wavefront $\Sigma\sse C\times\R$ in the target, as follows.

The construction of the front $\La(G)\sse (J^1C,\xi_\st)$ is obtained by gluing local wavefront models in $U_i\times\R$, $i\in I$, $U_i\cong\D^2$, which are the targets of front projections in the Darboux charts $(J^1 U_i,\xi_\st)\cong(J^1\D^2,\xi_\st)$, for $i\in I$. This is formalized in the following definition:

\begin{definition}\label{def:LegWeave1}
Let $\D_N=\D^2\times\{1\}\cup\ldots\cup \D^2\times\{N\}\sse\D^2\times\R$.  We consider $\D_N$ as a disconnected, horizontal wavefront.  Let $P\sse\D^2\times\{0\}$ be one of the following four local models of an $N$-graph $G\sse\D^2$:

\begin{itemize}
 \item[1.] A unique $i$-edge in $\D^2$, as drawn at the bottom of the second column in Figure \ref{fig:Weaving}.
 \item[2.] A unique trivalent $i$-vertex, as shown at the bottom of the third column in Figure \ref{fig:Weaving}.
 \item[3.] A unique hexagonal $(i,i+1)$-point, depicted in the fourth column in Figure \ref{fig:Weaving}.
 \item[4.] The empty set.
\end{itemize}

Here, recall that an $i$-edge is an edge belonging to the graph $G_i\sse G$ of the $N$-graph $G\sse \D^2$, for $1\leq i\leq (N-1)$. By definition, the Legendrian wavefront $\D_N(P)\sse\D^2\times\R$ associated to $P$ is obtained as follows:

\begin{itemize}
\item[-] If $P$ is a $i$-edge, insert an $A_1^2$-intersection along the two sheets $\D^2\times\{i\}$ and $\D^2\times\{i+1\}$ of the wavefront $\D_N$. This $A_1^2$ intersection must be inserted such that the image of the $A_1^2$ singular locus coincides with $P$ under the projection $\D^2\times\R\lr\D^2$ onto the first factor.\\

\item[-] If $P$ is a trivalent $i$-vertex, introduce a $D_4^-$-singularity between the two sheets $\D^2\times\{i\}$ and $\D^2\times\{i+1\}$ in the wavefront $\D^2_N$. This $D_4^-$ singularity must be introduced such that, under the projection $\D^2\times\R\lr\D^2$ onto the first factor, the image of the $A_1^2$-crossings coincides with the three edges of $P$ and the $D_4^-$ singular point is mapped to the unique trivalent vertex of $P$.\\

\item[-] If $P$ is a hexagonal $(i,i+1)$-point, insert an $A_1^3$-intersection along the three disjoint sheets $\D^2\times\{i\},\D^2\times\{i+1\}$ and $\D^2\times\{i+2\}$ of the wavefront $\D^2_N$. The pattern for the $A_1^3$-wavefront must be inserted such that, under the projection $\D^2\times\R\lr\D^2$ onto the first factor, the origin in the $A_1^3$-singularity maps to the unique vertex of $P$, and the six half-lines of $A_1^2$-crossings map to the six edges emanating from the vertex.
\end{itemize}

These wavefronts are depicted in Figure \ref{fig:Weaving}. For $P$ empty we use the front $\D^2\times\{1\}\cup\ldots\cup \D^2\times\{N\}\sse\D^2\times\R$. We refer to the wavefronts $\D_N(P)$ as being obtained from the wavefront $\D_N$ by weaving according to the pattern $P$.\hfill$\Box$
\end{definition}

\begin{center}
\begin{figure}[h!]
\centering
  \includegraphics[scale=0.8]{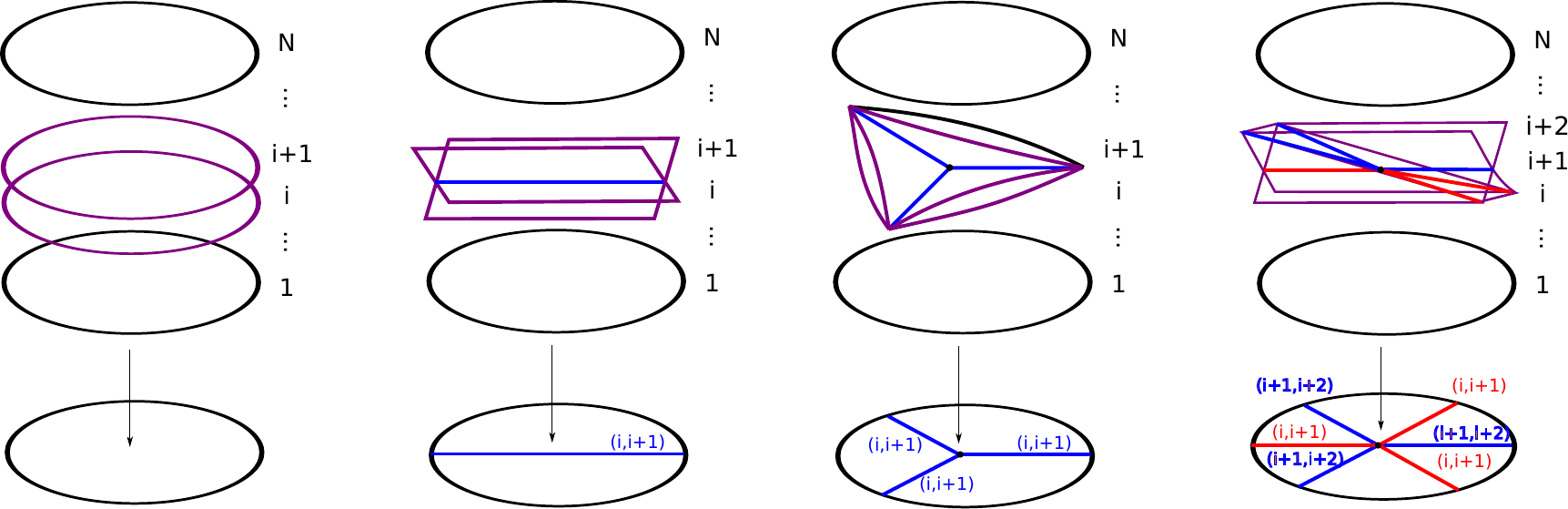}
  \caption{The leftmost wavefront is $\D_N$, then from left to right we find $\D_N(P)$ where $P$ is an edge, a trivalent vertex and a hexagonal vertex.}
  \label{fig:Weaving}
\end{figure}
\end{center}

Definition \ref{def:LegWeave1} describes how to weave the wavefont $\D_N\sse\D^2\times \R$, which we have fixed, according to a pattern $P\sse \D^2\times\{0\}$. 
To glue models, let $\{U_i\}_{i\in I}$ be a finite cover of $C$ by open 2-disks $U_i \cong \D^2$, refined as necessary so that each $U_i$ contains no more than one non-empty feature $P$ of the $N$-graph $G.$
Now, let us consider two 2-disks $U_1,U_2\sse C$ and two corresponding patterns $P_1,P_2$ therein.

Suppose that the patterns $P_1$ and $P_2$ coincide along the intersection $U_1\cap U_2$.  Then we say that $P_1\cup P_2$ defines a pattern in $U_1\cup U_2$. By definition, the wavefront $\Sigma(P_1\cup P_2)$ associated to $P_1\cup P_2$ is obtained by considering the set-theoretical union of $\D_N(P_1)$ and $\D_N(P_2)$ in $(U_1\cup U_2)\times\R$. For brevity of notation, we will say that $\Sigma(P_1\cup P_2)$ is obtained by weaving $\D^2_N\cup \D^2_N\sse (U_1\times\R)\cup (U_2\times\R)$ according to the pattern $P_1\cup P_2$. Finally, the Legendrian surface associated to an $N$-graph is defined as follows:

\begin{definition}\label{def:LegWeave2}
Let $C$ be a smooth surface and $G\sse C$ an $N$-graph, the Legendrian weave
$$\La(G)\sse (J^1C,\xi_\st)$$
is the embedded Legendrian surface whose wavefront $\Sigma(G)\sse C\times\R$ is obtained by weaving the wavefront $C\times\{1\}\cup\ldots\cup C\times\{N\}\sse C\times\R$ according to the pattern $G\sse C$.\hfill$\Box$
\end{definition}

Let $\{\varphi_t\}_{t\in[0,1]}\sse\Diff^c(C)$, $\varphi_0=\mbox{Id}$, be a compactly supported isotopy of the smooth surface $C$. Then the Legendrian surfaces $\La(\varphi_t(G))\sse(J^1C,\xi_\st)$, as described in Definition \ref{def:LegWeave2}, are Legendrian isotopic, relative to their boundaries. Hence, for the purposes of this article, our $N$-graphs $G\sse C$ are considered up to such planar isotopies. Similarly, Legendrian fronts in $\R^3$ are to be considered up to homotopy of fronts.\\

Thanks to Definition \ref{def:LegWeave2}, the wealth of contact topology invariants \cite{EliashbergGiventalHofer00,EkholmEtnyreSullivan05b,GKS_Quantization,STZ_ConstrSheaves,CasalsMurphy} can be used to define algebraic structures associated to $N$-graphs $G\sse C$. For instance, the articles \cite{CasalsMurphy2,TreumannZaslow} show that the chromatic polynomial of (the dual of) a trivalent graph $G$ -- which is a 2-graph -- is contained in the Floer-theoretical invariants of the Legendrian weave $\La(G)$. Conversely, from a contact topology perspective, the connection to combinatorics and algebraic geometry provides a new tool for computing contact invariants of higher-dimensional Legendrian submanifolds. This will be the focus of subsequent sections.

\begin{remark}\label{rmk:1dim}
The one-dimensional analogue of a Legendrian weave is a Legendrian braid, i.e. a positive braid. Indeed, an $N$-graph in a one-manifold $I$ is defined to be a set of points, each point labeled with a permutation in $\tau(N)\sse S_N$. The only planar front singularity that we can use is $A_1^2$, corresponding to a crossing, necessarily positive.  Thus, 1-dimensional weaving consists of introducing positive crossings to the $N$ strands
$$\D^1_N=I\times\{1\}\cup\ldots\cup I\times\{N\}\sse I\times\R$$
and concatenating them side by side. This is precisely the front for an $N$-strand positive braid \cite{PrasolovSossinsky}, which lifts to a Legendrian link in $(J^1S^1,\xi_\st)$ \cite[Section 3.3.1]{Geiges08}. The Legendrian weaves introduced in Definition \ref{def:LegWeave2} are thus the Legendrian surface generalization of Legendrian braids.\hfill$\Box$
\end{remark}


\subsection{Smooth Topology of Weaves}\label{ssec:homology} 
Let $G$ be an $N$-graph in a surface $C$, in this subsection we address the smooth topology of the Legendrian surface $\La(G)\sse(J^1C,\xi_\st)$.\footnote{This is necessary for our applications, especially in the study of microlocal monodromies and Lagrangian fillings in Section \ref{sec:app2} and the non-Abelianization map in Section \ref{sec:app3}.} The smooth invariants of $\La(G)$ are the first homology $H_1(\La(G),\Z)$, in particular its genus $g(\La(G))\in\N$, and the number of boundary components $|\dd\La(G)|$. For simplicity, we assume that $C$ is a closed surface, and thus $\dd\La(G)=\emptyset$.  We also assume that $G$ is a connected $N$-graph, i.e.~ the union of the graphs $G_i$, $i\in I$, is a connected topological subspace of $C$.

The surface $\La(G)$ is a branched $N$-fold cover over $C$ simply branched over the trivalent vertices of $G$. Indeed, the image of $\La(G)$ by the projection $J^1C\lr T^*C$ along the Reeb $\R$-direction yields an immersed surface $L(G)\sse T^*C$, and the canonical projection $T^*C\lr C$ restricts to $L(G)$ as an $N$-fold branched cover. The branch set is the image of the set of $D_4^-$ singularities. As a result, the genus of $\La(G)$ is provided by the Riemann-Hurwitz formula
$$\chi(\La(G))=N\chi(C)-v(G),\;\mbox{i.e. }g(\La(G))=\frac{1}{2}(v(G)+2-N\chi(C))$$
where $v(G)$ is the number of (trivalent) vertices of $G$.

\begin{remark}
If the surface $C$ has boundary, each boundary component of $\dd C$ contributes to a piece of the boundary $\dd\La(G)$ of the Legendrian surface $\La(G)\sse(J^1C,\xi_\st)$. Let $\kappa\in\N$ be the number of cycles in the (minimal length) factorization of the monodromy of the branched cover along a given boundary component of $\dd C$. Then, that one boundary component of $C$ contributes to $\kappa$ distinct boundary components for the Legendrian surface $\La(G)$.\hfill$\Box$
\end{remark}

\begin{example}
The Legendrian weaves $\La(G_1),\La(G_2)\sse (J^1(\S^2),\xi_\st)$ associated to the $3$- and $4$-graphs in Figure \ref{fig:NGraphExample} are closed Legendrian surfaces of genus $3$ and $4$, respectively. Should the graphs $G_1,G_2\sse\R^2$ be considered in the 2-plane $\R^2$, instead of the 2-sphere $\S^2$, the Legendrian surfaces $\La(G_1),\La(G_1)\sse (J^1(\R^2),\xi_\st)$ have genus $3$ and $4$, with $3$ and $4$ boundary components, respectively.\hfill$\Box$
\end{example}

Now, the $\Z_2$-monodromy of $\La(G)$ along a non-trivial 1-cycle of the base $C$ is trivial, and thus the contributions of the graph $G$ to $H_1(\La(G),\Z)$, as expressed by the above formula, can be considered by studying planar pieces. Let us then assume that $g(C)=0$ and construct 1-cycles in $H_1(\La(G),\Z)$ in terms of the edges of the $N$-graph.

\begin{center}
\begin{figure}[h!]
\centering
  \includegraphics[scale=0.6]{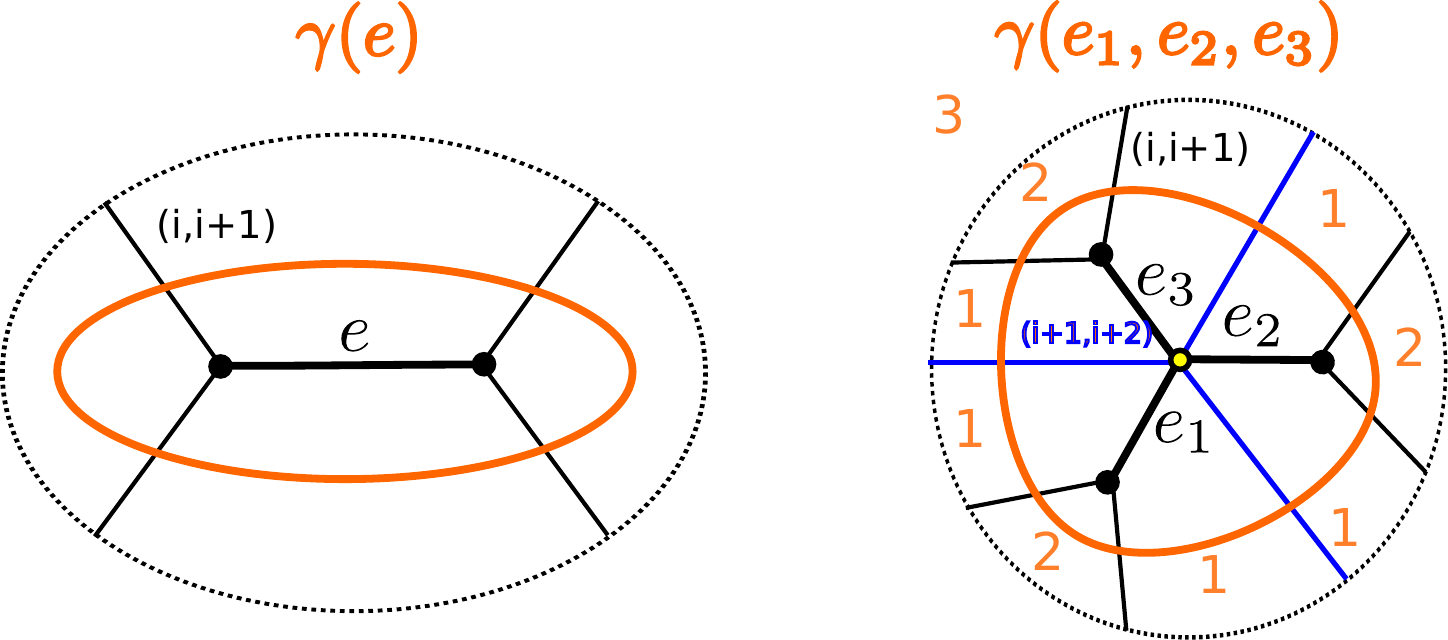}
  \caption{Two combinatorial descriptions of $1$-cycles in $H_1(\La(G),\Z)$.}
  \label{fig:1Cycle}
\end{figure}
\end{center}

There are two direct descriptions of 1-cycles $\gamma\in H_1(\La(G),\Z)$:

\begin{itemize}
 \item[1.] Each edge $e$ of the graph $G$ connecting two trivalent vertices defines a 1-cycle $\gamma(e)\in H_1(\La(G),\Z)$. The projection of this 1-cycle onto the pattern $P$ with two trivalent vertices is depicted in orange on the left of Figure \ref{fig:1Cycle}. In order to construct $\gamma(e)$ from the orange curve, lift a point in the orange curve to the annulus $\La(P)$, to either one of the two sheets, and uniquely follow the lift along the orange curve. Since the lift is isotopic to one of the boundary components of the annulus, it generates $H_1(\La(P),\Z)\cong\Z$. The 1-cycle $\gamma(e)$ is drawn directly in the wavefront projection in Figures \ref{fig:1CycleFront}. We refer to this type of 1-cycles as {\it monochromatic} edges or (short) $\sf I$-cycles.
 
 \begin{center}
 	\begin{figure}[h!]
 		\centering
 		\includegraphics[scale=0.7]{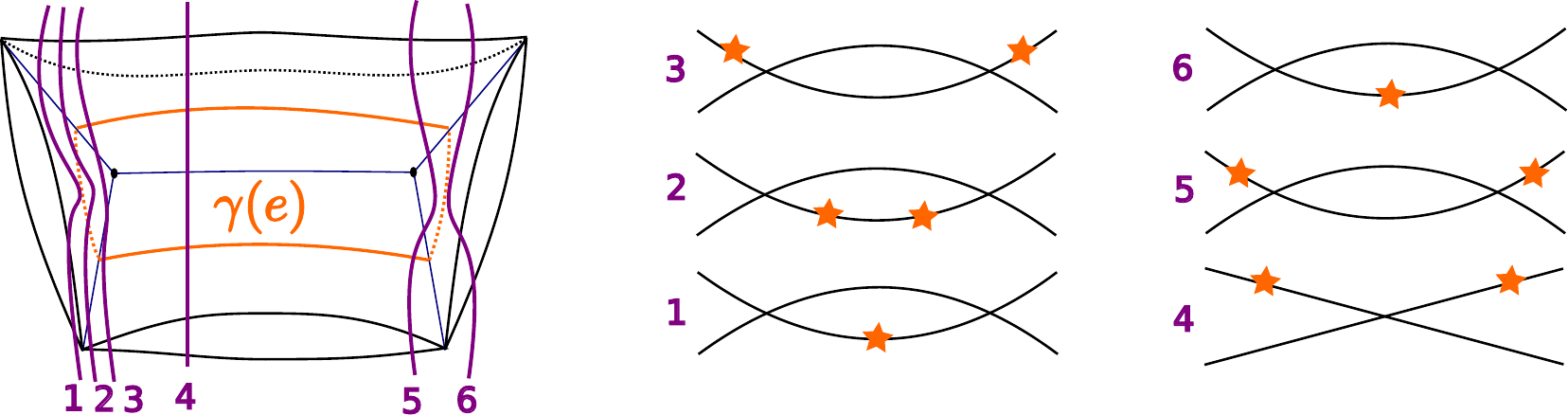}
 		\caption{The first type of $1$-cycle $\gamma(e)$ drawn in the wavefront (left) and in a vertical slicing (right). Each slice on the left is labeled by a number. The $1$-cycle $\gamma(e)$ appears as five-pointed stars in each slice as shown on the right.}
 		\label{fig:1CycleFront}
 	\end{figure}
 \end{center}
 
 There is a simple extension of this construction, depicted in Figure \ref{fig:1CycleExtended}. Consider a trivalent vertex $v\in G_i$ and a linear chain of edges $e_1,e_2,\ldots,e_k$ in $G$ such that $e_1$ connects $v$ to a hexagonal vertex, $e_i$ connect two hexagonal vertices for $2\leq i\leq k-1$ and $e_k$ connects the free hexagonal vertex in $e_{k-1}$ to a trivalent vertex.  Suppose further that $e_j$ and $e_{j+1}$ meet at opposite rays of the hexagonal vertex between them, $1\leq j \leq k-1.$  Then the orange curves in the patterns all lift to $1$-cycles which are essential\footnote{The topology of $\La(P)$ is that of an annulus union disjoint $2$-disks.} in the surfaces $\La(P)$ for the corresponding patterns $P$. These 1-cycles are referred to as {\it long edges} or {\it long $\sf I$-cycles}.\\
 
 \begin{center}
\begin{figure}[h!]
\centering
  \includegraphics[scale=0.7]{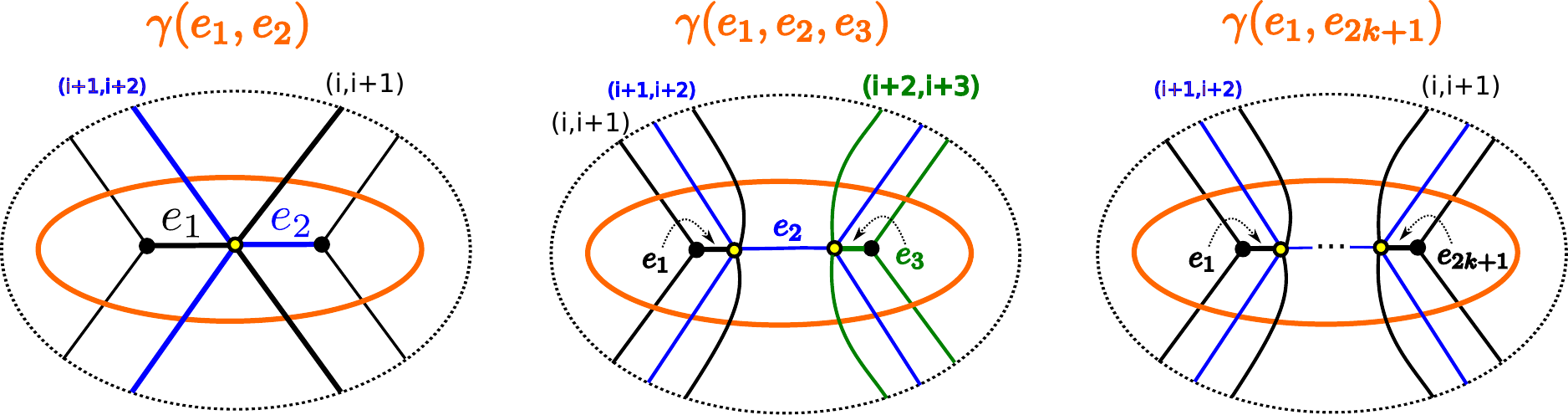}
  \caption{Descriptions of $1$-cycles in $H_1(\La(G),\Z)$ of the first type, generalizing $\gamma(e)$ on the left of Figure \ref{fig:1Cycle}. The lift of the orange curves generate the first homology $H_1(\La(P),\Z)\cong\Z$ for the corresponding patterns $P$.}
  \label{fig:1CycleExtended}
\end{figure}
\end{center}
 
 \item[2.] Three edges $e_1,e_2,e_3$ of a graph $G_i$ connecting a hexagonal vertex with three trivalent vertices in $G_i$ defines a cycle $\gamma(e_1,e_2,e_3)\in H_1(\La(G),\Z)$. This is depicted on the right in Figure \ref{fig:1Cycle}. The 1-cycle $\gamma(e_1,e_2,e_3)$ is drawn in the wavefront projection in Figure \ref{fig:1CycleFront2}. We refer to this type of 1-cycles $\gamma(e_1,e_2,e_3)$ as a $\sf Y$-cycle.
\end{itemize}

\begin{center}
\begin{figure}[h!]
\centering
  \includegraphics[scale=0.7]{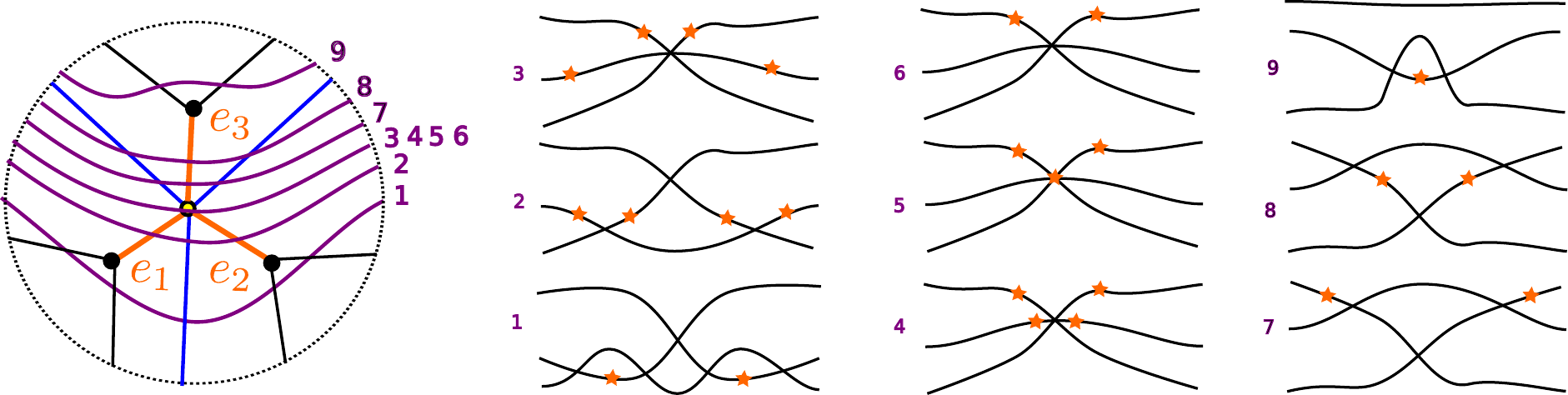}
  \caption{The second type of $1$-cycle $\gamma(e_1,e_2,e_3)$ drawn in a slicing of the wavefront associated to the pattern on the left.}
  \label{fig:1CycleFront2}
\end{figure}
\end{center}

We can also combine the above two constructions to associate a 1-cycle to any tree with leaves on trivalent vertices that passes directly through any hexagonal vertices, i.e.~entering and exiting along opposing edges, see Figure \ref{fig:ThurstonQuiver} for an example. For such a tree, we refer to the pieces corresponding to edges as $\sf I$-pieces, or edges, and the pieces that go through a hexagonal vertex as $\sf Y$-pieces. In addition, we can decorate such 1-cycles with a number, indicating higher multiplicity\footnote{Higher multiplicities will rarely feature in this manuscript, only in relation to Theorem \ref{thm:ThurstonLinks}.}. If we require the curves in the Legendrian surface to be connected, then higher multiplicity in general requires these curves to be immersed.

\begin{remark}
Let $\La(G)\sse(J^1C,\xi_\st)$ be a connected surface, and $G\sse C$ a connected $N$-graph. The trivalent vertices of the $N$-graph $G\sse C$ can be assumed to belong to $G_1$. This follows once we impose certain equivalence relations on the set of $N$-graphs, which is done in Section \ref{sec:moves}.\hfill$\Box$
\end{remark}


\subsection{Combinatorial Homology}\label{ssec:combhomology} Let $G\sse C$ be an $N$-graph.  We present a combinatorial model for the (chain-level) simplicial homology of $\La(G)$. This can be achieved in general, but for this subsection we assume that $G$ is a planar 3-graph, i.e.~$C=\S^2$ and $N=3$.  We will think of $G$ as bicolored --- see Remark \ref{rmk:bicolored}.  This will ease notation, while containing the essential idea for higher $N\in\N$ and higher-genus $C$. Note that the results in this subsection will not be used in the rest of the manuscript, we have included them for completeness.

The edges, faces and vertices of $G$ lift to edges, faces and vertices of the Legendrian surface $\Lambda:=\Lambda(G)$. Let us suppose that $G$ and $\Lambda$ are connected, and that the faces of $G$ define a polyhedral decomposition $(F,E,V)$ of the sphere.
This decomposition lifts to a polyhedral decomposition of $\Lambda$, as follows. Each face, edge and hexagonal vertex of $G$ has three lifts to $\Lambda;$ each trivalent vertex has two lifts. This yields
$$\chi(\Lambda) = 3\cdot 2 -v,$$
where $v=|V(G)|$ is the number of trivalent vertices. For a point $P \in \S^2$, we write $P_1, P_2, P_3$ for the (up to) three pre-images in non-decreasing order of the $z$-coordinate. If $P$ is on $G$, we must choose a nearby point to define the ordering of $z$ coordinates of sheets.
If $P$ is a trivalent vertex with label $(12)$, in blue, then $P_1 = P_2$ while $P_2 = P_3$ for a label $(23)$, in red. Lifts of edges and faces are labeled analogously. The chain complex $C_\bullet$ associated to this polyhedral decomposition of $\Lambda$ computes the homology $H_*(\Lambda;\bZ)$. There is a simplified chain complex that computes $H_1(\Lambda;\bZ),$ which we now explain.

Lift each edge $e = (P,Q)$ labeled $(i,i+1)$ to a one-chain as follows (here $i = 1$ or $2$).  In the (any) orientation of the plane, if $A$ is the sheet with lower $z$ value in the region to the left of $\overline{PQ}$ and $B$ is the sheet with lower $z$ value to the right of $\overline{PQ}$ then lift $e$ to the chain $\overline{P_BQ_B}-\overline{P_AQ_A}$; this only depends on $e$ and not the ordering of $P$ and $Q$. Write $\hat{e}$ for this lift of $e$.  Extending by linearity, we get a map $\bZ^E \to C_1.$

The embedded bicolored graph $G$ is the union $G = G_B \cup G_R$ of embedded blue and a red graphs intersecting at hexagonal vertices,
where $G_B = (F_B,E_B,V_B)$, $F_B$ denotes the number of faces of the graph $G_B$, and $E_B$ the number of edges and $V_B$ the number of vertices. (Similarly for $G_R= (F_R,E_R,V_R)$.)
We define a complex $A_\bullet$ as follows.  $A_2 := \bZ^{F_B}\oplus \bZ^{F_R},$ $A_1 := \bZ^E = \bZ^{E_B}\oplus \bZ^{E_R},$
and $A_0$ is the image $\partial \hat{A_1}$, where $\hat{A_1}$ is the image of $A_1$ in $C_1$ under $e\mapsto \hat{e}.$
A monochromatic face $f\in \bZ^{F_B} \subset A_2$ has a lift to $C_2$ as $f_1 - f_2$, whereas $f\in \bZ^{F_R}\subset A_2$ lifts to $f_2-f_3.$
Summarizing, we have
$$\xymatrix{C_2\ar[r]&C_1\ar[r]&C_0\\ A_2\ar[u]&A_1\ar[u]\ar[r]&A_0\ar@{^{(}->}[u]},$$
where the map $A_1\to A_0$ sends $e$ to $\partial \hat{e}.$
The missing differential $A_2\to A_1$ is defined as follows.  For a monochromatic face $f\in \bZ^{F_B}$ or $\bZ^{F_R} \subset A_2,$
$$\partial f = \sum_{\text{boundary edges}} e - \sum_{\text{interior edges}} e,$$
which we extend by linearity.

\begin{prop}
	\label{prop:chain}
	$A_\bullet$ is a chain complex and $A_\bullet \lr C_\bullet$ is a chain map.
\end{prop}
\begin{proof}
	Let $f\in A_2$ be a blue face. A similar argument will work for red faces. We need to check that $A_\bullet$ is a chain complex, i.e. the differential squares to zero. This reads
	$$\sum_{\text{boundary edges}} \partial \hat{e} - \sum_{\text{interior edges}} \partial \hat{e} = 0,$$
	for a face $f$. The left hand side of this equality is a (formal) sum of some of the vertices of the graph, some trivalent, some hexagonal. Thus, this imposes a condition at all the interior and exterior vertices of $f$.
	In fact, the condition is null at an interior vertex, since it must be monochromatic and hence trivalent, and $\partial \hat{e}$ is
	zero over any trivalent vertex.
	Likewise for an exterior trivalent vertex, there is nothing to check and it remains to discuss exterior hexagonal vertices.  For an exterior hexagonal vertex, a local study is needed, as follows.
	
	Let $h$ be a hexagonal vertex and let us study the differential restricted to it.  Let $e_1,e_2,e_3$ be three attached blue half-edges, with $e_4 = e_1' , e_5 = e_2' ,e_6 = e_3'$ the opposite red half-edges, respectively.
	Let $h_1,h_2,h_3$ be the three preimages of $h$.  We can restrict the differential $A_1\vert_h\to A_0\vert_h$ to edges intersecting $h$ and
	points over $h$, and in the chosen basis it takes the form
	\begin{equation}
	\label{eq:d}
	\partial\vert_h = \begin{pmatrix}-1&0&1\\1&-1&0\\0&1&-1\\
	1&0&-1\\-1&1&0\\0&-1&1\end{pmatrix}.
	\end{equation}
	The kernel is generated by $e_1+e_1', e_2+e_2', e_3+e_3', e_1 + e_2 + e_3.$  Note that the last generator could also have been taken to be $e_1 + e_2 - e_3'$. The first three represent long two-colored edges passing straight
	through the hexagonal vertex, while the last is a monochromatic $\sf Y$ shape. In more detail, the element $e_1+e_1'$ of the kernel is, diagramatically, given by a (bi-colored) edge passing through the hexagonal vertex. Similarly for $e_2+e_2'$ and $e_3+e_3'$, they represent straight edges passing through the hexagonal vertex, starting blue and then turning red, or viceversa. The element $e_1 + e_2 + e_3$ of the kernel is given by the $\sf Y$-shaped union of the three blue edges coming out of a hexagonal vertex. The element $e_1 + e_2 - e_3'$ also belongs to the kernel, and it represents two blue edges {\it and} the red edge in between (with a minus sign). The terms arising in $\partial^2 f$ which involve a boundary hexagonal vertex are given by the image of such configurations in the kernel, and thus they must (and do) vanish.  This concludes the calculation
	that $(A_\bullet,\partial)$ is a chain complex.
	
	To check that $A_\bullet \to C_\bullet$ is a chain map, we must show that for $f\in \bZ^{F_B}\subset A_2$, we have
	$$\partial f_1 - \partial f_2 =
	\sum_{e\; \text{exterior}} {\hat{e}} - \sum_{e\; \text{interior}} {\hat{e}}.$$  This is shown by direct calculation.

\end{proof}
Let us now prove the following lemma before showing that $A_\bullet$ is quasi-isomorphic to $C_\bullet$ in degree one, and thus computes the first homology $H_1(\Lambda;\bZ)$.
\begin{lemma}
	In the notation above, $H_2(A_2)=0$.
\end{lemma}
\begin{proof}
	This says that $\partial: A_2\to A_1$ is injective.  Suppose $\partial f=0$.  Let $h$ be a hexagonal vertex, which must exist since both $\Gamma$ and $\Lambda$ are assumed connected. Label the edges adjacent to $h$ by $e_1, e_2, e_3, e_4 = e_1', e_5 = e_2', e_6 = e_3'$ as in the proof of Proposition \ref{prop:chain}.
	For $1\leq c\leq 6$, let $f_c$ be the unique (opposite color) monochromatic face containing $e_c$ in its interior, and again we notate $f_4 = f_1'$, etc.
	Now for $i = 1,2,3,$ write $i,j,k$ for cyclically ordered elements of $\{1,2,3\},$ i.e. $j = i+1 \mod 3,$ etc.
	Then $e_i$ is an exterior edge of $f_j'$ and $f_k'$ and by definition an interior edge of $f_i.$
	If we write $f = \sum_{i=1}^3 a_i f_i + a_i' f_i' + \cdots,$ then we must have $a_i + a_j = a_k'$ for all $i$, and therefore $\sum a_i' = 2\sum a_i.$  By the same token, $\sum a_i = 2\sum a_i',$ and therefore
	all $a_i$ and $a_i'$ are zero.
	
	The faces $f_c$ with coefficients $a_c\neq 0$ must therefore have no hexagonal vertices on their boundary or interior. That said, the union $U$ of such faces must have a boundary, and therefore the coefficient of any face on the boundary of $U$ must be zero. By iterating this argument, all coefficients are zero.
\end{proof}

The 3-graphs associated to a 3-triangulation, and the 3-graph moves named {\it candy twists} and {\it push-through}, will be defined in Section \ref{sec:moves}. We will use them now just in this particular proposition as part of  this isolated subsection.\footnote{This subsection on combinatorial homology is included for completeness, but it will not be used in the rest of the article.} Now we establish the point of this subsection:

\begin{prop}
	Let $\Gamma$ be a 3-graph for a 3-triangulation, or any graph related by candy twists or push-through moves. Then
	$$H_1(A_\bullet) \cong H_1(C_\bullet) \cong H_1(\Lambda;\bZ).$$
\end{prop}

Before the proof, a {\it warning}: $H_1(A_\bullet) \ncong H_1(C_\bullet)$ in general.  Here is an example of a weave with topology of the twice-punctured plane. 
\begin{center}
	\begin{tikzpicture}
	\draw[ultra thick,blue] (-1.2,1)--(-.5,0)--(-1.2,-1);
	\draw[ultra thick,blue] (-.5,0)--(.5,0);
	\draw[thick,red] (-1.5,0)--(-.5,0);
	\draw[thick,red] (1.5,0)--(.5,0);
	\draw[ultra thick,blue] (1.2,1)--(.5,0)--(1.2,-1);
	\draw[thick,red] (-.5,1.3)--(-.5,-1.3);
	\draw[thick,red] (.5,1.3)--(.5,-1.3);
	\draw[thick,blue] (-1.5,1)--(-1.2,1)--(-1.2,1.3);
	\draw[thick,blue] (-1.5,-1)--(-1.2,-1)--(-1.2,-1.3);
	\draw[thick,blue] (1.5,1)--(1.2,1)--(1.2,1.3);
	\draw[thick,blue] (1.5,-1)--(1.2,-1)--(1.2,-1.3);
	\draw [thick, dotted, black] (-1.2,1) to[out=-20, in = 200]  (1.2,1);
	\end{tikzpicture}
\end{center}
Despite $b_1=2$, there is only one 1-cycle in $H_1(A_\bullet),$ represented by the tree with four leaves -- the sum of edges $\widehat{e}$ darkened in the picture.  A choice for another generating 1-cycle is clear:  it is a branch cut connecting the two trivalent vertices in the top (or bottom) -- pictured as a dotted black curve.  This class can be represented in $C_\bullet$, but the chain connecting the two hexagonal vertices is not in $A_\bullet$.  One could accommodate such chains with further notational complexity, but we will not require them for our applications.

\begin{proof}
	We need to prove the first equality only.  
	Since $A_\bullet \to C_\bullet$ is a chain map, we need only compare the dimensions of their first homology groups.
	We prove this first for the 3-graph $\Gamma_T$ of a 3-triangulation $T = (F_T,E_T,V_T)$, then show that the result is invariant under
	candy twist and push-through moves.  
	
	By definition, $\partial: A_1\to A_0$ is surjective, so since by the lemma, $\partial:  A_2\to A_1$ is injective, we know $\dim H_1(A_\bullet) = -\chi(A_\bullet).$
	On the other hand, we know $\chi_\Lambda = 6 - v = 6 - 3|F_T| = 2 - h_1(\Lambda),$ or $h_1(\Lambda) = 3|F_T| - 4.$
	We recall that each face of $T$ has three blue vertices.  It also has one hexagonal vertex which is a vertex of the blue and red graphs comprising $\Gamma_T.$ 
	It similarly easy to compute that
	$|F_B| = |V_T| + |E_T|,$ $|F_R| = |V_T|,$ $|E_B| = 2|E_T|+3|F_T|,$ $|E_R| =|E_T|,$ $|V_B| = 3|F_T|,$ $|V_R| = 0.$ 
	Now $|A_2| = |F_B| + |F_R| = 2|V_T| + |E_T|,$ $|A_1| = |E_B| + |E_R| = 3|E_T| + 3|F_T|$, and 
	$|A_0| = 2|F_T|$ is computed by noting that each hexagonal vertex contributes two possible dimensions to $|A_0|$ via the rank-two matrix in Equation \eqref{eq:d}, and these
	dimensions are realized as boundaries, while each trivalent vertex contributes nothing.
	We get $h_1(A_\bullet) = -\chi(A_\bullet) = -2\chi_T +3|F_T| = 3|F_T| - 4 = h_1(\Lambda),$
	as claimed.
	
	It remains to compute what happens after push-through or a candy-twist move.  In fact, since the result only depends on the Euler characteristic of $A_\bullet$,
	we only need to show that this is invariant under candy twist and push-through.  But these change the dimensions of $(A_2,A_1,A_0)$ by $(2,6,4)$ and $(1,3,2)$,
	respectively:  a local argument shows again that the two possible dimensions that a hexagonal vertex contributes to $A_0$ are in fact realized, and the result follows.	
\end{proof}


\section{Combinatorial Constructions}\label{sec:constr}

In this section we introduce two combinatorial constructions for $N$-graphs, focusing primarily on how to associate an $N$-graph to a given $N$-triangulation. The notion of an $N$-triangulation was introduced in \cite[Section 1.15]{FockGoncharov_ModuliLocSys}, and has since had an central role in higher Teichm\"uller theory \cite{Goncharov_IdealWebs,GoncharovLinhui_DT}.  Legendrian weaves associated to an $N$-triangulation, via our construction, place contact topology in the context of the recent developments in exact WKB analysis \cite{GMN_SpecNet13,GMN_SpecNetSnakes14,Kuwagaki20} and quiver Fukaya categories \cite{BridgelandSmith_QuadDiff,Smith_QuiverFuk}.


\subsection{$N$-Triangulations}

Let $N\in\N$ be a natural number, and consider the triangle $$t_N:=\{(x,y,z)\in\R^3:x+y+z=N\quad x,y,z\geq0\}.$$
Subtriangulate this triangle $t_N$ with the planes
$$(\{x=s\}\cup\{y=s\}\cup\{z=s\})\cap t_N,\quad 0\leq s\leq N,$$
which we refer to as an $N$-subdivision of the triangle $t_1$, following \cite{FockGoncharov_ModuliLocSys,GMN_SpecNetSnakes14}.  This subtriangulation has $N^2$ triangles.

\begin{center}
	\begin{figure}[h!]
		\centering
		\includegraphics[scale=0.85]{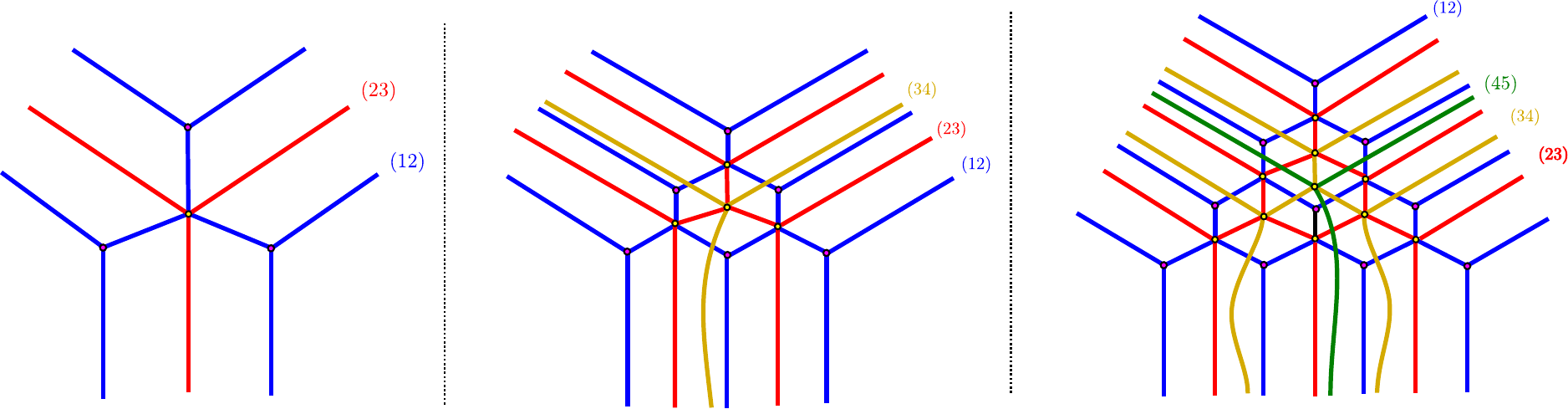}
		\caption{The Legendrian weaves $\La(G(t_3))$, $\La(G(t_4))$ and $\La(G(t_5))$ associated to a $3$-, $4$- and $5$-triangles.}
		\label{fig:NTriangles}
	\end{figure}
\end{center}

Now, let $(C,T)$ be a triangulation $T$ of a smooth closed surface $C$ and subdivide each triangle $t_1\in T$ according to the $N$-subdivision above. This yields a triangulation $T_N$ of the surface $C$. By definition, an $N$-triangulation on $C$ is any triangulation isotopic to $T_N$ for some triangulation $(C,T)$.


\subsection{Local Models}\label{ssec:Ngraph_Ntriangle} The $N$-graph associated to an $N$-triangulation is obtained by gluing local models for the $N$-graph $G(t_N)$ associated to each triangle $t_N$. We provide a definition of this local $N$-graph, in terms of the following construction. The reader content with using Figure \ref{fig:NTriangles} as a definition is invited to defer reading these technical descriptions.

{\it Construction.} Consider the triangles in $t_N$ which point up, i.e.~have a unique vertex with highest $z$-value. For each of these ${N \choose 2}$ triangles, we insert a $\tau_1$-trivalent vertex dual to it --- that is, a trivalent vertex associated with the permutation $(12)$ and such that the edges of this piece of $2$-graph intersects orthogonally with the edges of each triangle. By definition, the rest of the $N$-graph $G(t_N)$ is then {\it uniquely} determined by extending the edges from these $\tau_1$ vertices such that wherever three $\tau_i$-edges collide, we insert a hexagonal vertex with three edges in $\tau_i$ and three edges in $\tau_{i+1}$. That is, the two rules to generate the $N$-graph for an $N$-triangulation are:
\begin{itemize}
	\item[(i)] Insert exactly one $(12)$-trivalent vertex at the center of each upward pointing triangle,	
	\item[(ii)] In the collision of three $\tau_i$-edges, a $(\tau_i,\tau_{i+1})$ hexagonal vertex is inserted.
\end{itemize}

We stress that the original triangles are not part of the $N$-graph.

This construction of $G(t_N)$ can be considered as a {\it dynamical} description, in contrast with the static definition given by the second construction below Remark \ref{rmk:second}. Indeed, in this first construction one starts by placing the $\tau_1$-vertices and lets the edges grow symmetrically from these trivalent vertices, such that each edge intersects the interior edges of the $N$-triangulation at the middle point. These edges must collide in the interior of the triangle, and these collisions are resolved via the insertion of hexagonal vertices, creating $\tau_{i+1}$-edges. This insertion of hexagonal vertices is iteratively performed when the $\tau_{i+1}$-edges collide, creating $\tau_{i+2}$-edges, and the process terminates when exactly three $\tau_{N-1}$-edges are created at a unique hexagonal vertex.

Thus, given the triangle $t_N$, we obtain a local model for an $N$-graph. The boundary conditions for this local model are such that the $N$-graphs associated to two $N$-triangles $t_N$ and $t'_N$, which share an edge of the underlying $t_1$ and $t'_1$, match together.

\begin{remark}\label{rmk:second}
	This description, according to these two rules above, captures the properties of the spectral network associated to the WKB singular foliation for an $SU(2)$ quadratic differential lifted via the unique $N$-dimensional irreducible representation of $SU(2)$ --- see Sections 2 and 4 in \cite{GMN_SpecNetSnakes14}. The dynamical component, induced by the growing of the edges from vertices, corresponds with the time evolution of the differential equation defining the WKB system.\hfill$\Box$\end{remark}

We can also give a second succinct description of $G(t_N)$ as follows. Following Definition \ref{def:Ngraph}, it suffices to describe the image of the graphs $G_i$, $1\leq i\leq N-1$. The trivalent graph $G_i$ will be given by the vertices and edges of an hexagonal regular lattice: a finite number of vertices lying inside the triangle $t_N=\{x+y+z=N-1\}$ and with external edges extending to the boundary of $t_N$. Let $\vartheta_i$ be the set of points of $t_N$ all of whose coordinates lie in $\Z_{\geq0}+i/3$. Then the vertices of the trivalent graph $G_i$ are precisely the points in $\vartheta_i\cup\vartheta_{i+1}$. Note that the intersection between $G_i$ and $G_{i+1}$ is precisely given by the points in $\vartheta_{i+1}$, and the trivalent vertices are uniquely specified by the hexagonal lattice condition.

\begin{remark}
Both these constructions provide a Legendrian front for the Legendrian lift of certain exact Lagrangian spectral curve for a local spectral network. In particular, this shows that the BPS graphs studied in \cite{Gabella_BPSGraphs}, introduced as an interpolation between spectral networks and BPS quivers, are in fact the set of $A_1^2$ singularities of the Legendrian front $\Sigma(G(t_N))$ between the first two sheets.\hfill$\Box$
\end{remark}

Note that the boundary of this local $N$-graph $G(t_N)$ can be compactly described as follows. Consider the permutation
$$\Delta_N:=\prod_{i=1}^{N-1}\left(\prod_{j=i}^1 \tau_{j}\right)=\tau_1\cdot(\tau_2\tau_1)\cdot(\tau_3\tau_2\tau_1)\cdot\ldots\cdot (\tau_{N-1}\tau_{N-2}\cdot\ldots\cdot \tau_2\tau_1)\in S_N,$$
which is the projection to the Coxeter group $S_N$ of the Garside element of the braid group $B_N$ in $N$-strands, i.e. a braid half-twist in $N$-strands. Then the edges of the $N$-graph associated to $t_N$ along each of the three edges of $t_1$ are precisely given by the ordered terms in $\Delta_N$. That is, there exists an isotopy of the $N$-graph such that as one travels along an edge of $t_1$, the edges of the $N$-graph that we encounter are first $\tau_1$, then $\tau_2$ and $\tau_1$, then $\tau_3$, $\tau_2$, and $\tau_1$ and iteratively until reaching $\tau_1$ for the $(N-1)$th time. This is equivalent to the association $\tau(r_i)$, $i=1,\ldots,2N-3$, in the construction of $G(t_N)$ above.

For context, these permutations along the boundary are particularly relevant for the study of Legendrian surface weaves with boundary, whose Lagrangian projections yield interesting Lagrangian fillings of their Legendrian boundary links. The braid description of these Legendrian links is determined precisely by these permutations -- see Section \ref{sec:app2}. We see again, confer Remark \ref{rmk:1dim}, that it is useful to think of Legendrian weaves as two-dimensional Legendrian braids:  their one-dimensional boundaries are positive braids.


\subsection{Global Model} Given that the boundary conditions for the $N$-graphs in the local models for $t_N$ allow for gluing, we define the $N$-graph $G\sse C$ associated to a global $N$-triangulation of $C$ to be the $N$-graph obtained by concatenating the local models $G(t_N)$ along each triangle $t_N$ in the $N$-triangulation. We study the flag moduli space invariants for these $N$-graphs and their associated Legendrian weaves in Sections \ref{sec:app} and \ref{sec:app3}. 
Note that the genus of these Legendrian weaves increases as $N\in\N$, or the number of triangles, increases.

\begin{remark} Trivalent vertices are dual to triangulations of surfaces. In particular, triangulations of surfaces with a large group of symmetries yield particularly interesting $2$-graphs. From this perspective, Riemann surfaces with a conformal automorphism group of large order give rise to highly symmetric $2$-graphs. For instance, Riemann surfaces associated to tilings of the hyperbolic plane $\mathbb{H}^2$ with Schl\"afli symbol $\{n,3\}$ are highly symmetry, with $\{7,3\}$ being the Klein quartic, $\{8,3\}$ giving Bolza's surface and $\{12,3\}$ the $M(3)$ surface. We expect the flag moduli space associated to the Legendrian surfaces of these $2$-graphs, as defined in Section \ref{sec:flag}, to be algebraic spaces with correspondingly large symmetry. We begin an exploration of this kind with our Theorem \ref{thm:symmetries} in Section \ref{sec:app}.\hfill$\Box$
\end{remark}


\subsection{Bicubic graphs}\label{ssec:bipartite} Here is a second construction of $3$-graphs in a smooth surface $C$, strictly disjoint from the class of $3$-graphs arising from $3$-triangulations. 

By definition, a graph is \emph{bicubic} if it is both trivalent (cubic) and bipartite.  Now consider an embedded bicubic graph $G\sse C$, and replace each vertex of $G$ with a hexagonal vertex, doubling the edges as in Figure \ref{fig:BipartiteWeave}.

\begin{center}
	\begin{figure}[h!]
		\centering
		\includegraphics[scale=0.6]{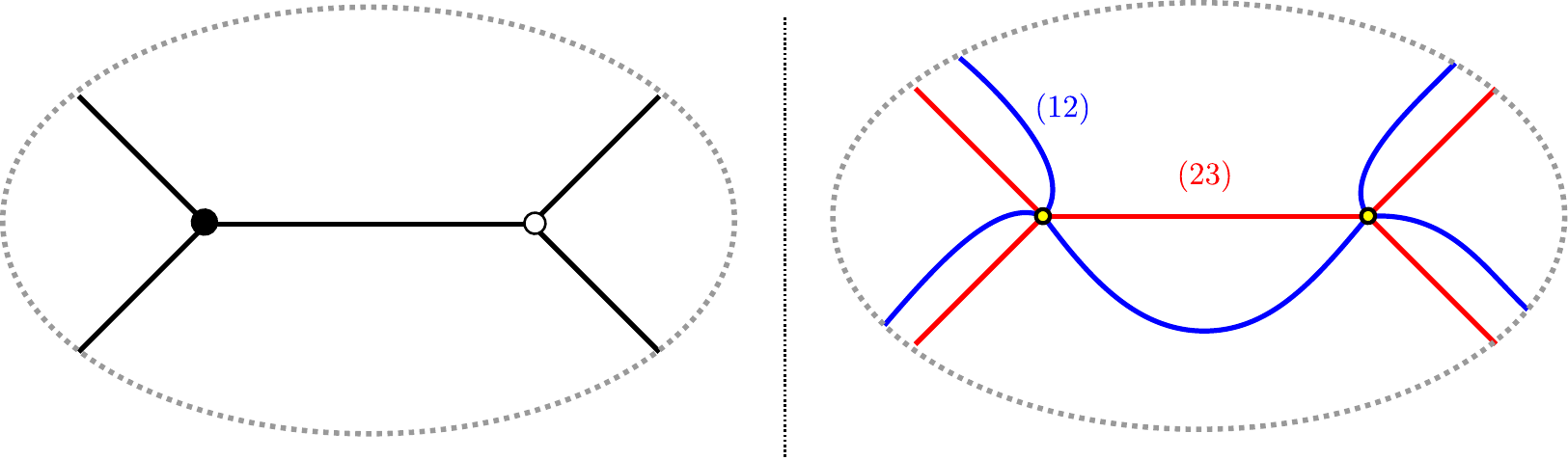}
		\caption{3-graph associated to an edge in a bicubic graph.}
		\label{fig:BipartiteWeave}
	\end{figure}
\end{center}

The bipartite condition on the graph guarantees that these local models can be glued together, uniquely up to isotopy, yielding a $3$-graph in $C$. Note that this $3$-graph is entirely built from hexagonal vertices, and no trivalent vertex is used.  As a result, the topology of the Legendrian weave associated to such a $3$-graphs is always that of a $3$-component link of Legendrian $2$-spheres. We will study a family of such $3$-graphs in Section \ref{sec:app}.

\begin{example}
\label{ex:cubegraph} The bicubic graph $G\sse \S^2$ associated to the 1-skeleton of a 3-dimensional cube, depicted in Figure \ref{fig:BipartiteWeave2}, yields a 3-component Legendrian link $\La(G)\sse(J^1(\S^2),\xi_\st)$. The flag moduli space $\SM(G)$ will show that these three Legendrian spheres, even after satellited to a Darboux ball $(\R^5,\xi_\st)$ are Legendrian knotted (and {\it smoothly} unknotted).\hfill$\Box$
\end{example}

\begin{center}
	\begin{figure}[h!]
		\centering
		\includegraphics[scale=0.6]{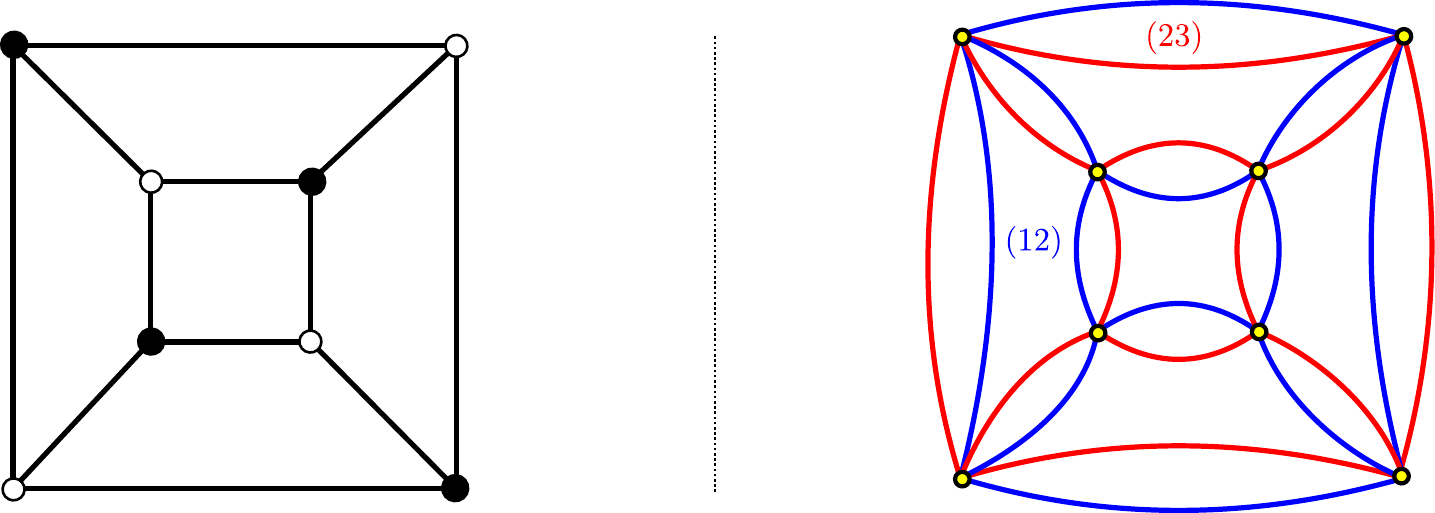}
		\caption{3-graph associated to the 1-skeleton of the cube and the 3-component Legendrian link of 2-spheres.}
		\label{fig:BipartiteWeave2}
	\end{figure}
\end{center}

\begin{remark} Not every 3-graph which is exclusively formed by hexagonal vertices arises from a bicubic graph, even up to candy-twist equivalence (see Section \ref{sec:moves}). In particular, two vertices may have just a single edge connecting them, with no vertices connected by three edges.  Figure \ref{fig:HexCounterExample} shows such an example.
	
	\begin{center}
		\begin{figure}[h!]
			\centering
			\includegraphics[scale=0.6]{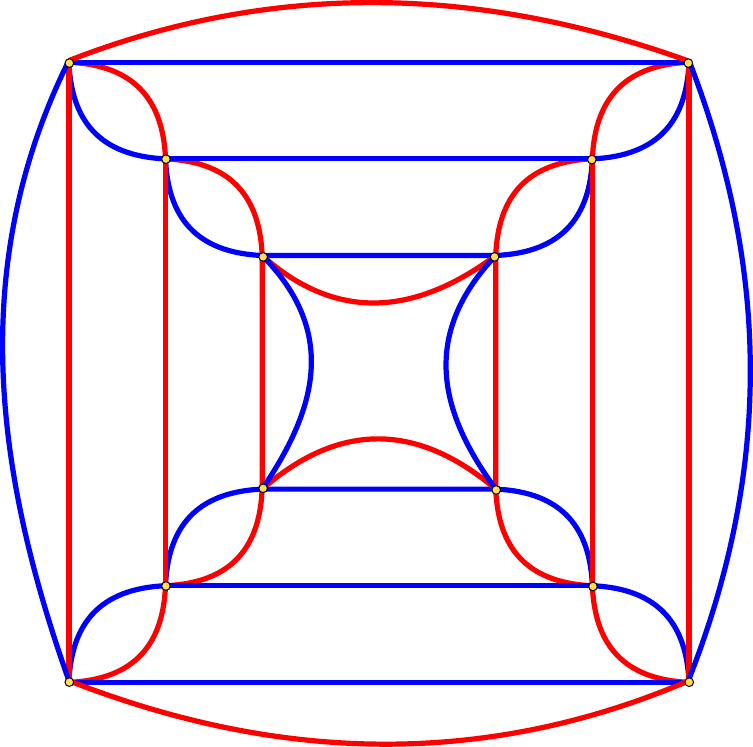}
			\caption{3-graph with only hexagonal vertices which does not arise from a bicubic graph, even up to Move I.}
			\label{fig:HexCounterExample}
		\end{figure}
	\end{center}

This example can be generalized in several ways: by adding more interior squares with one edge connecting adjacent vertices and/or replacing the squares with $2n$-gons.\hfill$\Box$
\end{remark}

\begin{ex}({\bf An Explosion of Examples.}) Bicubic graphs can be readily generated as follows.  Let $P'$ be a polytope, not necessarily regular, and $G'$ its edge graph, i.e. $G'$ is the one skeleton of $P'$. Suppose that $P'$ has $v'$ vertices, $e'$ edges and $f'$ faces. By definition, the \emph{explosion} of the polytope $P'$ is the polytope $P$ formed by first truncating at the vertices and then truncating the resulting polytope along the original edges of $P'.$  Then the 1-skeleton of $P$ is cubic and has a unique bipartite coloring, up to an overall black-white swap, so therefore is bicubic. Note that $P$ has $v = 4e'$ vertices, $e = 6e'$ edges, and $f = v' + e' + f'$ faces.

Even degenerate polytopes $P'$ give interesting examples.  For instance, if $P'$ is the degenerate polytope with two $n$-gon faces ($v' = n, e' = n, f' = 2$), then $P$ is a $2n$-gon prism ($v = 4n, e = 6n, f = 2n+2$).  The cube edge graph described in Example \ref{ex:cubegraph} is the bicubic graph which arises when $P'$ has just two bigon faces.\hfill$\Box$
\end{ex}


\section{Diagrammatic Calculus For Legendrian Weaves}\label{sec:moves}

Let $G\sse C$ be an $N$-graph.  The geometric objects that we are interested in are the Legendrian weaves $\La(G)\sse(J^1C,\xi_\st)$ and their invariants up to Legendrian isotopy. In this section we introduce a series of combinatorial operations that can be performed to an $N$-graph $G$, and we show how they affect the Legendrian isotopy type of $\La(G)$. The geometric understanding of the Legendrian isotopy type through this diagrammatic calculus allows us to significantly simplify computations of algebraic invariants associated to $\La(G)$ in Section \ref{sec:flag}.  Algebraic computations, using the results in this section, are detailed in Sections \ref{sec:app} and \ref{sec:app2}. Let us begin with the combinatorial moves in $G$ that preserve the Legendrian isotopy type of $\La(G)$.


\subsection{Surface Reidemeister Moves}\label{ssec:ReidemeisterMoves}

Let $\La\sse (J^1(C),\xi_\st)$ be a Legendrian surface, a Legendrian isotopy $\{\La_t\}_{\{t\in[0,1]\}}$ will generically induce singularities of the Legendrian fibration $J^1C\lr C\times \R$. As a result, the front sets $\Sigma(\La_t)$ and their singularities will restructure as the parameter $t\in[0,1]$ ranges along a 1-parameter family. These modifications of the Legendrian fronts are referred to as {\it perestroikas}, or Reidemeister moves \cite[Chapter 3]{ArnoldSing}.

\begin{remark} The three classical 1-dimensional Reidemeister moves have been the main method of study for smooth knots in geometric topology, since first introduced \cite{Reidemeister,AlexanderBriggs}. The corresponding seven moves for smooth surfaces are known as Roseman moves, after \cite[Theorem 1]{Roseman95}. The corresponding Legendrian Reidemeister, and Legendrian Roseman moves, for Legendrian knots, and Legendrian surfaces, follow from the classification of (stable) wavefront singularities in dimensions $\dim(\La)\leq3$ \cite[Theorem 13]{ArnoldICM74}. We will refer to Legendrian Roseman moves as surface Legendrian Reidemeister moves.\hfill$\Box$
\end{remark}

The combinatorial operations inducing surface Legendrian Reidemeister moves are the content of the following theorem. In the moves, the local pieces of the $N$-graphs are actually 3- or 4-graphs. The color code follows our standard notation: blue and red are adjacent colors (corresponding to adjacent transpositions), red and yellow are adjacent colors, and blue and yellow are disjoint colors.

\begin{thm}\label{thm:surfaceReidemeister} Let $G_1,G_2$ be one of the pairs of $N$-graphs depicted in Figures \ref{fig:Reidemeister1}, \ref{fig:Reidemeister2}, \ref{fig:Reidemeister3}, \ref{fig:Reidemeister4}, \ref{fig:Reidemeister5}, \ref{fig:Reidemeister6} and \ref{fig:Reidemeister6Prime}. Then the associated Legendrian surface $\La(G_1)$ is Legendrian isotopic to $\La(G_2)$ relative to their boundaries. That is, $\mbox{Moves I}, \mbox{II}, \mbox{III}, \mbox{IV}, \mbox{V}, \mbox{VI}$ and $\mbox{VI'}$ are local surface Legendrian Reidemeister moves.
\end{thm}

\begin{center}
		\begin{figure}[h!]
			\centering
			\includegraphics[scale=0.65]{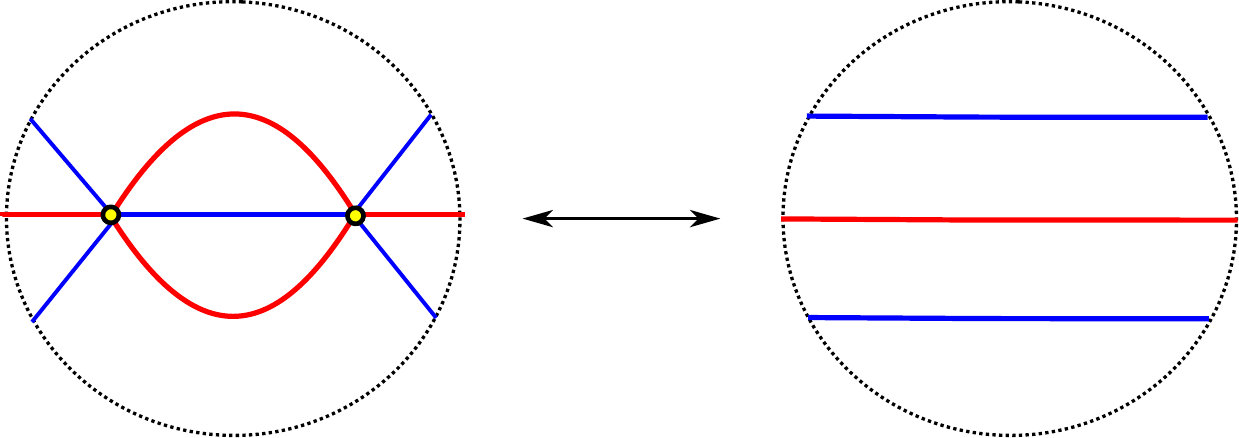}
			\caption{(Move I) The first pair of local $N$-graphs $G_1$, on the left, and $G_2$ on the right.  We refer to this move as \emph{candy twist}.}
			\label{fig:Reidemeister1}
		\end{figure}
\end{center}

\begin{center}
	\begin{figure}[h!]
		\centering
		\includegraphics[scale=0.65]{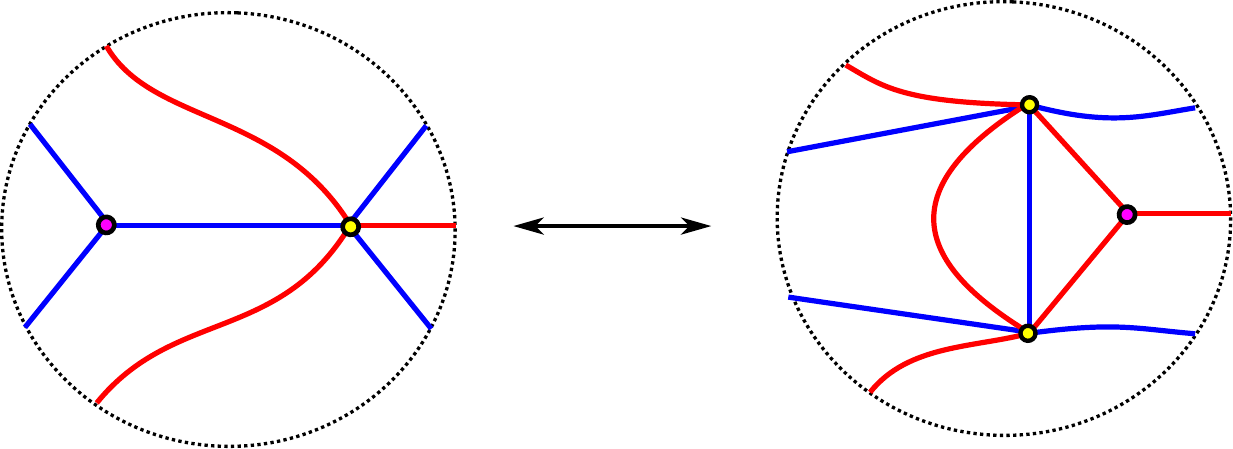}
		\caption{(Move II) The second pair of local $N$-graphs $G_1$, on the left, and $G_2$ on the right. We refer to this move as the \emph{push-through}, since the trivalent vertex gets pushed through the hexagonal vertex.}
		\label{fig:Reidemeister2}
	\end{figure}
\end{center}

\begin{center}
	\begin{figure}[h!]
		\centering
		\includegraphics[scale=0.65]{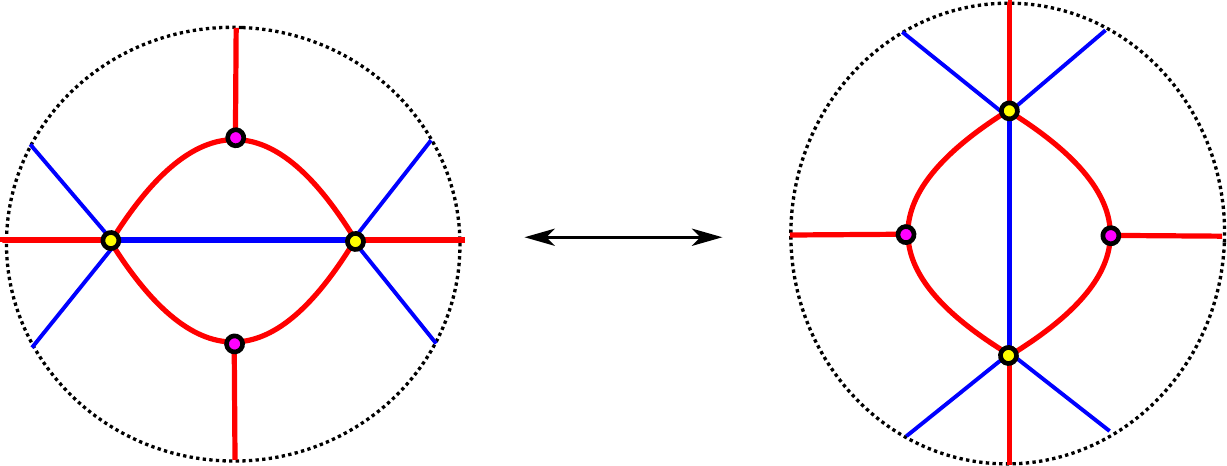}
		\caption{(Move III) The third pair of local $N$-graphs $G_1$, on the left, and $G_2$ on the right.  We refer to this move as the \emph{flop}.}
		\label{fig:Reidemeister3}
	\end{figure}
\end{center}

\begin{center}
	\begin{figure}[h!]
		\centering
		\includegraphics[scale=0.65]{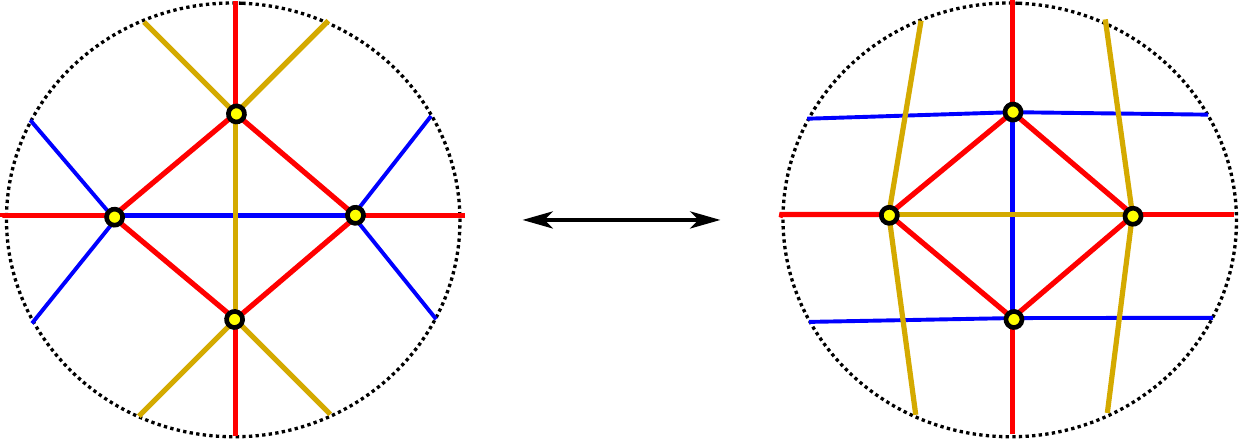}
		\caption{(Move IV) The fourth pair of local $N$-graphs $G_1$, on the left, and $G_2$ on the right. Note we must have $N\geq 4$. This moves implies the $A_3$ generalized Zamolodzhikov relation depicted in Figure \ref{fig:A3Zamolodzhikov}.}
		\label{fig:Reidemeister4}
	\end{figure}
\end{center}

\begin{center}
	\begin{figure}[h!]
		\centering
		\includegraphics[scale=0.65]{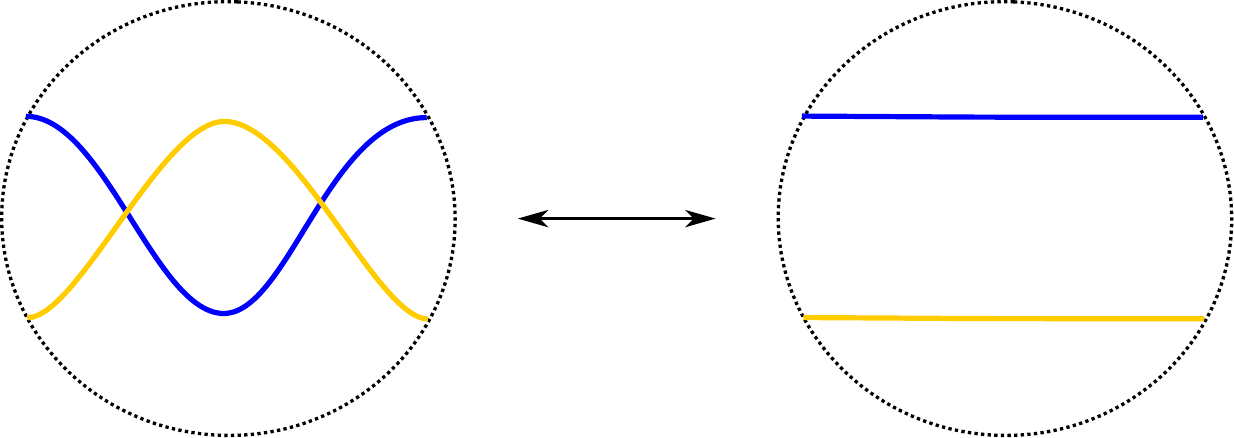}
		\caption{(Move V) The fifth pair of local $N$-graphs $G_1$, on the left, and $G_2$ on the right, with $N\geq 4$. The blue and yellow colors are associated to disjoint transpositions.}
		\label{fig:Reidemeister5}
	\end{figure}
\end{center}

\begin{center}
	\begin{figure}[h!]
		\centering
		\includegraphics[scale=0.65]{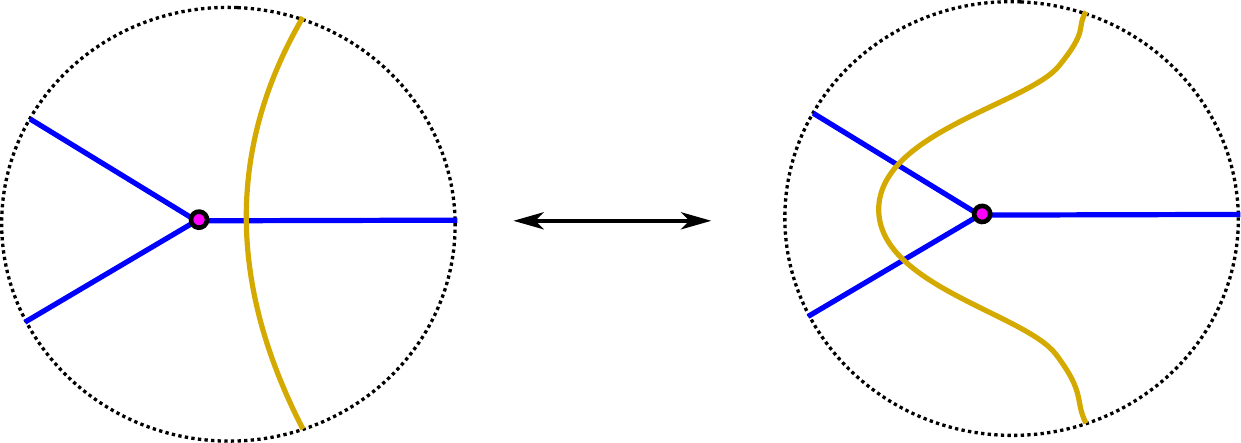}
		\caption{(Move VI) The sixth pair of local $N$-graphs $G_1$, on the left, and $G_2$ on the right, with $N\geq 4$.}
		\label{fig:Reidemeister6}
	\end{figure}
\end{center}

\begin{center}
	\begin{figure}[h!]
		\centering
		\includegraphics[scale=0.65]{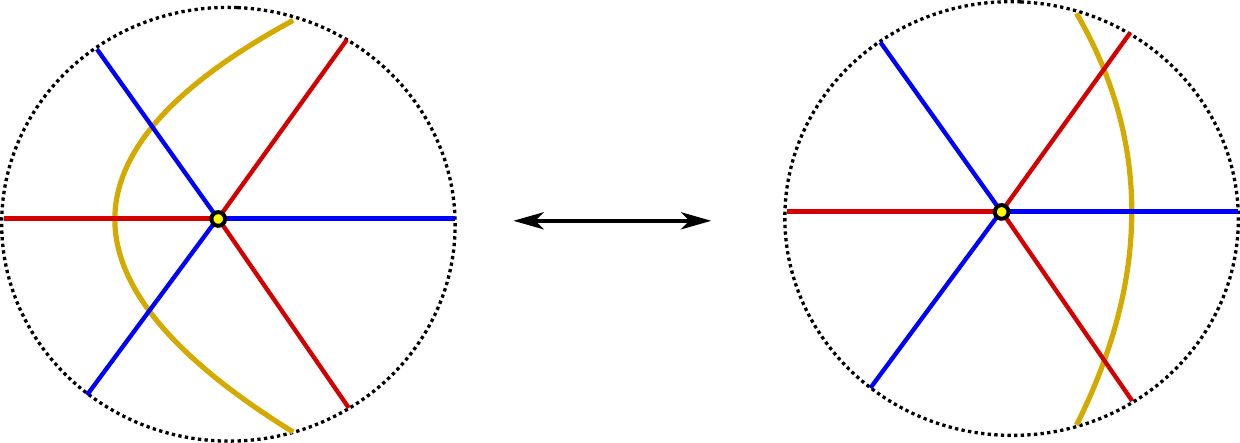}
		\caption{(Move VI') Variation on the sixth pair of local $N$-graphs $G_1$, on the left, and $G_2$ on the right, with $N\geq 4$.}
		\label{fig:Reidemeister6Prime}
	\end{figure}
\end{center}

\begin{proof} Let us start with Move I, the candy twist, as depicted in Figure \ref{fig:Reidemeister1}. It illustrates the method of proof for these surface Legendrian Reidemeister moves. There are essentially three equivalent viewpoints: exhibiting the Legendrian isotopy as $N$-graphs, visualizing the surface wavefronts explicitly in $(J^1\R^2,\xi_\st)$, or studying these surface wavefronts as families of (possibly singular) Legendrian links. In the first perspective, we need to justify that all the $N$-graphs lift to embedded Legendrian surfaces. In the second, the challenge is visualizing the actual front and ensuring that all the singularities lift to Legendrian embeddings. In the third perspective, visualization is simplified, with the trade-off of having to draw several movies of links. The second perspective is the strongest, as it readily implies the other two.
	
The first perspective is drawn in Figure \ref{fig:Reidemeister1Proof}. The left and rightmost 3-graphs lift to Legendrian weaves, yielding embedded Legendrian surfaces (with boundary). The diagram in the center of Figure \ref{fig:Reidemeister1Proof} does {\it not} immediately lift to an embedded Legendrian surface, as the six-valent vertex is {\it not} a hexagonal vertex -- the colors of the edges around it are not alternating, which is the condition for the hexagonal vertices introduced in Section \ref{sec:NGraphsLegWeaves}.
	
	\begin{center}
		\begin{figure}[h!]
			\centering
			\includegraphics[scale=0.8]{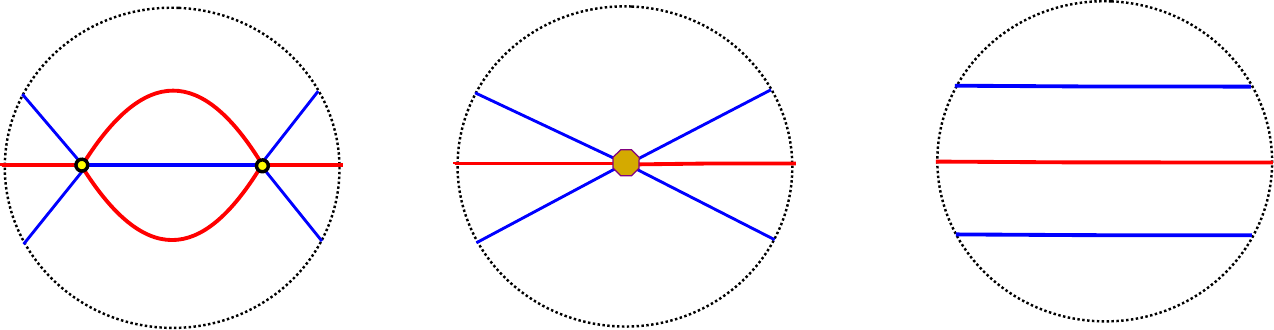}
			\caption{The 3-graph movie showing that the candy move - Move I - is a Legendrian isotopy. The geometric meaning of the central picture (not a 3-graph) is explained in the text.}
			\label{fig:Reidemeister1Proof}
		\end{figure}
	\end{center}

Nevertheless, the center diagram in Figure \ref{fig:Reidemeister1Proof} does in fact come from a Legendrian wavefront whose Legendrian lift is an embedded surface. Indeed, we have depicted such a front in the second front of Figure \ref{fig:Reidemeister1Proof3}.

\begin{center}
	\begin{figure}[h!]
		\centering
		\includegraphics[scale=0.85]{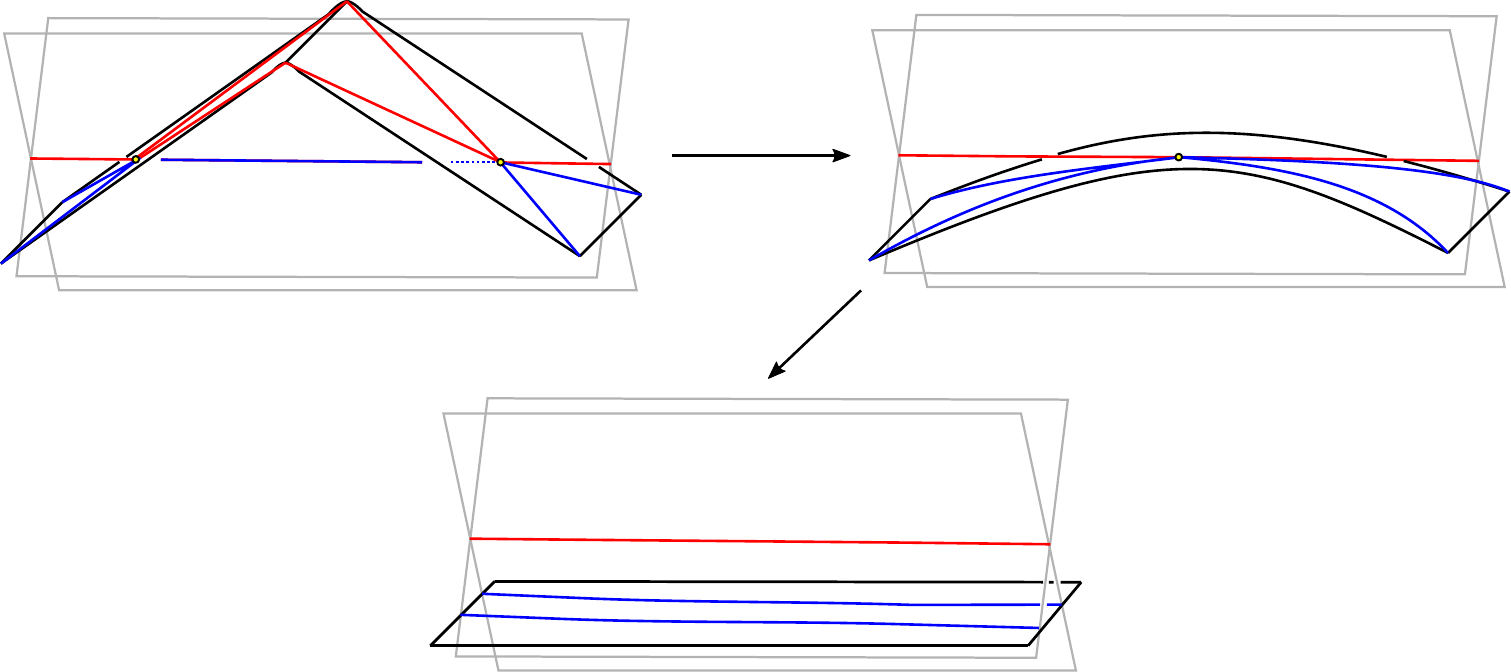}
		\caption{The homotopy of Legendrian wavefronts associated to Move I.}
		\label{fig:Reidemeister1Proof3}
	\end{figure}
\end{center}

The movie of wavefronts in Figure \ref{fig:Reidemeister1Proof3} geometrically constructs the homotopy of Legendrian fronts which lifts to the Legendrian isotopy corresponding to Move I. The three fronts in Figure \ref{fig:Reidemeister1Proof3} lift to {\it embedded} Legendrian surfaces, as the singularities are all Legendrian and there are no vertical tangent planes. The singularities at the beginning of Figure \ref{fig:Reidemeister1Proof3} are segments of $A_1^2$-crossings, and two isolated $A_1^3$ points. The singularities at the end of Figure \ref{fig:Reidemeister1Proof3} are just segments of $A_1^2$-crossings. The singularity in the middle of the movie, not corresponding to an $A_1^2$ segment, is {\it not} a stable front singularity, but it does lift to an embedded Legendrian surface, and thus the homotopy of fronts actually represents a Legendrian isotopy. Indeed, the tangent spaces at that singularity intersect transversely, and hence their lifts are disjoint. This concludes that Move I combinatorially represents a surface Legendrian Reidemeister move.

\begin{remark}
For completeness, in Figure \ref{fig:Reidemeister1Proof2} we have drawn the homotopy of surface fronts from Figure \ref{fig:Reidemeister1Proof3} as a movie (of movies).  It is thus a 2-homotopy of Legendrian links.  These three movies of links, one per each column, are obtained by slicing each of the respective fronts in Figure \ref{fig:Reidemeister1Proof3} from left to right. This is the third viewpoint we mentioned above.\hfill$\Box$

\begin{center}
	\begin{figure}[h!]
		\centering
		\includegraphics[scale=0.7]{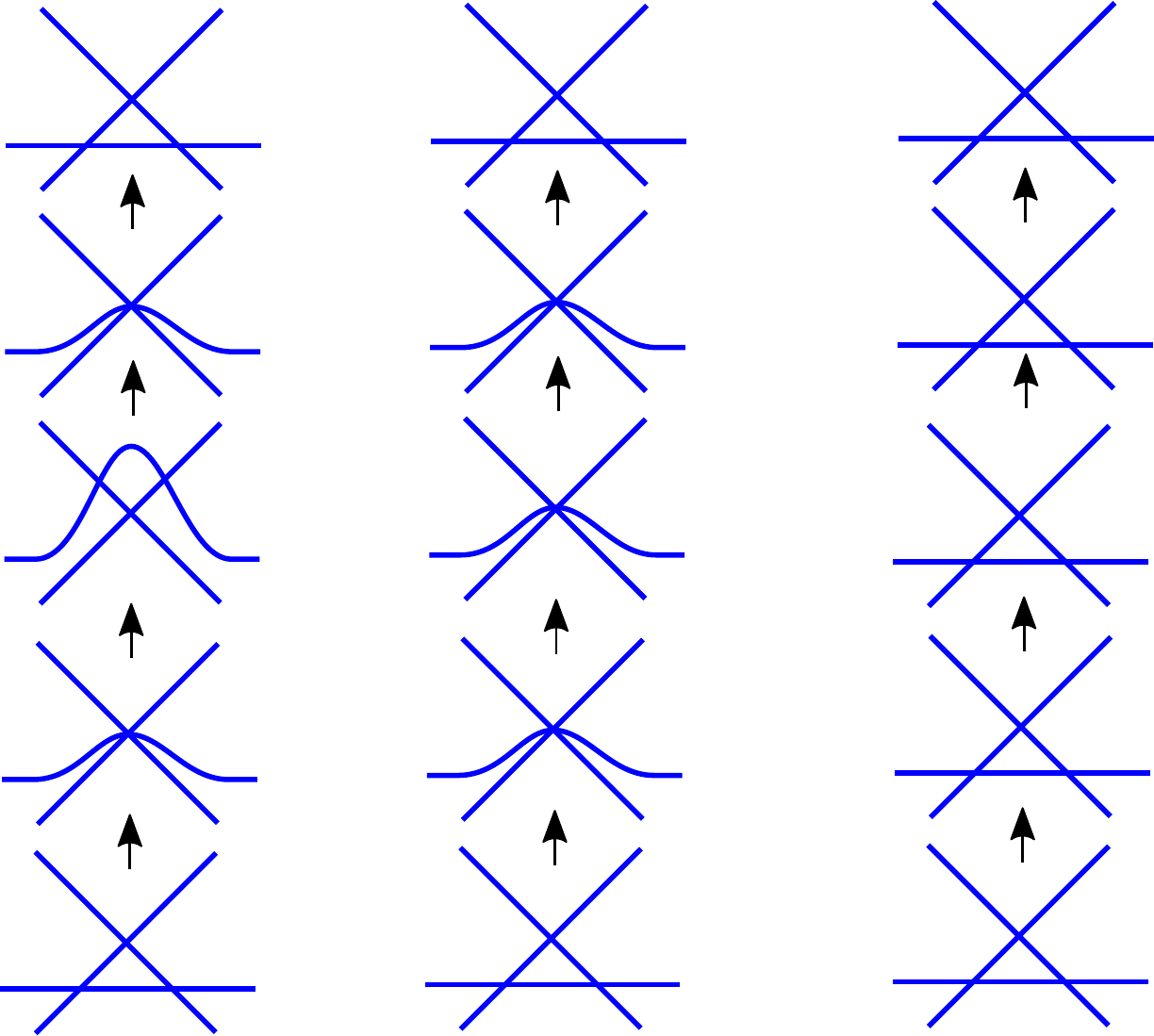}
		\caption{The proof that Move I is a Legendrian isotopy by (transversely) slicing each of the Legendrian wavefronts in Figure \ref{fig:Reidemeister1Proof3}.}
		\label{fig:Reidemeister1Proof2}
	\end{figure}
\end{center}
\end{remark}

Let us now justify Move II, where a $D_4^-$-singularity pushes-through an $A_1^3$-singularity. The resulting front has a $D_4^-$-singularity and {\it two} $A_1^3$-singularities. The clearest proof that this is a Legendrian isotopy comes from carefully drawing and examining the right homotopy of fronts. In this case, the required movie of fronts is depicted in Figure \ref{fig:Reidemeister2Proof1}. These Legendrian fronts start with the front whose $A_1^2$-singularities yield the 3-graph $G_1$ on the left of Move II, and end with the front whose $A_1^2$-singularities yield the 3-graph $G_2$ on the right of Move II.

\begin{center}
	\begin{figure}[h!]
		\centering
		\includegraphics[scale=1]{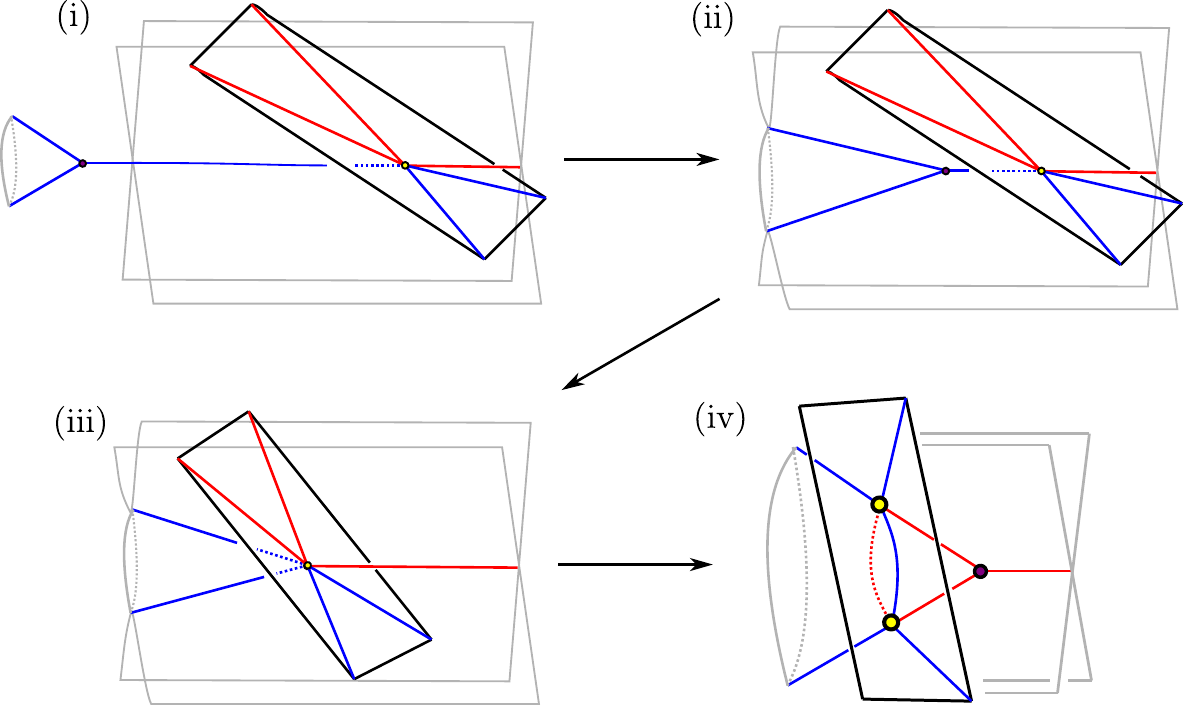}
		\caption{The homotopy of Legendrian fronts inducing Move II.}
		\label{fig:Reidemeister2Proof1}
	\end{figure}
\end{center}

These fronts describe a neighborhood $\R^3$ of a $D_4^-$-singularity with a 2-plane $\Pi\sse\R^3$ which starts away from the $D_4^-$-singularity. This 2-plane $\Pi$ is drawn with a tilt in its slope. The homotopy of fronts consists of this 2-plane $\Pi$ moving towards the $D_4^-$-singularity and {\it crossing} through it. There exists a unique moment in this isotopy in which the $D_4^-$-singularity is contained in the 2-plane $\Pi$. The $A_1^2$-singularities right before that moment give rise to $G_1$ for Move II, and right after this moment the $A_1^2$-singularities give rise to $G_2$ for Move II. Since the 2-plane $\Pi$ is not vertical, and the tangent 2-planes of the different branches at the $D_4^-$-singularity in all moments are distinct, each of the fronts in this homotopy lift to embedded Legendrian surfaces. Thus, the movie of fronts in Figure \ref{fig:Reidemeister2Proof1} shows that there exists a Legendrian isotopy with $A_1^2$-singularities as dictated by Move II, and $\La(G_1)$ and $\La(G_2)$ are Legendrian isotopic relative to their boundaries. This concludes Move II.

For Move III, we can proceed analogously by drawing a homotopy of fronts which lifts to a Legendrian isotopy. Nevertheless, Move III can actually be deduced as a combination of Moves I and II.  We leave it as an exercise for the reader to visualize the spatial Legendrian fronts, and instead explain how to deduce Move III from the previous two moves, as follows. Starting with one side of Move III, push both trivalent vertices through in the clockwise direction using Move II. This is depicted in the first two steps of Figure \ref{fig:Reidemeister3Seq}.

\begin{center}
	\begin{figure}[h!]
		\centering
		\includegraphics[scale=0.7]{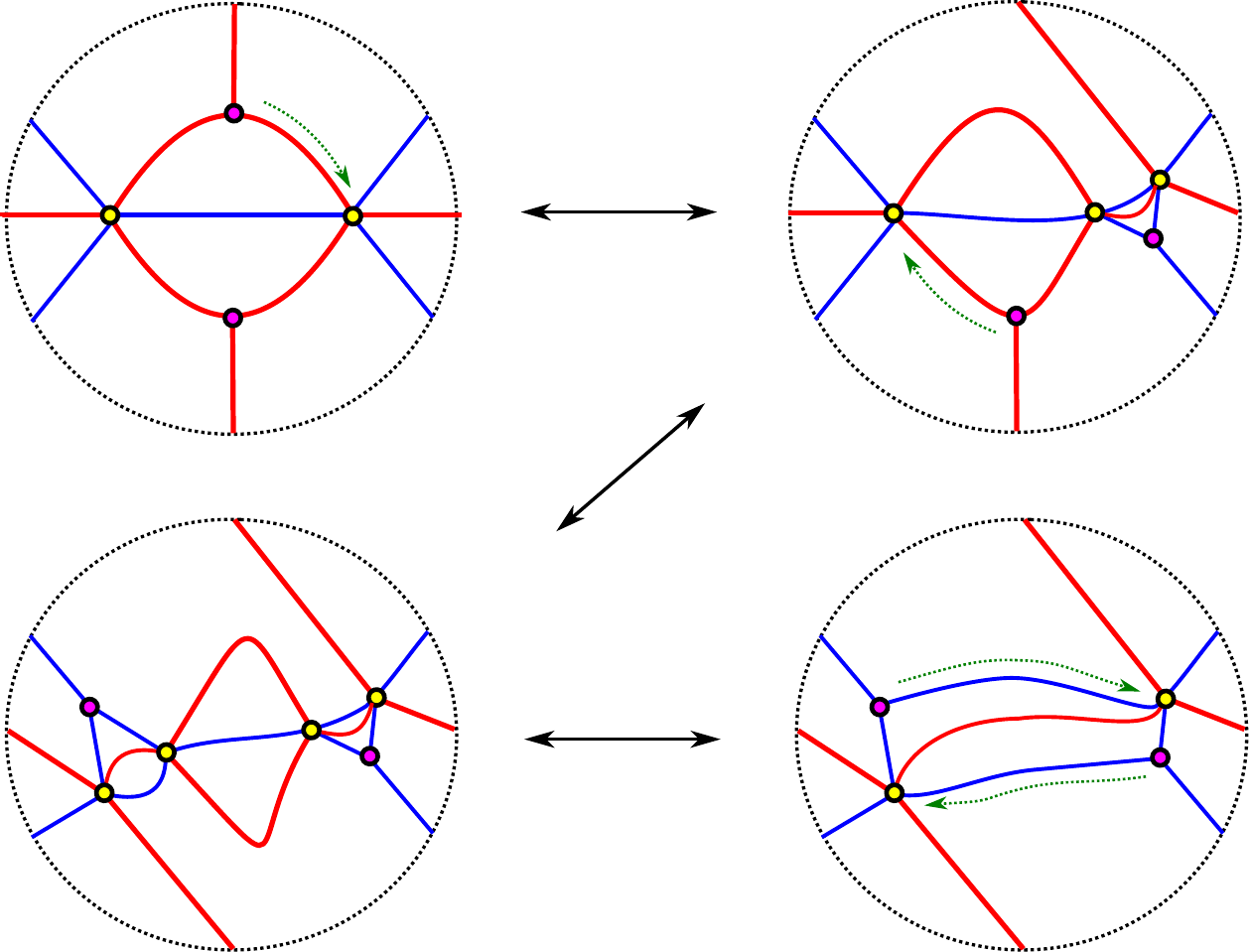}
		\caption{Deducing Move III from Moves I and II. The first step pushes a red trivalent vertex through a hexagonal vertex. The second step does the same for the other red trivalent vertex. The third step is a simplification undoing a candy twist. The dashed green lines indicate push-through moves that are about to occur.}
		\label{fig:Reidemeister3Seq}
	\end{figure}
\end{center}

This creates additional hexagonal vertices and the two trivalent vertices do change color. Perform Move II twice more, pushing-through these trivalent vertices again, and then cancel  two pairs of hexagonal vertices with a candy twist (Move I) to obtain the right hand side of Move III. Alternatively, first undo the candy twist as in the third step of Figure \ref{fig:Reidemeister3Seq}: this yields a 3-graph which is identical to a partial rotation of the initial 3-weave with red and blue switched. Iterating this again, i.e. pushing the two blue trivalent vertices through, as indicated by the dashed green lines in Figure \ref{fig:Reidemeister3Seq}, and undoing a candy twist yields the right hand side of Move III.

Let us now show that Move IV is a Legendrian isotopy. The corresponding spatial wavefronts consist of configurations of four 2-planes. The graph $G_1$ on the left of Move IV is obtained as the $A_1^2$-singularities, i.e. intersections, of the union of the four 2-planes

$$\pi_x=\{(x,y,z)\in\R^3:x+0.0001z=0\},\quad \pi_y=\{(x,y,z)\in\R^3:y+0.0001z=0\},$$
$$\pi_z=\{(x,y,z)\in\R^3:z=0\}, \quad \pi_1=\{(x,y,z)\in\R^3:x+y+z=1\}.$$

These intersections and 2-planes are depicted, with the corresponding colors, in Figure \ref{fig:Reidemeister4_3DVersion}. Now consider the 2-planes $\pi_t=\{(x,y,z)\in\R^3:x+y+z=t\}$, $t\in[-1,1]$. The homotopy of spatial wavefronts is locally given by the union $\pi_x\cup\pi_y\cup\pi_z\cup\pi_t$, $t\in[-1,1]$.

\begin{center}
	\begin{figure}[h!]
		\centering
		\includegraphics[scale=0.55]{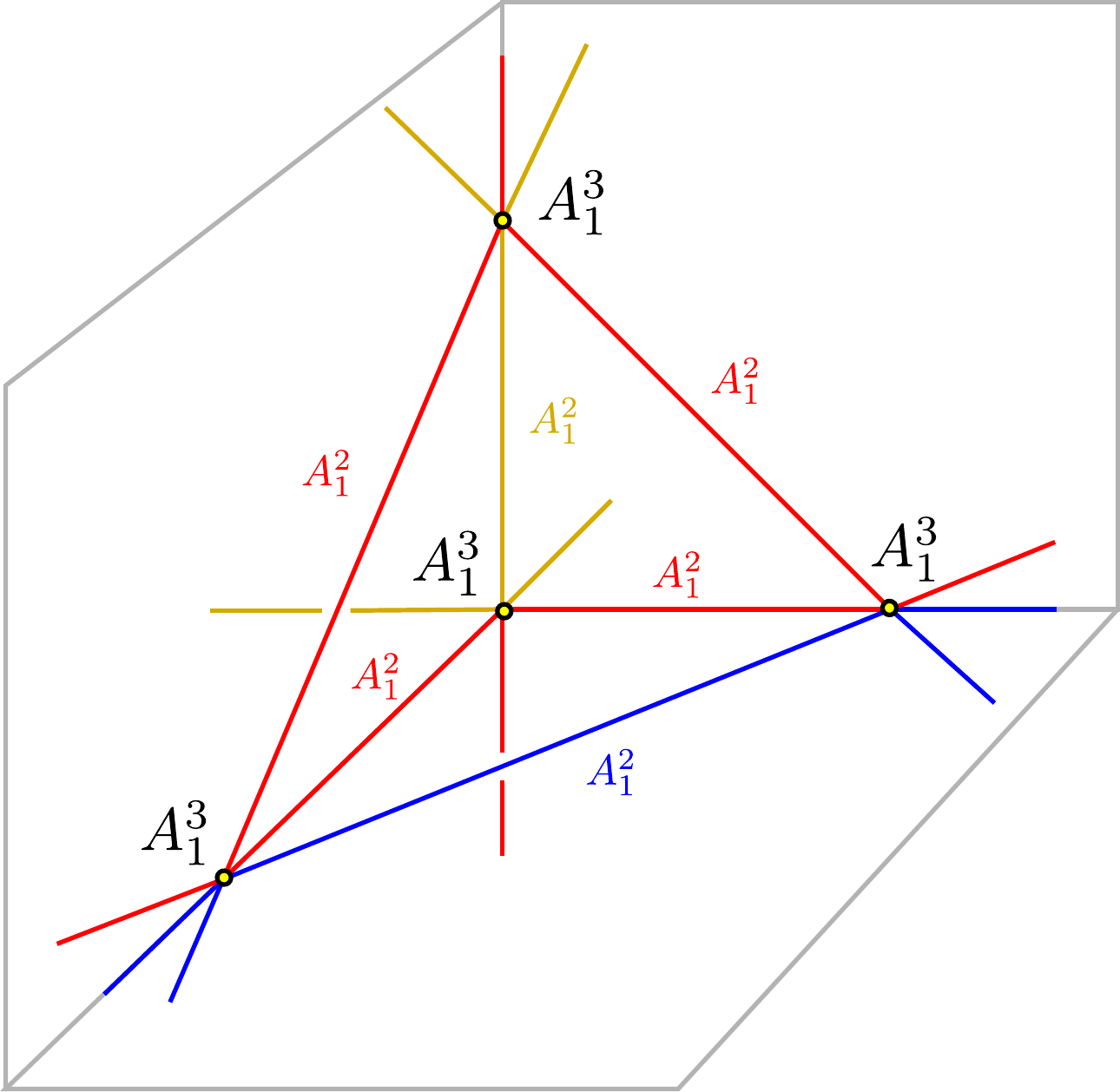}
		\caption{Front for the start of Move IV. The lines depict the intersections of the union of the four 2-planes $\pi_x\cup\pi_y\cup\pi_z\cup\pi_1$.}
		\label{fig:Reidemeister4_3DVersion}
	\end{figure}
\end{center}

This homotopy is not relative to the boundary, as the 2-planes $\pi_t$, $t\in[-1,1]$, change the boundary conditions --- but this is easily corrected by only pushing a compact piece of $\pi_t$, $t\in[-1,1]$ through the triple intersection point $\pi_x\cap\pi_y\cap\pi_z$. The $A_1^2$-singularity pattern of the resulting spatial wavefront is precisely as in the right graph $G_2$ in Move IV, as required.

Let us now address Move V, which depicts the local transition between two 4-graphs $G_1,G_2$ in Figure \ref{fig:Reidemeister5}. The corresponding spatial fronts consist of four 2-planes $\pi_1,\pi_2,\pi_3,\pi_4\sse\R^3$, where the only non-empty intersections are $\pi_1\cap\pi_2$, corresponding to the blue segment in $G_1$ (and $G_2$), and $\pi_3\cap\pi_4$, corresponding to the yellow segment in $G_1$, and $G_2$.
\begin{center}
	\begin{figure}[h!]
		\centering
		\includegraphics[scale=0.4]{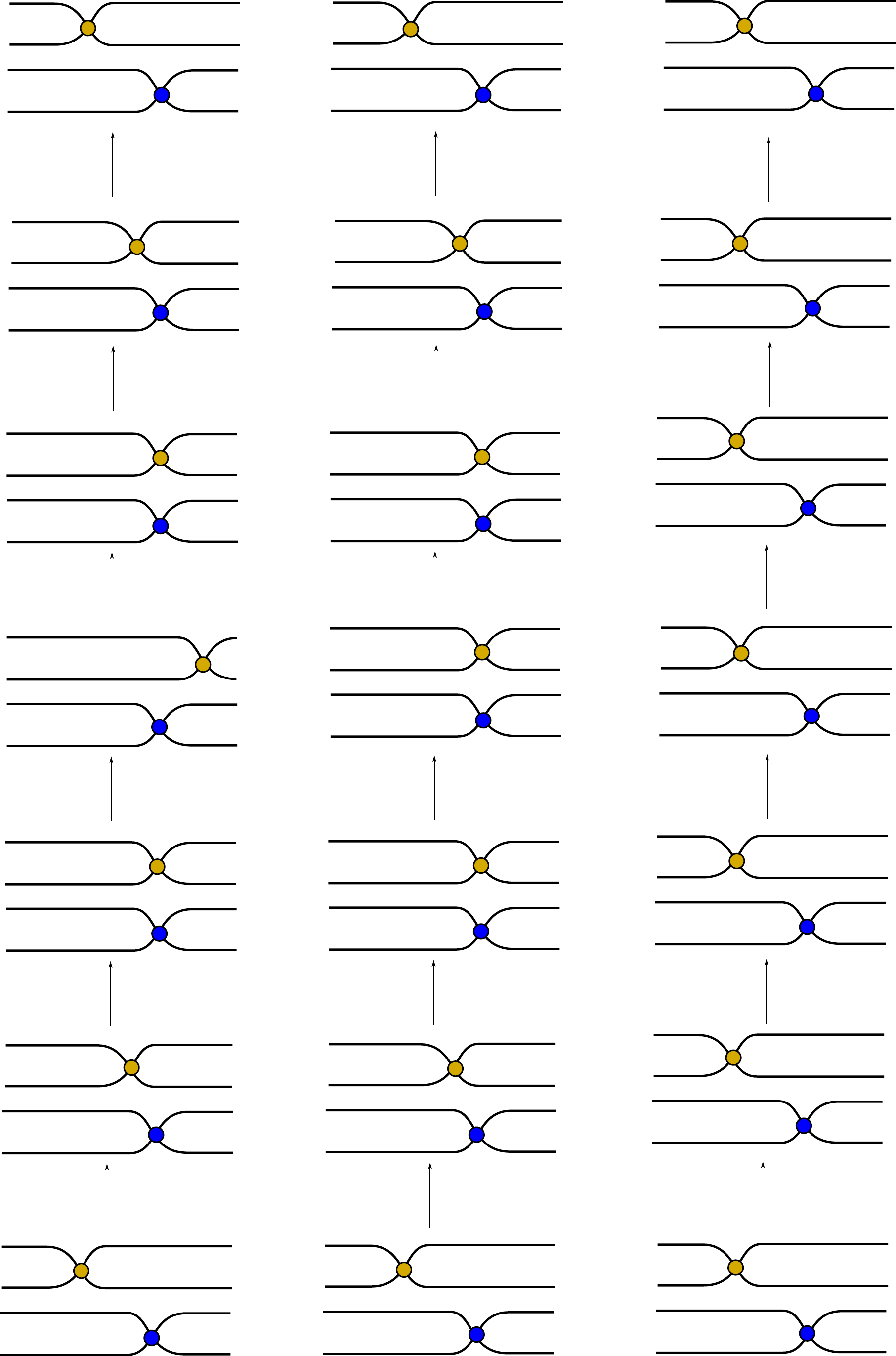}
		\caption{Each column represent slices of a spatial front. The $A_1^2$-singularities of the left column gives rise to $G_1$ in Figure \ref{fig:Reidemeister5}, and the $A_1^2$-singularities of the right column gives rise to $G_2$.}
		\label{fig:Reidemeister5Proof}
	\end{figure}
\end{center}
The fact that the fronts giving $G_1$ and $G_2$ are homotopic as Legendrian fronts is proven in Figure \ref{fig:Reidemeister5Proof}. Each of the columns in the figure represents a spatial surface front, with the links in the columns corresponding to slices. The corresponding intersections, dictating the $A_1^2$-singularities, are marked with the same color as in Figure \ref{fig:Reidemeister5}. The union of these slices in Figure \ref{fig:Reidemeister5Proof} yield spatial fronts which lift to embedded Legendrian surfaces, and thus the movie of columns in Figure \ref{fig:Reidemeister5Proof} exhibits a Legendrian isotopy from $\La(G_1)$ to $\La(G_2)$. Therefore, Move V is a surface Legendrian Reidemeister move. Move VI in Figure \ref{fig:Reidemeister6} follows with the same argument as for Move V, with a segment of $A_1^2$-singularities passing above, and disjointly, a $D_4^-$-singularity --- and likewise for Move VI'.  This concludes the proof of Theorem \ref{thm:surfaceReidemeister}.
\end{proof}

\begin{remark}
The Legendrian Reidemeister moves in Theorem \ref{thm:surfaceReidemeister} provide a symplectic geometric realization of A-type Soergel calculus. Moves I and V should be compared to \cite[Figure 4.4]{Soergel3}. Move II and Move VI are known as two-color associativity of type $A_1\times A_1$, with Coxeter exponent $m_{st}=2$, and of type $A_2$, with Coxeter exponent $m_{st}=3$, and Move IV corresponds to the $A_3$ relation \cite[Figure 4.7]{Soergel3}. It should be emphasized that the notation in Soergel calculus follows the notation for (rank three) parabolic subgroup of finite Coxeter groups, whereas we use the notation for Lie algebras whose irregular Weyl orbits yield spatial wavefronts. See Appendix \ref{ssec:SoergelCalculus} for further details.\hfill$\Box$
\end{remark}

Theorem \ref{thm:surfaceReidemeister} contains the Reidemeister moves that we use in the course of the article. They are all the possible (generic) Legendrian surface moves with only $D_4^-$ and $A_1^2$ Legendrian singularities in the endpoints of the Legendrian isotopy. The complete set of surface Reidemeister moves \cite[Section 3.3]{ArnoldSing} also includes the moves associated to the $A_4$ and $D_4^+$-singularities, which will require the interaction of $A_2$-cusp edges $A_2$ and $A_3$-swallowtails.

Theorem \ref{thm:surfaceReidemeister} allows one to make local modifications to an $N$-graph $G_1$ and obtain an $N$-graph $G_2$ such that the Legendrian surfaces $\La(G_1)\cong \La(G_2)\sse(J^1C,\xi_\st)$ are Legendrian isotopic. For the case $C=\S^2$, we define in Subsection \ref{ssec:Stabilization} an additional combinatorial move, which we refer to as a {\it stabilization}, going from an $N$-graph $G_1$ to a $(N+1)$-graph $G_2$. This requires a discussion on satellite constructions for Legendrian weaves, which is useful on its own, and also needed for Subsection \ref{ssec:Legsurgeries}.


\subsection{Legendrian Satellite Weaves}\label{ssec:Satellite} Let $G\sse C$ be an $N$-graph.  The Legendrian surface $\Lambda(G)$ defined by the weaving construction lies in the contact 5-manifold $(J^1C,\xi_\st)$.  Now, consider a contact 5-manifold $(Y,\xi)$ and a Legendrian embedding $\iota:C\longrightarrow (Y,\xi)$.  The Weinstein Neighborhood Theorem \cite[Section 7]{WeinsteinNeigh71} for Legendrian submanifolds gives a contactomorphism

$$\wt\iota:(J^1C,\xi_\st)\longrightarrow (\Op(\iota(C)),\xi|_{\Op(\iota(C))}),$$

where $\Op(A)$ is a sufficiently small neighborhood of $A\sse Y$, and such that the restriction to the zero section $C\sse J^1C$ is the initial Legendrian embedding $\iota$. In particular, any Legendrian $\La\sse (J^1C,\xi_\st)$ yields a Legendrian $\wt\iota(\La)\sse (Y,\xi)$. Thus, the contact 1-jet spaces serve as {\it local} contact manifolds, and a Legendrian embedding of $C$ in an arbitrary ambient contact 5-manifold allows one to embed a Legendrian weave there as well. In this context, the Legendrian surface $\wt\iota(\La)\sse (Y,\xi)$ is called the $\iota$-{\it satellite} of $\La\sse (J^1C,\xi_\st)$ and the Legendrian surface $\iota(C)\sse (Y,\xi)$ is called the {\it companion}. This terminology parallels the theory of satellite knots, as introduced in \cite{Schubert53_Satellite}, and see also \cite{Satellite2,Satellite1}. Notice that the smooth topology of $\La$ and its satellite $\wt\iota(\La)$ is identical, only the ambient contact manifold (and thus the Legendrian embedding type) are affected by this Legendrian satellite construction.

\begin{example}\label{ex:stdsatellite}
Let $(Y,\xi)=(\S^5,\xi_\st)$, $C=\S^2$, and let $\iota=\iota_0$ be the Legendrian embedding of the standard Legendrian unknot $\iota_0:\S^2\lr\S^5$. Given any Legendrian $\La\sse (J^1\S^2,\xi_\st)$, we will refer to $\wt\iota_0(\La)\sse (\S^5,\xi_\st)$ as the standard satellite of $\La$. Since $(\S^5\setminus \{{\rm pt}\},\xi_\st) \cong (\bR^5,\xi_\st)$, and the image $\wt\iota_0(\La)$ will avoid some point, this surface can be equivalently considered in $\wt\iota_0(\La)\sse (\R^5,\xi_\st)\cong (J^1(\bR^2),\xi_\st)$. It can thereupon be described by its front projection to $\bR^3 = \bR^2 \times \bR$. This is depicted in Figure \ref{fig:SatelliteUnknot}.\hfill$\Box$
\end{example}

\begin{center}
	\begin{figure}[h!]
		\centering
		\includegraphics[scale=0.65]{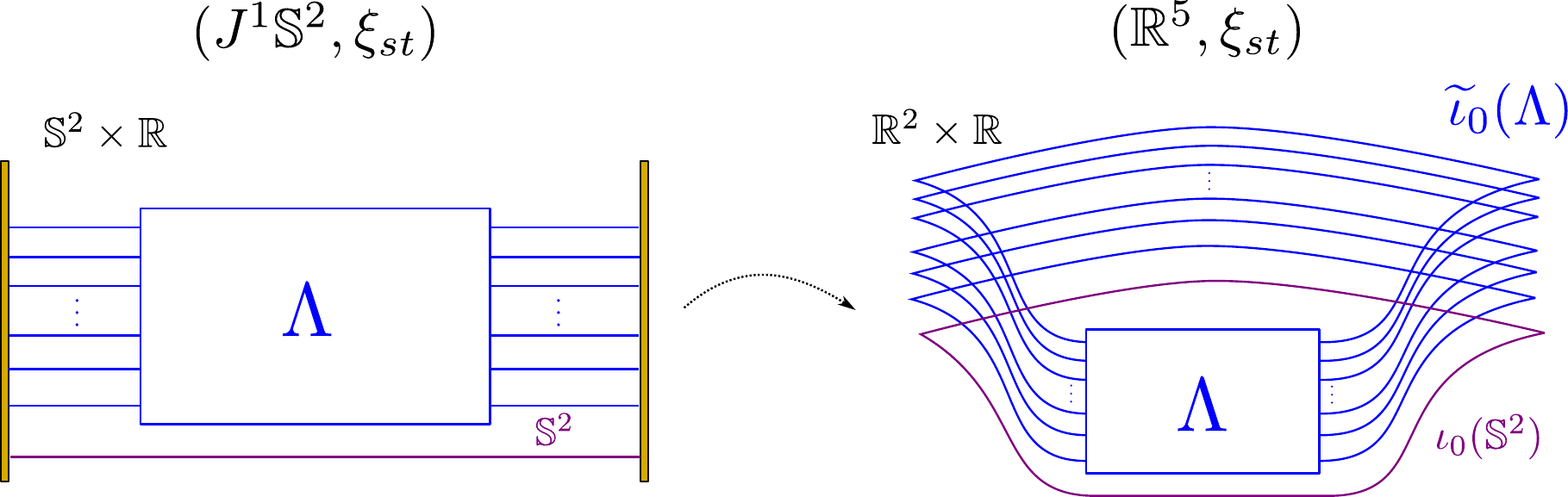}
		\caption{A Legendrian surface $\La\sse(J^1\S^2,\xi_\st)$ drawn in the front projection $\S^2\times\R$ (Left). The satellite $\wt\iota_0(\La)$ of $\La$ along the standard Legendrian unknot $\iota_0:\S^2\lr(\R^5,\xi_\st)$, drawn in the front projection $\R^2\times\R$ (Right). These pictures are schematic and ought to be rotated symmetrically along their central vertical axis so that the wavefronts for $\La$ and $\wt\iota_0(\La)$ are indeed surfaces in a 3-dimensional ambient space.}
		\label{fig:SatelliteUnknot}
	\end{figure}
\end{center}

In case no Legendrian embedding $\iota$ is specified and $C=\S^2$, the notation $\iota(\La)$ will implicitly refer to the {\it standard} satellite $\wt\iota_0(\La)\sse(\R^5,\xi_\st)$ as in Example \ref{ex:stdsatellite} and Figure \ref{fig:SatelliteUnknot}. It is often the case that the Legendrians $\La(G)\sse (J^1C,\xi_\st)$ that we introduce in this work do not have an {\it a priori} name nor they have been previously studied. Interestingly, for a certain variety of graphs $G\sse \S^2$ we will see how their standard Legendrian satellites are actually related to well-known Lagrangian surfaces, e.g. see Subsection \ref{ssec:FirstComputations}.

In addition, and in line with Markov's Theorem for smooth 1-dimensional braids \cite{Birman74Braids,PrasolovSossinsky}, the satellite operation is also required for a meaningful {\it stabilization} operation. Finally, note also that even if $\La(G)\sse (J^1C,\xi_\st)$ has no $A_2$-cusp edges, the spatial wavefronts for its standard satellite $\wt\iota_0(\La(G))$ will {\it always} have $A_2$-cusp edges, as any front for the standard Legendrian unknot $\La_0\sse(\R^5,\xi_\st)$ must have $A_2$-cusp edges. We now discuss $A_2$-cusp edges and $A_3$-swallowtail singularites, which are required for such a {\it stabilization} operation and Theorems \ref{thm:Legsurgeries} and \ref{thm:LegMutations} below, regarding Legendrian surgeries and Legendrian mutations.


\subsection{Cusp Edges and Swallowtail Singularities}\label{ssec:cuspedges} Let $G\sse C$ be an $N$-graph, the Legendrian weave $\La(G)\sse(J^1C,\xi_\st)$ associated to $G$ is determined by its front $\pi(\La(G))\sse C\times\R$. By definition, these fronts only have $D_4^-,A_1^2$ and $A_1^3$ singularities. The latter two are stable, i.e. a generic Legendrian isotopy $\La_t\sse(J^1C,\xi_\st)$, $t\in[0,1]$, such that $\La(G)=\La_0$, will have each of the $A_1^2$ and $A_1^3$ singularities of the front $\pi(\La_0)$ persist for $\pi(\La_t)$, $t\in(0,1]$. In contrast, $D_4^-$ is not:  the fronts $\pi(\La_t)$, $t\in(0,\varepsilon]$, will {\it not} have any $D_4^-$-singularity for $\varepsilon\in\R^+$ small enough.

	\begin{center}
	\begin{figure}[h!]
		\centering
		\includegraphics[scale=0.7]{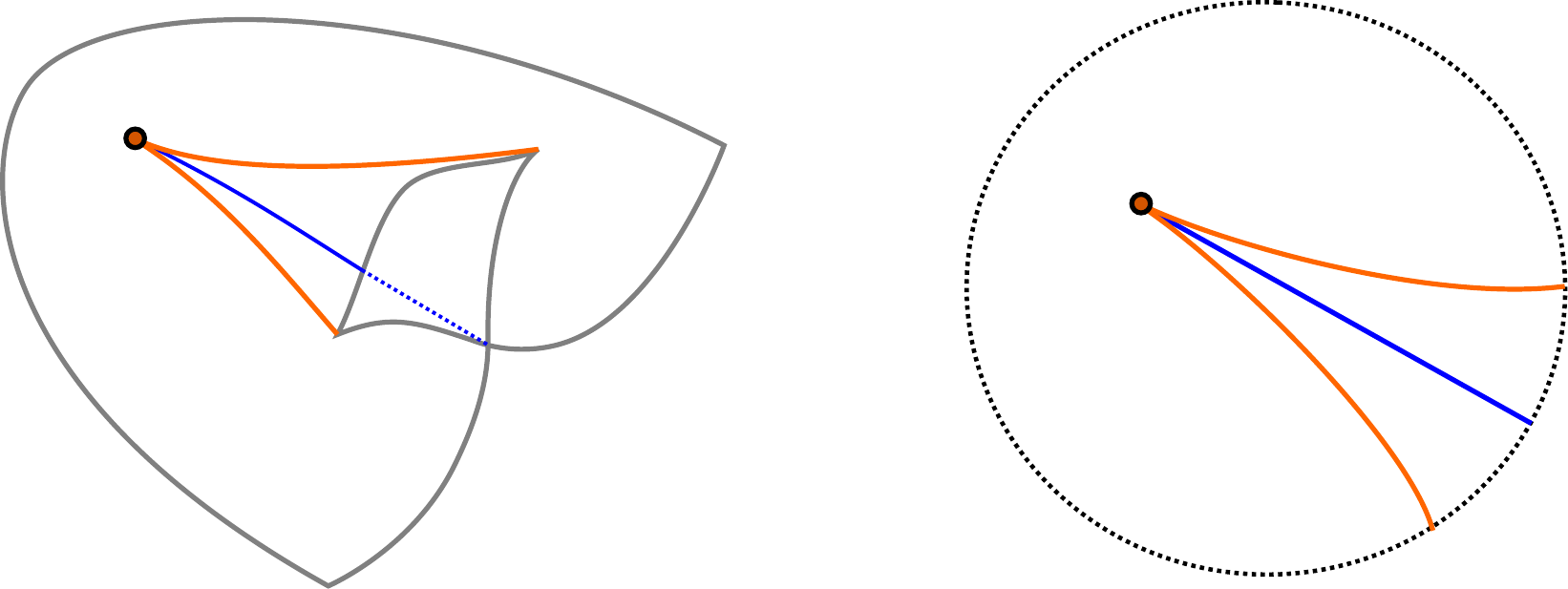}
		\caption{The Legendrian front of an $A_3$-swallowtail singularity (left). The planar diagrammatic depiction in our calculus (right).}
		\label{fig:A3Swallowtail}
	\end{figure}
\end{center}

The generic (stable) singularities of fronts in $3$-dimensional space are $A_1^2,A_1^3,A_2,A_2A_1$ and $A_3$, as shown in \cite[Section 3.2]{ArnoldSing}. These singularities are depicted in Figure \ref{fig:StableSing}. The appearance of $A_2,A_2A_1$ and $A_3$ singularities in a generic front forces us to extend our combinatorial diagrammatics, as our Legendrian isotopies will (typically) be generic. In the figures for this subsection, and only this subsection, we will draw edges around a hexagonal vertex with the same color -- this will simplify our diagrams, which are no longer $N$-graphs due to the presence of $A_2$-cusp edges.

\begin{center}
	\begin{figure}[h!]
		\centering
		\includegraphics[scale=0.7]{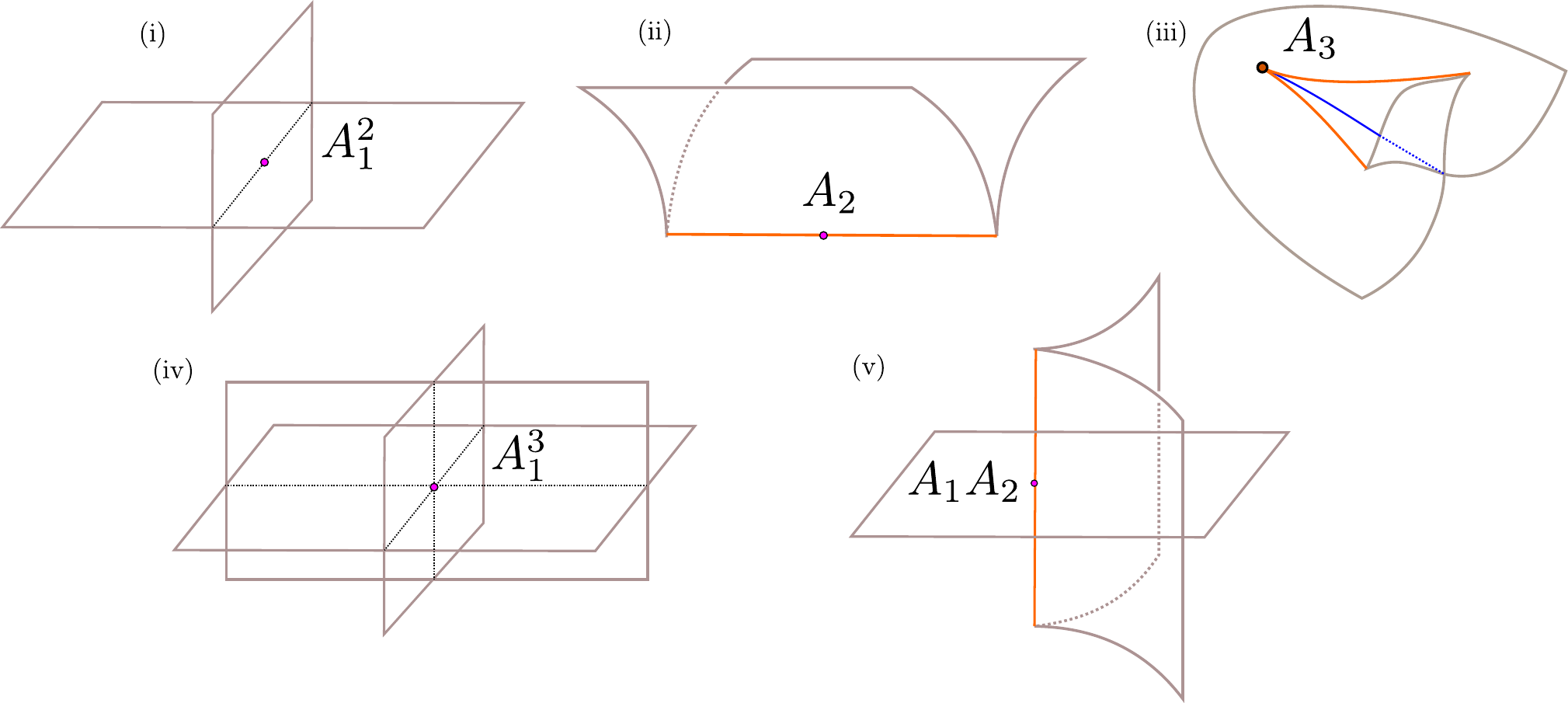}
		\caption{The generic Legendrian singularities of wavefronts in 3-space. The depicted $A_3$-singularity is known as the $A_3$-swallowtail, and the center $A_2$-singularity in the first row is referred to as the $A_3$-cusp edge. Note that the two $D_4^\pm$-singularities are {\it not} generic.}
		\label{fig:StableSing}
	\end{figure}
\end{center}

We extend the diagrammatics with the following rule: {\it orange} segments will denote $A_2$-cusp edges of singularities, and {\it orange} dots will stand for $A_3$-swallowtail singularities. Figure \ref{fig:A3Swallowtail} depicts on its left a genuine spatial front for the $A_3$-swallowtail singularity. The singularities of this front consist of a segment of $A_1^2$-crossings, shown in blue, two $A_2$-cusp edges, in orange, and a unique $A_3$-swallowtail point. The planar diagram through which we represent this front is shown on the right of Figure \ref{fig:A3Swallowtail}. It is simply a vertical view of the front (from above or below) with the $A_1^2,A_2$ and $A_3$-singularities marked.

\begin{remark} For the same reasons that we label $A_1^2$ singularities with transpositions, in order to indicate which two sheets are crossing, we should label $A_2$-cusp edges with the corresponding information. This is necessary information in order to recover the actual (homotopy type of the) Legendrian front, and thus the Legendrian itself. That said, in this article, it should be clear from context where such $A_2$-cusp edges lie, so these labels will be omitted.\hfill$\Box$
\end{remark}

The $D_4^-$-singularities are the central pieces in the construction of our Legendrian weaves $\La(G)\sse(J^1C,\xi_\st)$. It is important to emphasize that $D_4^-$ is {\it not} a generic singularity of a real spatial front, despite the fact that its complexification is a stable holomorphic Legendrian singularity. In particular, in our upcoming study of Legendrian surgeries, we will need generic Legendrian isotopies starting at $\La(G)$, whose fronts will break the non-generic $D_4^-$ into {\it generic} singularities of real spatial wavefronts.

The generic deformation of the $D_4^-$-singularity is depicted in Figure \ref{fig:D4Generic} (left). It contains three $A_3$-swallowtails arranged in a triangle and connected by $A_2$-cusp edges. Following our convention above, the associated planar diagram is shown in Figure \ref{fig:D4Generic} (right).

	\begin{center}
	\begin{figure}[h!]
		\centering
		\includegraphics[scale=0.65]{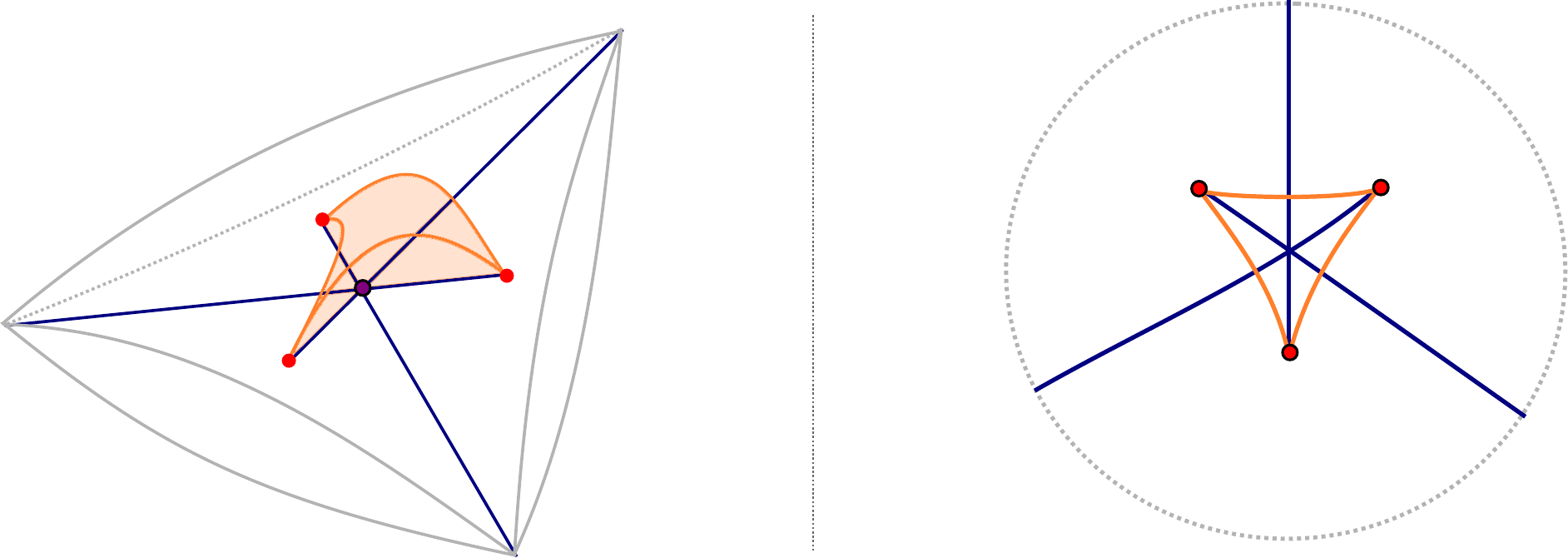}
		\caption{The spatial wavefront for a generic perturbation of the $D_4^-$-singularity (left). The associated planar diagram for this stable spatial wavefront (right). Note that the $A_1^2$-edges around the hexagonal vertex all drawn with the same color (blue), following the convention in this subsection.}
		\label{fig:D4Generic}
	\end{figure}
\end{center}


\subsection{Legendrian Front Calculus with Cusp Singularities}\label{ssec:CuspCalculus} Let us continue our development of a diagrammatic front calculus for Legendrian surfaces, this time including $A_2$-cusp edges and $A_3$-swallowtails. Proposition \ref{prop:CuspMoves} below is used to prove Proposition \ref{prop:CuspMovesII} and also Theorem \ref{thm:Legsurgeries}, in the upcoming Subsection \ref{ssec:Legsurgeries}.

\begin{prop}\label{prop:CuspMoves} Let $G\sse C$ be an $N$-graph, $N\in\N$. The four moves in Figure \ref{fig:CuspMoves} are achieved by compactly supported Legendrian isotopies, relative to the boundary.
	\begin{center}
		\begin{figure}[h!]
			\centering
			\includegraphics[scale=0.75]{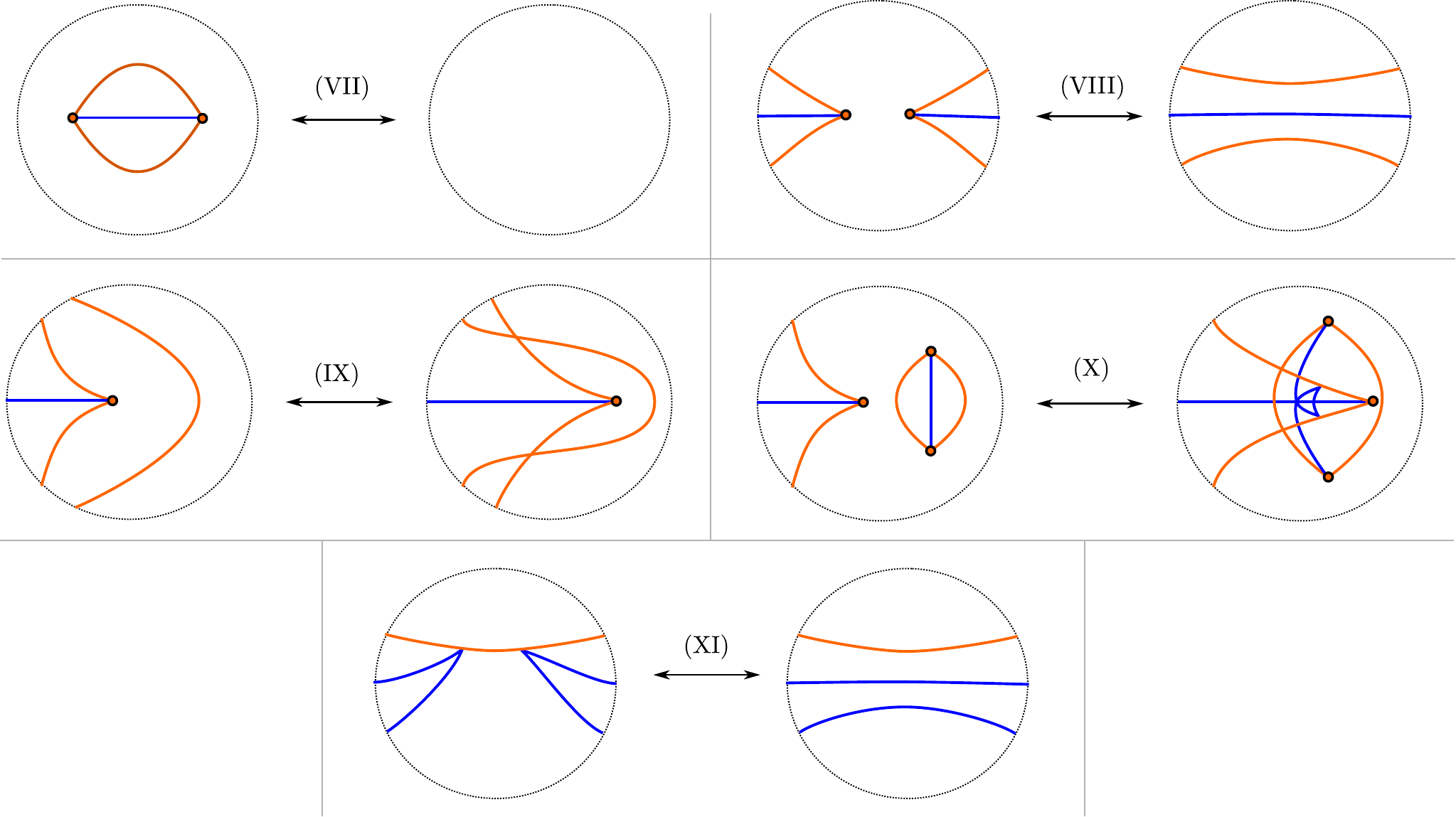}
			\caption{The five Legendrian front moves in Proposition \ref{prop:CuspMoves}. The moves are referred to as Move VII (upper left), Move VIII (upper right), Move IX (center left), Move X (center right) and Move XI (lower center).}
			\label{fig:CuspMoves}
		\end{figure}
	\end{center}
\end{prop}

\begin{proof} Moves VII and VIII, on the creation and fusion of two $A_3$-swallowtails singularities are immediate from the 3-dimensional First Reidemeister Move R1. Indeed, the left-to-right 1-dimensional Legendrian slices in Move VII correspond to a concatenation of R1 and its inverse, i.e. an R1 is performed, corresponding to the appearance of the leftmost $A_3$-swallowtail, and then the same R1 is undone, corresponding to the appearance of the rightmost $A_3$-swallowtail. This movie of 1-dimensional Legendrian slices can be isotoped to a movie with no R1 fronts, whose (big) front corresponds to the right of Move VII, with no swallowtails. For Move VIII, the R1 moves are performed in reverse order. That is, the left-to-right 1-dimensional Legendrian slices correspond to the inverse of an R1 move (a pair of cusps being undone) and then the exact same R1 move. This homotopy of 1-dimensional Legendrian fronts can be itself homotoped to a constant homotopy, which the local $N$-graph depicted in the right of Move VIII.
	
	For Move IX, we proceed with our slicing techniques. The 1-dimensional vertical left-to-right slices of the two fronts for Move IX are depicted in the left and right columns of Figure \ref{fig:MoveIXProof}. In the left column, the Reidemeister R1 move is performed for the upper piece of the 1-dimensional Legendrian knot.  In the right column, the R1 move is performed for the lower piece of the 1-dimensional Legendrian knot. The homotopy of Legendrian surface fronts is achieved by the center column in Figure \ref{fig:MoveIXProof}, where both R1 are performed simultaneously.
	
	\begin{center}
		\begin{figure}[h!]
			\centering
			\includegraphics[scale=0.5]{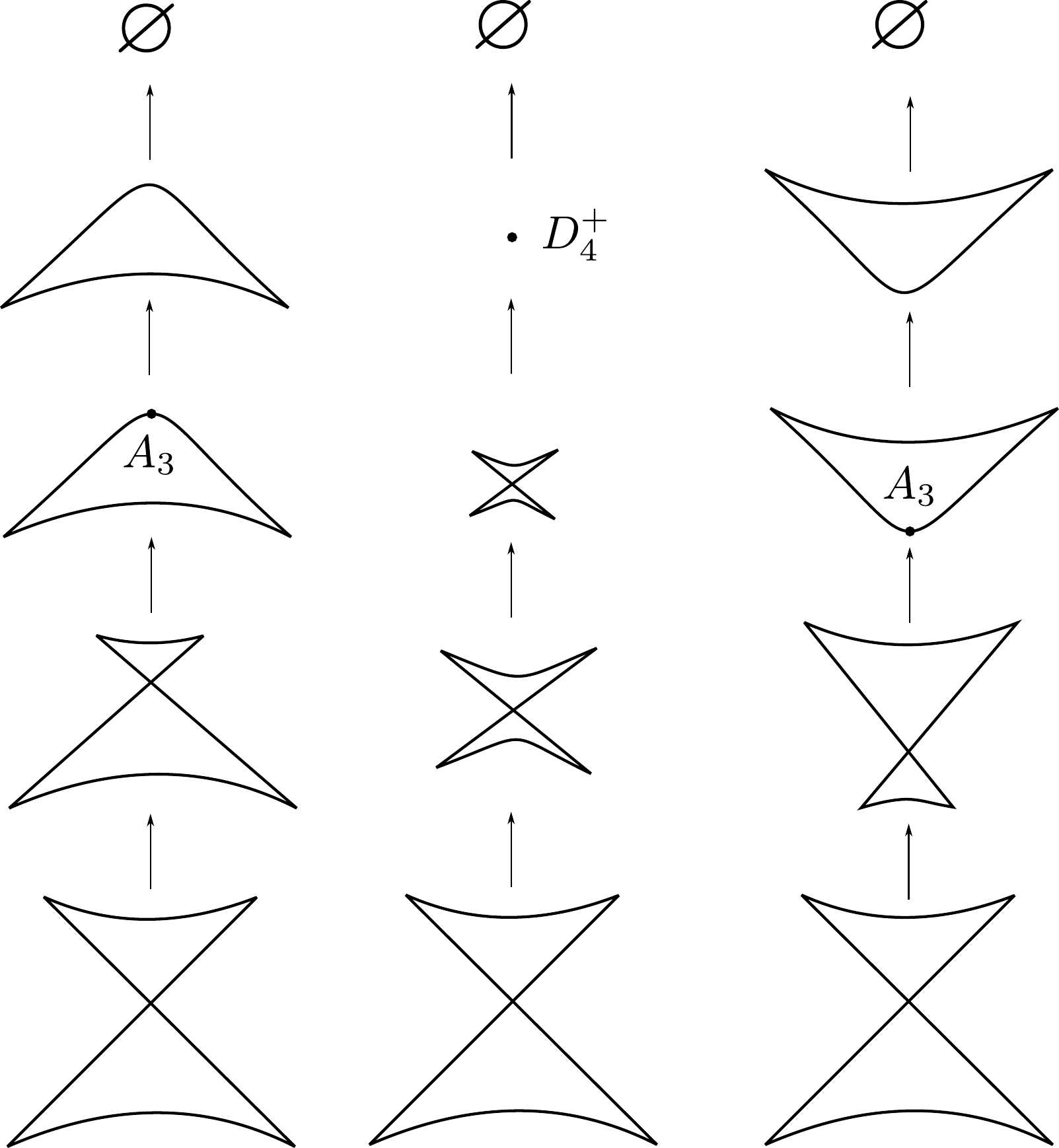}
			\caption{The homotopy of surface fronts showing that Move IX is a Legendrian Reidemeister move. The left-to-right slices for the left diagram in Move IX are depicted in the left column, whereas the slices for the right diagram in Move IX are depicted in the right column.}
			\label{fig:MoveIXProof}
		\end{figure}
	\end{center}
	
	Since the homotopy of fronts preserves the boundary conditions, this lifts to a Legendrian isotopy of embedded Legendrian surfaces, thus proving that Move IX is a Legendrian Reidemeister move. The fact that Move IX is a Legendrian Reidemeister move also follows carefully from visualizing the critical fronts associated to the generating family
	$$D_4^+:\quad F(x,y,\xi_1,\xi_2,\xi_3)=x^2y+y^3+\xi_1y^2+\xi_2y+\xi_3x,$$
	which leads to the above families in Figure \ref{fig:MoveIXProof}.
	
	Move X consists of a sliding for a $A_3$-swallowtail along an $A_1^2$-crossing line, as depicted in the top row of Figure \ref{fig:MoveXProof}, in Figures \ref{fig:MoveXProof}.(a) and \ref{fig:MoveXProof}.(b). The realistic surface fronts are depicted in the bottom row of Figure \ref{fig:MoveXProof}, in Figures \ref{fig:MoveXProof}.(A) and \ref{fig:MoveXProof}.(B), where the $A_3$-swallowtail singularity has been moved past the $A_1^2$-segment of singularities.
	
	\begin{center}
		\begin{figure}[h!]
			\centering
			\includegraphics[scale=0.9]{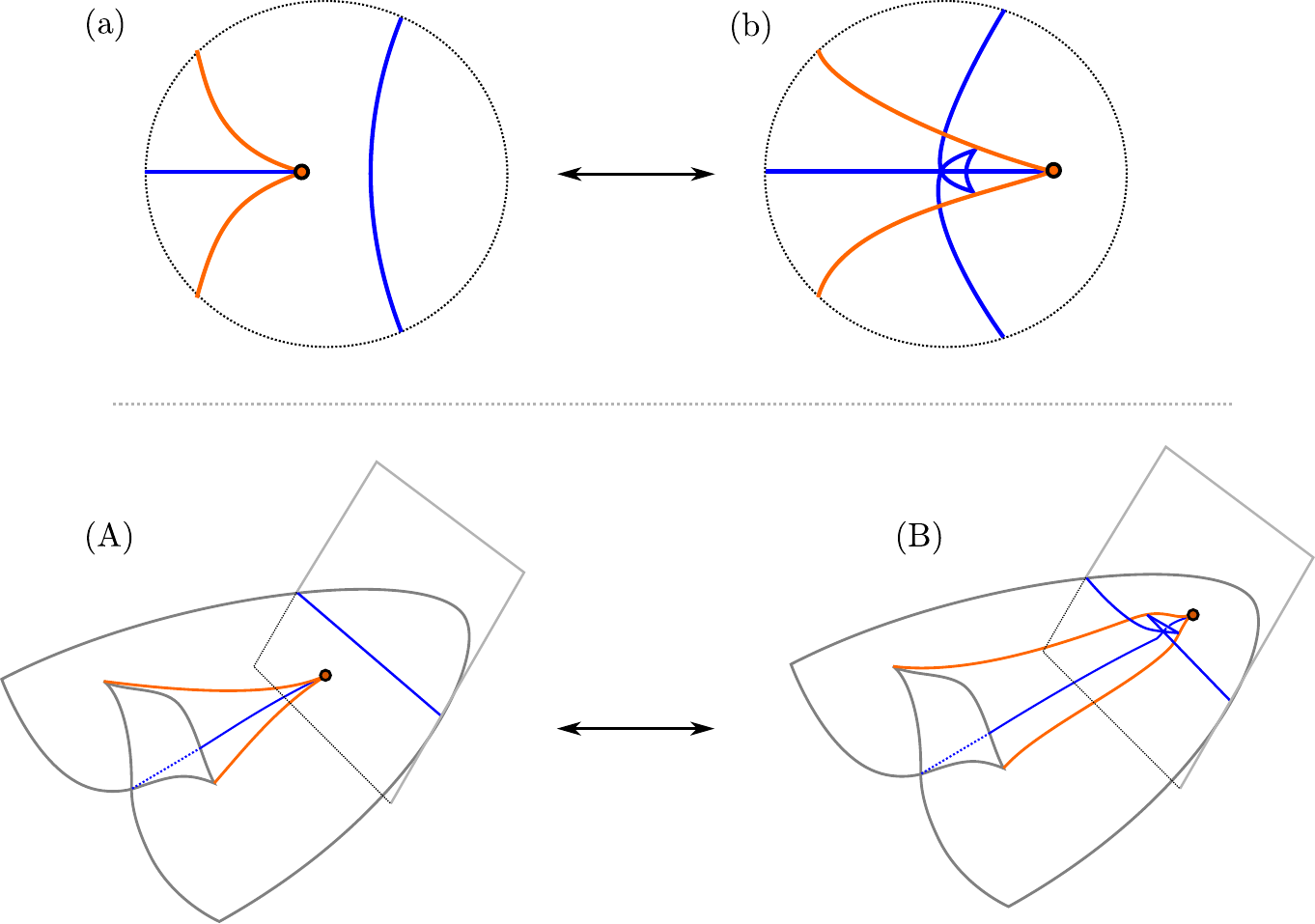}
			\caption{The front depiction of the non-trivial part in Move X. The $A_3$-swallowtail singularity slides across a orthogonal $A_1^2$-line, changing sheets as it slides through.}
			\label{fig:MoveXProof}
		\end{figure}
	\end{center}
	
	The sliding lifts to a Legendrian isotopy, as the interaction between the $A_3$-swallowtail and the $A_1^2$-line only sees a critical moment, where a $A_3A_1^2$ singularitiy appears. At this critical stage, the slopes are all distinct and non-vertical, thus the $A_3$-swallowtail is allowed to move past with a homotopy of fronts. This concludes that Move X is a Legendrian Reidemeister move.
	
	\begin{center}
		\begin{figure}[h!]
			\centering
			\includegraphics[scale=0.45]{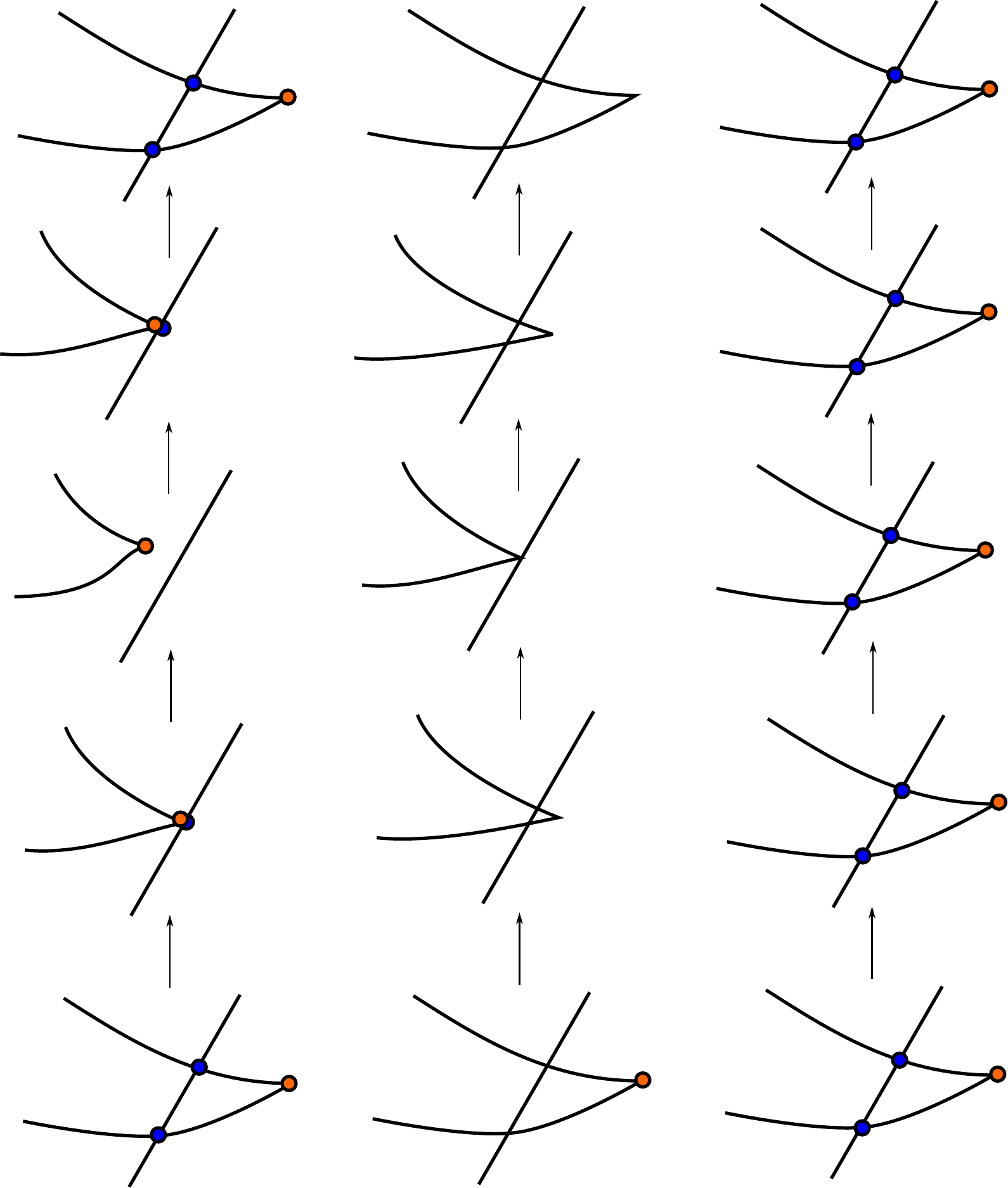}
			\caption{The homotopy of fronts for Move XI. The left front diagram of Move XI is obtained as the union of the slices in the left column, whereas the right front in Move XI is the union of the slices in the right column.}
			\label{fig:MoveXIProof}
		\end{figure}
	\end{center}
	
	Finally, Move XI is proven in Figure \ref{fig:MoveXIProof}. The middle singularity corresponds to the generic spatial front $A_2A_1$-singularity. In short, Move XI is obtained by performing a homotopy which interpolates between a constant movie of Legendrian links, and a movie consisting of doing a Reidemeister R2 move and then undoing it, as in the left column of Figure \ref{fig:MoveXIProof}.
\end{proof}

\begin{remark} It would appear that Reidemeister moves for Legendrian knots have been mastered by the vast majority of contact topologists. This does not seem to be the case in higher dimensions, including the Legendrian singularities appearing in surface fronts. Should the reader be interested in that, \cite{BennequinBourbaki,ArnoldSing} provides a starting presentation of the generic singularities of surface fronts. Our present manuscript develops the diagrammatic calculus adding to that classification, which allows us to manipulate fronts in a versatile manner. The combination of the results of this article, along with \cite{ArnoldSing}, should permit the reader to be fluent in the manipulation of wavefronts for Legendrian surfaces in contact 5-manifolds.\hfill$\Box$
\end{remark}

Let us now address the move shown in Figure \ref{fig:TechnicalMoves1}, which we prove in the following:

\begin{prop}\label{prop:CuspMovesII}
	The combinatorial move depicted in Figure \ref{fig:TechnicalMoves1} is realized by a compactly supported Legendrian isotopy of surfaces in a 5-dimensional Darboux ball $(J^1\R^2,\xi_\st)$, relative to the boundary.
	
	\begin{center}
		\begin{figure}[h!]
			\centering
			\includegraphics[scale=0.75]{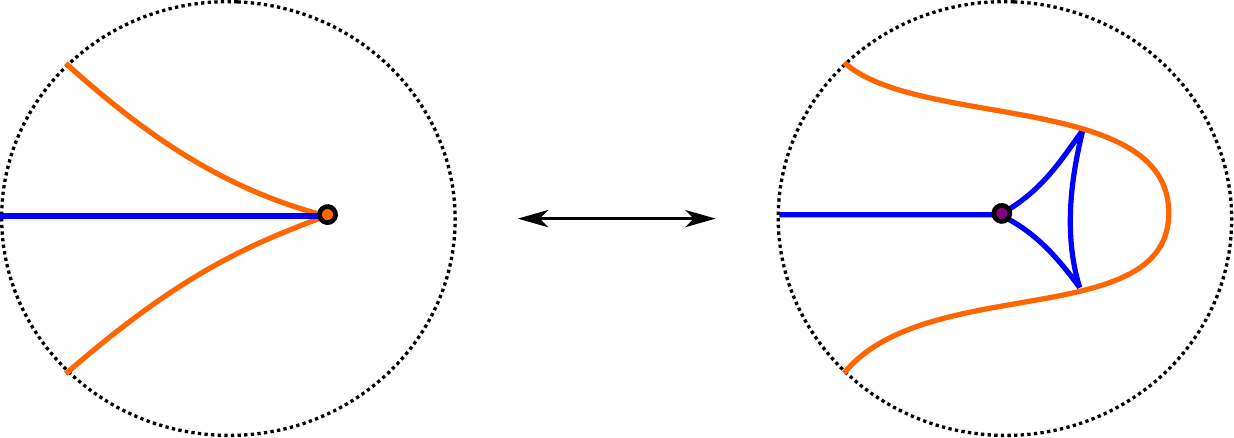}
			\caption{(Move XII) This move allows us to exchange $A_3$-swallowtail singularities with $D_4^-$-singularities in the presence of a $A_2$-cusp edge.}
			\label{fig:TechnicalMoves1}
		\end{figure}
	\end{center}
\end{prop}

\begin{proof} Let us start with the left front in Figure \ref{fig:TechnicalMoves1}. Apply Move VII to create a canceling pair of $A_3$-swallowtails, as shown in the beginning of Figure \ref{fig:TechnicalMoves1Proof}. Now slide the $A_3$-swallowtail by performing a Move X, and use the $D_4^+$-singularity, i.e. Move IX to exchange the $A_2$-cusp edge where the $A_3$-swallowtail connects. This is depicted in the first and second steps of Figure \ref{fig:TechnicalMoves1Proof}. The next two steps in Figure \ref{fig:TechnicalMoves1Proof} consists of Legendrian isotopies where no singularities interact with each other, it is a plain homotopy of fronts with the {\it same} singularities. Finally, the last step consists in joining the three existing $A_3$-swallowtails into a {\it single} $D_4^-$-singularity, as depicted at the end of Figure \ref{fig:TechnicalMoves1Proof}.
	
	\begin{center}
		\begin{figure}[h!]
			\centering
			\includegraphics[scale=0.75]{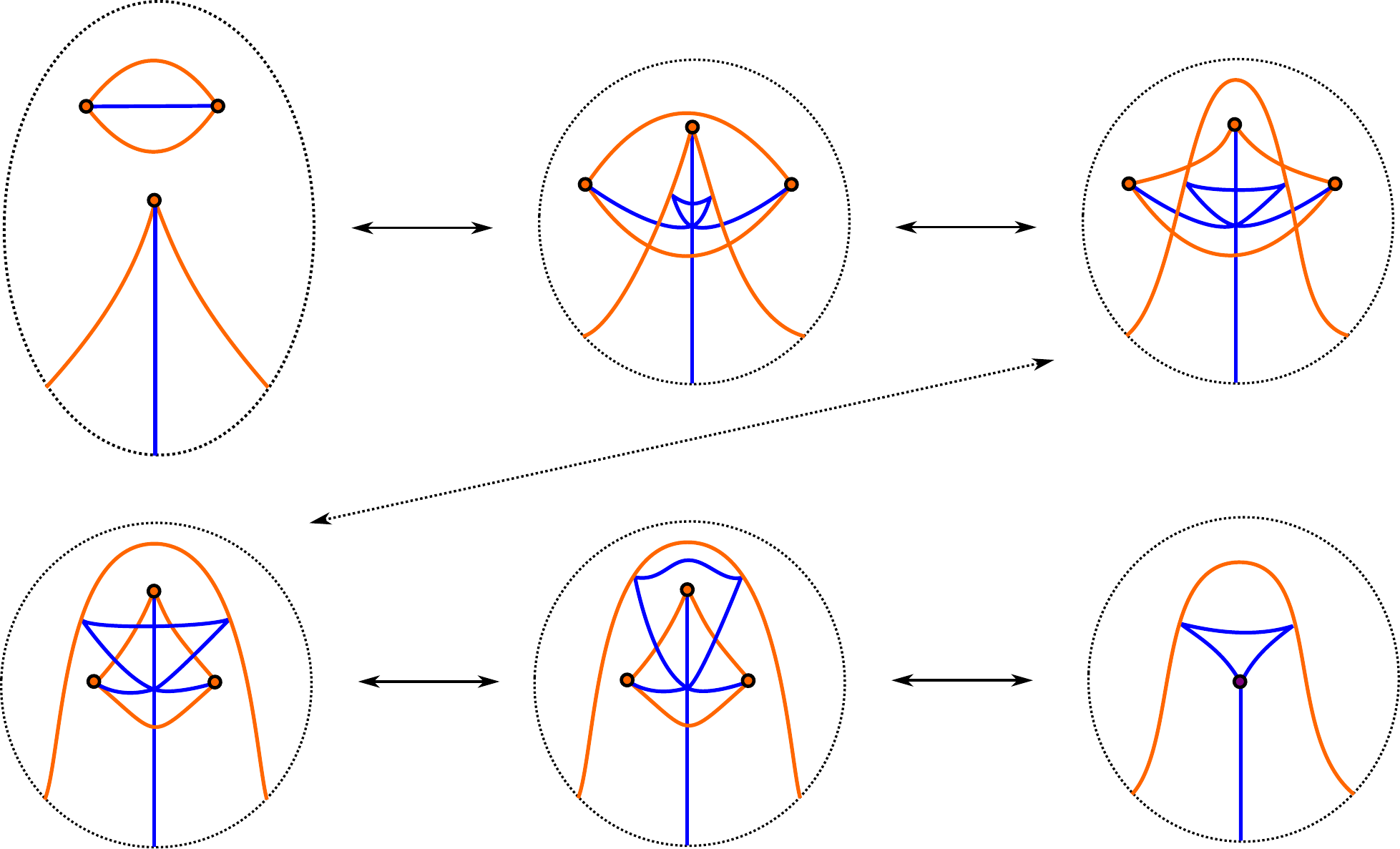}
			\caption{The homotopy of fronts for Move XII. The initial $A_3$-swallowtail requires two additional swallowtails to become a $D_4^-$-singularity, and certain intermediate moves. The homotopy realizing this can be read in this picture.}
			\label{fig:TechnicalMoves1Proof}
		\end{figure}
	\end{center}
\end{proof}


\subsection{Legendrian Surgeries}\label{ssec:Legsurgeries} The theory of Legendrian surgeries was initiated in \cite{Arnold76SurgeryI,Arnold79SurgeryII} in the study of critical points of the time function with respect to a Legendrian wavefront. Its modern description in terms of Lagrangian handle attachments is  described in \cite[Theorem 4.2]{BourgeoisSabloffTraynor15} and \cite[Section 4]{Rizell_Surgeries}. A Legendrian surgery on $\La\sse(Y,\xi)$ is an operation which inputs an isotropic sphere within $\La$, bounding ambiently, and outputs a Legendrian $\wt\La\sse (Y,\xi)$. The Legendrians $\La$ and $\wt\La$ are not even homotopy equivalent, and thus Legendrian surgery is a useful method to create {\it new} Legendrians by modifying the topology of a given Legendrian $\La$.

In the context of Legendrian surfaces, there are different types of Legendrian surgeries \cite[Figure 48]{ArnoldSing}. The following result characterizes the combinatorial operations that correspond to Legendrian $0$-surgeries, $1$-surgeries and Legendrian connected sums.

\begin{thm}[Legendrian Surgeries]\label{thm:Legsurgeries}	
Let $G\sse C$, $G_1\sse C_1$ be $N$-graphs and $G_2\sse C_2$ an $M$-graph, for $N,M\in\N$. The following statements hold:
\begin{itemize}
	\item[1.] $($0-Surgery$)$ The combinatorial move of adding an $i$-edge and two vertices along an existing $i$-edge corresponds to a Legendrian $0$-surgery. This move is shown in the upper right diagram in Figure \ref{fig:LegendrianSurgeries}.\\
	
	\item[2.] $($1-Surgery$)$ The combinatorial move of removing an $i$-edge between two trivalent vertices corresponds to a Legendrian $1$-surgery. This move is shown in the lower left of Figure \ref{fig:LegendrianSurgeries}.\\
		
	\item[3.] $($Connect Sum$)$ The kissing of two trivalent vertices $v_1\in G_1$ and $v_2\in G_2$, where $G_1\sse C_1, G_2\sse C_2$ are two disjoint graphs, corresponds to a connect sum
	$$\iota(\La(G_1))\#\iota(\La(G_2))\sse(\R^5,\xi_\st),$$
	for any satellite $\iota:\La\lr(\R^5,\xi_\st)$. This is shown in the upper left of Figure \ref{fig:LegendrianSurgeries}.\\
	
	\item[4.] $($Clifford Sum$)$ The combinatorial move of substituting a trivalent vertex by a triangle corresponds to a connected sum of $\iota(\La(G))$ with a Clifford 2-torus $\bT_c^2\sse(\R^5,\xi_\st)$. This move is shown in the lower right of Figure \ref{fig:LegendrianSurgeries}.\\
\end{itemize}

The 0-surgeries, 1-surgeries are local in any $\La(G)\sse (J^1C,\xi_\st)$. In contrast, the connected sum in the third item requires to geometrically satellite the Legendrian weaves $\La(G_1)\sse (J^1 C_1,\xi_\st)$ and $\La(G_2)\sse (J^1 C_2,\xi_\st)$ via any Legendrian embedding
$$\iota:C_1\cup C_2\lr(\R^5,\xi_\st).$$

\begin{center}
	\begin{figure}[h!]
		\centering
		\includegraphics[scale=0.9]{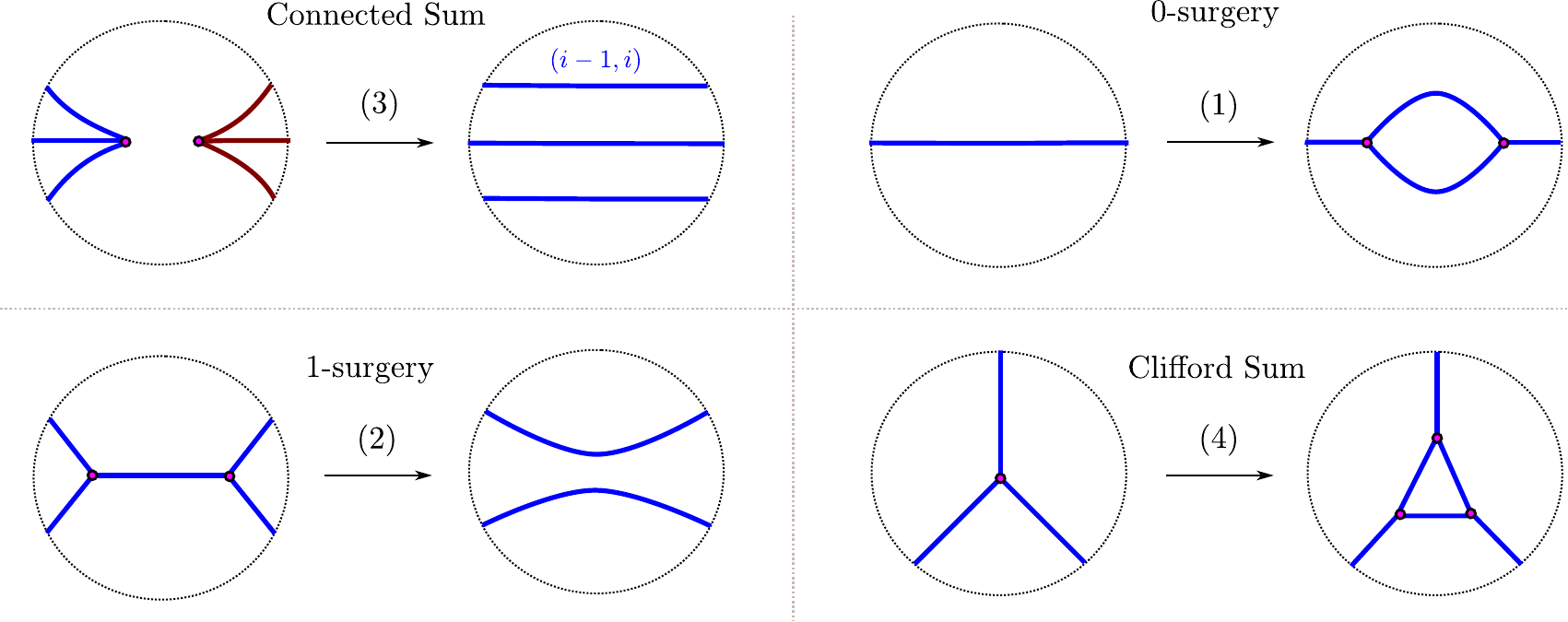}
		\caption{The Legendrian Surgery Moves in Theorem \ref{thm:Legsurgeries}}
		\label{fig:LegendrianSurgeries}
	\end{figure}
\end{center}
\end{thm}

Theorem \ref{thm:Legsurgeries} will be proven below. The Legendrian weaves in the statements involve only $D_4^-$ and $A_1^2$ (and $A_2$-cusp edges for the connected sum, due to the satellite operation). Nevertheless the manipulation of their fronts in the proof of Theorem \ref{thm:Legsurgeries} requires the use of further Legendrian front moves, involving $A_3$-swallowtails and $A_2A_1$-singularities and their interaction with the $A_2,A_1^2$ and $D_4^-$-germs, as developed in Subsection \ref{ssec:CuspCalculus} above.

\begin{remark}\label{rkm:surgeries} $(i)$ Should the reader be solely interested in the satellited Legendrian surface $\iota(\La(G))\sse(\R^5,\xi_\st)$, the connected sum operation in Theorem \ref{thm:Legsurgeries}.(3) is the strongest of the four statements (and the hardest to prove). Indeed, the satellite analogue of Items 1,2 and 4 follow from Item 3. That said, Items 1,2 do {\it not} follow from Item 3 locally.\\
	
\noindent $(ii)$ Note also that the $\iota$-satellite of the Legendrian 0-surgery depicted in Move (1) of Figure \ref{fig:LegendrianSurgeries}, and Theorem \ref{thm:Legsurgeries}.(1), corresponds to a Legendrian connected sum with the standard Legendrian 2-torus in $(\R^5,\xi_\st)$. This is the 2-torus whose front is obtained by $\S^1$-front spinning of the saucer front for the standard Legendrian unknot in $(\R^3,\xi_\st)$. See \cite[Section 4.1]{Rizell_TwistedSurgery}, and Figure 6 therein, and also \cite{BourgeoisSabloffTraynor15,Rizell_Surgeries}.
	\hfill$\Box$
\end{remark}

We recall that, by definition, the index of an elementary exact Lagrangian cobordism is the Morse index of its unique critical point, see \cite[Section 4.1]{BourgeoisSabloffTraynor15} and \cite[Section 4]{Rizell_Surgeries}. Note that elementary index-$k$ exact Lagrangian cobordisms are also referred to as Lagrangian $k$-handle attachments.  In particular, the Legendrian convex end of an elementary index-$k$ exact Lagrangian cobordism is a Legendrian $(k-1)$-surgery on the Legendrian concave end. In combination with Theorem \ref{thm:surfaceReidemeister}, Theorem \ref{thm:Legsurgeries} yields the following two moves:

\begin{cor}\label{cor:alt2handle} The two $N$-graph moves in Figure \ref{fig:Alternative2Handle} corresponds to a Legendrian 1-surgery, i.e. upon performing (2'), or (2''), there exists an elementary index-2 exact Lagrangian cobordism from the Legendrian weave on the left to the Legendrian weave on the right.
	
In fact, in Move (2') the Lagrangian 2-disk is attached along the 1-cycle represented by the (bi)chromatic horizontal edge between the two trivalent vertices. In Move (2'') the Lagrangian 2-disk is attached along the 1-cycle represented by the (blue) tripod at the hexagonal vertex uniting the three trivalent vertices.
\begin{center}
	\begin{figure}[h!]
		\centering
		\includegraphics[scale=0.9]{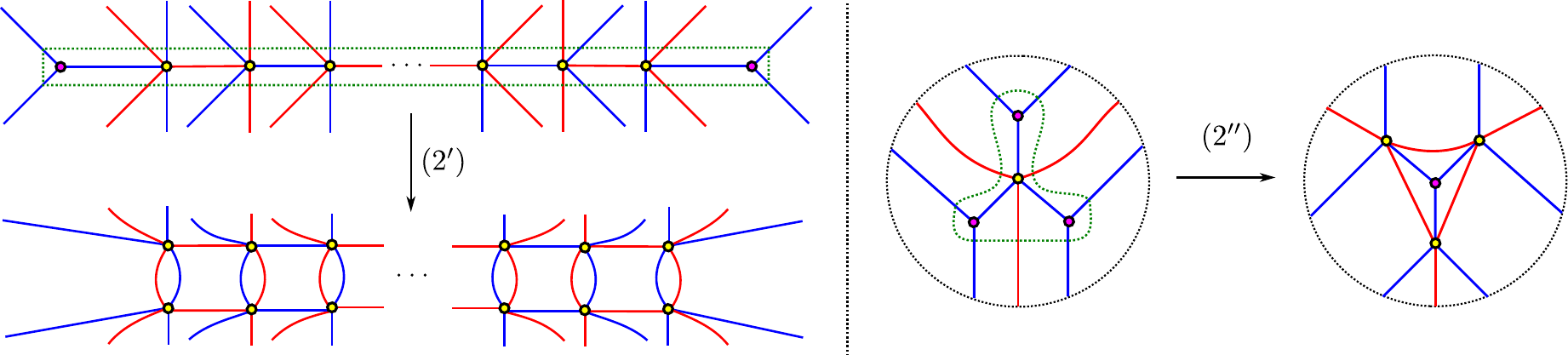}
		\caption{The two Legendrian Surgery Moves in Corollary \ref{cor:alt2handle}, both representing Lagrangian 2-handle attachments.}
		\label{fig:Alternative2Handle}
	\end{figure}
\end{center}

\end{cor}

\begin{proof}[Proof of Theorem \ref{thm:Legsurgeries}] We start by proving that adding an $i$-edge with two trivalent vertices to an existing $i$-edge effects a Legendrian 0-surgery, i.e. a Lagrangian $1$-handle attachment. The homotopy of spatial fronts is depicted in Figure \ref{fig:1Surgery}, according to the conventions in Subsection \ref{ssec:cuspedges}. The detailed description reads as follows. We first generically perturb the two $D_4^-$-singularities in the first spatial front, which yields the second front. Performing Move VIII and then Move I yields the third and fifth fronts, respectively, in Figure \ref{fig:1Surgery}. Note that the homotopy from the third to the fourth front does not involve any change in the singularities of fronts, as the blue segment of $A_1^2$-singularities intersecting the orange $A_2$-cusp segment lies strictly below it in 3-space. The homotopy from the fifth to the sixth front emphasizes the yellow band where the (reverse) 1-surgery is to be performed. The step from the sixth to the seventh fronts is precisely the reverse surgery: the $A_2$-cusp edges in the seventh front are surgered along the yellow band \cite{Arnold76SurgeryI,BourgeoisSabloffTraynor15}, in the sixth front, to obtain the fifth front. The seventh front is homotopic to the eighth front by Move VII.

\begin{center}
	\begin{figure}[h!]
		\centering
		\includegraphics[scale=0.5]{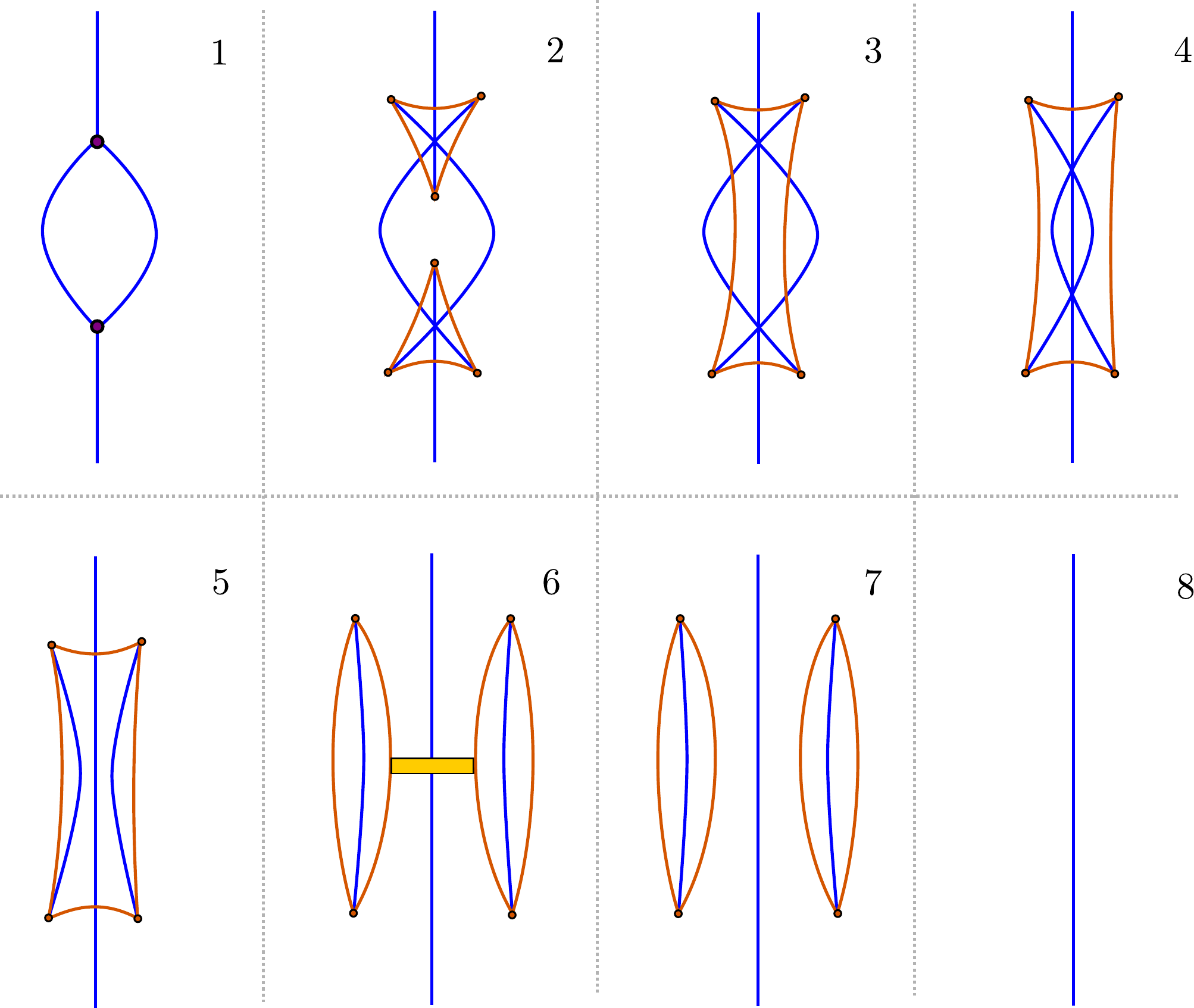}
		\caption{The diagrammatic homotopy of spatial fronts associated to the Legendrian 1-surgery move. It shows that the first front is a Legendrian 0-surgery on the eight front, i.e. the result of a Lagrangian 1-handle attachment.}
		\label{fig:1Surgery}
	\end{figure}
\end{center}

Let us now show that removing an $i$-edge corresponds to a Lagrangian $2$-handle attachment, i.e. a Legendrian 1-surgery. The homotopy of fronts is depicted in Figure \ref{fig:2Surgery}. Starting with the first front, generically perturbing yields the second front and two applications of Move VIII give the third front. In the fourth front we have shown the Legendrian 2-disk (in yellow) along which we perform the 1-surgery \cite{ArnoldSing,BourgeoisSabloffTraynor15},  the result of which is the fifth front. Indeed, the 1-surgery opens up the inner circle of $A_2$-cusp edges and adds two horizontal (Legendrian) 2-disks.  As a result, the effect on its diagrammatic representation is removing the inner circle of $A_2$-cusps, as shown in the fifth front. The application of Move I gives the sixth front, which is readily homotopic to the seventh front. The eighth front is then obtained by performing a Move VII.

\begin{center}
	\begin{figure}[h!]
		\centering
		\includegraphics[scale=0.5]{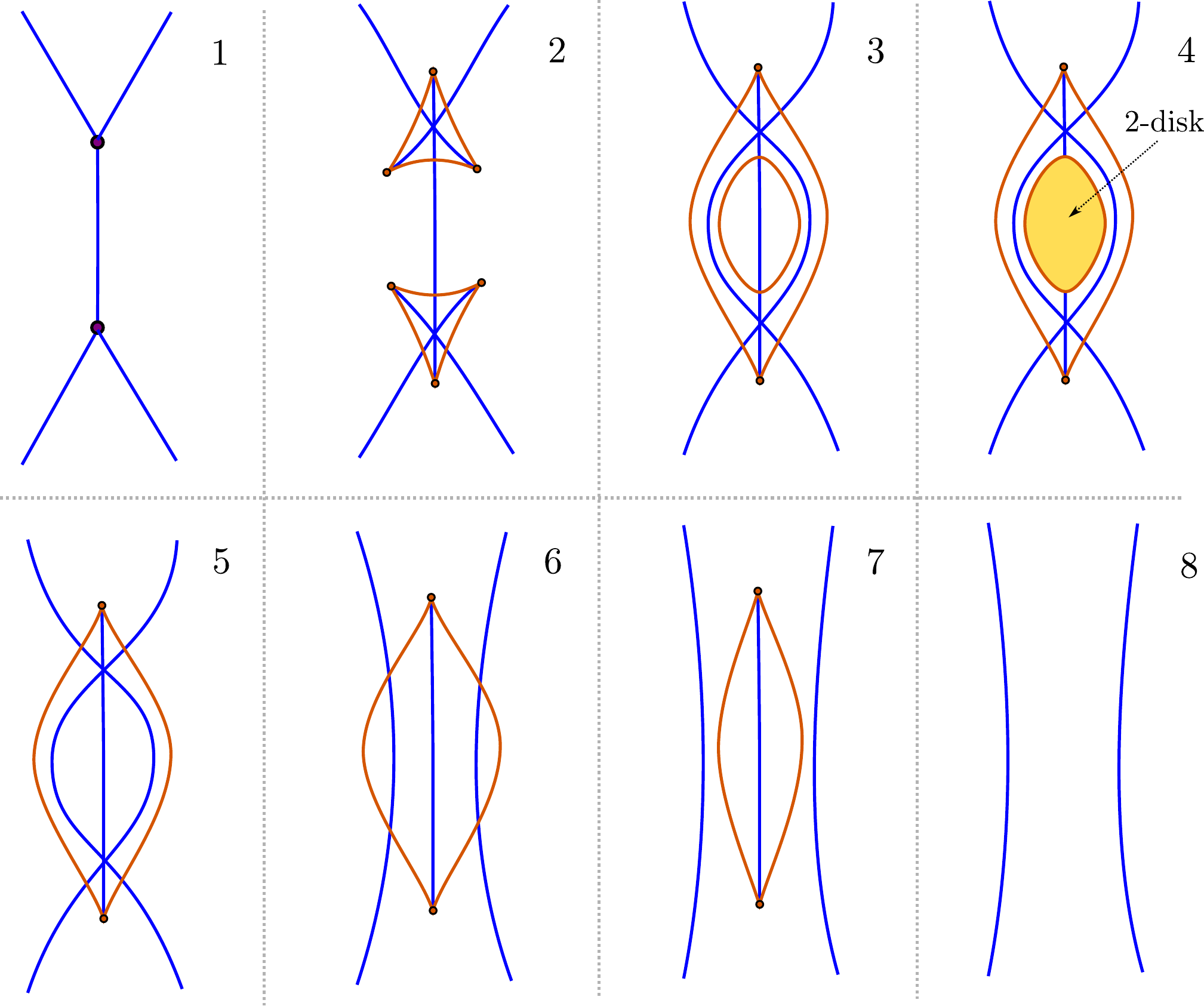}
		\caption{The diagrammatic homotopy of spatial fronts associated to the Legendrian 1-surgery move. It shows that the eighth front is a Legendrian 1-surgery on the first front, i.e. the result of a Lagrangian 2-handle attachment. Note that diagrams 3 and 4 in the first row are the same. The difference is that in diagram 4 we have depicted (in yellow) the 2-disk along which the Legendrian 1-surgery is performed, which overlaps with part of diagram 3 (and thus this part is not depicted in diagram 4).}
		\label{fig:2Surgery}
	\end{figure}
\end{center}

Now, we prove that joining two trivalent vertices in distinct graphs $G_1\sse C_1,G_2\sse C_2$ is realized by a Legendrian surface connected sum, which is a Lagrangian $1$-handle attachment (a Legendrian 0-surgery) whose attaching $0$-sphere has its two points belonging to different boundary components. The required homotopy of fronts is shown in Figure \ref{fig:ConnectedSum}. In this case, we must satellite the Legendrian weaves $\La(G_1),\La(G_2)$ to a Darboux ball $(\R^5,\xi_\st)$. From the perspective of spatial fronts, we must locally add a $A_1^2$-curve and two $A_2$-cusp edges as depicted in the first front of Figure \ref{fig:ConnectedSum}. The Legendrian 0-surgery is performed from the first front to the second, along the Legendrian band given by the red dotted line. The homotopy from the second front to the third consists of four applications of Move XI. Then, we use Move XII to obtain the fourth front. The fifth front is achieved by applying Move VII, and the sixth front consists of two applications of Move XI.

\begin{center}
	\begin{figure}[h!]
		\centering
		\includegraphics[scale=0.5]{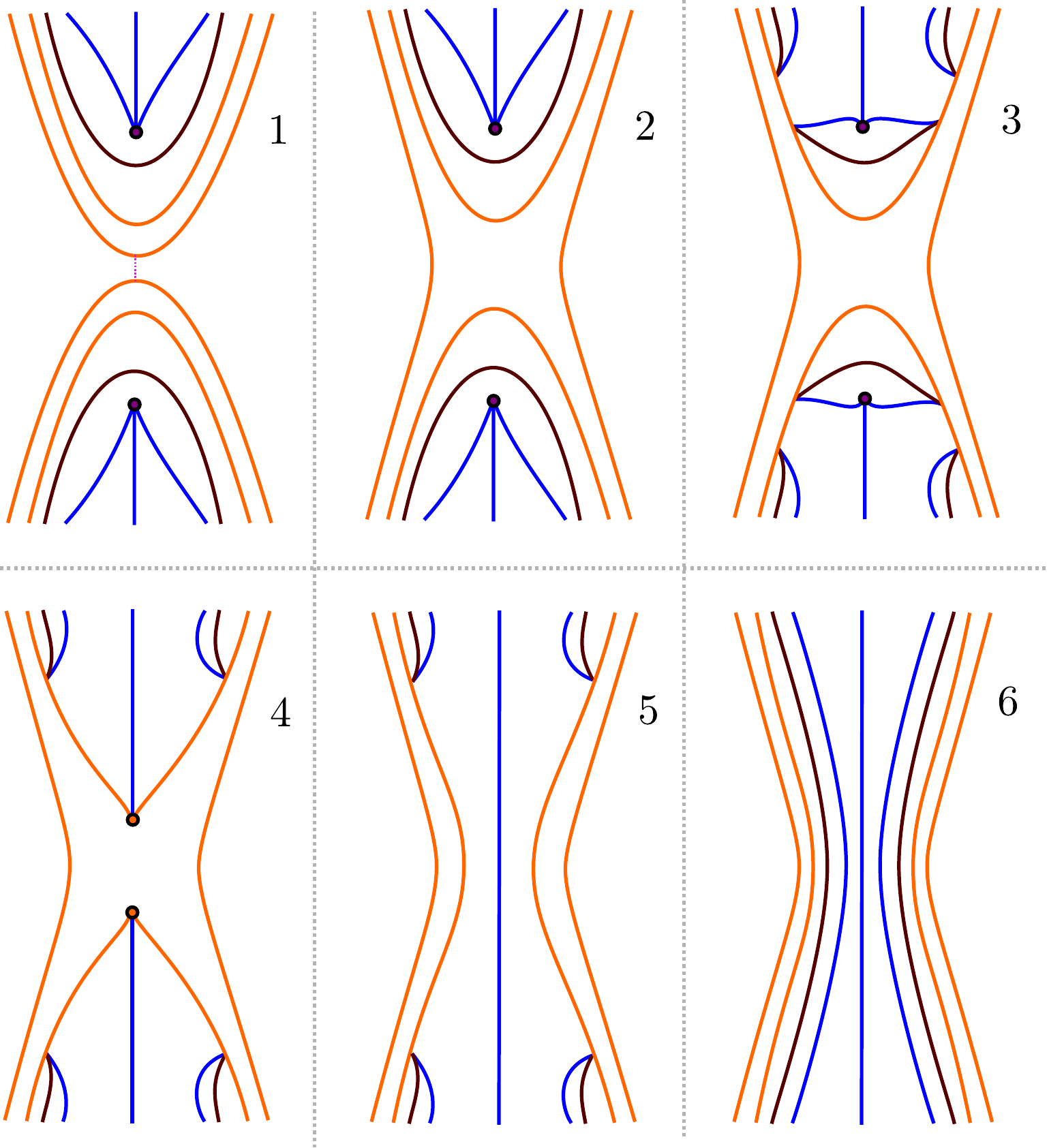}
		\caption{The diagrammatic homotopy of spatial fronts associated to the Legendrian connected sum.}
		\label{fig:ConnectedSum}
	\end{figure}
\end{center}

Finally, substituting a trivalent vertex by a triangle corresponds to a connected sum with the four vertex graph $G_c\sse\S^2$ in the left of Figure \ref{fig:CliffordSum}. One then shows that the spatial front of the Legendrian weave $\iota(\La(G_c))\sse(\R^5,\xi_\st)$ is front equivalent to the front on the right of Figure \ref{fig:CliffordSum}, which is known to be the Legendrian lift of the Clifford torus \cite{Rizell_TwistedSurgery,CasalsMurphy}. In brief, this can be shown by first identifying the Legendrian 2-torus associated to the Clifford graph with the vanishing cycle associated to the superpotential $W:\C^3\lr\C$, $W(z_1,z_2,z_3)=z_1z_2z_3$. This superpotential has a singular Lagrangian thimble
$$L=\{(z_1,z_2,z_3)\in\C^3: W(z_1,z_2,z_3)\in\R_{\geq0},|z_1|=|z_2|=|z_3|\},$$
whose intersection with the contact unit 5-sphere $(\S^5,\xi_\st)\sse \C^3$ is a Legendrian 2-torus $L_W$. It is shown in \cite[Section 3.3]{Nadler_LG1} that the Clifford graph is a front for this 2-torus $L_W$. In order to obtain the cone front from Figure \ref{fig:CliffordSum} (on the right), one stereographically projects from $(\S^5,\xi_\st)\sse \C^3$ to $(\R^5,\xi_\st)$ with the contactomorphism provided in \cite[Proposition 2.1.8]{Geiges08} and draws the (image of $L_W$ in the) front projection. The resulting front for $L_W$ is precisely the one drawn on the right of Figure \ref{fig:CliffordSum}.

\begin{center}
	\begin{figure}[h!]
		\centering
		\includegraphics[scale=0.6]{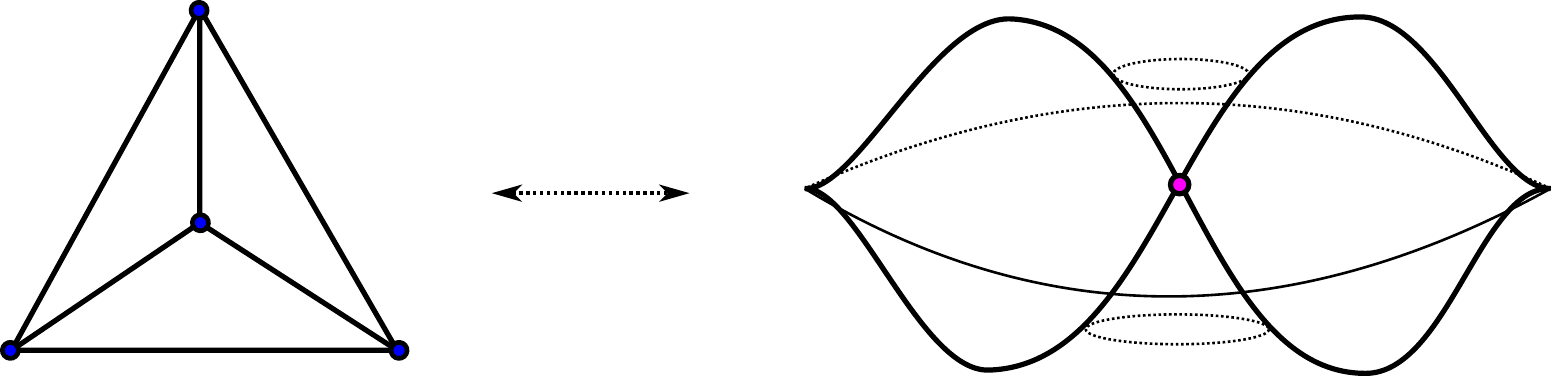}
		\caption{The Clifford graph $G_c\sse\S^2$ and a simplified spatial front for the satellited Legendrian $\iota(\La(G_c))\sse(\R^5,\xi_\st)$.}
		\label{fig:CliffordSum}
	\end{figure}
\end{center}
\end{proof}

\begin{proof}[Proof of Corollary \ref{cor:alt2handle}]
In Figure \ref{fig:Alternative2Handle}, Move (2') follows by applying a sequence of Moves II to the leftmost trivalent vertex, pushing that vertex through {\it all} the hexagonal vertices -- until it is connected to the rightmost trivalent vertex with a monochromatic edge -- and then using Move (3) in Theorem \ref{thm:Legsurgeries}. Move (2'') is more interesting, and its proof is shown in Figure \ref{fig:Alternative2HandleProof}.

\begin{center}
	\begin{figure}[h!]
		\centering
		\includegraphics[scale=0.95]{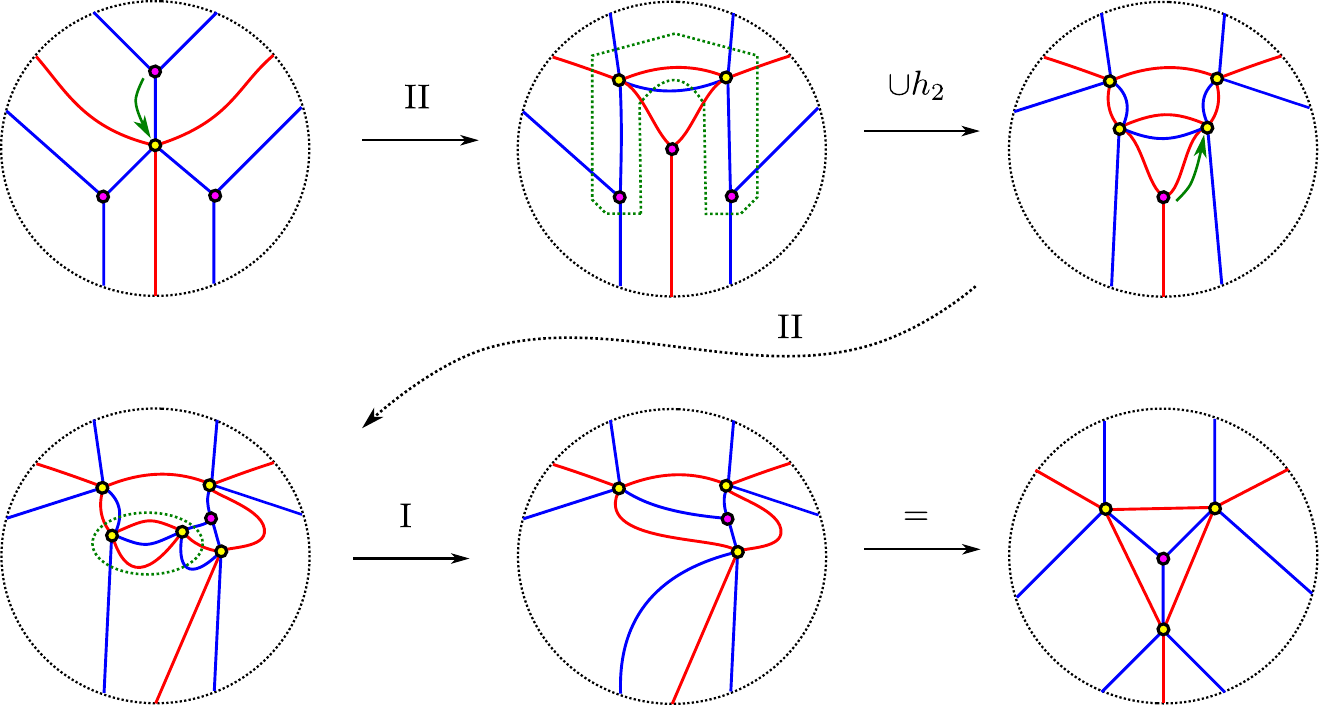}
		\caption{The Lagrangian 2-handle attachment in Move (2'') decomposed as a sequence of surface Reidemeister moves, from Theorem \ref{thm:surfaceReidemeister}, and Move (2) in Theorem \ref{thm:Legsurgeries}, in the guise of Corollary \ref{cor:alt2handle}.}
		\label{fig:Alternative2HandleProof}
	\end{figure}
\end{center}
\end{proof}

Theorem \ref{thm:Legsurgeries} provides a useful and efficient way to describe Legendrian surfaces in terms of $N$-graph combinatorics. Its statement is as strong as possible, in that the conclusion is on the Legendrian isotopy type of the associated Legendrian weaves. The computation of algebraic invariants then follows as a consequence of our geometric understanding.

In particular, we have following.

\begin{cor}\label{cor:blowup}
Let $G\sse C$ be an $N$-graph and $v\in G$ a trivalent vertex. The blow-up combinatorial move on $G$, given by an insertion of a triangle at the vertex $v$, is a twisted $0$-surgery on $\iota(\La(G))$.\hfill$\Box$
\end{cor}

The blow-up procedure was first studied in \cite[Section 5]{TreumannZaslow}.  It is depicted in Figure \ref{fig:LegendrianSurgeries} (lower right). By definition, a twisted $0$-surgery is a connected sum with a non-standard Legendrian torus in $(\S^5,\xi_\st)$. For now, we refer to \cite[Section 4]{Rizell_TwistedSurgery} for more details.

A consequence of Corollary \ref{cor:blowup} is that the Legendrian isotopy type of $\iota(\La)$ is independent of the choice of vertex $v\in G$, because a twisted $0$-surgery is independent of the choice of $0$-sphere at which it is performed (since all pair of points are isotopic in a connected surface). This question was initially asked in \cite{TreumannZaslow} in the study of the dependence of the sheaf invariants in terms of $v$. Since the Legendrian isotopy type of $\iota(\La)$ is independent of $v$, the algebraic invariants are also independent of $v$.

Finally, note that the Legendrian 0-surgery in Theorem \ref{thm:Legsurgeries}.(1) can be understood as a Legendrian connected sum with the 2-graph $G\sse\S^2$ shown in Figure \ref{fig:StandardTorusGraph} (Left). In fact, the standard Legendrian satellite $\iota(\La(G))$ for this 4-vertex 2-graph is the standard Legendrian 2-torus, a Legendrian front of which is shown in Figure \ref{fig:StandardTorusGraph} (Right). Indeed, they are both obtained from the standard Legendrian unknot by a 0-surgery (which yields a unique Legendrian isotopy class of Legendrian 2-tori) and thus they must be Legendrian isotopic.

\begin{center}
	\begin{figure}[h!]
		\centering
		\includegraphics[scale=0.8]{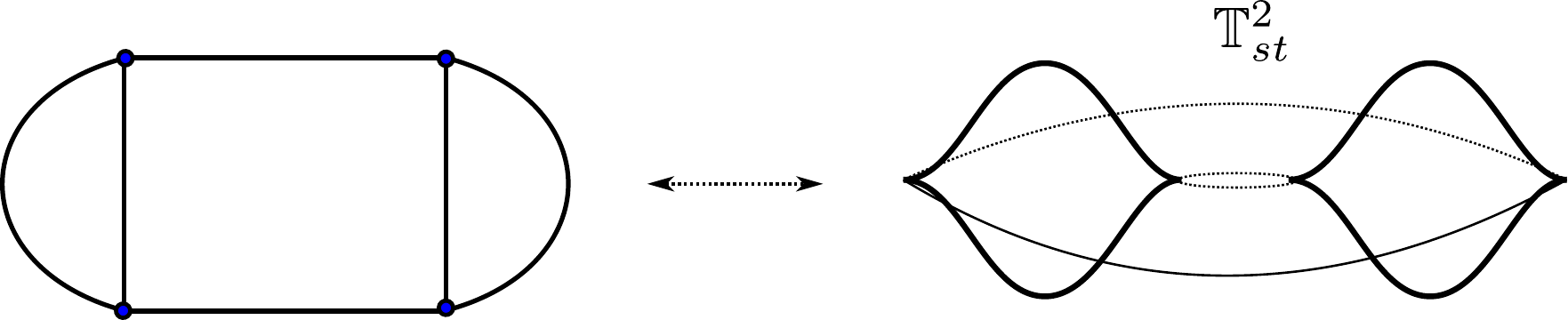}
		\caption{A $2$-graph $G$ in the 2-sphere $\S^2$ (Left) and a Legendrian front for its Legendrian weave $\iota(\La(G))$ (Right). This is the standard Legendrian 2-torus $\bT^2_{\st}\sse(\R^5,\xi_\st)$, given by Legendrian front spinning of the 1-dimensional standard Legendrian unknot $\La_0\sse(\S^3,\xi_\st)$.}
		\label{fig:StandardTorusGraph}
	\end{figure}
\end{center}

\begin{remark}
The Legendrian 0- and 1-surgeries in Theorem \ref{thm:Legsurgeries} physically correspond to partial puncture degenerations in the context of spectral networks \cite{GMN_SpecNet13,GMN_SpecNetSnakes14}. Indeed, the Legendrian weaves obtained as the Legendrian lift of the Lagrangian hyperk\"ahler rotation of the spectral curve of a diagonalizable Higgs field are related by the Legendrian surgeries in Theorem \ref{thm:Legsurgeries}. For instance, the process of a full puncture [1,1,1] degenerating to a simple [2,1] puncture in a punctured $3$-sphere is precisely a Legendrian $0$-surgery \cite[Section 6]{Gabella_BPSGraphs}.  
\hfill$\Box$
\end{remark}

The Reidemeister moves in Subsection  \ref{ssec:ReidemeisterMoves} and the stabilization operation in Subsection \ref{ssec:Stabilization} preserve the Legendrian isotopy type of the (satellite) Legendrian weaves. The Legendrian surgeries discussed in Theorem \ref{thm:Legsurgeries} generically change the topology of $\La(G)$. The natural next step is to modify the Legendrian isotopy type of $\La(G)$ without changing its topology, which we will discuss in Subsection \ref{ssec:legmutation}. For now, we study an explicit example and present the stabilization operation.


\subsection{Example of a Closed Legendrian Weave}\label{ssec:RealLifeExample} Let us illustrate our spatial front calculus in an example. Consider the triangulation of $C=\S^2$ given by a tetrahedron, and the 3-graph $G$ associated to this triangulation according to Section \ref{sec:constr}. This 3-graph is shown in Figure \ref{fig:3TriangN3Part1} (upper left). The 3-graph $G$ is depicted in the plane as an {\it unfolded} triangulation, thus the triangles should be identified according to the faces of the tetrahedron: the outer three vertices of the dashed triangle are identified, and the dashed lines are glued accordingly. In particular, the 3-graph $G$ has twelve trivalent vertices and four hexagonal vertices. The question is to describe the Legendrian isotopy type of this Legendrian surface $\iota(\La(G))\sse(\R^5,\xi_\st)$. In addition, we would like to compute Legendrian invariants, such as the augmentation variety of 3-dimensional Lagrangian fillings in $(\D^6,\omega_\st)$. In this context, understanding the Legendrian isotopy type readily implies the computation of this Legendrian invariant.

We will exploit Theorem \ref{thm:surfaceReidemeister} and Theorem \ref{thm:Legsurgeries} to understand this Legendrian weave, and note that the closed surface $\iota(\La(G)):=\wt\iota_0(\La(G))$ has genus $4.$
First, we describe the sequence of Legendrian moves and surgeries in Figure \ref{fig:3TriangN3Part1}. 
In Diagram (1) on the upper left, first note that there are three blue triangles each having one vertex in the central triangle, one each in two outer triangles, and passing through one glued edge.
There is another blue triangle with one vertex on each of the outer triangles.
By Theorem \ref{thm:Legsurgeries}, we conclude that Diagram (1) corresponds geometrically to a connected sum of the weave from Diagram (2) with four copies of the Clifford 2-torus
$\mathbb{T}^2_c$.  The 3-graph of Diagram (2) is still complicated, so we use Theorem \ref{thm:surfaceReidemeister} to simplify.  First apply Move III, flopping the four vertices in the upper right of the 3-graph.  This brings us to Diagram (3).  Now do a Move I to undo the newly appearing candy twist.  

\begin{center}
	\begin{figure}[h!]
		\centering
		\includegraphics[scale=0.85]{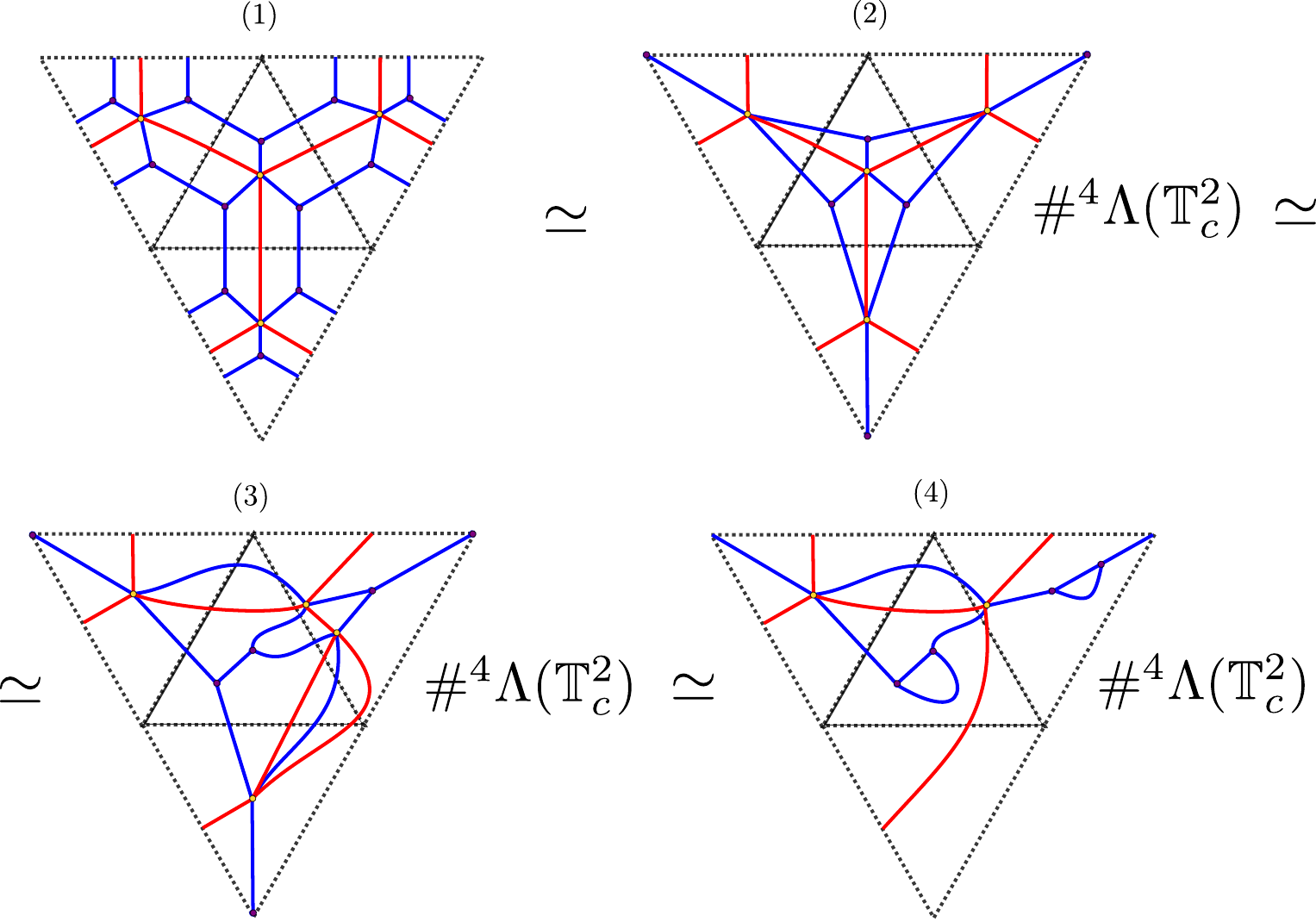}
		\caption{Simplification of a 3-graph with Theorem \ref{thm:surfaceReidemeister} and Theorem \ref{thm:Legsurgeries}.}
		\label{fig:3TriangN3Part1}
	\end{figure}
\end{center}

This brings us to Diagram (4).  So we have proven that the standard satellite $\iota(\La(G))$ is Legendrian isotopic to $\iota(\La(G'))\#_{i=1}^4\mathbb{T}^2_c$, where $G'$ is the 3-graph in Diagram (4) of Figure \ref{fig:3TriangN3Part1}.  It now suffices to understand the Legendrian $\iota(\La(G'))\sse(\R^5,\xi_\st)$.\\

\noindent {\bf Assertion}: Let $G'\sse\S^2$ be the 3-graph in Figure \ref{fig:3TriangN3Part2} (upper left). The Legendrian 2-sphere $\iota(\La(G'))$ is Legedrian isotopic to the standard Legendrian unknot $\La_0\sse(\R^5,\xi_\st)$.\\

\noindent {\bf Proof of the assertion}: By Theorem \ref{thm:Legsurgeries}, we can undo the two bigons in Diagram (5) of Figure \ref{fig:3TriangN3Part2}, and understand them as two connect sums with the standard Legendrian 2-torus $\mathbb{T}^2_\st$, defined as any Lagrangian 1-handle attachment to the standard Legendrian unknot $\La_0\sse(\R^5,\xi_\st)$.

\begin{center}
	\begin{figure}[h!]
		\centering
		\includegraphics[scale=0.85]{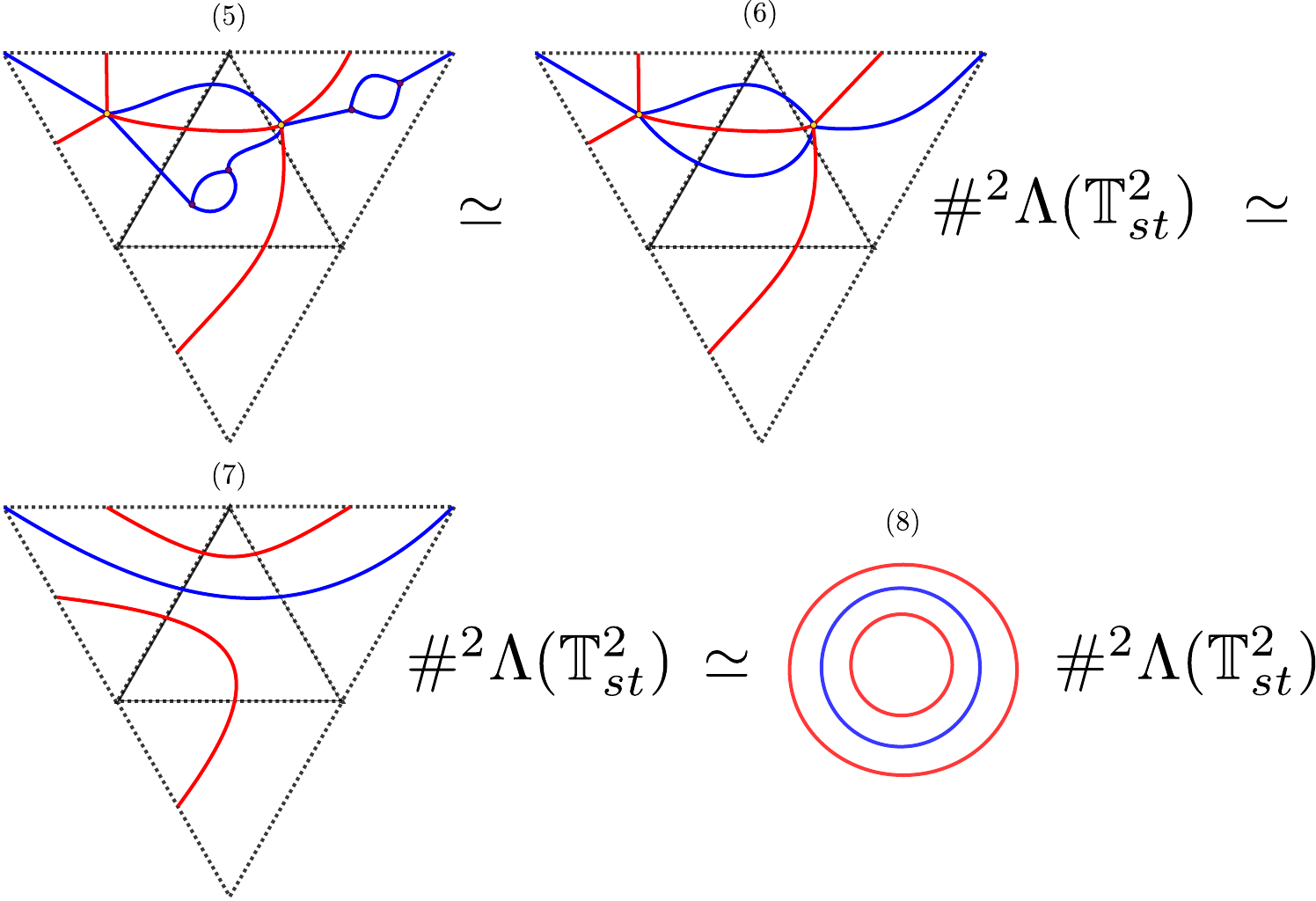}
		\caption{Diagrammatic proof that the standard satellite of the Legendrian 2-sphere associated to the 3-graph in Diagram (5) is the standard Legendrian unknot two-sphere $\La_0\sse(\R^5,\xi_\st)$.}
		\label{fig:3TriangN3Part2}
	\end{figure}
\end{center}

By applying Move I in Theorem \ref{thm:surfaceReidemeister} to the 3-graph in Diagram (6), we arrive at the 3-graph $G''$ in Diagram (7) of Figure \ref{fig:3TriangN3Part2}, which simplifies to the three concentric circles of alternating colors in Diagram (8). The Legendrian weave $\iota(\La(G''))\sse(\R^5,\xi_\st)$ is readily seen to be the standard 3-component unlink $\La_0\cup\La_0\cup\La_0\sse(\R^5,\xi_\st)$. Indeed, in 3-dimensional contact topology, the standard satellite of an $N$-stranded braid along the unknot -- with its standard saucer front -- creates a $w_0^{-1}$ of crossings at each side of the braid, where $w_0\in S_N$ is the longest element. That is, a $S^0$-worth of $w_0^{-1}$ crossings. In particular, a positive braid given by $w_0^2$, which consist of a $S^0$-worth of $w_0$-crossings, will get satellited to the standard Legendrian $N$-component unlink. See e.g. \cite[Section 2.2]{CasalsNg}. By $S^1$-symmetrically rotating this picture, we conclude that an $N$-weave in $S^2$ given by ${N\choose 2}$ concentric circles whose colors exactly give $w_0$ will be satellited along the standard 2-dimensional unknot to a standard Legendrian $N$-component unlink. Here the case at hand is $N=3$ and it suffices to note that red-blue-red represents $w_0\in S_3$. 

In conclusion, $\iota(\La(G'))$ is obtained by performing Lagrangian 1-handle attachments to $\La_0\cup\La_0\cup\La_0\sse(\R^5,\xi_\st)$, and thus $\iota(\La(G'))$ must be the standard Legendrian unknot.\hfill$\Box$

The conclusion of the above discussion is that the Legendrian isotopy type of the Legendrian surface $\iota(\La(G))\sse(\R^5,\xi_\st)$ associated to 3-triangulation of the tetrahedron, i.e.~Diagram (1) of Figure \ref{fig:3TriangN3Part1}, is that of the connected sum of four copies of the Clifford 2-torus $\mathbb{T}^2$. Hence, we now have a complete geometric understanding of $\iota(\La(G))$. In particular, this readily implies \cite{Sivek_Bordered,Rizell_TwistedSurgery} that the $\C$-moduli of objects of the category of microlocal rank-one sheaves in $\R^3$ supported in $\iota(\La(G))$ is isomorphic to $(\C\setminus\{0,1\})^4$.\hfill$\Box$


\subsection{$N$-Graph Stabilization}\label{ssec:Stabilization} The Reidemeister moves introduced in Theorem \ref{thm:surfaceReidemeister} constitute combinatorial operations on a given $N$-graph $G$ which yield the same Legendrian isotopy type for the associated Legendrian weave $\La(G)$, as a Legendrian in $(J^1C,\xi_\st)$. In particular, the resulting graph is still an $N$-graph.

In this section we discuss a different type of combinatorial move, where the number of sheets $N\in\N$ is increased. This operation, which we call {\it stabilization}, inputs an $N$-graph $G\sse C$ and outputs an $(N+1)$-graph $s(G)\sse C$. The main property of stabilization, proven in Theorem \ref{thm:graph_stab} below, is that it preserves the Legendrian isotopy type of the standard Legendrian satellite $\iota(\La(G))\sse(\R^5,\xi_\st)$, and as a result it is a non-characteristic operation.

\begin{remark}
	The relative homology class of the surface $\La(G)\sse J^1C$ has order $N$, and thus no combinatorial operation that modifies the number $N\in\N$ of sheets for a Legendrian weave will ever yield a Legendrian isotopic surface in the 1-jet space $J^1C$. Therefore, preserving the Legendrian isotopy type for the (standard) {\it satellite} is the optimal statement for a stabilization operation.\hfill$\Box$
\end{remark}

Let us describe the Legendrian weave stabilization. Given an $N$-graph $G$, the first step is to introduce a {\it ladybug} trivalent graph $B$ in $(N,N+1)$ as depicted in blue in the left of Figure \ref{fig:Stabilization_Steps} in such a way that $G$ is completely contained in one face\footnote{The construction is independent of the choice of such face.} of $B$, i.e.~ $G$ is inside one of the wings of the ladybug $B$. The second step is the introduction of {\it descending halos} centered at an $(N+1)$-graph $G$, which consists of a nested set of $N-1$ circles of $A_1^2$-crossings indexed by the permutations $(N-1,N), (N-2,N-1), \cdots (23), (12)$ reading outward. This is depicted in the right of Figure \ref{fig:Stabilization_Steps}.

\begin{center}
	\begin{figure}[h!]
		\centering
		\includegraphics[scale=0.75]{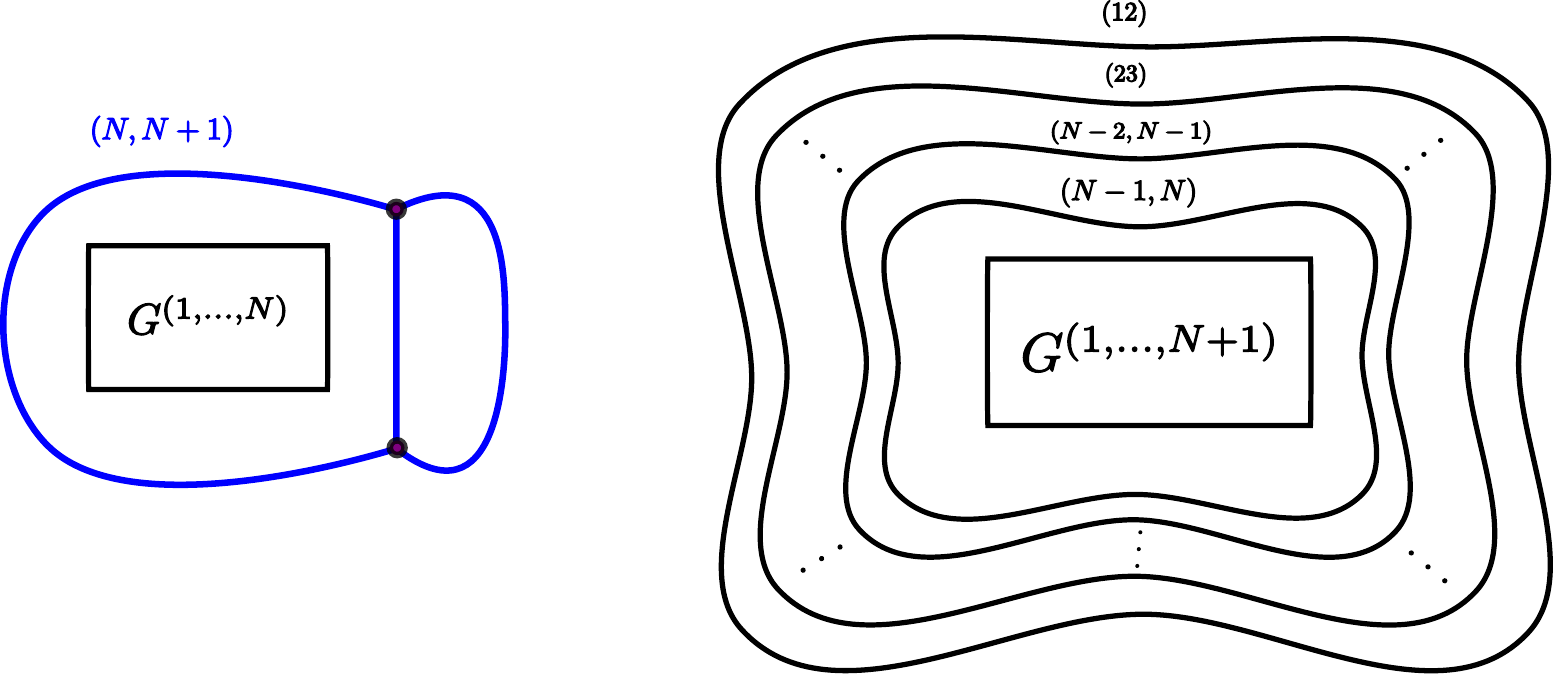}
		\caption{Ladybug graph $B$ around $G$ (left) and halos centered at $G$ (right).}
		\label{fig:Stabilization_Steps}
	\end{figure}
\end{center}

The concatenation of these two operations leads to the following:

\begin{definition}\label{def:Legstabilization}
	Let $G\sse C$ be an $N$-graph.  The stabilization of $G$ is the $(N+1)$-graph $s(G)\sse C$ obtained from $G$ by placing a ladybug $B$ around $G$, labeled with the transposition $(N,N+1)$, and a sequence of descending halos centered at the $(N+1)$-graph $G\cup B$.\hfill$\Box$
\end{definition}

Figure \ref{fig:Stabilization23} depicts the stabilization for the cases $N=2,3$.  The ladybug graph $B$ is shown in blue.

\begin{center}
	\begin{figure}[h!]
		\centering
		\includegraphics[scale=0.85]{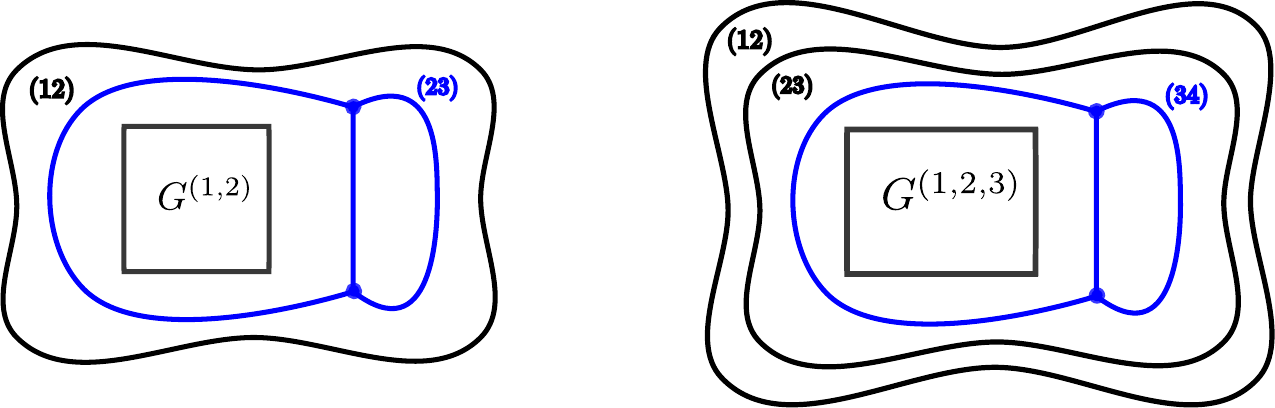}
		\caption{Stabilization of a $2$-graph (left) and of a $3$-graph (right).}
		\label{fig:Stabilization23}
	\end{figure}
\end{center}

The stabilization in Definition \ref{def:Legstabilization} is the Legendrian surface generalization of the Type II Markov move for smooth $N$-strand braids \cite{Markov35Moves,Birman74Braids}. The main property of graph stabilization is the following geometric result:

\begin{thm}\label{thm:graph_stab}
	Let $G\sse \S^2$ be an $N$-graph. Then the standard satellites $\iota(\La(G))$ and $\iota(\La(s(G)))$ are Legendrian isotopic in $(\S^5,\xi_\st)$.
\end{thm}

\begin{proof} Let us provide a detailed proof for the case $N=2$, where the stabilization is a 3-graph. The argument for higher $N\geq3$ is identical. Consider the standard satellite closure $\iota(\La(s(G)))$, which yields the diagram on the left of Figure \ref{fig:Stabilization23_Proof1} -- we refer the reader to Figure \ref{fig:SatelliteUnknot} for the front of the standard satellite closure. The standard satellite closure of a 3-graph introduces three circles of $A_1^2$-crossings, drawn in dark grey, and three circular cusp edges, drawn in orange.\footnote{For a general $N$-graph, a front for the standard satellite closure of the Legendrian weave contains $N$ additional sheets, $(N+1),(N+2),\ldots,2N$. The bottom $N$ sheets $1,\ldots,N$ are woven according to $G$, and the top horizontal $N$ sheets are parallel. The bottom and top sheets are then connected by circles worth of $A_1^2$-crossings, according to the half-twist $\Delta\in\mbox{Br}^+_N$, and $N$ circles worth of $A_2$-cusp edges -- see Figure \ref{fig:SatelliteUnknot}.} Perform a Legendrian isotopy which exchanges the $(12)$-circle of $A_1^2$-crossings with the adjacent $(34)$-circle of $A_1^2$-crossings; this gives the diagram in the right of Figure \ref{fig:Stabilization23_Proof1}. This move is possible thanks to the cusp sliding shown in the first two columns of Figure \ref{fig:Stabilization23_Proof1Slice}.
	
	\begin{center}
		\begin{figure}[h!]
			\centering
			\includegraphics[scale=0.75]{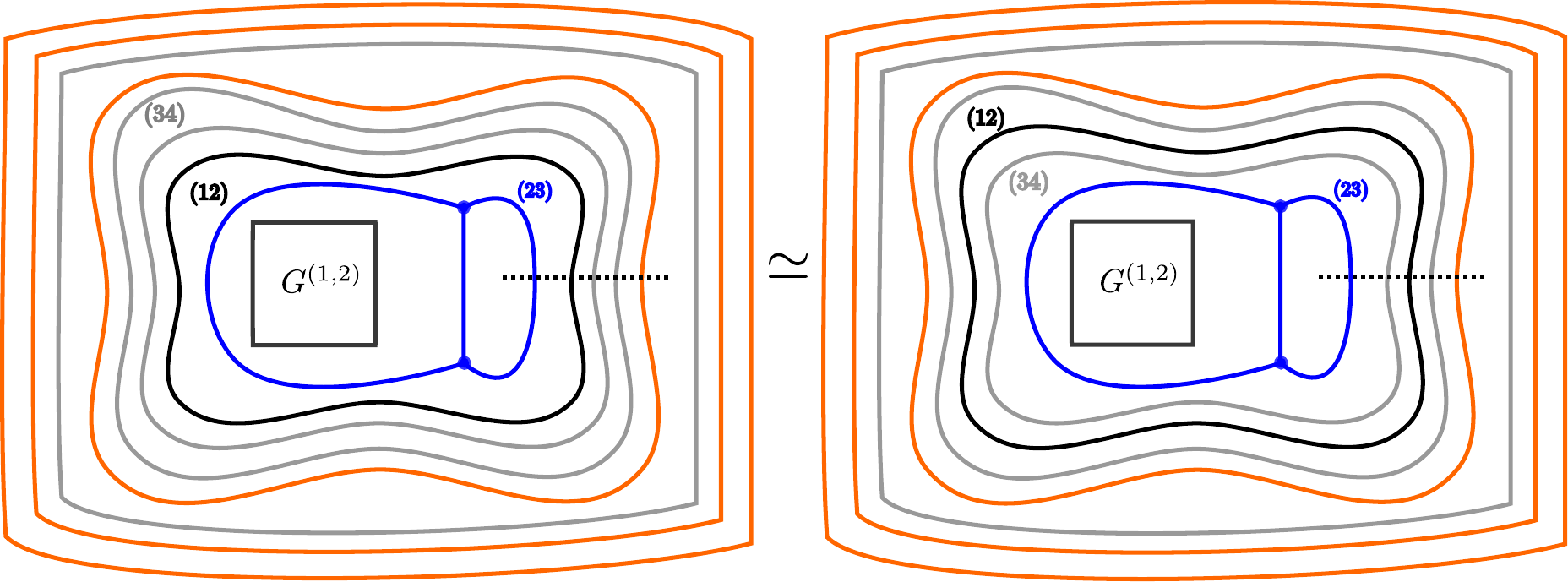}
			\caption{Exchange of $(12)$ and $(34)$ circles of $A_1^2$-crossings.}
			\label{fig:Stabilization23_Proof1}
		\end{figure}
	\end{center}
	
	\begin{center}
		\begin{figure}[h!]
			\centering
			\includegraphics[scale=0.7]{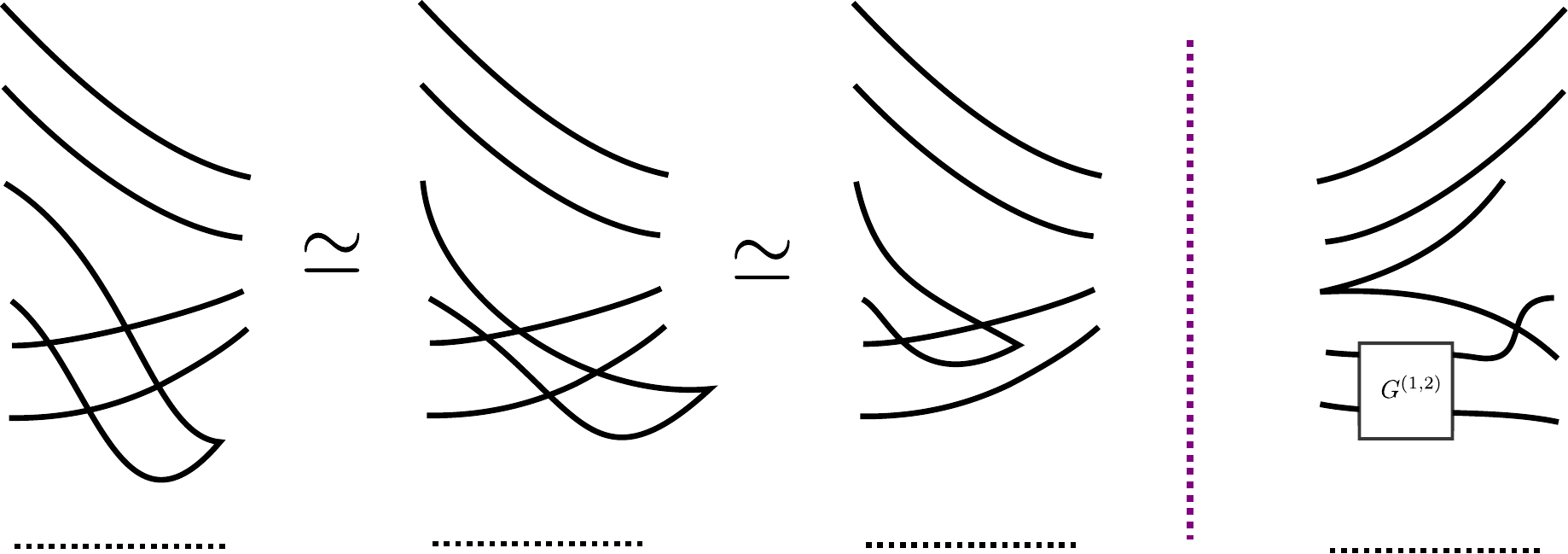}
			\caption{The left three diagrams depict slices in the dotted segments for Figure \ref{fig:Stabilization23_Proof1}. The rightmost diagram depicts a slice for the dotted segment in the right of Figure \ref{fig:Stabilization23_Proof2}.}
			\label{fig:Stabilization23_Proof1Slice}
		\end{figure}
	\end{center}
	
	Then use the innermost cusp circle and perform a Move XI, also denoted $R_1^2$ as it consists of two Reidemeister I moves, to remove two of the $A_1^2$-crossings as in the left of Figure \ref{fig:Stabilization23_Proof2}, this corresponds in the slice to the third column of Figure \ref{fig:Stabilization23_Proof1Slice}. Iterate with an $R_1^2$ in the same cusp edge with the $(34)$-circle of crossings and the ladybug piece $B$, arriving at rightmost diagram in Figure \ref{fig:Stabilization23_Proof2}.
	
	Finally, eliminate the two half-moons in the cusp edge and isotope the cusp edge above the graph $G^{(1,2)}$, which is possible thanks to the configuration shown at the rightmost column of Figure \ref{fig:Stabilization23_Proof1Slice}. The resulting diagram is that on the left of Figure \ref{fig:Stabilization23_Proof3}, which is Legendrian isotopic to the diagram on its right.
	 by applying two Moves XII, from Figure \ref{fig:TechnicalMoves1}, and an inverse Move VII from Figure \ref{fig:CuspMoves}.
	
	\begin{center}
		\begin{figure}[h!]
			\centering
			\includegraphics[scale=0.75]{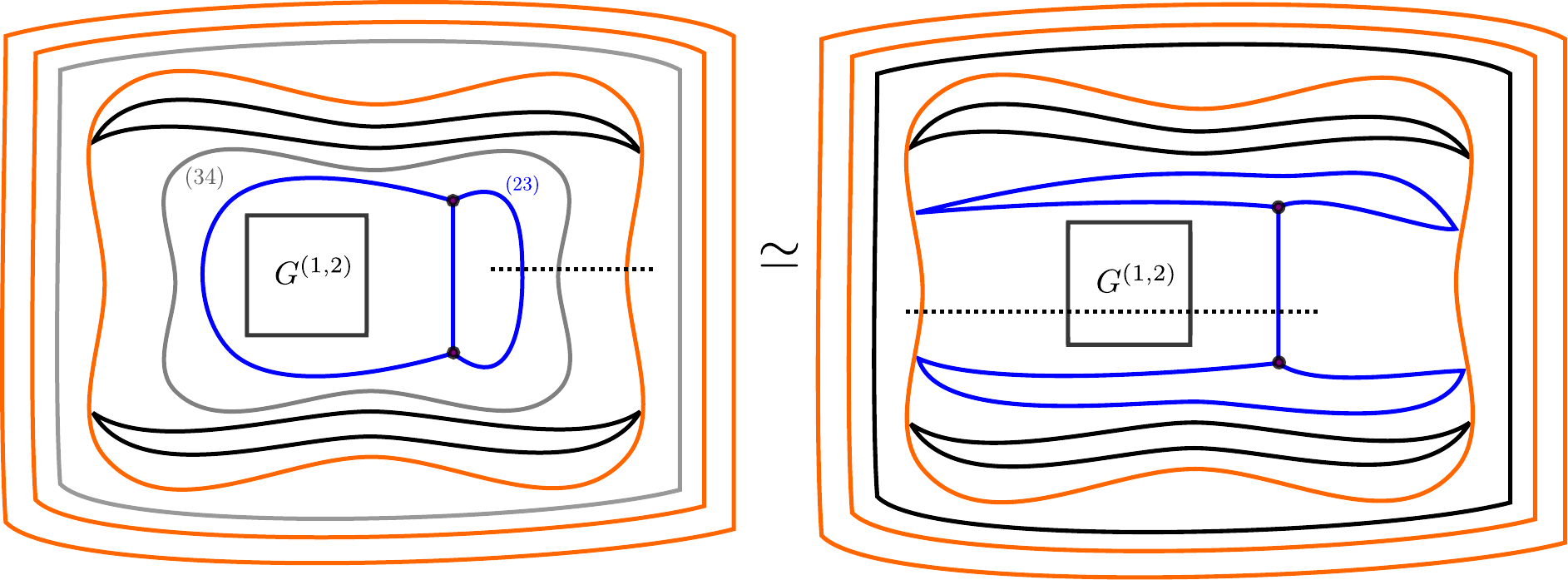}
			\caption{Performing an $R_1^2$-move with $(34)$ and the ladybug.}
			\label{fig:Stabilization23_Proof2}
		\end{figure}
	\end{center}
	
	\begin{center}
		\begin{figure}[h!]
			\centering
			\includegraphics[scale=0.75]{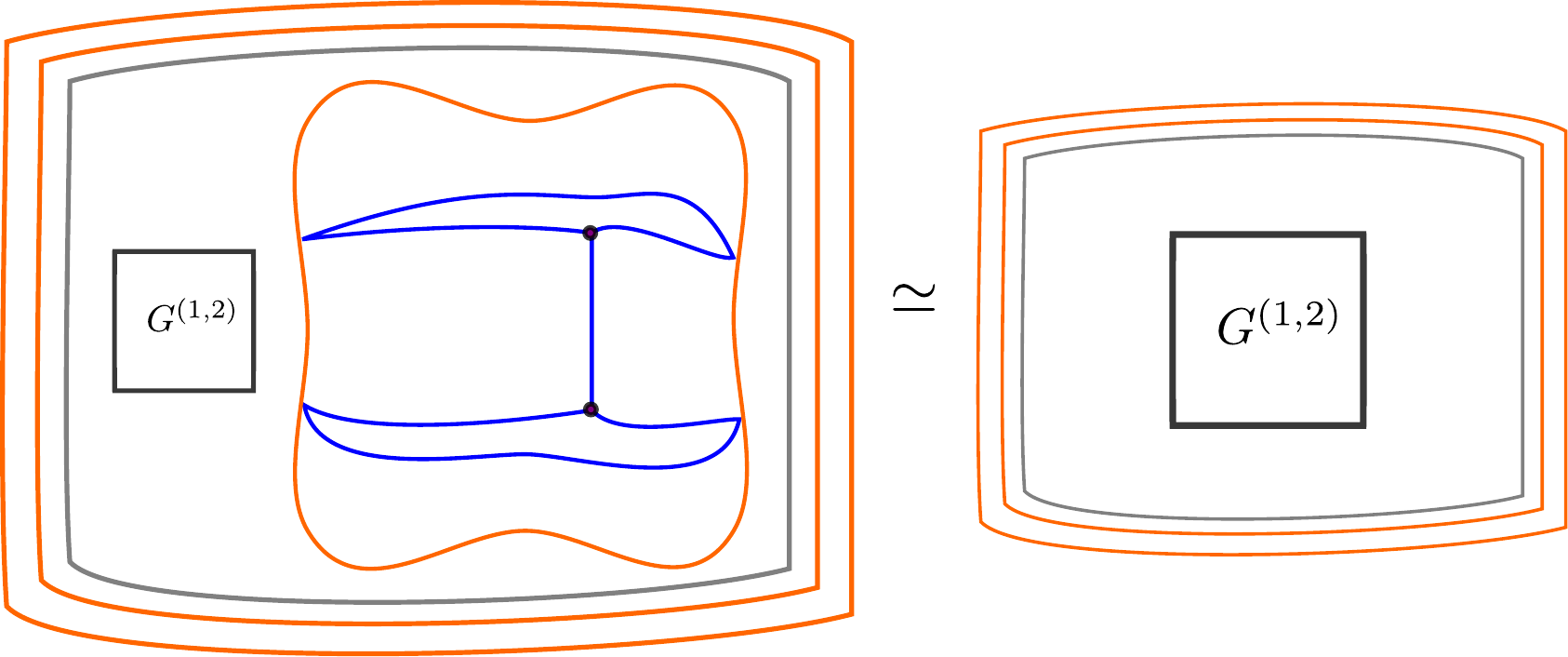}
			\caption{From $N=2$ to $N=3$ (left) and $N=3$ to $N=4$ (right).}
			\label{fig:Stabilization23_Proof3}
		\end{figure}
	\end{center}
\end{proof}

In this manuscript, Reidemeister moves in Subsection \ref{ssec:ReidemeisterMoves} and the Stabilization in Theorem \ref{thm:graph_stab} form the set of combinatorial moves that is available to us when manipulating an $N$-graph, if the Legendrian isotopy type of the associated (satellite) Legendrian weave is to be preserved.


\subsection{Legendrian Mutations}\label{ssec:legmutation} We now discuss the $N$-graph combinatorics of Legendrian mutations, a new geometric operation that we define in this manuscript. This operation inputs a Legendrian surface $\La\sse (\R^5,\xi_\st)$ and an isotropic 1-cycle $\gamma\sse\La$, and outputs a Legendrian surface $\mu_\gamma(\La)\sse (\R^5,\xi_\st)$. The Legendrian surface $\mu_\gamma(\La)\sse (\R^5,\xi_\st)$ will be ambiently (relatively) smoothly isotopic to $\La$, and oftentimes not Legendrian isotopic to $\La$. The choice of notation aims at emphasizing its relation to the wall-crossing phenomenon \cite{GMN_Wallcrossing,KontsevichSoibelman_MotivicDT,KontsevichSoibelman_Wallcrossing}, Lagrangian mutation \cite{Polterovich_Surgery,Auroux_Wallcrossing2,Auroux_Wallcrossing} and \cite[Chapter 10]{FOOO_Anomaly}, and quiver mutations \cite{FominZelevinsky_ClusterI,Vianna_Thesis}.

\begin{definition}\label{def:graphmut}
Let $G\sse C$ be an $N$-graph and $e\in G$ and $i$-edge between two trivalent vertices. The mutation of $G$ along $e$ is the $N$-graph $\mu_e(G)$ obtained by performing the exchange depicted in Figure \ref{fig:LegendrianMutation} (left), also shown in Figure \ref{fig:IntroSurgeries} (3).\hfill$\Box$
\end{definition}

By Theorem \ref{thm:LegMutations} below, the Legendrian weaves $\La(G)$ and $\La(\mu_e(G))$ will be mutation-equivalent, according to the upcoming \ref{def:LegendrianMutation} -- this motivates Definition \ref{def:graphmut} from the perspective of contact topology. Note that the operation in Definition \ref{def:graphmut} is the simplest possible mutation, corresponding to the combinatorics associated to a Whitehead move, i.e.~an edge flip in the context of triangulations dual to $2$-graphs. Indeed, consider the two unique non-degenerate triangulations $T_1,T_2$ of the square, the dual $2$-graphs $G_1,G_2$ differ precisely by a mutation along their unique internal edge.


Correspondingly, the standard satellites of their associated Legendrian weaves are two Legendrian cylinders with coinciding Legendrian boundary, smoothly isotopic relative to their boundary but which are {\it not} Legendrian isotopic relative to their boundary.

In general, given a $1$-cycle $\gamma\in\La(G)$ which is expressed combinatorially in $G$, it is possible to describe the mutation of $G$ along such $1$-cycle $\gamma$. The mutated graph $\mu_\gamma(G)$ can either be defined in an {\it ad hoc} way, or rather be understood as a graph which is equivalent via Reidemeister moves, as in Subsection \ref{ssec:ReidemeisterMoves}, to the mutated graph $\mu_e(\gamma)(G')$. Here $G'$ is Reidemeister equivalent to $G$ and $e(\gamma)$ is an $i$-edge between trivalent vertices such that $[e(\gamma)]=[\gamma]\in H_1(\La(G),\Z)$ under the canonical identification $H_1(\La(G),\Z)\cong H_1(\La(G'),\Z)$ given by a Legendrian isotopy. Here is the definition:

\begin{definition}[Legendrian Mutation]\label{def:LegendrianMutation} Two Legendrian surfaces $\La_0,\La_1\sse(\R^5,\xi_\st)$ are mutation-equivalent if and only if there exists a compactly supported Legendrian isotopy $\{\wt\La_t\}_{t\in[0,1]}$ relative to the boundary $\dd\La_0$, with $\wt\La_0=\La_0$, and a Darboux ball $(B,\xi_\st)$ such that
\begin{itemize}
	\item[(i)] The two restrictions $\wt\La_1|_{(\R^5\setminus B)}=\La_1|_{(\R^5\setminus B)}$ coincide away from this Darboux ball,
	\item[(ii)] There exists a global front projection $\pi:\R^5\lr\R^3$ such that each of the spatial fronts $\pi|_B(\wt\La_1)$ and $\pi|_B(\La_1)$ respectively coincide with each of the two fronts in Figure \ref{fig:LegendrianMutationFront}.
\end{itemize}	
\hfill$\Box$
	
	\begin{center}
		\begin{figure}[h!]
			\centering
			\includegraphics[scale=0.6]{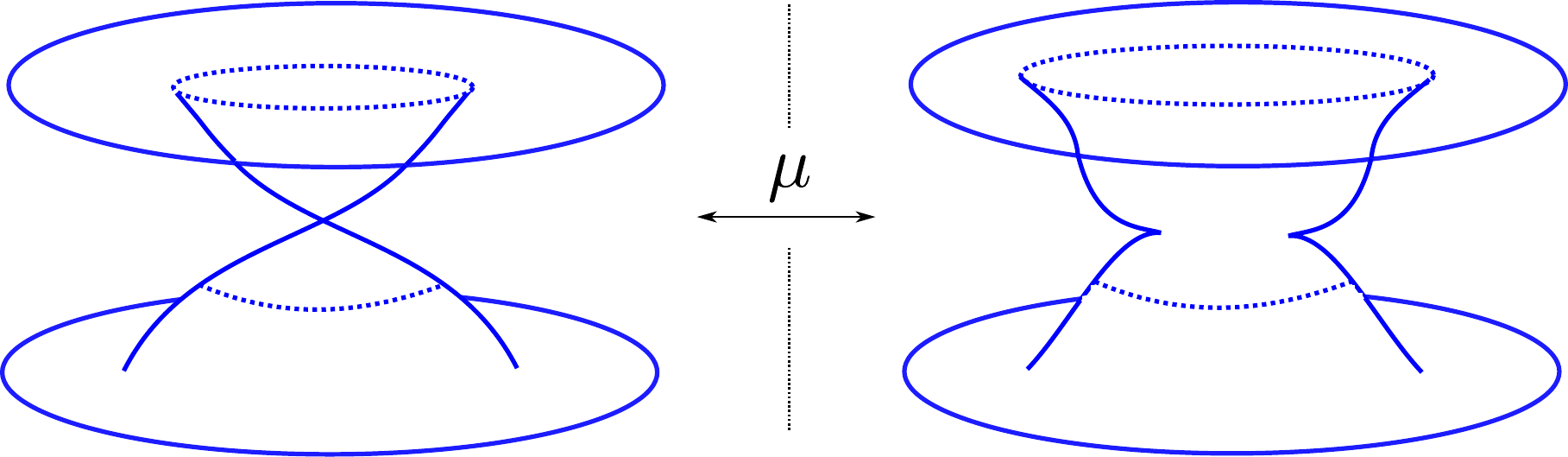}
			\caption{Legendrian mutation in a local spatial wavefront.}
			\label{fig:LegendrianMutationFront}
		\end{figure}
	\end{center}
\end{definition}

The two fronts depicted in Figure \ref{fig:LegendrianMutationFront} coincide at their boundaries and lift to Legendrian cylinders. These Legendrian cylinders are {\it not} Legendrian isotopic relative to their boundary. Indeed, compactifying the upper sheet of the fronts with an $A_2$-cusp edge and a flat 2-disk, and the lower sheet with a different $A_2$-cusp edge and a flat 2-disk, yields the standard Legendrian unknot $\La_0\sse(\R^5,\xi_\st)$ for the left front in Figure \ref{fig:LegendrianMutationFront}, and a loose Legendrian 2-sphere $\mathfrak{s}(\La_0)$ for the right front in Figure \ref{fig:LegendrianMutationFront}. The Legendrians $\La_0,\mathfrak{s}(\La_0)\sse(\R^5,\xi_\st)$ are not Legendrian isotopic \cite{EkholmEtnyreSullivan05b,EkholmEtnyreSullivan05a}.

A strong motivation for the study of the above mutations is the production of Legendrian surfaces which are {\it not} Legendrian isotopic, even though they belong to the same formal Legendrian isotopy class \cite{Gromov86,EliashbergMishachev02}. In order to distinguish Legendrian isotopy classes we will be using {\it flag moduli spaces}, which synthesize Legendrian invariants coming from the study of microlocal sheaves in terms of algebraic geometry.

\begin{remark}
The conic Legendrian singularity for the front in Figure \ref{fig:LegendrianMutationFront} (left) is {\it not} a generic singularity. It is explained in detail in \cite{Rizell_TwistedSurgery,CasalsMurphy}, and its generic perturbation contains four $A_3$-swallowtail singularities.\hfill$\Box$
\end{remark}

\begin{thm}[Legendrian Mutations]\label{thm:LegMutations} Let $G_1,G_2$ be one of the pairs of $N$-graphs depicted in Figure \ref{fig:LegendrianMutation}. Then the associated Legendrian surface $\La(G_1)$ is a Legendrian mutation of $\La(G_2)$ relative to their boundaries.
	
	\begin{center}
		\begin{figure}[h!]
			\centering
			\includegraphics[scale=0.9]{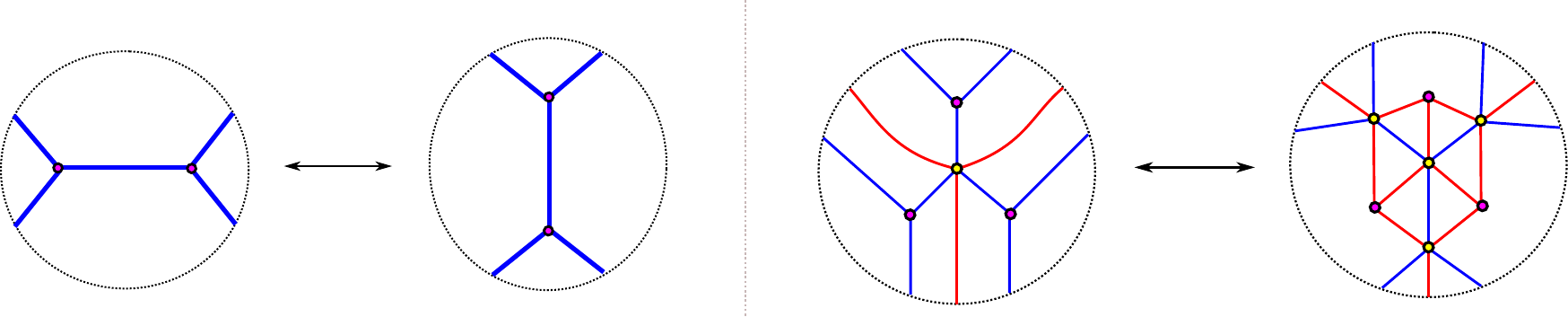}
			\caption{The Legendrian Mutation Moves in Theorem \ref{thm:LegMutations}}
			\label{fig:LegendrianMutation}
		\end{figure}
	\end{center}
\end{thm}

\begin{proof} Let us start by showing that the exchange move in Figure \ref{fig:LegendrianMutation} (left) corresponds to a Legendrian mutation,
as in Definition \ref{def:LegendrianMutation}. By \cite[Theorem 6.3]{CasalsMurphyPresas}, the Lagrangian projections $\Pi(\La_0),\Pi(\La_1)\sse\R^4$ of
the Legendrian lifts of the fronts $\pi(\La_0),\pi(\La_1)\sse\R^3$ in Figure \ref{fig:LegendrianMutationFront} correspond to the two
Polterovich surgeries associated to the normal crossing of two Lagrangian planes $\R^2\times\{0\},\{0\}\times\R^2\sse(\R^4,\omega_\st)$.
The Lagrangian projection of the Legendrian lifts for each two $2$-graphs in the exchange move in Figure \ref{fig:LegendrianMutation} (left) are exact Lagrangian
fillings $L_1,L_2$ of the Hopf link $\La_{Hopf}\sse(\S^3,\xi_\st)\cong\dd(\R^4,\omega_\st)$. Indeed, the 2-stranded braid word at the boundary of the 2-weave is $\sigma_1^4$, as there are four blue edges arriving at the boundary, and then note that the $(-1)$-framed closure of $\sigma_1^4$ in $(\R^3,\xi_\st)$ is the Hopf link. See Section \ref{sec:app2} for more details on Lagrangian fillings.
Thus, it suffices to show that $L_1,L_2\sse(\R^4,\xi_\st)$ are the positive and negative Polterovich surgeries of the two
Lagrangian planes $\R^2\times\{0\},\{0\}\times\R^2\sse(\R^4,\omega_\st)$ at their intersection points.
Indeed, Figure \ref{fig:LegendrianMutationsProof} (center) depicts the 2-graph for the {\it singular} Legendrian whose Lagrangian projections is the Lagrangian union $(\R^2\times\{0\})\cup(\{0\}\times\R^2)$.
	
\begin{center}
	\begin{figure}[h!]
	\centering
	\includegraphics[scale=0.7]{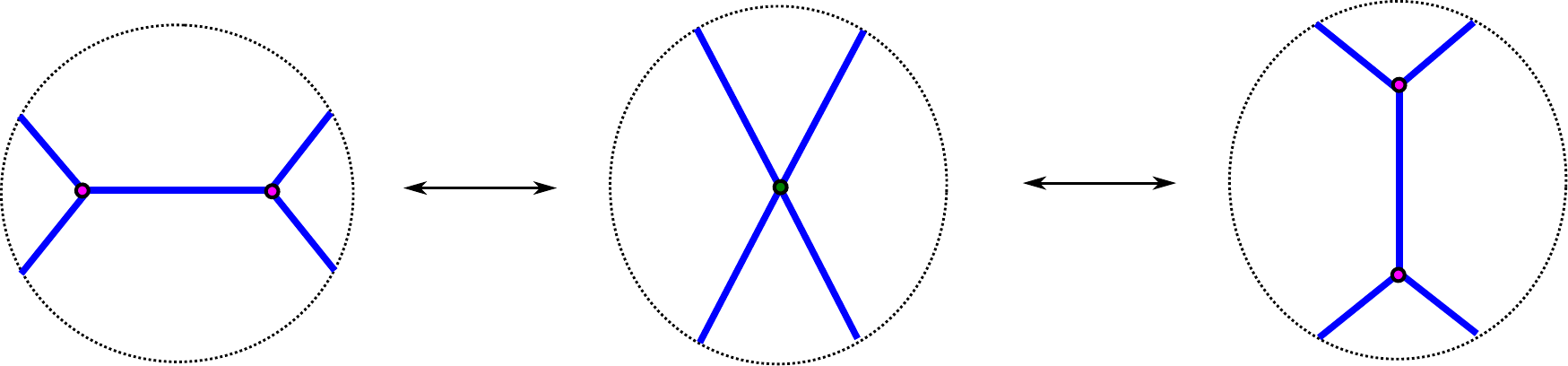}
	\caption{The 2-graphs associated to a Legendrian Mutation. The middle 2-graph yields a spatial front which lifts to a {\it singular} Legendrian surface, consisting of the union of two 2-planes intersecting at a point.}
	\label{fig:LegendrianMutationsProof}
	\end{figure}
\end{center}

The 2-graph in Figure \ref{fig:LegendrianMutationsProof} (center) describes a topological surface which is  the union of 2-planes intersecting at a point, both for the Lagrangian surfaces in $(\R^4,\omega_\st)$ {\it and} the Legendrian surfaces in $(\R^5,\xi_\st)$. Topologically, the front in Figure \ref{fig:LegendrianMutationsProof} (center) is the cone over the annular projection of the $(2,4)$-braid, with singular crossings\footnote{This is consistent with the fact that the Hopf link is the boundary of two transversely intersecting planes in the 4-ball $\D^4$. For the max-tb {\it Legendrian} Hopf link, these two planes should be taken to be Lagrangian.}.

Finally, the Lagrangian projections of the Legendrian lifts of Figure \ref{fig:LegendrianMutationsProof} (left) and Figure \ref{fig:LegendrianMutationsProof} (right) are realized as Polterovich surgeries of the corresponding Lagrangian projection in Figure \ref{fig:LegendrianMutationsProof} (center). Since the Legendrian lifts of Polterovich surgeries are Legendrian mutations \cite[Theorem 6.3]{CasalsMurphyPresas}, this concludes the first part of Theorem \ref{thm:LegMutations}.

Let us now show that the exchange move in Figure \ref{fig:LegendrianMutation} (right) also corresponds to a unique Legendrian mutation. This is proven directly through the homotopy of fronts in Figure \ref{fig:LegendrianMutationsN2N3Proof}.

Indeed, the first step in Figure \ref{fig:LegendrianMutationsN2N3Proof}, starting from the upper left, consists of applying Move II, pushing a trivalent vertex through a hexagonal vertex. The second and third steps are also a direct application of a Move II, pushing the remaining two trivalent vertices through the newly created {\it two} hexagonal vertices. The fourth move, starting at the left of the second row, is a mutation of 2-graphs. This yields the 3-graph at the center of the second row, the arrow being labeled by the letter $\mu$. Finally, we apply a Move III, flopping the four vertices nearest to the center, in order to achieve the 3-graph at the right of Figure \ref{fig:LegendrianMutation} (right). This shows that the exchange move in Figure \ref{fig:LegendrianMutation} (right) is a Legendrian mutation.

\begin{center}
	\begin{figure}[h!]
		\centering
		\includegraphics[scale=0.8]{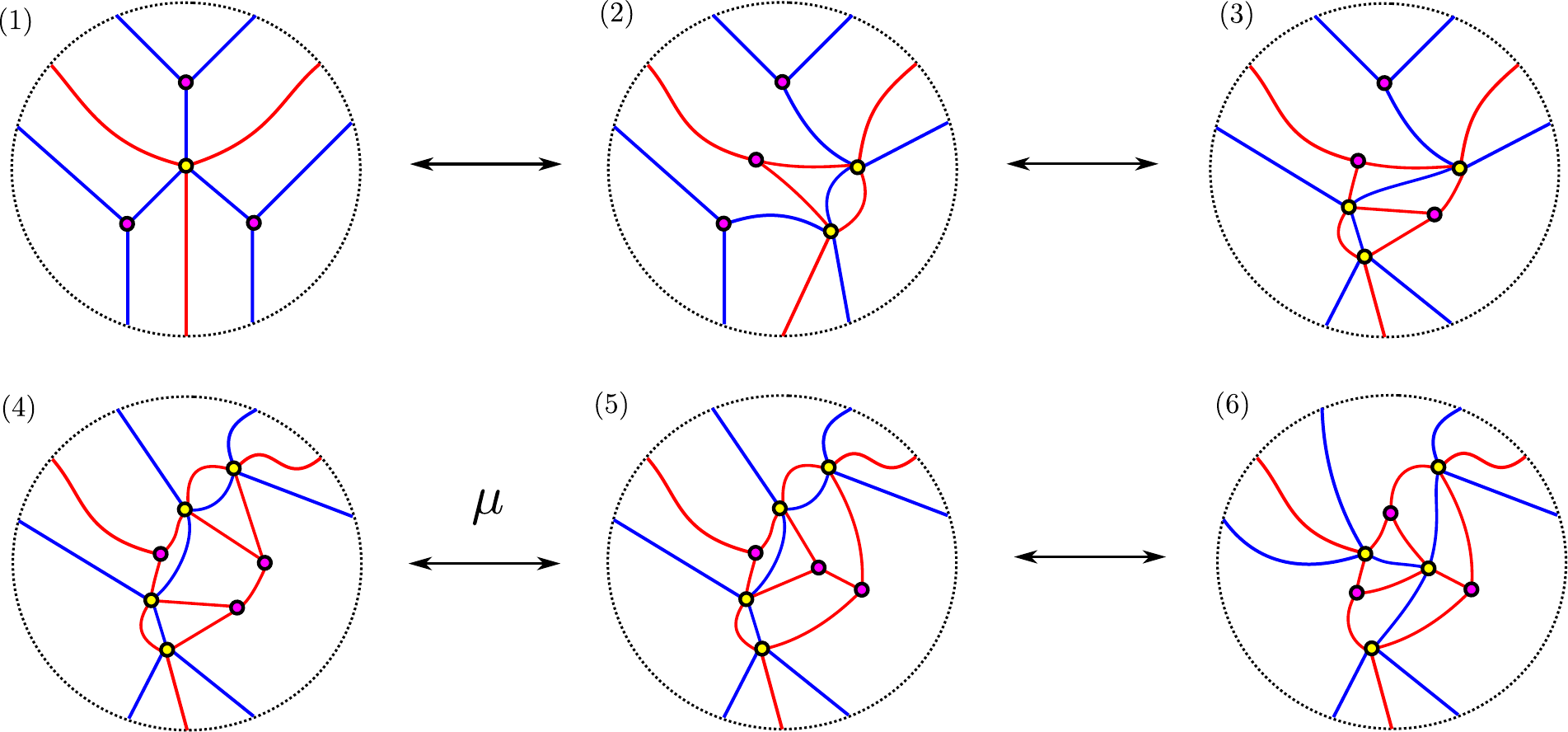}
		\caption{The Legendrian mutation for $3$-graphs as a sequence of Legendrian isotopies and  $2$-graph mutation.}
		\label{fig:LegendrianMutationsN2N3Proof}
	\end{figure}
\end{center}
\end{proof}

For our applications to Lagrangian fillings, it is important to understand how 1-cycle representatives of classes in $H_1(\La(G),\Z)$ change under the mutations depicted in Figure \ref{fig:LegendrianMutation}. Following Subsection \ref{ssec:homology}, we focus on 1-cycles represented by monochromatic edges -- or more generally {\it long} edges -- and by $\sf Y$-cycles. Figure \ref{fig:Mutation_GraphCycles} explicits shows how to transport certain $\sf I$-cycles along the mutation. (See Section \ref{sec:NGraphsLegWeaves}, specifically Subsection \ref{ssec:homology}, for the definition of $\sf I$-cycles.) In addition, mutation along a long edge is dictated by the following:

\begin{center}
	\begin{figure}[h!]
		\centering
		\includegraphics[scale=0.85]{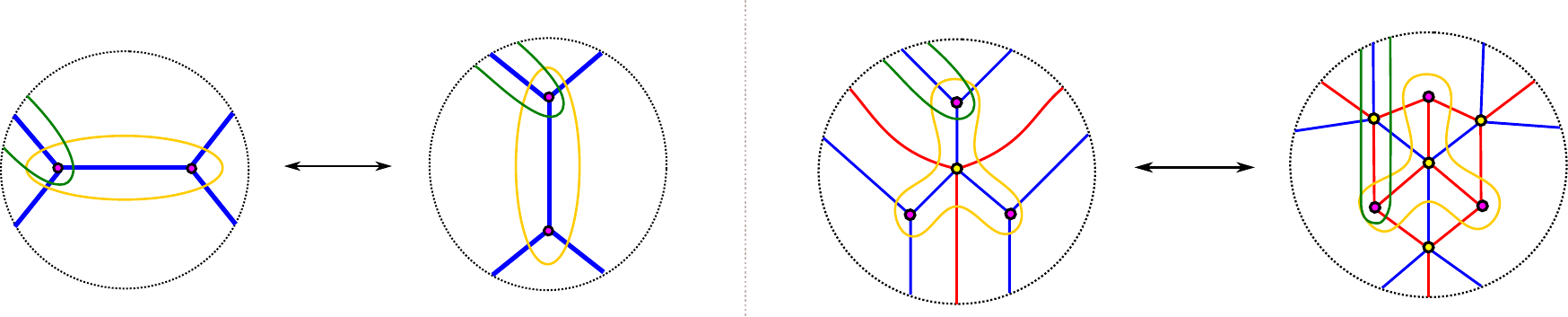}
		\caption{The 2-graph mutation with the additional information of the 1-cycles, before and after the 2-graph mutation (Left). The $\sf Y$-cycle and an incident 1-cycle transforming before and after a mutation along the $\sf Y$-cycle (Right).}
		\label{fig:Mutation_GraphCycles}
	\end{figure}
\end{center}

\begin{cor}\label{cor:LongEdgeMutation}
Let $[\gamma]\in H_1(\La(G),\Z)$ be represented by a long edge in an $N$-graph $G$, as shown in the first row of Figure \ref{fig:Mutation_LongEdge}. Then the Legendrian mutation $\mu_\gamma(\La(G))$ is the Legendrian weave associated to the graph $\mu_\gamma(G)$ as depicted in the second row of Figure \ref{fig:Mutation_LongEdge}.

\begin{center}
	\begin{figure}[h!]
		\centering
		\includegraphics[scale=0.75]{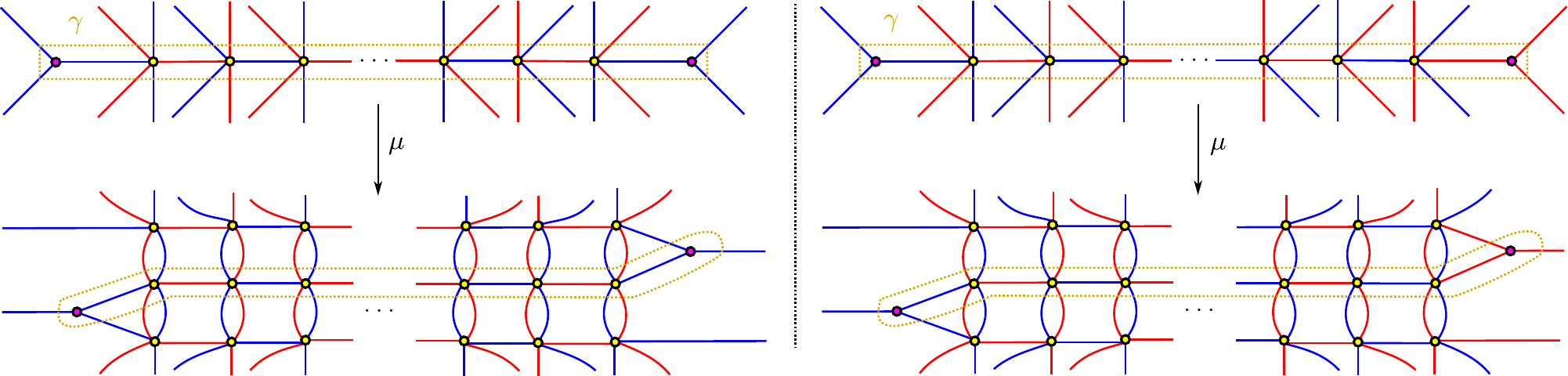}
		\caption{The two cases, left and right, of a Legendrian mutation along a 1-cycle $\gamma$ represented by a long-edge.}
		\label{fig:Mutation_LongEdge}
	\end{figure}
\end{center}
\end{cor}

Theorem \ref{thm:LegMutations} and Corollary \ref{cor:LongEdgeMutation} describe mutations along $\sf Y$-cycles and $\sf I$-cycles, either monochromatic or long edges. In general, we might be interested in mutating along a cycle $\gamma$ which is a tree, both with $\sf Y$-pieces and $\sf I$-piece, as introduced in Section \ref{ssec:homology}. Thus, we now develop {\it local} rules for Legendrian mutations that will allow us to mutation along any such cycle $\gamma$. These rules also imply Corollary \ref{cor:LongEdgeMutation}. All these rules are obtained and proven in the same manner: one simplifies the weave with equivalence moves -- using  Section \ref{sec:moves} -- until the cycle to be mutated becomes a {\it short} $\sf I$-cycle. Then we apply the short $\sf I$-cycle mutation in Figure \ref{fig:LegendrianMutation} (left) and rearrange the weave with moves to the required configuration. For instance, for Corollary \ref{cor:LongEdgeMutation} (and so Figure \ref{fig:Mutation_LongEdge}), we push-through the left-most trivalent vertex through all the hexagonal vertices until the long $\sf I$-cycle becomes a short $\sf I$-cycle. Then we mutate at the short cycle, and push-through one of the trivalent vertices back to the left.



\subsection{Diagrammatic Rules for $N$-graph Mutations}\label{ssec:legmutation_local} Let $\gamma$ be a 1-cycle in an $N$-graph, given by a tree with $\sf Y$-pieces and $\sf I$-pieces. In this subsection we gather the necessary rules for performing a general mutation along $\gamma$ and also diagrammatically carrying a 1-cycle after the mutation at $\gamma$. The rules are {\it local}, either near a hexagonal vertex or a trivalent vertex, and there are three cases that we need to draw: Legendrian mutation being performed at a $\sf Y$-piece, at a $\sf I$-piece, and mutation near a trivalent vertex.

First, we draw the rules for the effect of mutating at a cycle which contains $\sf Y$-pieces:

\begin{itemize}
	\item[(i)] Figure \ref{fig:InternalYVertex} shows how the $\sf Y$-cycle at which we mutate transforms, this cycle is depicted in green. Note that the resulting cycle locally contains only one $\sf Y$-piece.
	\item[(ii)] Figure \ref{fig:InternalYVertex1} explains how to transform the {\it other} $\sf Y$-cycle, in ochre (a darker yellow), under mutation at the green $\sf Y$-cycle in Figure \ref{fig:InternalYVertex}.
	\item[(iii)] Figure \ref{fig:InternalYVertex3} then depicts the transformation of edge $\sf I$-cycles through a hexagonal vertex under mutation at the green $\sf Y$-cycle in Figure \ref{fig:InternalYVertex}.
	\item[(iv)] Finally, Figure \ref{fig:InternalYVertex2} provides the last information needed for carrying any cycle upon mutating at the green $\sf Y$-cycle in Figure \ref{fig:InternalYVertex}. These are the three ways in which a 1-cycle must be continued if the 1-cycle is coming from the extremes of one of the sides.
\end{itemize}

Second, the rules for mutating at a long edge of an $\sf I$-piece of a 1-cycle:
\begin{itemize}
	\item[(v)] Figure \ref{fig:InternalYVertex4} shows how to transform an $\sf I$-piece upon mutation at the green $\sf I$-piece.
	\item[(vi)] Figure \ref{fig:InternalYVertex5} then depicts the transformation of a $\sf Y$-piece of a cycle, in ochre, upon mutation at the green $\sf I$-piece in Figure \ref{fig:InternalYVertex4}.
\end{itemize}

Finally, the local rules for mutating near a trivalent vertex are shown in Figure \ref{fig:InternalYVertex6}. These rules are derived by performing Legendrian Reidemeister moves, especially Move II, until the given cycle at which we want to mutate becomes a monochromatic (short) edge. Then a monochromatic edge mutation is performed, as in Theorem \ref{thm:LegMutations}, and Legendrian Reidemeister moves are performed back to the starting configuration. The two non-canceling applications of a push-through move, before and after a monochromatic edge mutation, are responsible for the tripling behavior seen in the diagrams.

\begin{center}
	\begin{figure}[h!]
		\centering
		\includegraphics[scale=0.8]{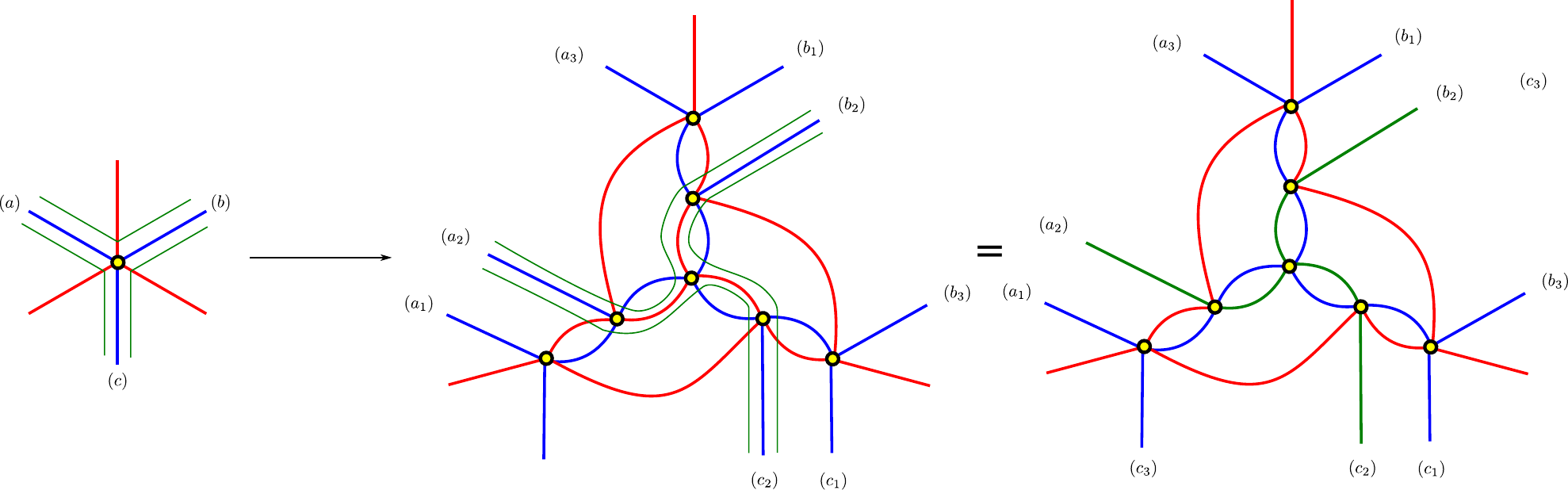}
		\caption{Case Mutation at $\sf Y$-cycle: Internal Mutation along $\sf Y$-piece in green.}
		\label{fig:InternalYVertex}
	\end{figure}
\end{center}

\begin{center}
	\begin{figure}[h!]
		\centering
		\includegraphics[scale=0.5]{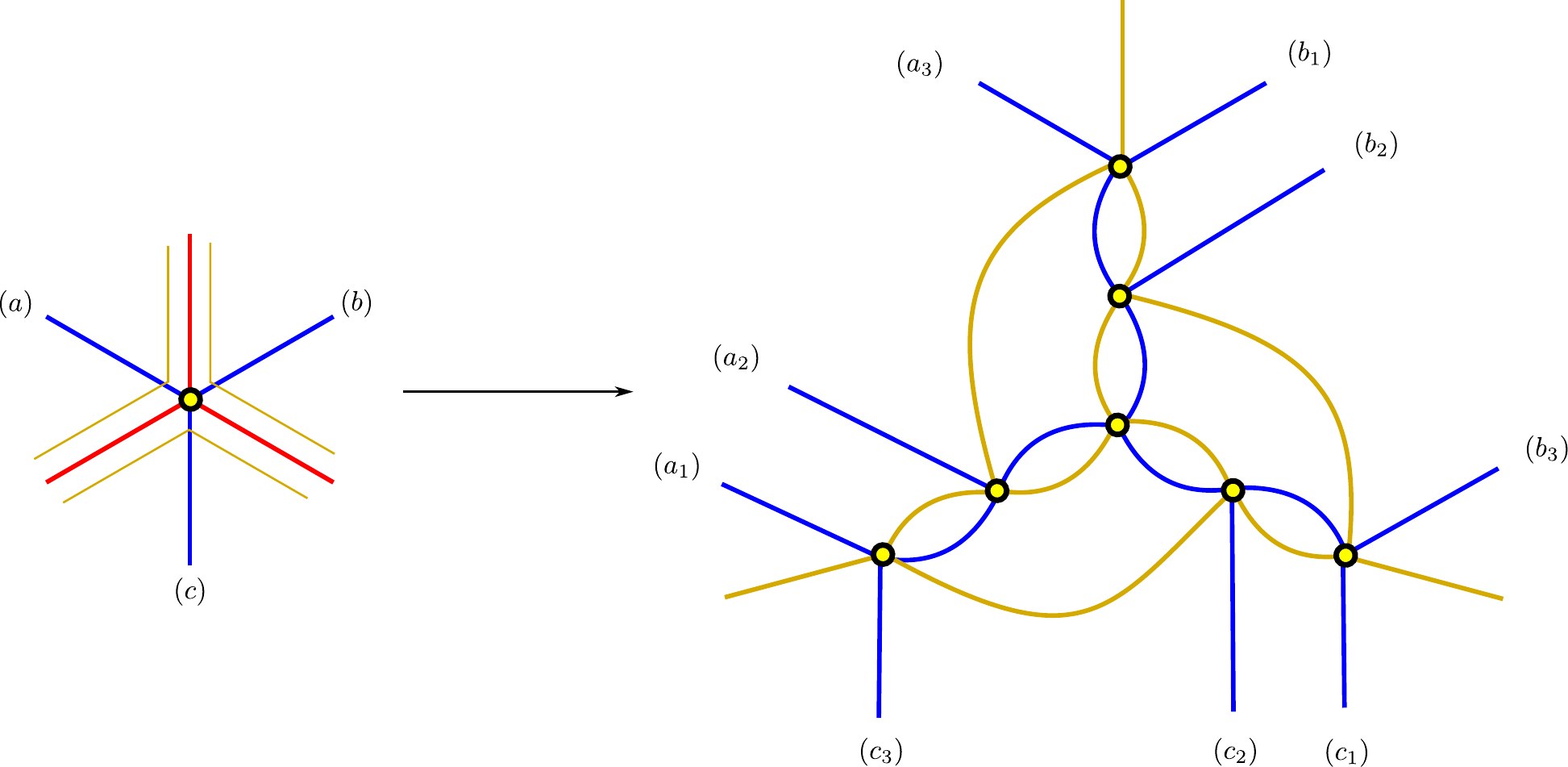}
		\caption{Case Mutation at $\sf Y$-cycle in Figure \ref{fig:InternalYVertex}: Effect for ochre $\sf Y$-cycle of Internal Mutation along $\sf Y$-cycle in green in Figure \ref{fig:InternalYVertex}.}
		\label{fig:InternalYVertex1}
	\end{figure}
\end{center}

\begin{center}
	\begin{figure}[h!]
		\centering
		\includegraphics[scale=0.5]{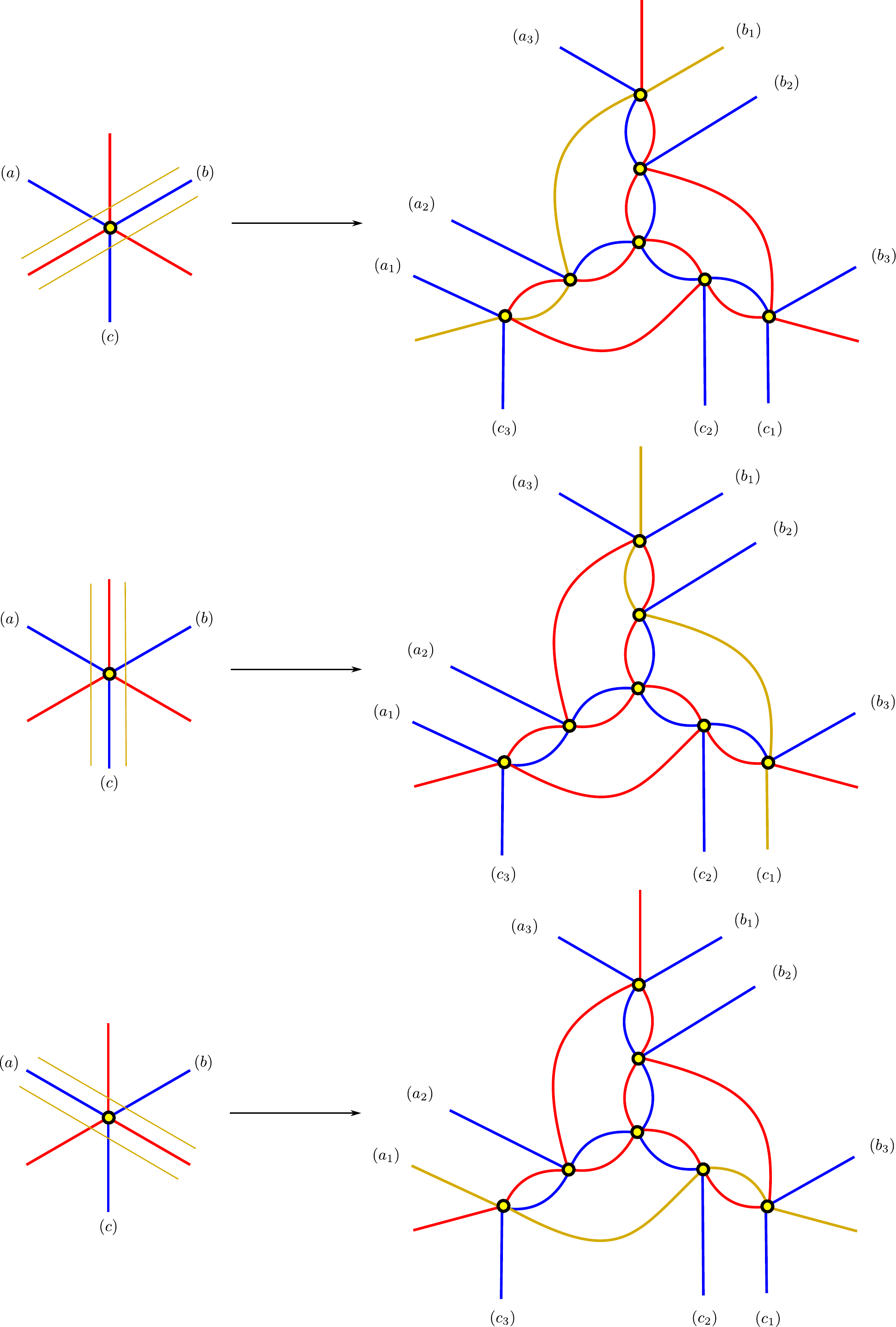}
		\caption{Case Mutation at $\sf Y$-cycle in Figure \ref{fig:InternalYVertex}: Effect for ochre $\sf I$-cycle of Internal Mutation along $\sf Y$-cycle in green.}
		\label{fig:InternalYVertex3}
	\end{figure}
\end{center}

\begin{center}
	\begin{figure}[h!]
		\centering
		\includegraphics[scale=0.6]{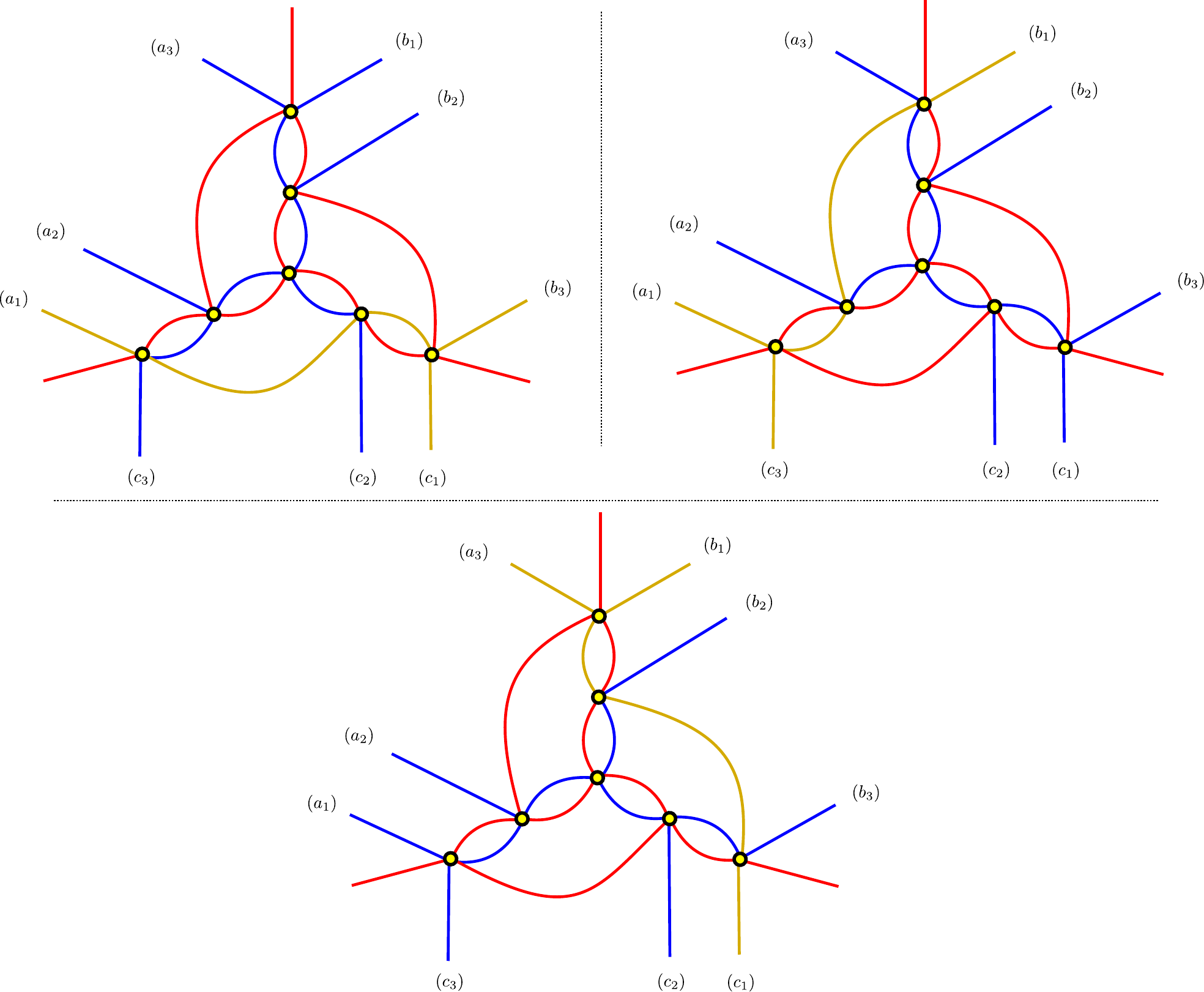}
		\caption{Case Mutation at $\sf Y$-cycle in Figure \ref{fig:InternalYVertex}: Effect for side $\sf I$-cycles of Internal Mutation along $\sf Y$-cycle in green.}
		\label{fig:InternalYVertex2}
	\end{figure}
\end{center}

\begin{center}
	\begin{figure}[h!]
		\centering
		\includegraphics[scale=1.1]{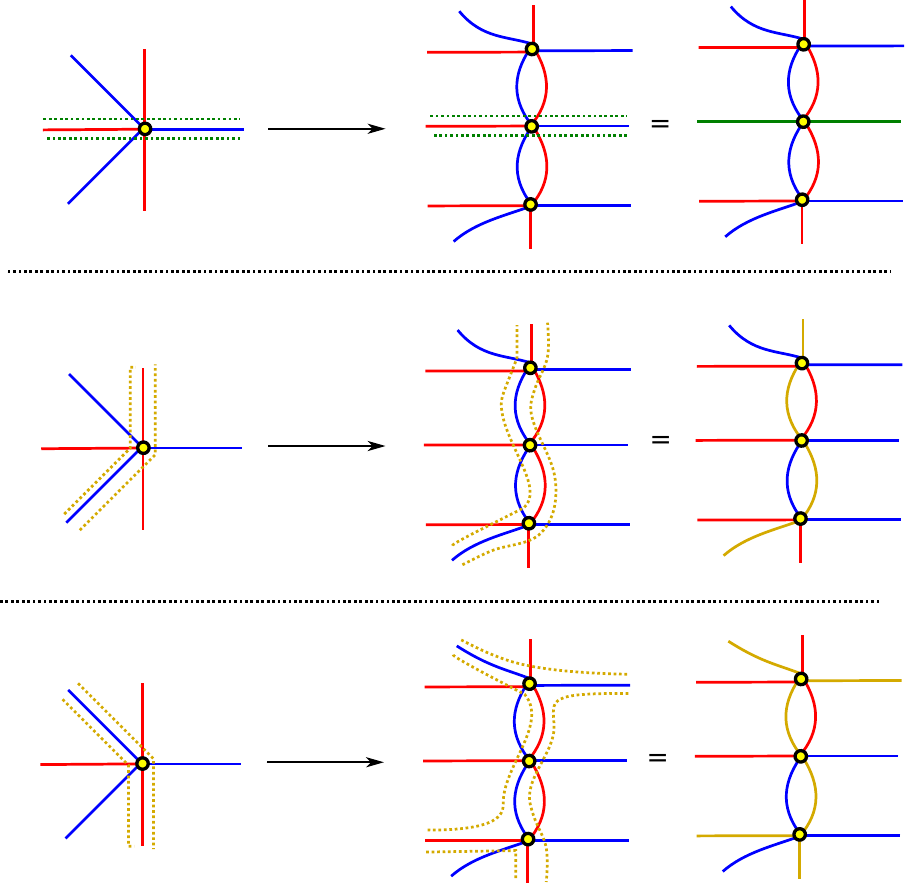}
		\caption{Case Mutation at $\sf I$-cycle in green (upper Left). In second and third row: effect of this mutation for ochre $\sf I$-cycle of Internal Mutation along $\sf I$-cycle in green.}
		\label{fig:InternalYVertex4}
	\end{figure}
\end{center}

\begin{center}
	\begin{figure}[h!]
		\centering
		\includegraphics[scale=1]{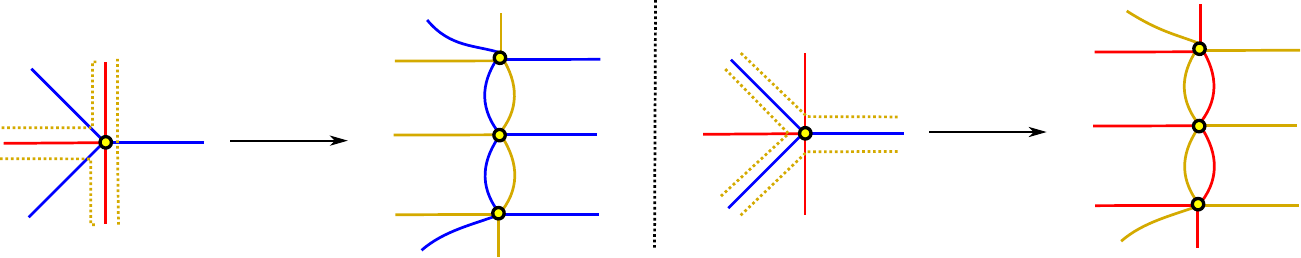}
		\caption{Case Mutation at horizontal $\sf I$-cycle as in Figure \ref{fig:InternalYVertex4}: Effect for ochre $\sf Y$-cycle of Internal Mutation along $\sf I$-cycle in green in Figure \ref{fig:InternalYVertex4} (left).}
		\label{fig:InternalYVertex5}
	\end{figure}
\end{center}

\begin{center}
	\begin{figure}[h!]
		\centering
		\includegraphics[scale=1]{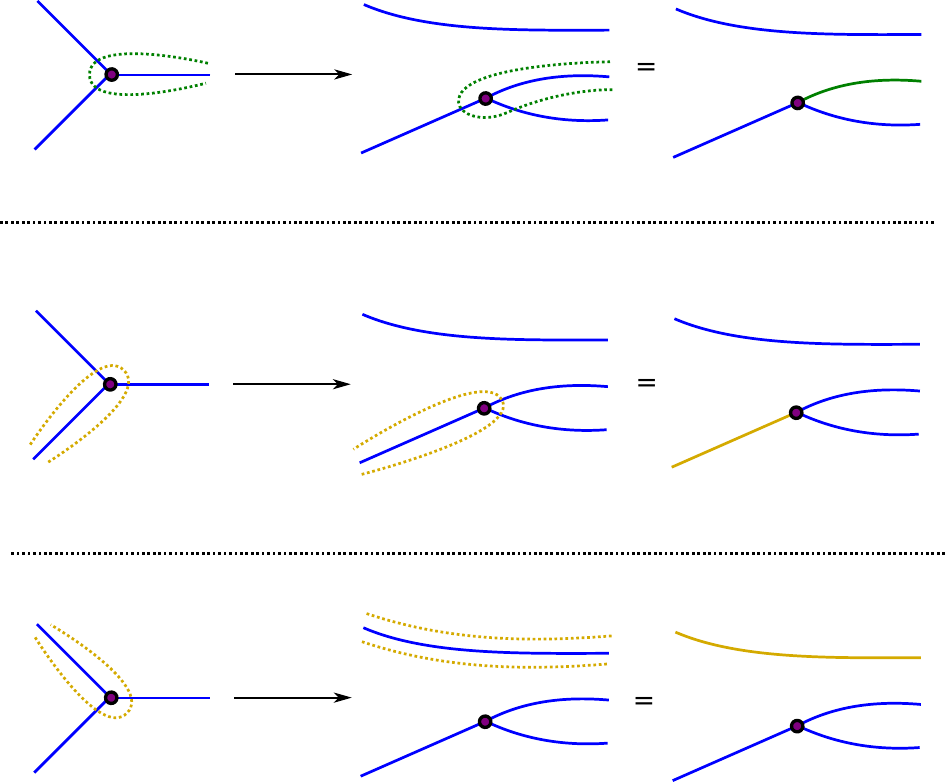}
		\caption{Case Mutation near trivalent vertex for green cycle (first row). Second and third rows: Effect for ochre $\sf Y$-cycle of Internal Mutation at green cycle near trivalent vertex.}
		\label{fig:InternalYVertex6}
	\end{figure}
\end{center}

\newpage


\subsection{Sufficiency For Stabilized Legendrians} Finally, we conclude this section by introducing the following combinatorial idea, motivated by the topology of Legendrian surfaces in 5-dimensional contact manifolds.

\begin{definition}\label{def:bridge} An $N$-graph $G\sse C$ is said to have a \emph{bridge} if there exists two disjoint 2-disks $D_1,D_2\sse C$ such that the complement $G\setminus(G\cap D_1\cup G\cap D_2)$ consists of $(N-1)$ disjoint strands with labels $\tau_1,\tau_2,\ldots,\tau_{N-1}$ consecutive with respect to a transverse oriented curve in $C\setminus(D_1\cup D_2)$.\hfill$\Box$
\end{definition}

For the $N=2$ case, where $G$ is a trivalent graph, a bridge for $G$ according to Definition \ref{def:bridge} coincides with the standard graph-theoretic notion of a bridge \cite{Graph1,Graph2}. A general $N$-graph $G\sse C$ with a bridge is depicted in Figure \ref{fig:StabilizedChart} (left), and an example of a $4$-graph with a bridge is shown in Figure \ref{fig:StabilizedChart} (right).

\begin{center}
	\begin{figure}[h!]
		\centering
		\includegraphics[scale=0.8]{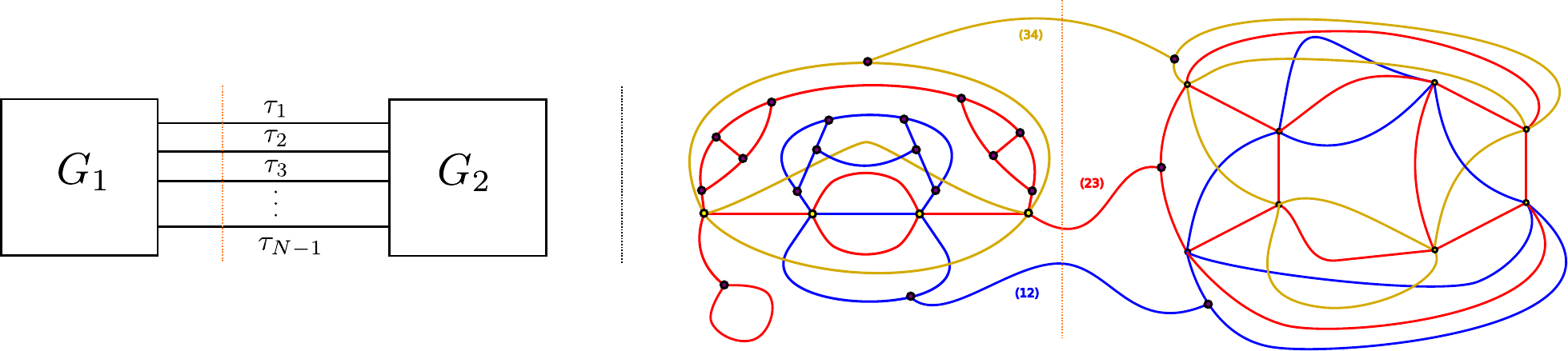}
		\caption{Structure of an $N$-graph with a bridge (left) and instance of a $4$-graph with a bridge (right).}
		\label{fig:StabilizedChart}
	\end{figure}
\end{center}

The geometric motivation for this definition is based on the theory of loose Legendrian surfaces, also known as stabilized Legendrians \cite{Loose}. This class of loose Legendrians are known to satisfy an $h$-principle and has proven to be very useful in the study of Weinstein manifolds \cite{CieliebakEliashberg12,CasalsMurphy}. The reader is referred to \cite{CieliebakEliashberg12,Loose} for further details. For the present manuscript, we will assume known its definition and state the following property:

\begin{prop}\label{prop:bridge}
Let $G\sse C$ be an $N$-graph with a bridge. Then $\iota(\La(G))$ is a loose Legendrian surface.
\end{prop}

\begin{proof}
The proof is a simple argument in the theory of spatial fronts. Indeed, consider the 1-dimensional front slice along the dashed orange line in Figure \ref{fig:StabilizedChart}. The braid shown along this slice is depicted in Figure \ref{fig:StabilizedChartProof} (left). Its closure as a satellite of the standard Legendrian unknot is shown in Figure \ref{fig:StabilizedChartProof} (center). This Legendrian link is isotopic, via a sequence of Reidemeister II moves, to the Legendrian link given by the front in Figure \ref{fig:StabilizedChartProof} (right). The loose chart is exhibited in yellow in this figure. Note that this chart has arbitrarily large thickness due to the dilation freedom in $(\R^5,\xi_\st)$ and the fact that our front is global. This proves that $\iota(\La(G))$ is a loose Legendrian if $G$ has a bridge.
\end{proof}

\begin{center}
	\begin{figure}[h!]
		\centering
		\includegraphics[scale=0.7]{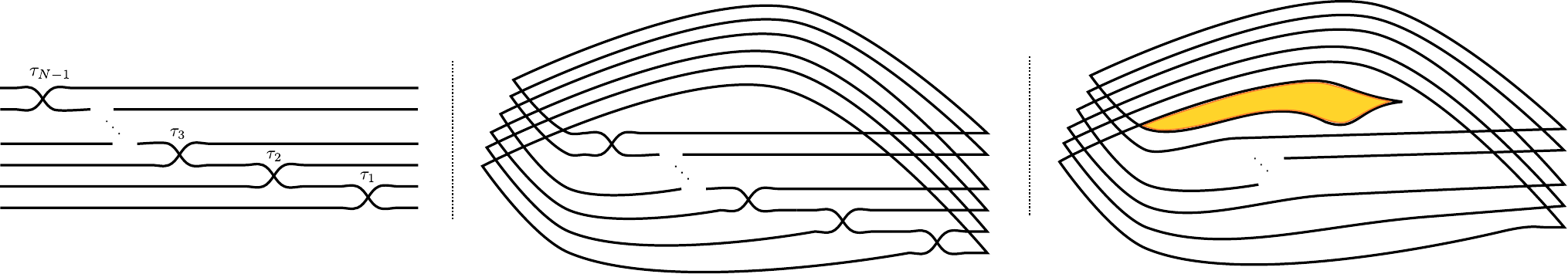}
		\caption{The front for the Legendrian link obtained in a 3-dimensional slice of a bridge (left). The front for the corresponding satellite closure (center) and a homotopic front exhibiting a loose chart (right).}
		\label{fig:StabilizedChartProof}
	\end{figure}
\end{center}

Proposition \ref{prop:bridge} immediately has the following consequence.

\begin{cor}\label{coro:bridge}
Let $G\sse C$ be an $N$-graph with a bridge. Then $\iota(\La(G))\sse(\S^5,\xi_\st)$ admits no exact Lagrangian filling $L\sse(\D^6,\omega_\st)$.\hfill$\Box$
\end{cor}

Corollary \ref{coro:bridge} should be contrasted with the fact that many of the Legendrian surfaces $\iota(\La(G))\sse(\S^5,\xi_\st)$ admit exact Lagrangian fillings. For instance, it follows from Theorem \ref{thm:Legsurgeries} that any 2-graph $G$ obtained from the unique two-vertex 2-graph by adding {\it bigons}, i.e. a 1-surgery, yields a Legendrian surface $\iota(\La(G)))$ which admits exact Lagrangian fillings. On the other hand, {\it simple} 2-graphs do not -- see \cite[Theorem 1.3]{TreumannZaslow}.\\

\begin{ex}[Exact Lagrangian Cobordisms To a Loose Legendrian] Consider the Legendrian Clifford 2-torus $\bT^2_c\sse(\S^5,\xi_\st)$ associated, via the standard satellite, to the 2-graph in Figure \ref{fig:CliffordToBarbell} (Left). By applying our combinatorial Legendrian surgery from Theorem \ref{thm:Legsurgeries}, Figure \ref{fig:LegendrianSurgeries}.(2), we obtain an (index-2) exact Lagrangian cobordism from $\bT^2_c$ to the Legendrian 2-sphere $\La_{l}$ associated Figure \ref{fig:CliffordToBarbell} (Right). By Proposition \ref{prop:bridge}, the Legendrian $\La_{l}$ is a loose Legendrian surface. This proves that the Legendrian Clifford 2-torus $\bT^2_c\sse(\S^5,\xi_\st)$ is a subloose Legendrian surface, and we will show in Section \ref{sec:app} that $\bT^2_c$ is {\it not} a loose Legendrian. In particular, this also proves that $\bT^2_c\sse(\S^5,\xi_\st)$ admits no 3-dimensional exact Lagrangian fillings $L\sse(\D^6,\la_\st)$ in the standard symplectic 6-disk. The points in the non-empty flag moduli associated to $\bT^2_c$ will in fact be geometrically represented by {\it non}-exact Lagrangian fillings. \hfill$\Box$
\end{ex}

\begin{center}
	\begin{figure}[h!]
		\centering
		\includegraphics[scale=0.75]{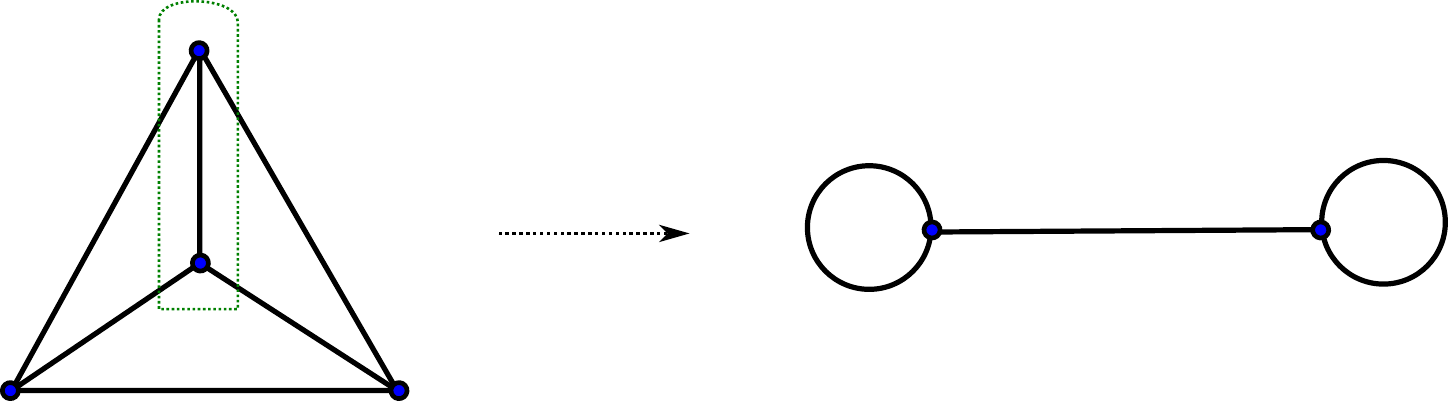}
		\caption{An exact Lagrangian cobordism from a non-loose Legendrian 2-torus to a loose Legendrian 2-sphere.}
		\label{fig:CliffordToBarbell}
	\end{figure}
\end{center}


\section{Flag Moduli Spaces}\label{sec:flag}

In this section we introduce one of the central algebraic invariants in this article, the flag\footnote{``Vexillary'' is the appropriate adjectival form of ``flag''. Hence, it should technically be named the {\it vexillary} moduli space. The word is possibly too obscure, and we thus favor {\it flag} moduli space, as in {\it flag} variety.} moduli space $\SM(G)$ of an $N$-graph $G$ and its associated Legendrian weave. We will prove that these flag moduli spaces are moduli spaces of constructible sheaves associated to a Legendrian weave, but we first present their explicit and self-contained definition.


\subsection{Preliminaries on the Flag Variety}\label{ssec:flagvariety}
Let $N\in\N$ be a natural number and $R$ a commutative ground ring, which will oftentimes be a field. We denote by $\GL_N$ the general linear group, a scheme whose value over $R$ is $\GL(N,R)$, and likewise for $\PGL_N$, the projective general linear group. By definition, a (full or complete) flag is an element

$$\SF^\bullet\in\{\SF^0\subset\SF^1\subset\SF^2\subset\cdots\subset \SF^{N-1}\subset\SF^N: \mbox{dim}\SF^i=i,0\leq i\leq N\},$$
i.e.~a sequence of nested linear subspaces $\SF^i\sse R^N=R\oplus\stackrel{(N)}{\ldots}\oplus R$, $0\leq i\leq N$. Let $B\sse GL_N$ be the Borel subgroup\footnote{This is a maximal Zariski closed and connected solvable algebraic subgroup. Since $B$ is a minimal parabolic subgroup of $\GL_N$ it preserves the {\it most} geometric linear structure in $R^N$, which is precisely a flag $\SF^\bullet$.} of upper triangular matrices preserving the standard coordinate flag. Since $GL_N$ acts transitively on the set of bases, the space that parametrizes such full flags is the homogeneous space $\cB=\GL_N/B$. This is an algebraic variety, known as the flag variety.

The relative position of two flags $(\SF^\bullet,\SG^\bullet)\in \cB\times\cB$ is encoded algebraically by the Bruhat decomposition
$$\GL_N=\displaystyle\bigsqcup_{w\in S_N}BwB,$$
where the symmetric group $S_N=W(\GL_N)$ is identified with the Weyl group. That is, the orbits of the diagonal action of $\GL_N$ on a pair of flags are indexed by the symmetric group $S_N.$  The dimension of the Bruhat cell $BwB$ is the length $\ell(w)$ of the permutation $w\in S_n$. By definition, $\SF^\bullet$ and $\SG^\bullet$ are in {\it transverse} position (or {\it totally transverse} or {\it completely transverse}) if their relative position is $w_0\in S_N$, where $w_0$ denotes the longest element in the Coxeter group $S_N$. Note that $\ell(w_0)={N \choose 2}$, $w_0\in S_N$, and that {\it totally transverse} is the generic relative position between two points in the flag variety $\cB$.
In particular, an elementary transposition $\tau_i\in S_N$ determines a relative position between two flags $\SF^\bullet$ and $\SG^\bullet$ in which only their $i$th vector spaces differ, and no others.

We will require a slight generalization of the above when the surface $C$ is not simply connected: compatible local systems of flags, rather than flags of subspaces of a fixed vector space. 
This will not be required for our applications in Sections \ref{sec:app}, \ref{sec:app2} and \ref{sec:app3}, so the reader is welcome to skip this paragraph.  Let $E\lr X$ be a local system on a topological space $X$. By a local system of flags, we mean a complete filtration (flag) $\SE^\bullet$ of $E$ by local systems $\SE^k$ such that the monodromy preserves the filtration.  In this sense, the flag itself makes global sense.   
Let $U\subset X$ be a subspace and let $\SF^\bullet$ be a flag of sub-local systems on $U$, so that $\SF^k \sse E$ for all $0\leq k \leq N$.  We say that $\SF^\bullet$ is {\it compatible} with $\SE^\bullet$ if the monodromies are:  specifically, for $\gamma\in \pi_1(U,u)$ and $v\in \SF^k$, $i_k(\gamma\cdot v) = i_{k,*}(\gamma) \cdot i_k(v),$ where the symbol $\cdot$ denotes (ambiguously) the action of any group on a vector space.
Note that by monodromy invariance, we may speak of the relative position of two compatible sub-local systems of flags $\SF^\bullet$ and $\SG^\bullet$ on subspaces $U$ and $U'$ of $X$.

With these algebraic preliminaries, we turn to describing the flag moduli space associated to an $N$-graph.


\subsection{Description of the Flag Moduli Space of an $N$-graph}\label{ssec:description_flagmoduli}

Let $G$ be an $N$-graph on a connected surface $C$, thought of as the union of the embedded graphs $G_i$.  By a {\it face} of $G$ we mean the closure of a connected component of the complement $C\setminus G$.  

We first give a general description of the flag moduli space for $C$ not necessarily simply connected.  We will not use this in our applications, so the reader is welcome to skip to the simpler Definition \ref{def:flagmoduli}, which is equivalent when $C$ is simply connected.

Let $\Sigma(G) \subset C\times \bR$ be the wavefront of the Legendrian weave, woven according to $G\sse C$. Call a {\it region} a connected component of the complement $(C\times \bR) \setminus \Sigma(G).$  

\begin{definition}\label{def:flagmoduli-general}
	
	Let $C$ be a connected surface and let $G \sse C$ be an $N$-graph.  The {\it framed flag moduli space} $\wt\SM(C,G)$ associated to $G$ is comprised of the following data.  
	
	\begin{itemize}
		\item[i)]  A rank-$N$ local system $E\lr C$, equivalently a vector space $V$ and a representation of the based fundamental group $\pi_1(C)$ on $V$.\\
		
		\item[ii)] For each face $F$ of the $N$-graph $G$, a compatible local system of flags $\SF^\bullet(F)$.\\
		
		\item[iii)] For each pair of adjacent faces $F_1,F_2$, sharing an $i$-edge $e$, their two associated compatible local systems of flags $\SF^\bullet(F_1),\SF^\bullet(F_2)$ are in relative position $\tau_i\in S_N$, and along the common edge $e$ we have chosen isomorphisms
		$$\SF^j(F_1)\cong \SF^j(F_2),\quad 0\leq j\leq N, j\neq i,$$
		and no other information, as $\SF^i(F_1)\ncong\SF^i(F_2)$.\\
		
		\item[iv)] By gluing, these isomorphisms define local systems in each region, since the $j$th step of a flag of local systems $\SF^j$ compatible with $E$ defines a local system on the
		region between the $j$th and $(j+1)$st sheets --- and these are not separated by a $\tau_i$ crossing of sheets when $j \neq i$.  We require that such local systems in regions, each of which are sub-local systems of $E$ via upward generization morphisms, are compatible with $E$.\footnote{This condition is not local in the $N$-graph, $G$.}
		
	\end{itemize}
	The group $\PGL_N$ acts on the space $\wt\SM(C,G)$ diagonally, i.e.~as isomorphisms of $E$ and on all flags of local systems at once. By definition, the {\it flag moduli space} of the $N$-graph $G$ is the quotient stack
	$$\SM(C,G):=\wt\SM(C,G)/\PGL_N.$$  We simply write $\SM(G)$ when $C$ is understood.
	\hfill$\Box$
\end{definition}

\begin{definition}\label{def:flagmoduli}
	
	Let $C$ be a connected, simply connected surface and let $G\sse C$ be an $N$-graph.  The {\it framed flag moduli space} $\wt\SM(C,G)$ associated to $G$ is comprised of tuples of flags, specifically:  
	
	\begin{itemize}
		\item[i)] There is a flag $\SF^\bullet(F)$ assigned to each face $F$ of the $N$-graph $G$.\\
		
		\item[ii)] 
		\label{def:flagmoduli-general-ii}
		For each pair of adjacent faces $F_1,F_2\sse C\setminus G$, sharing an $i$-edge, their two associated flags $\SF^\bullet(F_1),\SF^\bullet(F_2)$ are in relative position $\tau_i\in S_N$, i.e.~ they must satisfy
		$$\SF^j(F_1)=\SF^j(F_2),\quad 0\leq j\leq N, j\neq i,\mbox{ and }\SF^i(F_1)\neq\SF^i(F_2).$$
	\end{itemize}
	The group $\GL_N$ acts on the space $\wt\SM(C,G)$ diagonally, i.e. on all flags at once. By definition, the {\it flag moduli space} of the $N$-graph $G$ is the quotient stack
	$$\SM(C,G):=\wt\SM(C,G)/\PGL_N.$$
	We simply write $\SM(G)$ when $C$ is understood.
	\hfill$\Box$
\end{definition}

We will equivalently exchange between the linear and projective perspective for a full flag. In the projective setting, flags $\SF^\bullet$ (or local systems of flags) are understood as a sequence of nested projective planes $\P(\SF)^\bullet$, given by the projectivization of the linear spaces of the linear flag $\SF^\bullet$. For a ground field $R$, the moduli space $\SM(C,G;R)$ is representable by an Artin stack of finite type \cite{Artin1,Artin2}, and is typically an algebraic variety (unless $G$ is so symmetric that an admissible configuration of flags might be fixed by $\PGL_N$).

In Subsection \ref{ssec:constsheaves} we explain why the moduli space $\SM(C,G;R)$ is an invariant of the Legendrian isotopy type of the associated Legendrian weave $\La(G)\sse (J^1C,\xi_\st)$. The algebraic questions we are interested in this article are about the different properties and computations of the moduli $\SM(C,G;R)$ --- for instance the cardinality of $|\SM(C,G;\F_q)|$ over a finite field or how $\SM(C,G;R)$ changes upon performing the combinatorial moves in Section \ref{sec:moves}, including Legendrian mutations and surgeries. To ease notation, we will denote flags $\SF^\bullet$ by $\SF$.


\subsection{Sheaf Description of Flag Moduli and Invariance}\label{ssec:constsheaves}

Let $C$ be a smooth surface, $R$ a commutative ring, and $\Sh(C\times\R)$ the category of \emph{constructible sheaves,} i.e.~the $R$-linear dg-derived category of complexes of sheaves of $R$-modules on $C\times\R$ with constructible cohomology sheaves. For algebraic preliminaries on (derived) dg-categories we refer the reader to \cite{Keller_DerivingDG,DGDerived2,DGDerived1,OrlovLunts_DGCat}, and for simplicity we will choose $R$ a field. In this section, we use the identification $J^1(C) \cong T^{\infty,-}(C\times \bR)$ of the first jet bundle of $C$ with downward covectors of $C\times \bR$ --- see \cite[Section 2.1]{NRSSZ}. Now given an $N$-graph $G\sse C$,
the Legendrian $\Lambda(G) \subset J^1(C) \cong T^{\infty,-}(C\times \bR)\subset T^\infty(C\times\R)$ can
be used to define the subcategory $\Sh_{\Lambda(G)}(C\times\R) \subset \Sh(C\times\R)$ whose objects are constructible sheaves whose singular support at contact infinity is contained in $\Lambda(G) \subset T^\infty(C\times \bR)$ --- see \cite[Section 4]{TreumannZaslow}.  

We write $\cC(C,G) := \Sh^1_{\Lambda(G)}(C\times \R)_0 \subset \Sh_{\Lambda(G)}(C\times\R)$ for the subategory of microlocal
rank-one sheaves which are zero in a neighborhood of $C\times \{-\infty\},$ or $\cC(G)$ for short.  This has a simple description, which we now explain.
The dg-category $\Sh_{\Lambda(G)}(C\times\R)$ is itself a subcategory of sheaves constructible with respect to the stratification defined
by the front projection $\Sigma(G)$, and thus has a combinatorial description. 
By \cite[Theorem 8.1.11]{KashiwaraSchapira_Book}, it is equivalent to the dg-category of functors from the poset of strata to $k$-mod (chain complexes) --- see also
\cite[Section 2.3]{Nadler_MicrolocalBrane} and \cite[Section 3.3]{STZ_ConstrSheaves}. 
The subcategory cut out by $\cC(G)$ is the one whose objects are isomorphic to ones with the following properties:  the chain complex assigned to a neighborhood of $C\times \{-\infty\}$ is zero; the complexes in each region of $(C\times \R)\setminus\Sigma(G)$ are rank-one local systems (or just vector spaces if $C$ is simply connected); the morphisms assigned to all downward restriction maps are isomorphisms; and the upward morphisms from small open sets intersecting $\Sigma(G)$ to the regions above them which do not are codimension-one inclusions.

The combinatorial model for this description leads to the flag moduli space $\SM(G)$ of isomorphism classes of objects in $\cC(G).$ Indeed, the flag moduli space $\SM(G)$ associated to an $N$-graph $G\sse C$, as introduced in Definition \ref{def:flagmoduli}, relates to the category $\cC(G) := \Sh^1_{\Lambda(G)}(C\times\R)_0$ according to the following result, which itself generalizes \cite[Section 4.3]{TreumannZaslow} to $N$-graphs with $N\geq3$:

\begin{thm}\label{thm:sheaves}
	The flag moduli space $\SM(C,G;R)$ is isomorphic to the moduli space of objects in $\cC(G) := \Sh^1_{\Lambda(G)}(C\times\R)_0$, the subcategory of microlocal rank-$1$ objects in $\Sh_{\Lambda(G)}(C\times\R)$ supported away from $C\times \{-\infty\}.$ 
\end{thm}

\begin{proof}
	We first assume that $C$ is simply connected.  The argument parallels that of \cite[Sections 6.2 and 6.3]{STZ_ConstrSheaves}, with the additions required by the strictly two-dimensional behavior.  The moduli space of objects is defined locally, meaning that it is the fiber product over its restriction-to-boundary maps of the moduli spaces $\wt\SM$ in neighborhoods of $C$.  We can assume that these neighborhoods of $C$ are chosen small enough so that they are contractible and contain no more than one ``feature'' of the given $N$-graph $G$.  That is, for some such neighborhood $U\sse C$, either $U\cap G$ is empty or contains part of an edge, a single trivalent vertex, or a single hexagonal vertex.  We then have a local study for each of these cases.
	
	In the case where $U$ is empty or contains part of an edge, the front of the Legendrian weave over $U$ is either $N$ parallel sheets or $N$ sheets with a single crossing labeled $\tau_i$, and can be identified with $\sigma \times \bR$, where $\sigma$ is a front of a one-dimensional Legendrian knot being either $N$ parallel lines or $N$ lines with a single crossing.  Then, since the $\bR$ factor is contractible, we can identify the moduli space $\wt\SM$ over $U$ using the one-dimensional study in \cite[Sections 6.2 and 6.3]{STZ_ConstrSheaves}, concluding that it is either the flag variety or pairs of $\tau_i$-transverse flags, respectively.
	
	The moduli $\Sh^1_{\Lambda(G)}(C\times\R)$ is local with respect to $G\sse C$ and the topology of the surface $C$, i.e. it is globally described as fibered products for the local pieces of $G\sse C$. 
	It therefore remains to show that $\Sh^1_{\Lambda(G)}(C\times\R)$ coincides with $\SM(C,G)$ for the local graphs $G_{tri}\sse\D^2$ and $G_{hex}\sse\D^2$, respectively given by a trivalent vertex and a hexagonal vertex, as introduced in Section \ref{sec:NGraphsLegWeaves}.  We do these in turn.
	
	The trivalent vertex case was studied in \cite[Section 4]{TreumannZaslow} for $2$-graphs, and we will make the needed adjustments to $N$-graphs.
	The computation for the local $N$-graph $G_{tri}$ consists of an analysis of the moduli of constructible sheaves supported at the $D_4^-$-wavefront singularity, as directly carried out in \cite{TreumannZaslow}.  The boundary conditions for an object in $\Sh^1_{\Lambda(G_{tri})}(\D^2\times\R)$ consist of a triple of flags $(\SF_1,\SF_2,\SF_3)$ such that $\SF_i\in S_{\tau_{k}}(\SF_j)$ for $i\neq j$, $1\leq i,j\leq 3$, if the edges of $G_{tri}$ are labeled by $\tau_k$.  This can be seen by combining the result for a neighborhood of a single crossing edge above, taking the fiber product over the spaces of flags in the empty neighborhoods in-between.  Then \cite[Section 4.1]{TreumannZaslow} implies that these are all the required conditions (and strata) and thus $\Sh^1_{\Lambda(G_{tri})}(\D^2\times\R)$ coincides with $\SM(\D^2,G_{tri})$. Note that the analysis in \cite[Subsection 4.1.2]{TreumannZaslow} restricts to the case where the local model is a 2-graph $G_{tri}\sse\D^2$, it is readily seen that this model suffices for the analysis of the local model $N$-graph $G_{tri}\sse\D^2$.
	
	\begin{center}
		\begin{figure}[h!]
			\centering
			\includegraphics[scale=0.75]{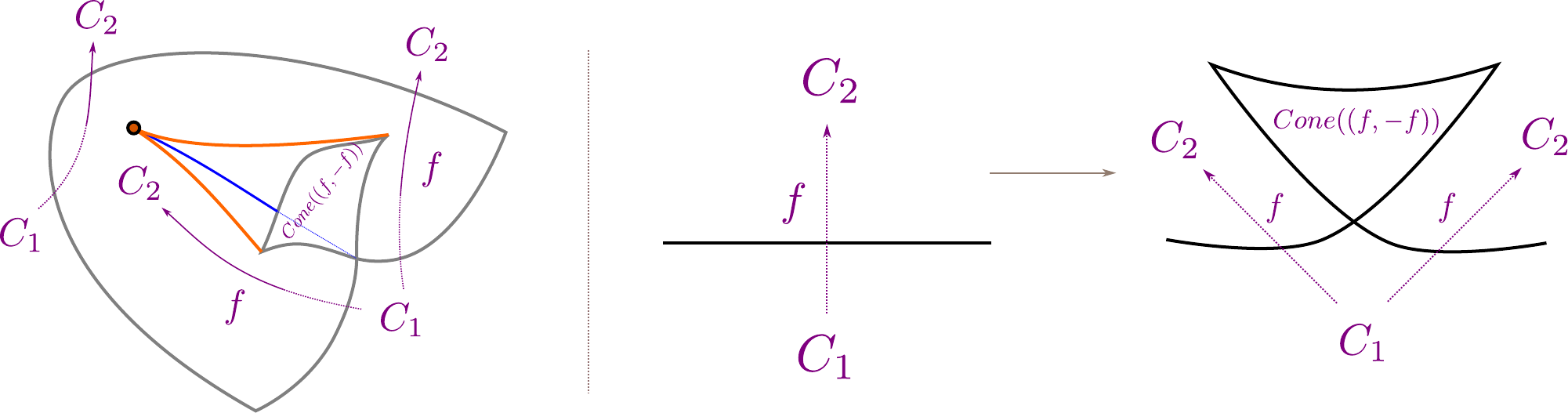}
			\caption{Constructible sheaf in $\R^3$ microlocally supported on the $A_3$-swallowtail singularity (left). The sheaf convolution given by the Guillermou-Kashiwara-Schapira \cite{GKS_Quantization} quantization upon performing a Reidemeister R1 move (right).}
			\label{fig:SheavesSwallowtail}
		\end{figure}
	\end{center}
	
	Alternatively, it is possible to directly conclude the analysis of the $D_4^-$-singularity by performing a generic perturbation of the $D_4^-$-wavefront, as depicted in Figure \ref{fig:D4Generic}, and studying the category of constructible sheaves supported at a $A_3$-swallowtail singularity. Indeed, Figure \ref{fig:SheavesSwallowtail} (left) shows the conditions for a constructible sheaf microlocally supported along the front of an $A_3$-swallowtail singularity, which consists of a choice of injective map $f:C_1\lr C_2$, where $C_1\cong R^k$ and $C_1\cong R^{k+1}$, for some $k\in\N$. The crucial fact is that the (stalk of the) sheaf in the remaining 3-dimensional open strata $I$ is uniquely determined to be the cone of the map $(f,-f):C_1\lr C_2\oplus C_2$. This is a consequence of the Guillermou-Kashiwara-Schapira quantization \cite[Theorem 3.7]{GKS_Quantization} of Legendrian isotopies: since the $A_3$-swallowtail is the big wavefront \cite{ArnoldSing} of the first Reidemeister move for 1-dimensional Legendrian fronts, it follows that the sheaves in the strata $I$ are uniquely determined by $f:C_1\lr C_2$ by the sheaf kernel associated to the first Reidemeister move. It is readily seen \cite{STZ_ConstrSheaves} that the result of the convolution with such a kernel yields the sheaf transformation in Figure \ref{fig:SheavesSwallowtail} (right). By the non-characteristic property of the category of microlocal sheaves \cite{GKS_Quantization}, the sheaves microlocally supported on the wavefront of the $D_4^-$-singularity is equivalent to that for a generic perturbation of such $D_4^-$-singularity. The generic perturbation consists of three $A_3$-swallowtails and the conditions for the constructible sheaves on these stratification follow from the above analysis. In conclusion, we obtain an isomorphism $\Sh^1_{\Lambda(G_{tri})}(\D^2\times\R)\cong\SM(\D^2,G_{tri})$.
	
	Let us now address the hexagonal vertex $G_{hex}$. Since the Legendrian weave $\Sigma(G_{hex})$ is the big wavefront of the third Reidemeister move for 1-dimensional Legendrian fronts, it suffices to understand the kernel of its quantization. Figure \ref{fig:SheavesHexVertex} shows the local transformation for constructible sheaves near the third Reidemeister move \cite[Section 4.4.3]{STZ_ConstrSheaves}.
	
	\begin{center}
		\begin{figure}[h!]
			\centering
			\includegraphics[scale=0.75]{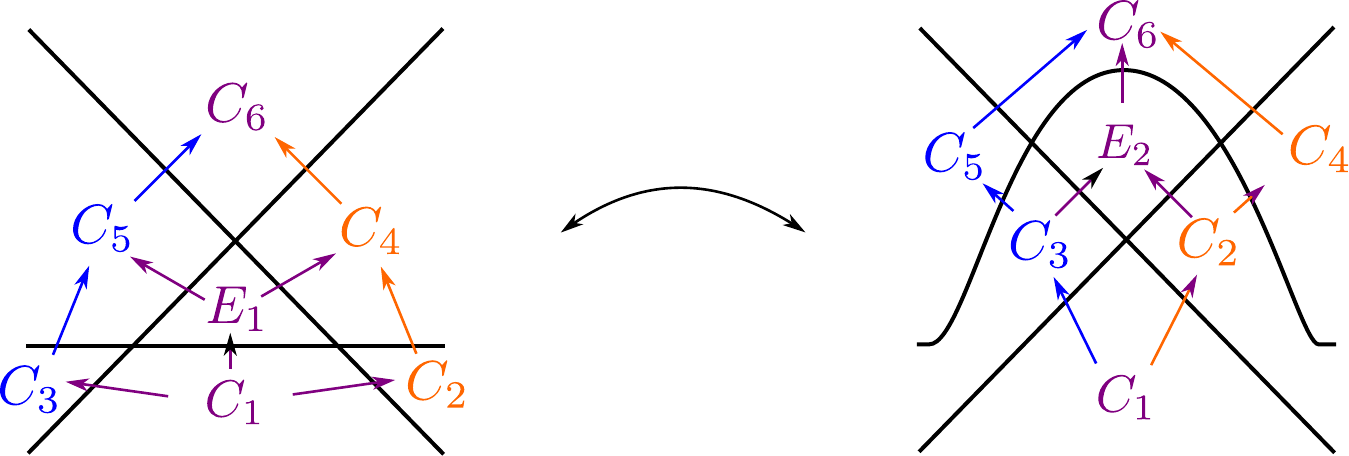}
			\caption{The explicit flag exchange given by the Guillermou-Kashiwara-Schapira \cite{GKS_Quantization} quantization upon performing a Reidemeister R3 move.}
			\label{fig:SheavesHexVertex}
		\end{figure}
	\end{center}
	
	In Figure \ref{fig:SheavesHexVertex}, the $C_i$, $1\leq i\leq5$ and $E_1,E_2$ are complexes of vectors spaces, which we can actually assume to be vector spaces \cite[Section 3.3]{STZ_ConstrSheaves}. If $C_1\cong R^k$, for some $k\in\N$, the microlocal rank 1 condition implies that $E_1,E_2,C_2,C_3\cong R^{k+1}$, $C_4,C_5\cong R^{k+2}$ and $C_6\cong R^{k+3}$. The four flags at one of the sides of the hexagonal vertex are
	$$\SF^{(1)}_1=C_1\lr C_3\lr C_5\lr C_6,\quad \SF^{(1)}_2=C_1\lr E_1\lr C_5\lr C_6,$$ 
	$$\SF^{(1)}_3=C_1\lr E_1\lr C_4\lr C_6,\quad \SF^{(1)}_4=C_1\lr C_2\lr C_4\lr C_6,$$
	and the four flags at the other side of the hexagonal vertex are
	$$\SF^{(2)}_1=C_1\lr C_3\lr C_5\lr C_6,\quad \SF^{(2)}_2=C_1\lr C_3\lr E_2\lr C_6,$$ 
	$$\SF^{(2)}_3=C_1\lr C_2\lr E_2\lr C_6,\quad \SF^{(2)}_4=C_1\lr C_2\lr C_4\lr C_6.$$
	The three crossings in Figure \ref{fig:SheavesHexVertex} (left) imply, from left to right, that
	$$\SF^{(1)}_1\in S_{\tau_{k+1}}(\SF^{(1)}_2),\SF^{(1)}_2\in S_{\tau_{k+2}}(\SF^{(1)}_3),\SF^{(1)}_3\in S_{\tau_{k+1}}(\SF^{(1)}_4).$$
	Similarly, the three crossings in Figure \ref{fig:SheavesHexVertex} (right) imply, from left to right, that
	$$\SF^{(2)}_1\in S_{\tau_{k+2}}(\SF^{(2)}_2),\SF^{(2)}_2\in S_{\tau_{k+1}}(\SF^{(2)}_3),\SF^{(2)}_3\in S_{\tau_{k+2}}(\SF^{(2)}_4).$$
	These are precisely the conditions for the flag moduli space $\SM(\D^2,G_{hex})$ in Definition \ref{def:flagmoduli}, and hence $\Sh^1_{\Lambda(G_{hex})}(\D^2\times\R)\cong\SM(\D^2,G_{hex})$.
	
	This concludes the argument for the case where $C$ is simply connected.
	We now turn to the case where $C$ is not simply connected.  There are no further local conditions.  The only additional concerns regard compatibilities of local systems.
	
	Let $\Sigma(G) \subset C\times \bR$ be the wavefront of the Legendrian weave, and recall that we call a {\it region} a connected component of the complement $(C\times \bR) \setminus \Sigma(G).$  A constructible sheaf in $\cC(G)$ restricts to a local system on each region, since there is no singular support away from the wavefront.  There are two distinguished regions $R_{\rm top}$ and $R_{\rm bot}$ containing neighborhoods of $C\times \{\infty\}$ and $C\times \{-\infty\}$, respectively.  A constructible sheaf in $\cC(G)$ restricts to $0$ in $R_{\rm bot}$ (by definition) and to a local system on $R_{\rm top} \sim C$ that we assign to be the data $E$ from Definition \ref{def:flagmoduli-general}(i). 
	Now, as explained in Definition \ref{def:flagmoduli-general}(iv), the data of a point in $\wt\SM(G)$ defines a local system in each region.
	Commutativity of sheaf restriction maps requires that a section which is parallel transported around a region and then included into $E$ arrives at the same place as a section which is included first and then parallel transported around $R_{\rm top},$ and this is the requirement of Definition \ref{def:flagmoduli-general}(iv).	
\end{proof}


\subsection{Local Flag Moduli Computations}\label{ssec:computations_flagmoduli} Let us prove the following useful lemmas on the flag moduli, which can be implicitly used when performing computations on $\SM(C,G;R)$. In this section, and subsequent computations, we will consider a ground field $R=k$, with $k=\C$ and finite fields $k=\F_q$ as the main fields of interest.

We start with the study of the flag moduli space at a trivalent vertex, as depicted in the left of Figure \ref{fig:VertexType}, and characterize that local flag moduli space.

\begin{lemma}\label{lem:trivalent}
Consider the neighborhood $\Op_3(N)$ of a $\tau_i$-trivalent vertex in an $N$-graph. Then the local moduli of flags $\SM(\Op_3,G;k)$ is set-theoretically a point, and the $\PGL_N$-action on $\wt\SM(\Op_3,G;k)$ has stabilizer $(k^*)^{N-2}\times k^{{N\choose 2}-1}$.
\end{lemma}

\begin{proof}
Let $\SF_1,\SF_2,\SF_3$ be the three flags in $\Op_3(N)$. The $\GL_N$-action is transitive on the space of flags, and thus $\SF_1$ can be mapped to the standard flag $\SS_1$, defined by $$\SS^j_1=\{x_{j+1}=x_{j+2}=\ldots=x_{N-1}=x_N=0\},\mbox{ where }k=\Spec k[x_1,\ldots,x_N].$$
The $\GL_N$-action allows us to also map the two flags $\SF_2$ and $\SF_3$, respectively, to $\SS_2$ and $\SS_3$, defined by
$$\SS^j_2=\SS^j_3=\SS^j_1,\quad 0\leq j\leq N,\quad j\neq i,$$ $$\SS_2^j=\{x_{i}=x_{i+2}=\ldots=x_{N-1}=x_N=0\},$$ $$\SS_3^j=\{x_{i}-x_{i+1}=x_{i+2}=\ldots=x_{N-1}=x_N=0\}.$$
This implies that the quotient of the moduli $\wt\SM(\Op_3,G;R)$ by the gauge group $\PGL_N$ is set-theoretically a point. In order to recover its structure as a quotient stack, it suffices to identify the stabilizer of the triple of flags $\SS_1,\SS_2,\SS_3$. For that, notice that the stabilizer of $\SS_1$ is the projectivization of the Borel subgroup of upper triangular matrices, isomorphic to $(k^*)^{N-1}\times k^{{N\choose 2}}$. The condition of fixing the flag $\SS_2$ transversely cuts out a $k$-coordinate in the interior of the upper triangle, since it sets the $(i,i+1)$ entry equals to zero. This cuts the stabilizer down to $(k^*)^{N-1}\times k^{{N\choose 2}-1}$, and finally stabilizing $\SS_3$ imposes the equality of the two diagonal entries $(i,i)$ and $(i+1,i+1)$, thus transversely cutting down a $k^*$. The resulting stabilizer is $(k^*)^{N-2}\times k^{{N\choose 2}-1}$, as claimed.
\end{proof}

In its simplest instance of $N=2$, this is the statement that three distinct points in the projective line $\P^1(k)$ can be sent to $\{0,1,\infty\}$ with trivial stabilizer. A lesson from Lemma \ref{lem:trivalent} is that for any $N$, near at least one trivalent vertex of an $N$-graph, we are allowed to use the gauge group $\PGL_N$ and fix the flags around that vertex. The (proof of the) lemma also provides the (geometric) degrees of freedom left after this choice.

\begin{example}
Consider the $2$-graph $G$ associated to the triangulation of $C=S^2$ with two triangles. This $2$-graph $G$, dual to the triangulation, has two vertices, three edges and three faces. Then the flag moduli space $\SM(\S^2,G;\C)$ consists of a point $\{*\}$. In fact, this point $\{*\}$ of the flag moduli space geometrically corresponds to the conjecturally unique Lagrangian 3-disk filling of the standard Legendrian unknot $\La_0\sse(\S^5,\xi_\st)$.\hfill$\Box$ 
\end{example}

Lemma \ref{lem:trivalent} is a statement about a particular triple of flags. It ought to be noted that a {\it generic} triple of flags is part of a moduli space of dimension ${N-1\choose 2}$, with birational coordinates given by generalized triple ratios -- see Section \ref{sec:app2} and \cite[Section 9]{FockGoncharov_ModuliLocSys}. The flags appearing in the context of our $N$-graphs are in general a combination of non-generic flags, arising from the local vertices, with a flag being modified at exactly one degree when crossing an edge.

Let us now address our second local model at a vertex, that of a hexagonal vertex, as depicted in the right of Figure \ref{fig:VertexType}.

\begin{lemma}\label{lem:hexagonal} Consider the neighborhood $\Op_6(N)$ of a hexagonal vertex, with edges $\tau_i$, $\tau_{i+1}$ and consecutively ordered flags $\SF_j$, $j \in \bZ/6\bZ$. Then any pair of opposite flags $\SF_k, \SF_{k+3}$ determines the others.	
\end{lemma}

\begin{proof}
By symmetry, it suffices to show that the flags $\SF_1, \SF_4$ determine $\SF_5$ and $\SF_6$.  We assume that $\SF_4$ and $\SF_5$ are separated by a $\tau_{i+1}$ edge --- a similar argument will work if it is of type $\tau_i$.
By the prescribed transversality, we have $\SF^{j}_6=\SF^j_1$ and $\SF^{j}_5=\SF^{j}_4$ for $j\neq i+1$.
Now since $\SF_5^i\neq\SF_6^i$, and $\SF_1^{i+2}=\SF_4^{i+2}$, there exists a unique linear subspace $V\sse\SF_1^{i+2}$ which contains $\SF_5^i,\SF_6^i$.  So we must have $\SF^{i+1}_5=\SF^{i+1}_6=V$, uniquely determining the flags $\SF_5$ and $\SF_6$.
\end{proof}

A direct application of Lemma \ref{lem:hexagonal} is the invariance of the moduli of flags under the $N$-graph Reidemeister Move I from our Theorem \ref{thm:surfaceReidemeister} above:

\begin{cor}\label{cor:CandyInv} The flag moduli space $\SM(C,G;R)$ is invariant under the {\it candy} twist.\hfill$\Box$
\end{cor}

The candy twist -- Move I -- is the move depicted in Figure \ref{fig:Reidemeister1} above, and the proof of Corollary \ref{cor:CandyInv} follows immediately from Lemma \ref{lem:hexagonal}, since the interior faces of the local model are uniquely determined by two opposing boundary flags, and they in turn determine the remaining ones. Corollary \ref{cor:CandyInv} also follows from Theorem \ref{thm:DiagrammaticsI} and the Legendrian invariance proven in \cite[Theorem 3.7]{GKS_Quantization}. The invariance of the moduli of flags under the other moves in Theorem \ref{thm:surfaceReidemeister} can be proven similarly by direct means.

Lemma \ref{lem:hexagonal} discusses the flags in a neighborhood of a hexagonal vertex and allows for a computation of the local flag moduli space $\SM(\Op_6(N);R)$ at a hexagonal vertex, since it reduces it to the study of a quadruple of flags.

\begin{example}
Let us illustrate this point by computing $\SM(\Op_6(3);\C)$, which we claim is isomorphic to a point stabilized by the subgroup $(\C^*)^2\sse\PGL(3,\C)$. Indeed, the incidence problem at a hexagonal vertex is given by six flags
$$\SF_1=(p_1,l_1),\quad \SF_2=(p_2,l_1),\quad \SF_3=(p_2,l_2),$$
$$\SF_4=(p_3,l_2),\quad\SF_5=(p_3,l_3),\quad\SF_6=(p_1,l_3)$$
where $p_i$ and $l_i$, for $1\leq i\leq 4$, are points and lines in $\P^2(\C)$ and the notation $(p_i,l_i)$ stands for the projectivized flag $p_i\in l_i$. Since the three points $p_1,p_2,p_3$ are pairwise distinct, $\PGL(3,\C)$ acts on them transitively, and their stabilizer is the (projectivization) of a maximal torus in $\GL(3,\C)$, which is isomorphic to $(\C^*)^2$. Lemma \ref{lem:hexagonal} provides a more direct route: it suffices to observe that the $\PGL(3,\C)$-stabilizer of the two completely transverse flags $\SF_1,\SF_4$ is the set of diagonal matrices in $\PGL(3,\C)$, i.e. $(\C^*)^2$.\hfill$\Box$
\end{example}

It is an exercise to extend the argument for Lemma \ref{lem:trivalent} above in this context and show that:

\begin{lemma}\label{lem:hexagonal2}
Consider the neighborhood $\Op_6(N)$ of a $(\tau_i,\tau_{i+1})$-hexagonal vertex in an $N$-graph. Then the local moduli of flags $\SM(\Op_6,G;k)$ is set-theoretically a point, and the $\PGL_N$-action on $\wt\SM(\Op_6,G;k)$ has stabilizer $(k^*)^{2}\times\left((k^*)^{N-3}\times k^{{N\choose 2}-3}\right)$.\hfill$\Box$
\end{lemma}

Having computed the local models at trivalent and hexagonal vertices, in Lemmas \ref{lem:trivalent} and \ref{lem:hexagonal2}, we now address the local flag moduli space around a $\tau_i$-edge connecting two trivalent vertices for $1\leq i\leq N-1$. Thanks to our discussion in Subsection \ref{ssec:homology} on the homology of the associated Legendrian weaves, we know that this is the flag moduli space associated to a Legendrian cylinder. In contrast to Lemmas \ref{lem:trivalent} and \ref{lem:hexagonal2} above, we will now discover that the local flag moduli space around a monochromatic edge is (set-theoretically) non-trivial.

\begin{lemma}[Flag Cross-ratio]\label{lem:crossratio}
Let $G$ be an $N$-graph, and $e\in G$ a monochromatic edge between two trivalent vertices. The local flag moduli space $\SM(\Op(e),G;k)$ in a neighborhood $\Op(e)$ is isomorphic to $k^*$ with stabilizer $(k^*)^{N-2}\times k^{{N\choose 2}-1}$, under the $\PGL_N$-action. 
\end{lemma}

Lemma \ref{lem:crossratio} appears in the study of cluster coordinates for $2$-graphs in the works \cite{FockGoncharov_ModuliLocSys,TreumannZaslow}, yet a treatment of it here, in the context of $N$-graphs, seems in order. The interesting part in Lemma \ref{lem:crossratio} is the existence of a non-trivial flag moduli space around the edge $e\in G$. The stabilizer only appears due to the dependence on $N$. Note also that, by using Lemma \ref{lem:hexagonal}, the statement in Lemma \ref{lem:crossratio} can readily be generalized for a long edge $e$, i.e. an $\sf I$-cycle between two trivalent vertices, as described in Section \ref{sec:NGraphsLegWeaves}.

\begin{proof}
For an edge $e\in G$ between two trivalent vertices, it suffices to discuss the case of a monochromatic edge, since the push-through move preserves the flag moduli. In this case, let $v_1,v_2$ be the two endpoints of $e$. By Lemma \ref{lem:trivalent}, the local flag moduli space around $v_1$ can be fixed to be a point with stabilizer $(k^*)^{N-2}\times k^{{N\choose 2}-1}$. In this normalization, the flag moduli space around $v_2$ is determined in two of the sectors, and thus it is uniquely described by the remaining choice of flag. This is tantamount to the choice of a fourth point in $\P^1(k)\setminus\{0,1,\infty\}$, which yields a modulus of $k^*$. 
\end{proof}

In general, the existence of a non-trivial $1$-cycle $\gamma\in H_1(\La(G),\Z)$ provides the flag moduli space with a $k^*$ factor, which can be geometrically interpreted as being a contribution of the microlocal monodromy of the associated local system induced in the Legendrian surface $\La(G)$, as we explain in Section \ref{sec:app2}. The following example illustrates this point in the case of a $\sf Y$-cycle in $G$.

\begin{example}\label{ex:Ycycle}
Let us compute the local flag moduli space in an $N=3$ neighborhood of a $\sf Y$-cycle, as depicted in Figure \ref{fig:1Cycle} (Right). The configurations of points for this incidence problem are given by the following conditions:

\begin{itemize}
\item[(a)] Three distinct points $p_0,p_1,p_2$, and three points $q_i\in l_i=\langle p_i,p_{i+1}\rangle$, where the index $0\leq i\leq 2$ is understood modulo $3$,\\

\item[(b)] The triples $\{p_i,p_{i+1},q_i\}$, $0\leq i\leq 2$, are triples of distinct points.
\end{itemize}

The action of $\PGL_3$ allows us to set $p_0=[1:0:0]$, $p_1=[0:1:0]$ and $p_2=[0:0:1]$ with a $(k^*)^2$ Cartan stabilizer, and this stabilizer can then be used to fix $q_0=[1:1:0]\in l_0$ and $q_1=[0:1:1]$. The remaining choice of $q_2$ yields the $k^*$ contribution to the flag moduli space since it is a choice of a point $q_2\in l_2$ distinct from $p_1,p_2$.\hfill$\Box$
\end{example}

This concludes our local computations of flag moduli spaces $\SM(C,G)$. We now study the behavior of the invariant $\SM(C,G)$ under Legendrian surgery, and Sections \ref{sec:app} and \ref{sec:app3} will develop global computation of flag moduli spaces. Given an $N$-graph $G\sse C$, we ease notation by writing $\SM(G)$ for $\SM(C,G)$.


\subsection{Flag Moduli under Legendrian Surgeries} Let $G,G'$ be $N$-graphs such that $G'$ is obtained by Legendrian surgery on $G$, as described in Theorem \ref{thm:Legsurgeries}. The following result relates the flag moduli spaces $\SM(G)$ and $\SM(G')$ before and after Legendrian surgery.

\begin{thm}\label{thm:FlagLegSurgeries} Let $k$ be a field and $G$ an $N$-graph. For any $\tau_i$-edges of $G$, the flag moduli space $\SM(G)$ satisfies the following local relations:
	\begin{figure}[h!]
		\centering
		\includegraphics[scale=0.5]{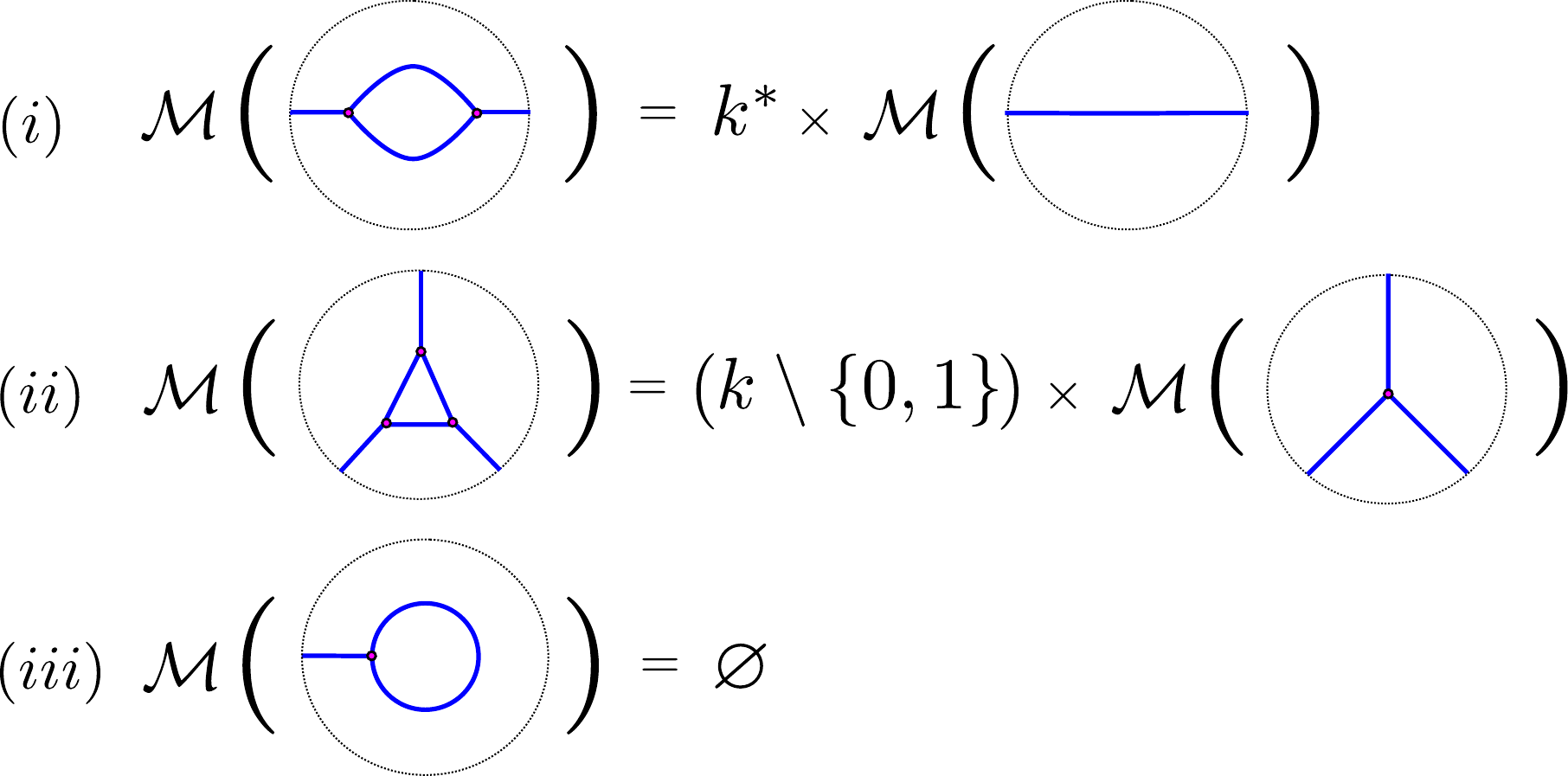}
		\caption{The change of the flag moduli spaces $\SM(G)$ under combinatorial changes in a piece of an $N$-graph $G$.}
		\label{fig:FlagModuliUnderLegendrianSurgeries}
	\end{figure}
\end{thm}

\begin{proof}
The relations can be verified with our description of the flag moduli space in Subsection \ref{ssec:description_flagmoduli}. We can also argue directly thanks to the geometry developed in Section \ref{sec:moves}. Indeed, the moduli of objects of the category of constructible sheaves microlocally supported at a Legendrian connected sum $\La_1\#\La_2$ is a direct product of the moduli of objects microlocally supported at $\La_1$ and those microlocally supported at $\La_2$. By Theorem \ref{thm:Legsurgeries}, the right and left graphs $G_r,G_l$ for the Relations $(i)$ and $(ii)$ geometrically correspond to Legendrian connected sums with the standard Legendrian 2-torus $\bT^2_\st$, and the Legendrian Clifford 2-torus $\bT^2_c$, respectively. The flag moduli for the former is $k^*$, and for the latter it is $k\setminus\{0,1\}$, which concludes $(i),(ii)$. Finally, the relation $(iii)$ follows from Proposition \ref{prop:bridge}, as there do not exist constructible sheaves microlocally supported at a loose Legendrian.
\end{proof}

Note that, by construction, there exists a 3-dimensional exact Lagrangian cobordism $L(G,G')$ from $\La(G)$ to $\La(G')$, in the symplectization of $(J^1C,\xi_\st)$. Thus, from the standard results in Floer theory \cite{EkholmEtnyreSullivan05a,EliashbergGiventalHofer00}, we expect\footnote{To our knowledge, these maps have yet to be studied in the context of microlocal sheaf theory. The expectation that they exist comes from the fact that the flag moduli space $\SM(G)$ should correspond to an augmentation variety for $\La(G)$, and these maps are known to exist between augmentation varieties.} a map from $\SM(G)\times H_1(L(G,G'),k)$ to $\SM(G')$. Theorem \ref{thm:FlagLegSurgeries} gives a strong indication of what these maps should be, i.e. for $(i),(ii)$, $\SM(G')$ is a $k^*$- or a $(k\setminus\{0,1\})$-bundle over $\SM(G)$, with the map being a section for this bundle projection.

\subsection{Non-characterstic Property of Stabilization} We conclude Section \ref{sec:flag} with an interesting and direct computation of flag moduli spaces. First, note that the proof of Theorem \ref{thm:graph_stab}, showing that the standard satellites of $\La(G)$ and $\La(s(G))$ are Legendrian isotopic, and Theorem \ref{thm:sheaves} imply the isomorphism $$\SM(\Lambda(G))\cong\SM(\Lambda(s(G))),$$
where $s(G)$ is the stabilization we introduced in Subsection \ref{ssec:Stabilization}. We will nevertheless provide a self-contained sheaf-theoretical proof of that equivalence, which we now illustrate in the case $N=2$.

\begin{proof}[Proof of flag moduli space equivalence $N=2$]
	In that case, the moduli of objects in the category $\SM(\Lambda(s(G)))$ parametrizes flags in $\P^2$ up to $\PGL(3,\C)$ equivalence abiding the constraints imposed by the 3-graph on the left of Figure \ref{fig:Stabilization_Sheaf}. We assume that the 2-graph $G=G^{1,2}$, before stabilizing, contains at least a vertex.
	
	\begin{center}
		\begin{figure}[h!]
			\centering
			\includegraphics[scale=0.75]{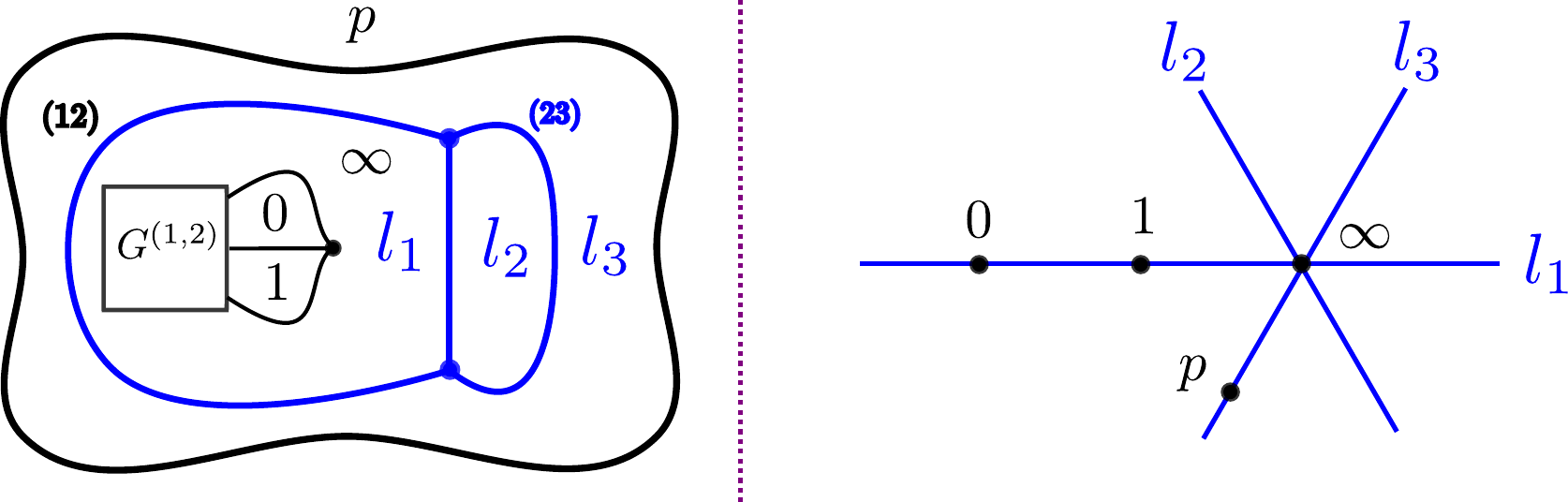}
			\caption{The flag configuration for $N=2$ stabilization.}
			\label{fig:Stabilization_Sheaf}
		\end{figure}
	\end{center}
	
	The graph $G^{(1,2)}$ imposes constraints on the points lying in a line $l_1\sse\P^2$, the ladybug changes this line to distinct lines $l_2,l_3$, also different from $l_1$, and the descending $(12)$-halo provides the freedom of a point $p\in l_3$. The fact that $G^{(1,2)}$ is contained in a wing of the ladybug implies that $l_1\cap l_2\cap l_3$ is a point, which for now we denote $\infty$. Let us show that this moduli space coincides with the moduli space of points in $l_1$ imposed by $G^{(1,2)}$. For that, note that the stabilizer of three non-collinear points $p_1,p_2,p_3\in\P^2$ is isomorphic to $\C^*\times\C^*$; indeed, it is isomorphic to the space of invertible diagonal matrices in $\PGL(3,\C)$. Geometrically, each of the $\C^*$ allows us to move any point in one of the three possible lines spanned by two of the three points $\{p_1,p_2,p_3\}$ around that line, on the complement of these two spanning points.
	
	Hence we can start by using the $\PGL(3,\C)$ and fix the points $1,p,\infty\in\P^2$ in the configuration shown in the right of Figure \ref{fig:Stabilization_Sheaf}, which determine the lines $l_1,l_3$. From the $\C^*\times\C^*$ we can use the first $\C^*$ in order to send the third point in $l_1$ imposed by $G^{(12)}$ to $0\in l_1$, and the second $\C^*$ to choose a point in the line $l=\langle 1,p\rangle$, which in turn determines a line $l_2\sse\P^2$ by taking its span with $\infty\in l_1\cap l_3$. This fixes the configuration of lines $l_1,l_2,l_3$ and the points ${0,1,\infty,p}\sse\P^2$ with $\{0,1,\infty\}\sse l_1$, and that is precisely the three points being fixed by the $\PGL(2,\C)$ symmetry acting in $G^{(12)}$.
\end{proof}

This argument is self-contained, yet hopefully illustrates how in general the {\it geometric} conclusion from Theorem \ref{thm:graph_stab}, and the invariance of the flag moduli space $\SM$ under Legendrian isotopy, are stronger and neater tools than the strict algebraic invariance of the flag moduli space. Let us now move forward with the following Sections \ref{sec:app}, \ref{sec:app2} and \ref{sec:app3}, which display several applications of the techniques developed in Sections \ref{sec:NGraphsLegWeaves}, \ref{sec:constr}, \ref{sec:moves} and \ref{sec:flag}, and in particular prove Theorems \ref{thm:intro1}, \ref{thm:symmetries_intro}, \ref{thm:ThurstonLinksIntro} and \ref{thm:FlagModuli_Ntriangle_Intro} stated in the introduction.


\section{Applications and Vexillary Computations}\label{sec:app}

In this section we study applications of our diagrammatic calculus for Legendrian weaves $\La(G)$ associated to an $N$-graph $G$, and their flag moduli spaces $\SM(G)$. In particular, we will prove Theorem \ref{thm:intro1} and Theorem \ref{thm:symmetries_intro}.


\subsection{First Pair of Computations}\label{ssec:FirstComputations} Let us start with two simple examples of Legendrian weaves and their flag moduli: the Legendrian Clifford torus and the double $t_4\cup t_4\sse\S^2$ of the 4-triangle $t_4$ in the 2-sphere $\S^2$.


\subsubsection{The Legendrian Clifford Torus} Let us consider the 2-graph $G=(\dd\Delta^3)^{(1)}\sse\S^2$ in Figure \ref{fig:CliffordSum2}, which has already featured in the proof of Theorem \ref{thm:Legsurgeries}. The flag moduli space $\SM(G)$ is readily seen to be the pair of pants $\P^1\setminus\{0,1,\infty\}$. Indeed, there are four contractible connected components in $\S^2\setminus G$, which implies that
$$\wt\SM(G)=\{(p_1,p_2,p_3,p_4)\in(\P^1)^4:p_i\neq p_j,\quad i\neq j\}$$

where $\P^1\cong\GL(2,\C)/B$ is the flag variety of lines in $\C^2$. Since $\PGL(2,\C)$ acts 3-transitively on $\P^1$, we can assume that $(p_2,p_3,p_4)=(0,1,\infty)$, and the quotient $\wt\SM(G)/PGL(3,\C)$ is given by
$$\SM(G)=\{\lambda\in\P^1:\lambda\neq 0,1,\infty\}.$$
This flag moduli space is shown in Figure \ref{fig:CliffordSum2} (left), which is uniquely determined by the choice of $\lambda\in\P^1\setminus\{0,1,\infty\}$.
\begin{center}
	\begin{figure}[h!]
		\centering
		\includegraphics[scale=0.7]{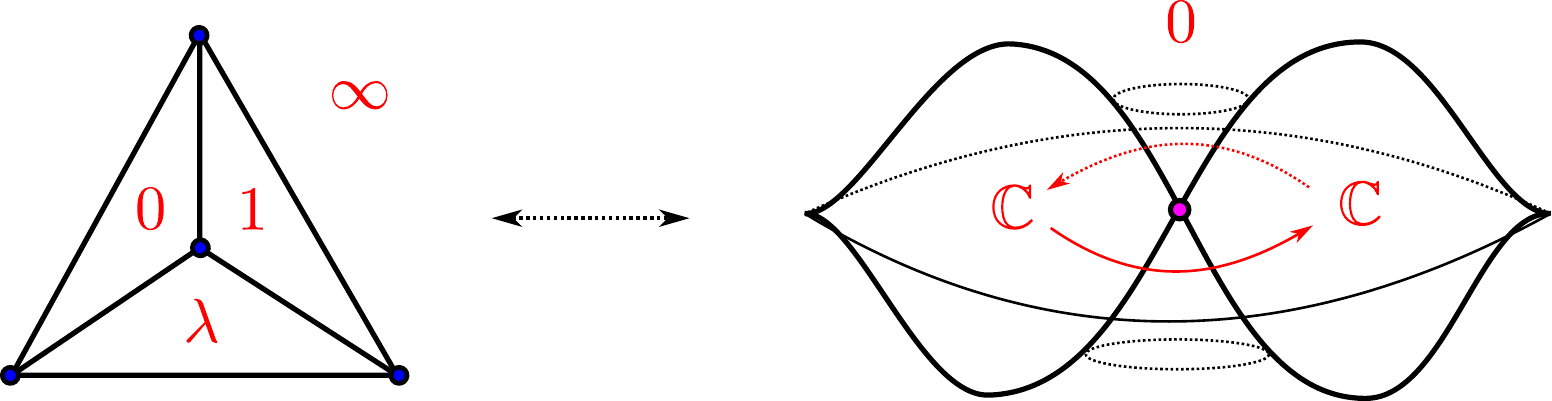}
		\caption{The tetrahedral 2-graph $G$ as a planar projection of the 1-skeleton $(\Delta^3)^{(1)}$ of the tetrahedron $\dd\Delta$ (left). A front projection for the Legendrian 2-torus $\iota(\La(G))$ (right).}
		\label{fig:CliffordSum2}
	\end{figure}
\end{center}

Let us illustrate the Legendrian geometry in this case. The Euler characteristic of the Legendrian weave $\La(G)$ is $\chi(\La(G))=2\cdot\chi(\S^2)-4=0$, and thus $\La(G)$ is a closed 2-torus. A different front for $\La(G)$ is depicted in Figure \ref{fig:CliffordSum2} (right), where the cone singularity \cite[Section 2]{CasalsMurphy} is used, in line with the description in \cite[Section 3]{Rizell_TwistedSurgery}. The flag moduli space $\SM(G)$ for the 2-graph $G$ is read in this front as the moduli space of constructible sheaves in $\R^3$ microlocally supported with rank-1 in the front Figure \ref{fig:CliffordSum2} (right). This latter moduli is given with the data of a 1-dimensional vector space $\C$ in the bounded region in the interior of the front and a linear monodromy map $\lambda:\C\lr\C$. The monodromy must be an isomorphism, and thus $\lambda\in\GL(1,\C)\cong\C^*$, and also satisfy the additional constraint imposed by the cone singularity. By generically perturbing this singularity, it is readily seen that the condition is that the monodromy $\la$ does not have $1$ has an eigenvalue, which in this case reduces to $\la\in\C\setminus\{0,1\}\cong\P^1\setminus\{0,1,\infty\}$. This is precisely the flag moduli space $\SM(G)$.\hfill$\Box$

\begin{remark}[\cite{Nadler_LG1,TreumannZaslow}]\label{rmk:HMS} This particular wavefront allows for a direct Legendrian analysis of the Landau-Ginzburg model $(\C^3,z_1z_2z_3)$, as follows. The regular fiber $F\sse\C^3$ of the superpotential is isomorphic to $F\cong(\C^*)^2$, and its Lagrangian skeleton is thus an exact 2-torus $\bT^2\sse F$, i.e. the vanishing cycle for the (non-isolated) singularity $W$. Its Legendrian lift
$$\La:=\{(z_1,z_2,z_3)\in\C^3:|z_1|=|z_2|=|z_3|=1/3,\arg(z_1)+\arg(z_2)+\arg(z_3)=0\}\sse(\S^5,\xi_\st),$$
has vanishing (singular) thimble the conic Lagrangian
$$L=\{(z_1,z_2,z_3)\in\C^3:W(z_1,z_2,z_3)\in\R^+,|z_1|=|z_2|=|z_3|\}.$$
By performing a real blow-up at the origin, we introduce a real 2-sphere $\S^2$ at the origin and a projection map $\pi:\La\lr\S^2$ from our Legendrian 2-torus onto this exceptional 2-sphere $\S^2$. In coordinates, the map $\pi(z_1,z_2,z_3)=(\Re(z_1),\Re(z_2),\Re(z_3))$ is just given by taking the real parts of the complex coordinates and realizes the Legendrian surface $\La\sse(\S^5,\xi_\st)$ as the Legendrian weave $\iota(\La(G))$ associated to the four-vertex 2-graph $G\sse\S^2$, given by the 1-skeleton of the tetrahedron. Thus, the mirror of the Landau-Ginzburg model $(\C^3,z_1z_2z_3)$ is the Legendrian 2-torus in $(J^1\S^2,\xi_\st)$ which satellites to the Clifford 2-torus $\bT^2_c\sse(\S^5,\xi_\st)$. This leads to the description of the A-model Landau-Ginzburg model $(\C^3,z_1,z_2,z_3)$, given by the category $\mu\mbox{Sh}_L(\C^3)$ of wrapped sheaves, as the bounded dg-category of finitely-generated torsion complexes on the flag moduli space $\SM(\bT^2_c)\cong\P^1\setminus\{0,1,\infty\}$.\hfill$\Box$
\end{remark}


\subsubsection{The Double of the 4-Triangle} Let us consider the 4-graph $G(t_4)$ associated to a 4-triangle $t_4$, as depicted in Figure \ref{fig:Example3Graph2Triang} (left), and described in Section \ref{sec:constr}. Let $G=G(t_4)\cup_{\dd}G(t_4)\sse\S^2$ be the 4-graph obtained by gluing two copies of this 4-graph along their boundaries, i.e. $G$ is the 4-graph associated to the 4-triangulation of $\S^2$ with two underlying $t_1$-triangles. The 4-graph $G$ is depicted in Figure \ref{fig:Example3Graph2Triang} (right), where the circle at the boundary is identified to a unique point, which is a hexagonal vertex.

\begin{center}
	\begin{figure}[h!]
		\centering
		\includegraphics[scale=0.7]{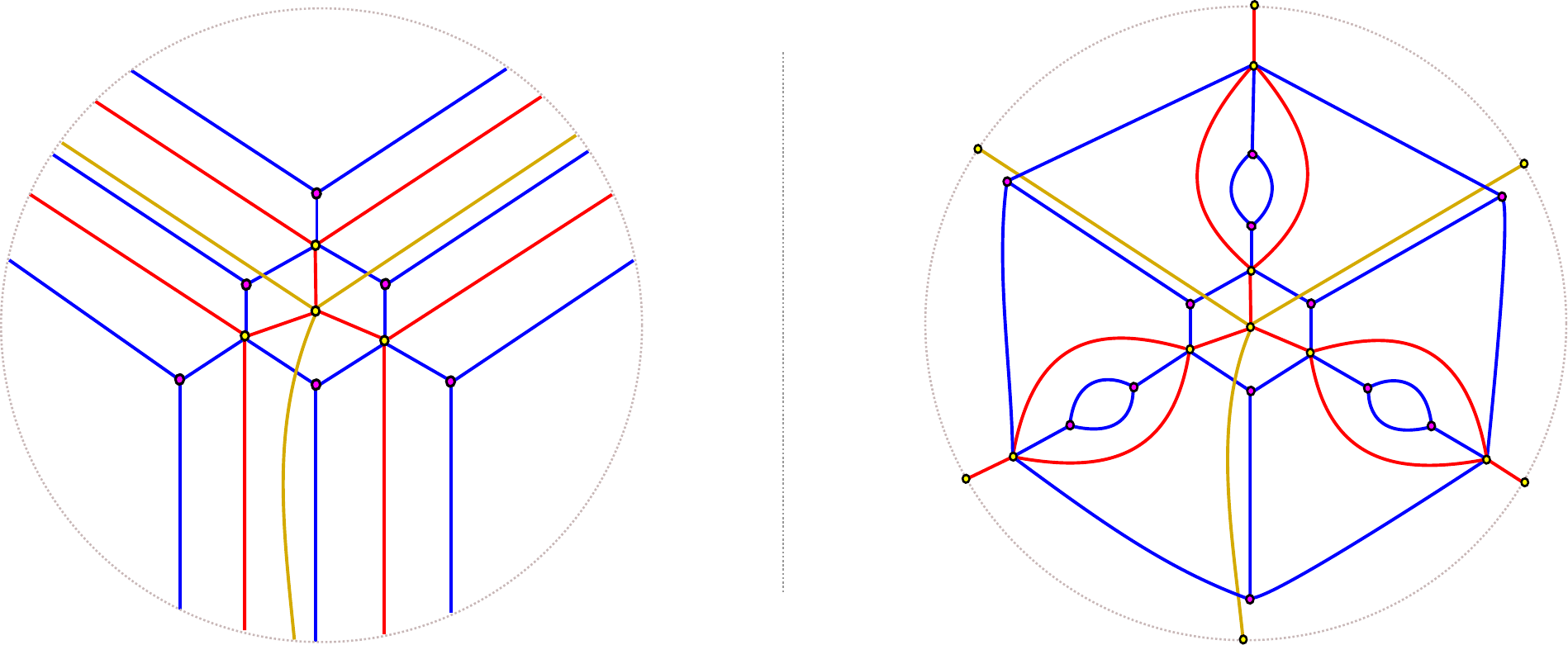}
		\caption{The local 4-graph $G(t_4)$ associated to a 4-triangle $t_4$ (left). The global 4-graph $G(\tau_4)$ given by the 4-triangulation $\tau_4$ of the 2-sphere $\S^2$ with two triangles.}
		\label{fig:Example3Graph2Triang}
	\end{figure}
\end{center}

For the computation of the flag moduli space $\SM(G)$, we employ our geometric techniques in Section \ref{sec:moves}. Theorem \ref{thm:Legsurgeries} allows us to remove the initial three (blue) $\tau_1$-bigons, by considering a direct sum with three copies of the standard Legendrian 2-torus $\mathbb{T}^2_{st}$, see Section \ref{sec:moves}. By applying Move I in Theorem \ref{thm:surfaceReidemeister} three times, we obtain the 3-graph in Figure \ref{fig:Example3Graph2Triang2} (left). Further removing three of the bigons, we reach the 3-graph $G_0$ in Figure \ref{fig:Example3Graph2Triang2} (right). The framed flag moduli space $\wt\SM(G_0)$ for the 3-graph $G_0$ is given by the choice of two flags $\SF_1=(p_1,l_1,\pi_1),\SF_2=(p_2,l_2,\pi_2)\in \GL_4/B$ in projective 3-space, and a choice of three points $p_3,p_4,p_5\in\P_k^3$ such that
\begin{itemize}
	\item[-] $(l_1,\pi_1)$ and $(l_2,\pi_2)$ are completely transverse, i.e. $l_1\not\in\pi_2$ and $l_2\not\in\pi_1$, and $p_1\neq p_2$,
	\item[-] $p_3\in l_1$, $p_3\neq p_1$,
	\item[-] $p_4\in l_2$, $p_3\neq p_2$,
	\item[-] $p_5\in \pi_1\cap\pi_2$, $p_3\neq p_1$.
\end{itemize} 
In particular, $\SM(G_0)\cong\wt\SM(G_0)/\PGL_4$, and the flag moduli space $\SM(G_0)$ is described by the data above. By Theorem \ref{thm:Legsurgeries}, and the fact that each bigon contributes to $k^*$ once the Legendrian weave is connected, we deduce that our original flag moduli space must be isomorphic to $\SM(G)\cong\SM(G_0)\times(k^*)^4$.

\begin{center}
	\begin{figure}[h!]
		\centering
		\includegraphics[scale=0.75]{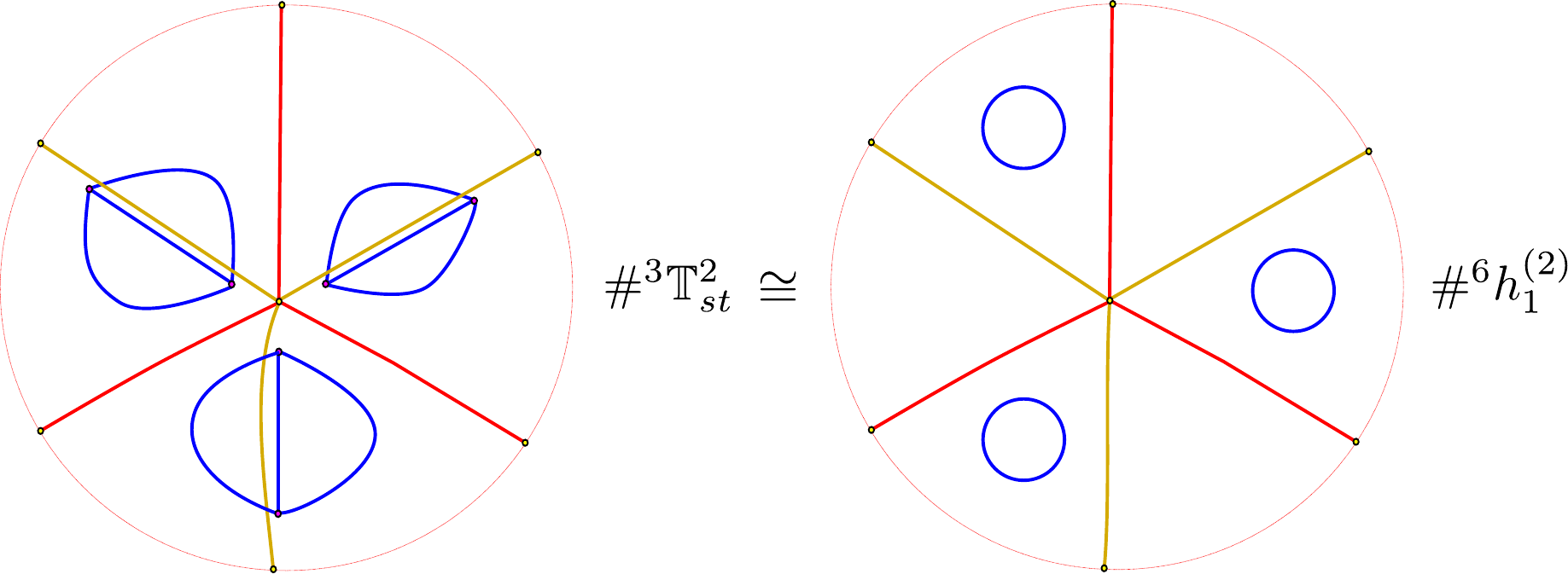}
		\caption{The 4-graph $G(\tau_4)$ in Figure \ref{fig:Example3Graph2Triang} after three index 1 anti-surgeries - accounted by the connected sums with $\mathbb{T}^2_{st}$ - and simplified with Move I (right). The 4-graph obtained by three additional index anti-surgeries (right).}
		\label{fig:Example3Graph2Triang2}
	\end{figure}
\end{center}
This simplification, from the original $4$-graph $G$ to $G_0$, allows for a direct description above of the flag moduli space $\SM(G)$, from which further information can be readily extracted. For instance, the $\F_q$-rational count for $\SM(G)(\F_q)$ is immediately:

$$|\SM(G)(\F_q)|=\frac{(q-1)^{5}}{(q^4-1)(q^4-q)(q^4-q^2)(q^4-q^3)}\cdot q^3\cdot\frac{(q^4-1)(q^3-1)(q^2-1)}{(q-1)^3}(q+1)q,$$

as $|\mbox{PGL}(4,\F_q)|=(q^4-1)(q^4-q)(q^4-q^2)(q^4-q^3)(q-1)^{-1}$, the rightmost multiplicative factor is the count for the two flags $\SF_1,\SF_2$,  and the $q^3$ factors stands for the final choice of $(p_3,p_4,p_5)$.\hfill$\Box$


We conclude this initial gallery of computations with the following:

\begin{ex}[Concentric Circles]\label{ssec:concentric} Let $\tau=(\tau_{i_1},\tau_{i_2},\ldots,\tau_{i_n})$ be an ordered collection of $n$ simple transpositions $\tau_{i_j}\in S_{N-1}$, $1\leq j\leq n$, $n\in\N$. Consider the $N$-graph $G(\tau)\sse \S^2$ described by $n$ concentric circles $C_i\sse\S^2$, $1\leq i\leq n$, with center on the North Pole, and strictly increasing radius. This $N$-graph is depicted in Figure \ref{fig:ConcentricCircles} (left).

\begin{center}
	\begin{figure}[h!]
		\centering
		\includegraphics[scale=0.55]{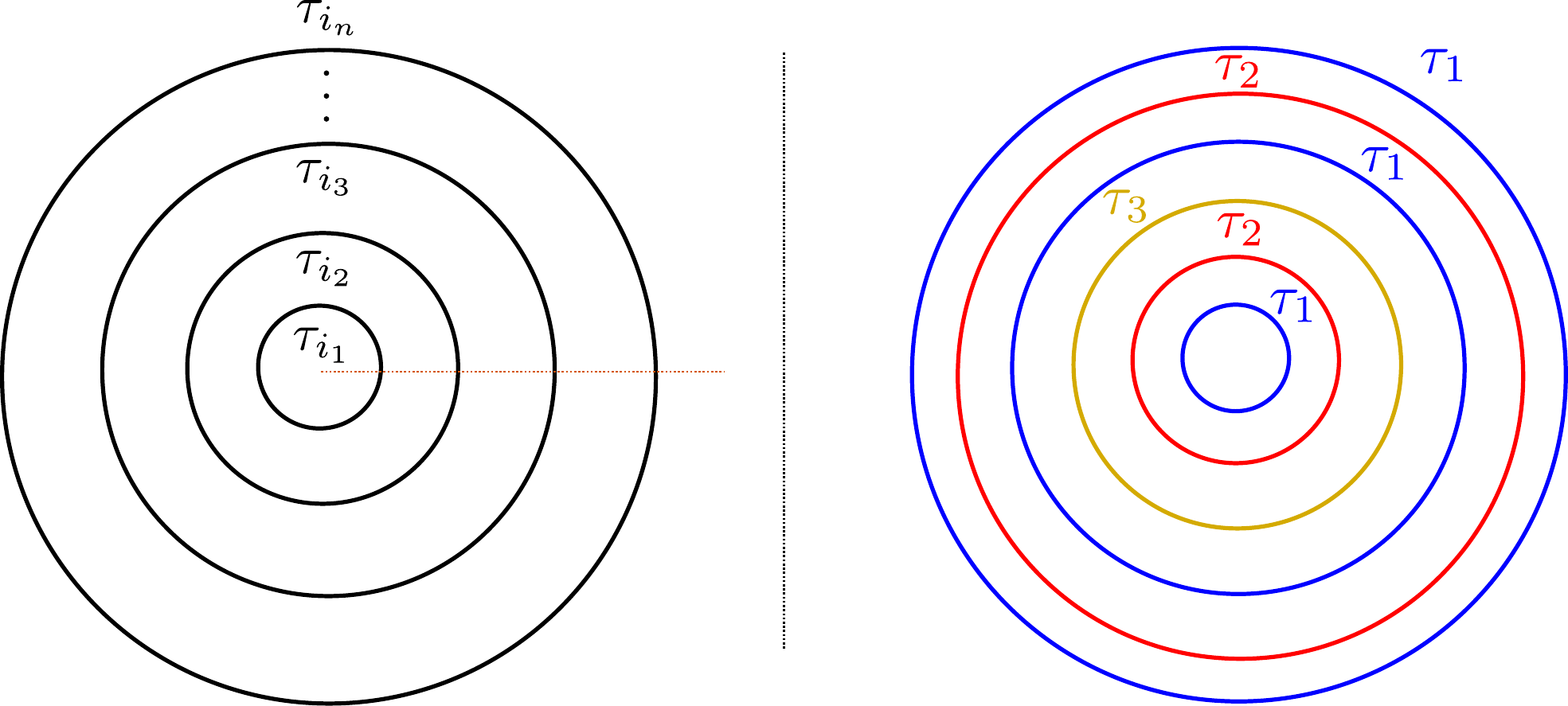}
		\caption{The $N$-graph $G(\tau)$ associated to the sequence of transpositions $\tau=(\tau_{i_1},\tau_{i_2},\ldots,\tau_{i_n})$ (left). The 4-graph $G(\tau)$ associated to $\tau=(\tau_1,\tau_2,\tau_3,\tau_1,\tau_2,\tau_1)$ (right).}
		\label{fig:ConcentricCircles}
	\end{figure}
\end{center}

The Legendrian weave $\La(G(\tau))\sse (J^1\S^2,\xi_\st)$ is a radial version of the $N$-stranded positive braid closure of $\beta=\sigma_{i_1}\sigma_{i_2}\cdot\ldots\cdot \sigma_{i_n}$. Smoothly, it is a link of $N$ two-spheres $\S^2$.  The moduli space of rank-one sheaves in $\R^2$ supported along the positive braid $\beta$ is the open Bott-Samelson variety $O(\beta)$ \cite{STZ_ConstrSheaves,OBS,CasalsHonghao}. 
By Section \ref{ssec:description_flagmoduli}, since $C = \S^2$ is simply connected, there is no further monodromy information and $\cM(G(\tau)) = O(\beta)$.
In particular, the links with different $n$ have a different number of points over $\bF_q$ and cannot be Legendrian iosotopic.
We note further that \cite[Theorem 6.34]{STZ_ConstrSheaves} relates this number to the HOMFLY-PT polynomial of the (topological) knot in $\bR^3$ defined by the braid $\beta.$\hfill$\Box$
\end{ex}

\subsection{Symmetry groups for Legendrian weaves} Let $\bG$ be an arbitrary finite group and $\La\sse(\S^5,\xi_\st)$ a Legendrian surface, with underlying smooth surface $S(\La)$. Let $\cL(\La)$ be the space of embedded Legendrian surfaces in $(\S^5,\xi_\st)$ Legendrian isotopic to the Legendrian surface $\La$, with base point $\La$. In addition, let $\L(\La)$ be the monoid of 3-dimensional exact Lagrangian concordances in the symplectization $(\S^5\times\R(t),e^t\la_\st)$, up to Hamiltonian isotopy, based on the Legendrian surface $\La\sse(\S^5,\xi_\st)$. Let $\varphi_t:S(\La)\lr(\S^5,\xi_\st)$ be a $\S^1$-family of Legendrian embeddings, $t\in\S^1$. Then the graph map
$$\mbox{gr}:\pi_1(\cL(\La))\lr\L(\La),\quad [\varphi_t]\longmapsto (\varphi_t(S(\La)),t),$$
allows us to relate loops of Legendrian surfaces with Lagrangian concordances.

These spaces $\cL(\La),\L(\La)$ are challenging to study. Already in the 1-dimensional case of Legendrian links $\La\sse(\S^3,\xi_\st)$, it was only established recently that there exist Legendrian links such that the fundamental groups $\pi_1(\cL(\La))$ can admit (infinite order) non-Abelian subgroups \cite[Corollary 1.6]{CasalsHonghao}, and $\L(\La)$ actually contains elements of infinite order \cite[Corollary 1.7]{CasalsHonghao}. To our knowledge, the only previous result about the fundamental group $\pi_1(\cL(\La))$ or the monoid $\L(\La)$ for $\La\sse(\R^5,\xi_\st)$ a Legendrian surface was proven in \cite{SabloffSullivan16}, where Legendrian surfaces $\La_{\Z_n}$, $n\in\N$, were built such that $\pi_1(\cL(\La_{\Z_n}))$ admits the finite cyclic group $\Z_n$ as a subgroup. Legendrian weaves and their flag moduli space are well-suited to address these questions. We present the following result for Legendrian surfaces in $(\S^5,\xi_\st)$:

\begin{thm}\label{thm:symmetries} Let $\bG$ be an arbitrary finite group. Then there exists a Legendrian surface $\La_{\bG}\sse(\S^5,\xi_\st)$ such that
\begin{itemize}
	\item[(i)] $\bG$ is a subquotient of the fundamental group $\pi_1(\cL(\La_\bG))$,
	\item[(ii)] $\bG$ is a subquotient of the 3-dimensional Lagrangian concordance monoid $\L(\La_\bG)$.
\end{itemize}
In fact, the latter is the image of the former via the graph map $\mbox{gr}:\pi_1(\cL(\La))\lr\L(\La)$.
\end{thm}

\begin{proof} The argument is structured in two parts. First, we describe a construction of a 2-graph $G'$ given a triangulation of a surface. Second, we use this construction to prove the statement. The second part has itself two steps: in the first step, the statement is proven only for those finite groups $\bG$ which are Hurwitz groups $\mathbb{H}$.\footnote{A Hurwitz  group $\mathbb{H}$ is  any  finite  group  which  can  be  generated  by  an  element   $x$   of  order   2  and  an  element  $y$   of  order   3  whose  product   $xy$    has  order  7.   Equivalently,  a  Hurwitz  group  is  any  finite  nontrivial  quotient  of  the  $(2,3,7)$-triangle  group $T(2,3,7):=\langle x,  y \|  x^2   =   y^3   =  (xy)^7  =  1\rangle$.} In the second step, the case of Hurwitz groups is used to conclude the statement for an arbitrary finite group.
	
First, we begin by describing a construction of $2$-graphs. Let $(C,T)$ be a closed smooth surface, $T$ a triangulation with $e(T)$ edges, and $G(T)$ the trivalent 2-graph dual to the triangulation $T$. Consider the 2-graph $G'$ obtained by adding a bigon at each edge of $G(T)$, using Move 4 in Figure \ref{fig:IntroSurgeries}. By Theorem \ref{thm:Legsurgeries}, specifically Remark \ref{rkm:surgeries}.$(ii)$, the Legendrian $\iota(\La(G'))$ is obtained by performing a connected sum of $\iota(\La(G(T)))$ with $e(T)$ copies of the standard Legendrian torus $\bT_\st^2\sse(\S^5,\xi_\st)$. Then \cite[Proposition 4.6]{Rizell_TwistedSurgery}, or Theorem \ref{thm:FlagLegSurgeries}, implies that the complex flag moduli space $\SM(G')$ is isomorphic to the product $\SM(G(T))\times(\C^*)^{e(T)}$, and thus $H^*(\SM(G'),\Q)\cong H^*(\SM(G(T),\Q)\otimes H^*((\C^*)^{e(T)},\Q)$ by the K\"unneth formula.

Second, we will now prove the statement in the case that $\bG$ is assumed to be an arbitrary but fixed Hurwitz group $\mathbb{H}$. By virtue of Hurwitz' theorem \cite{Hurwitz,Hurwitz2}, there exists a compact Riemann surface $C=C(\mathbb{H})$ whose automorphism group is (isomorphic to) $\mathbb{H}$; this surface $C$ is called a Hurwitz surface in the literature. The topological surface underlying the Riemann surface $C$ admits a triangulation $T(\mathbb{H})$ with symmetry group $\mathbb{H}$. In particular, the dual graph $G=G(T(\mathbb{H}))$ also has symmetry group $\mathbb{H}$. Let us now consider the 2-graph $G'$, associated to $G$ as in the paragraph above, where the edge bigons are added such that $\mathbb{H}$ is still a subgroup of the symmetry group of $G'$. Note that, by construction, $\mathbb{H}$ acts faithfully on the set of edges of the triangulation $T$, and thus $\mathbb{H}$ also acts faithfully on the $1\otimes H^*((\C^*)^{e(T)},\Q)\sse H^*(\SM(G'),\Q)$ piece of the cohomology of the flag moduli space $\SM(G')$.

Now, the generators $x,y$ of the triangle group $T(2,3,7)$ are geometrically given by rotations of the Poincar\'e hyperbolic disk, namely $x$ is a rotation of angle $\pi$ about the vertices of the $(2,3,7)$-Schwarz triangle and $y$ corresponds to a rotation of angle $2\pi/3$. Since a rotation $\rho$ is smoothly isotopic to the identity, as a diffeomorphism, there exists a contact isotopy $\phi_t(\rho)$, $t\in[0,1]$, of $(J^1 C,\xi_\st)$ such that $\phi_0(\rho)=id$ and $\phi_1(\rho)$ set-wise fixes the weave front associated to $G'$, and thus the Legendrian surface $\La(G')$ associated to it. This contact isotopy $\phi_t(\rho)$ defines an element of $\pi_1(\cL(\La(G')))$, and its graph an element of $\L(\La(G'))$. The flag moduli space $\SM(G')$ is a Legendrian isotopy invariant of the Legendrian surface $\La(G')\sse(\S^5,\xi_\st)$, and this contact isotopy induces an automorphism of $\SM(G')$. In particular, there are Legendrian isotopies $\phi_t(x)$ and $\phi_t(y)$ associated to the generators $x,y$ of any Hurwitz group $\mathbb{H}$, $x$ rotating $\pi$ and $y$ rotating $2\pi/3$. Thus the subgroup $\langle\phi_t(x),\phi_t(y)\rangle\sse \pi_1(\cL(\La(G')))$ acts by automorphisms in $\SM(G')$. Since $\mathbb{H}$ acts faithfully in the cohomology $H^*(\SM(G'))$, as pointed out above, $\mathbb{H}$ is a subquotient of $\pi_1(\cL(\La(G')))$, namely, it is a quotient of the subgroup $\langle\phi_t(x),\phi_t(y)\rangle$. The argument for $\L(\La(G'))$ is identical, and this concludes the required statement for Hurwitz groups $\mathbb{H}$.

Finally, to conclude the general statement, let $\bG$ be an arbitrary finite group and assume the result holds for Hurwitz groups, which is proven above. Then $\bG$ is a subgroup of the alternating group $A_n$ for large enough $n\in\N$. By \cite[Section 3]{AlternatingHurwitz}, see also \cite{Hurwitz2}, $A_n$ is a Hurwitz group $\bG(C)$ for $n\geq168$, and thus $\bG$ injects into such a Hurwitz group $\bG(C)$.\footnote{Note that $A_n$, for $n\leq 168$, is a subgroup of $A_m$, for a greater $m\geq n$, and thus all cases $A_n$ are covered.} The argument above thus implies that $\bG$ is a subquotient for $\pi_1(\cL(\La(G')))$ and $\L(\La(G'))$. Hence, the choice of weave $\La(G')$ completes the proof of Theorem \ref{thm:symmetries}.
\end{proof}

We do not know whether or not a result analogous to Theorem \ref{thm:symmetries} holds for 1-dimensional Legendrian knots $\La\sse(\S^3,\xi_\st)$. That could be a good question in low-dimensional contact topology. Any answer -- positive or negative -- would be of interest.


There is a complement to Theorem \ref{thm:symmetries} for certain groups $\bG$ of infinite order, including non-Abelian groups such as $\mbox{PSL}(2,\Z)$, by using results of the first author. Indeed, the Legendrian weave associated to the 4-graph $G(\tau)$ with the eighteen concentric circles
$$\tau=(\tau_1,\tau_2,\tau_1,\tau_2,\tau_1,\tau_2,\tau_1,\tau_2,\tau_1,\tau_2,\tau_1,\tau_2,\tau_1,\tau_2,\tau_1,\tau_2,\tau_1,\tau_2)$$
represents a 3-component Legendrian link $\La(G(\tau))$ of $2$-spheres. The geometric $\mbox{Br}_3$-braid action constructed in \cite{CasalsHonghao}, modulo its center $Z(\mbox{Br}_3)$, acts faithfully on the flag moduli space  $\SM(\La(G(\tau)))$. This flag moduli space is described in Example \ref{ssec:concentric}. Then \cite[Theorem 1.1]{CasalsHonghao} shows that the modular group $\mbox{PSL}(2,\Z)$ acts faithfully on the cluster charts for the space obtained by forgetting the monodromies in the Grothendieck resolution $\SM(\La(G(\tau)))$. Hence, $\mbox{PSL}(2,\Z)$ is a subquotient of $\pi_1(\L(\La(G(\tau))))$ and $\cL(\La(G(\tau)))$ for these Legendrian weaves $\La(G(\tau))$.


\subsection{Flag Moduli and Bipartite Graphs}\label{ssec:HexVertex} In Section \ref{sec:constr}, we introduced the construction of a 3-graph $G\sse C$ associated to an embedded eponymous bipartite graph $G$. This subsection explains how to compute flag moduli spaces for such 3-graphs.

We will employ a useful notation, local to this subsection. If $a,b\in V^3$ are distinct vectors in a 3-dimensional vector space $V$, we denote by $ab$ the unique 2-plane spanned by $a,b$. Similarly, given two 2-planes $A,B,\sse V^3$, the intersection $A\cap B$ will be denoted by $AB$.

At a hexagonal vertex, traveling between opposite faces requires crossing three edges of alternating colors, and thus opposite
faces are assigned completely transverse flags $\cA=(a,A)=aA$ and $\cB=(b,B)=bB$. Note that a single such pair $\cA,\cB$ determines the remaining four regions, by Lemma \ref{lem:hexagonal}: if crossing {\it red, blue, red} from $\cA$ to $\cB$, the flags in succession are $(a,A), (AB,A), (AB, B), (b,B)$. If crossing {\it blue, red, blue,} the flags are $(a,A), (a,ab), (b,ab), (b,B)$. This is depicted as follows:

$$
\begin{tikzpicture}[scale=1.5]
\tikzstyle{hex}=[rectangle,draw=black!100,fill=white!20,thick,minimum size=3.5mm]
\node[hex] (a) at (0,0) {};
\node (b) at (1,0) {};
\node (c) at (.5,.866) {};
\node (d) at (-.5,.866) {};
\node (e) at (-1,0) {};
\node (f) at (-.5,-.866) {};
\node (g) at (.5,-.866) {};
\draw[thick,red] (a) -- (b);
\draw[thick,blue] (a) -- (c);
\draw[thick,red] (a) -- (d);
\draw[thick,blue] (a) -- (e);
\draw[thick,red] (a) -- (f);
\draw[thick,blue] (a) -- (g);
\node at (0,.8) {$(a,A)$};
\node at (0,-.8) {$(b,B)$};
\node at (.866,.4) {$(a,ab)$};
\node at (-.866,-.4) {$(AB,B)$};
\node at (-.866,.4) {$(AB,A)$};
\node at (.866,-.4) {$(b,ab)$};
\end{tikzpicture}
\qquad
\qquad
\begin{tikzpicture}[scale=1.5]
\node (a) at (0,0) {$\bullet$};
\node (b) at (1,1) {$\bullet$};
\node (c) at (1.3,0) {};
\node (d) at (1,-.3) {};
\node (e) at (-.3,-.3) {};
\node (f) at (1.3,1.3) {};
\draw[thick] (-.3,0) -- (1.3,0);
\path[thick] (1,1.3) edge (1,-.3);
\draw[thick] (e) edge (f);
\node at (.15,-.2) {$a$};
\node at (1.15,.8) {$b$};
\node at (1.3,-.2) {$AB$};
\node at (.3,.55) {$ab$};
\node at (-.5,0) {$A$};
\node at (1,1.5) {$B$};
\end{tikzpicture}
$$

Now consider an edge of the bicubic graph $G$. In the associated 3-graph, this edge generates two hexagonal vertices which are connected by two adjacent edges of different colors. This local configuration is said to be a hexagonal edge. Let us denote the two flags on opposite regions along the axis connecting the hexagonal vertices by $\cA = aA$ and $\cC = cC$. Let $\cB$ be the flag in the interior region of the hexagonal edge, transverse to both $\cA$ and $\cC$. There are two further conditions on the flag $\cB$:
$$\quad AB \subset C, \qquad c \subset ab.$$

The Weyl group $W(A_2)\cong S_3$ is the symmetric group on three elements, and thus there are six possible relative positions for the two flags $\SA,\SC\in\GL_3/B$. Here we consider the case of a finite field $k=\F_q$. In a hexagonal edge, the relative position of the two outer flags $\cA,\cC$ is restricted:

\begin{lemma}\label{lem:ACflags}
The two outer flags $\cA,\cC$ in a hexagonal edge must coincide or be completely transverse. In addition, with $\cA,\cC$ fixed, number the of choices of flag $\cB$ in the interior of the hexagonal edge is $q^3$, in the case $\cA=\cC$, and $q-1$, in the case $\cA\neq\cC$.
\end{lemma}

\begin{proof} Let us analyze their possible relative positions, labeled according to the elements $W(A_2)=\{0,1,2,12,21,121\}$:
\begin{itemize}
	\item[-] Type 0:  $\cC = \cA.$
	Then the conditions are automatic, and $\cB$ is
	simply transverse to $\cA = \cC.$
	There are $q^3$ such choices.\\
	
	\item[-] Type 1:  $c = a, C \neq A.$  The second condition is then automatic, but $C \supset a=c$ and $C \supset AB$ means $C=A.$ This is a contradiction.\\
	
	\item[-] Type 2:  $c \neq a, C = A.$  The first condition is then automatic, but $c \subset C = A$ and $c\subset ab$ means $c = a$.
This is a contradiction.\\

	\item[-] Type 12:  $a\neq c, C \neq A$ but $a\subset C$.  Then $a\subset C$ and $AB \subset C$ means $C = A$  
This is a contradiction.\\

	\item[-] Type 21:  $a\neq c, C \neq A$ but $c\subset A.$  Then $c\subset A$ and $c \subset ab$ means $c = a.$ 
This is a contradiction.\\

	\item[-] Type 121: In this case, the flag $\cB$ is determined by \emph{either} equivalent choice:  a line $b$ in $ac$ not equal to $a$ or $c$ (then $B$ is the plane $bAC$) or a plane $B$ containing $AC$
	not equal to $A$ or $C$ (then $b$ is $acB$).
	The number of such choices is $q-1$.
\end{itemize}
Therefore, this flag $\cB$ has either $q^3$ or $q-1$ internal degrees of freedom, respectively, after fixing the outer flags $\cA$ and $\cC$ to be
either equal or completely transverse.  The other configurations have no solutions.
\end{proof}

We now apply Lemma \ref{lem:ACflags} and the discussion above to prove Theorem \ref{thm:intro1} in the introduction.

\subsection{Non-isotopic Links of Legendrian Spheres} Let $n\in\N$ and consider the bipartite Ladder Graph $L_n\sse\S^2$ depicted in Figure \ref{fig:NLadder} (bottom). The number $n\in\N$ denotes half the number of square faces, and the right and left sides of the bipartite graph are identified in $\S^2$. In particular, $\S^2\setminus L_n$ has $2n+2$ connected components, $2n$ squares and two 2-disks, at the north and south poles of $\S^2$. We consider its associated 3-graph $\SL_n\sse\S^2$, as described in Section \ref{sec:constr}, which is shown in Figure \ref{fig:NLadder} (bottom). The Legendrian weave $\La(\SL_n)\sse (J^1(\S^2),\xi_\st)$ consists of a 3-component link of Legendrian 2-spheres, independent of $n\in\N$.

\begin{center}
	\begin{figure}[h!]
		\centering
		\includegraphics[scale=0.65]{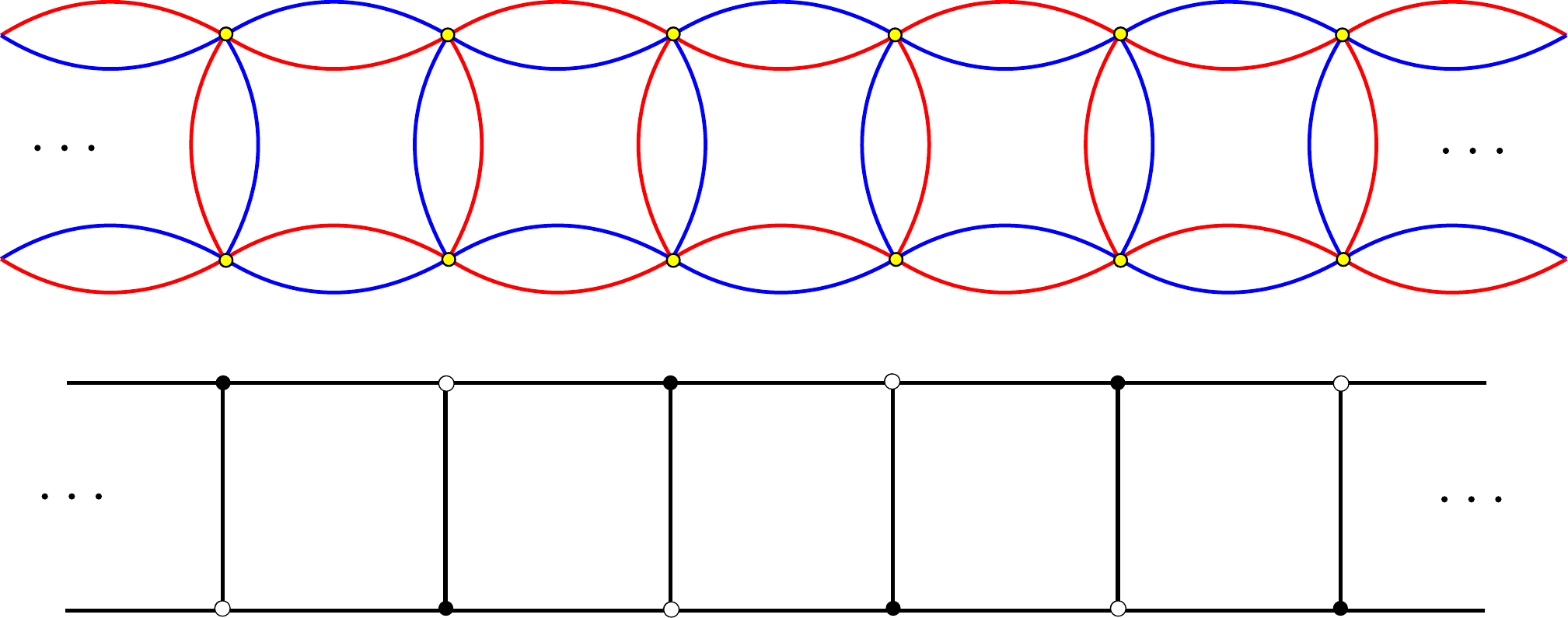}
		\caption{The bipartite Ladder Graph $L_n$, where the right and left sides are identified after $n$ rungs (bottom). The 3-graph $\SL_n$ associated to $L_n$ (top).}
		\label{fig:NLadder}
	\end{figure}
\end{center}

Note that the Legendrian link $\La(\SL_n)\sse(J^1(\S^2),\xi_\st)$ is smoothly isotopic to the surface unlink, as the codimension of this smooth embedding is three. We now show that the Legendrian isotopy type of the Legendrian link $\La(\SL_n)\sse(J^1(\S^2),\xi_\st)$ is different for each $n\in\N$. This will be achieved by counting the number of points of their flag moduli spaces $\SM(\SL_n)$ over a finite field. The precise statement reads:

\begin{thm}[Theorem \ref{thm:intro1}]\label{thm:finitecounts} Let $\SL_n\sse\S^2$ be the $(2n)$-runged ladder graph and $\F_q$ a finite field, $q$ a prime power. Then the flag moduli space $\SM(\SL_n)$ has orbifold point count
	
$$|\SM(\SL_n)(\F_q)|=\frac{q^{2n-3}-q^{n-2}+q^{n-1}+q-1}{(q-1)^2}.$$

Hence, the Legendrian surface links $\La(\SL_n)$ and $\La(\SL_m)$ are Legendrian isotopic iff $n=m$.
\end{thm}

\begin{proof}
Let us consider the two flags $\mathcal{A},\SC\in\GL(3,\C)/B$ located in the strata corresponding to the neighborhoods of the north and south poles. We have shown these flags in Figure \ref{fig:NLadderProof}. The flags in the vertical regions will be denoted $\SB_i$, $0\leq i\leq 2n-1$, with the cyclic condition $\SB_0=\SB_{2n}$.

\begin{center}
	\begin{figure}[h!]
		\centering
		\includegraphics[scale=0.75]{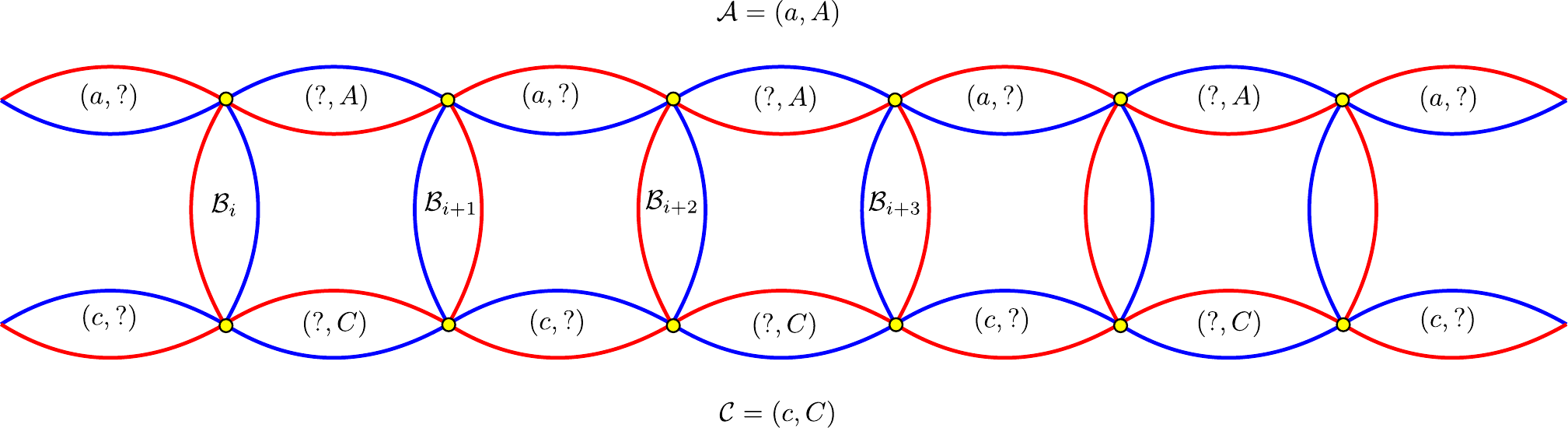}
		\caption{The flag configuration at a point of the flag moduli space $\SM(\SL_n)$ where $\SA=(a,A)$, $\SC=(c,C)$ are the inner and outer flags. Observe that the choice of $\SA,\SC$ partially fills the flags in the horizontal eye-shaped regions.}
		\label{fig:NLadderProof}
	\end{figure}
\end{center}

By Lemma \ref{lem:ACflags}, the existence of the flags $\mathcal{B}_i$ in the vertical hexagonal edges, $0\leq i\leq 2n-1$, as in Figure \ref{fig:NLadderProof}, implies that the relative position of $\SA,\SC$ must either be trivial, i.e.~$\SA=\SC$, or completely transverse, i.e. the projective lines $A\neq C$ are distinct, and $a\not\in C$ and $c\not\in A$. The $\F_q$-count is divided into these two cases.

First, let us consider the case where $\SA$ and $\SC$ are completely transverse, i.e. they belong to the Bruhat $\GL(3,\C)$-orbit labeled by $w=(12)(23)(12)\in W(A_2)$. We claim that after choosing the flag $\SB_0=(b,B)$, the remaining flags $\SB_i$, $1\leq i\leq 2n-1$ are
uniquely determined. The resulting flag configuration is shown in Figure \ref{fig:NLadderProof2}.

\begin{center}
	\begin{figure}[h!]
		\centering
		\includegraphics[scale=0.75]{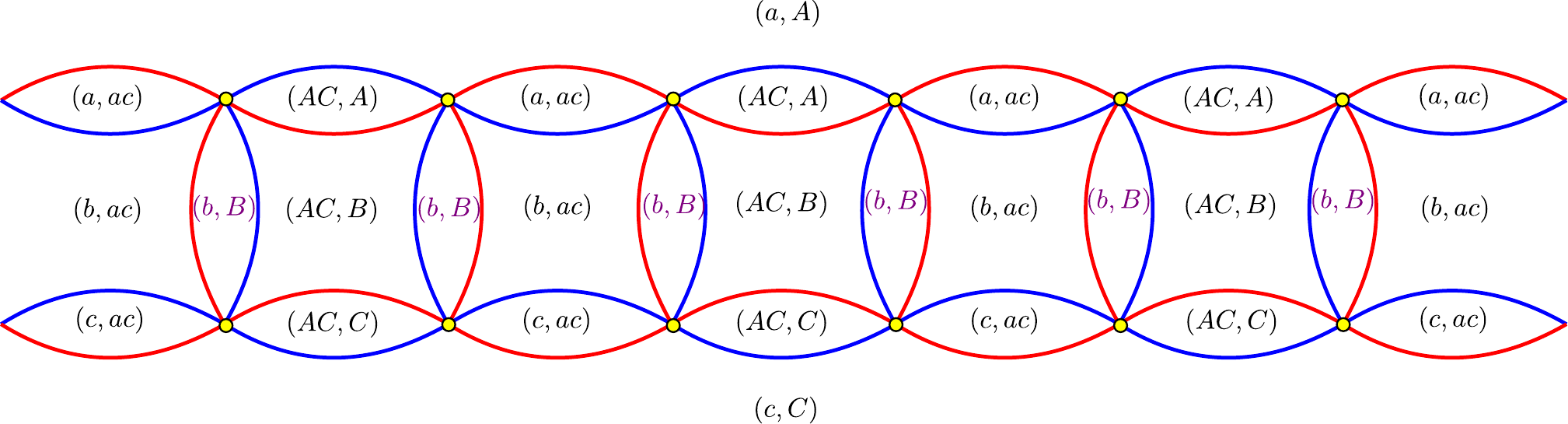}
		\caption{Flag configuration at a point of the flag moduli space $\SM(\SL_n)$ in the case $\SA=(a,A)$ is completely transverse to $\SC=(c,C)$. For these configurations, the choice of flag $(b,B)$ uniquely determines the point in the flag moduli.}
		\label{fig:NLadderProof2}
	\end{figure}
\end{center}

Let us prove this. Since $\SA$ and $\SC$ are completely transverse, they determine the flags $(AC,A)$,$(a,ac)$ in the horizontal eye-shaped spaces in the upper row, and the flags $(AC,C)$,$(c,ac)$ in the corresponding horizontal spaces along the bottom. The additional choice of $\SB_0=(b,B)$ determines the flags $(AC,B),(b,ac)$ in the left and right regions adjacent to that of $\SB_0$. Note that $B\neq ac$ and $b\in B\cap ac$. Similarly, $b\neq AC$ and the two points $AC,b\in\P^2$ span the line $B$. The flag $\SB_1$ must have $b\in\P^2$ as its point, and its line must contain $AC,b\in\P^2$. Hence the flag $\SB_1=(b,B)$ is uniquely determined, and coincides with $\SB_0$. By an analogous reasoning, $\SB_1$ determines the flag $(AC,B)$ on the adjacent region at its right, and hence the line in $\SB_2$ must be $B$. Since the point in $\SB_2$ must be the intersection $B\cap ac$, we conclude $\SB_2=(b,B)$ and thus $\SB_2=\SB_1=\SB_0$. Iteratively applying these two steps, we show that $\SB_i=\SB_0$ for all $1\leq i\leq 2n-1$. The cyclic condition $\SB_0=\SB_{2n}$ is automatically verified in this case. In conclusion, in this completely transverse case, the choices are the three flags $\SA,\SB_1,\SC$, being pairwise completely transverse. This configuration is depicted in Figure \ref{fig:NLadderLines} (left).
	
	\begin{center}
		\begin{figure}[h!]
			\centering
			\includegraphics[scale=0.75]{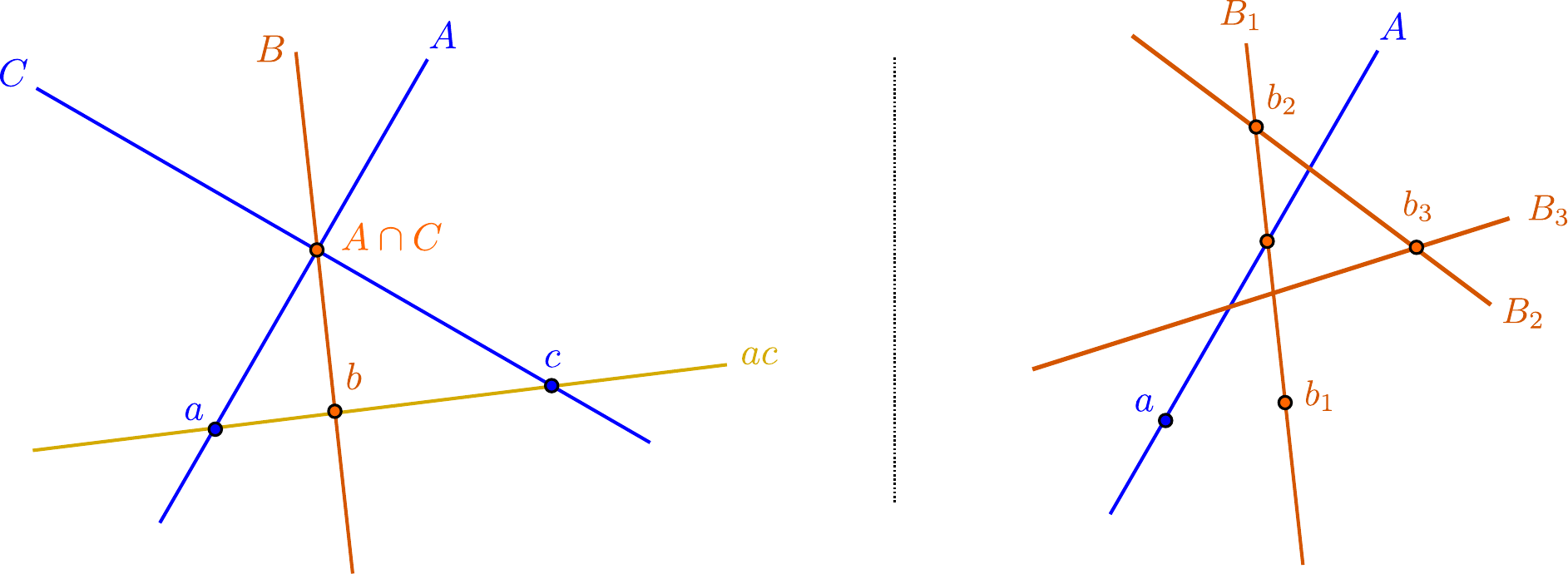}
			\caption{The projective flags $\SA,\SC$ and $\SB=(b,B)$ in the case $\SA,\SC$ are completely transverse (left). The configuration of projective flags in the case $\SA=\SC$, where admissible flags $(b_1,B_1),(b_2,B_1),(b_2,B_2),(b_3,B_2),(b_3,B_3)$ are depicted (right).}
			\label{fig:NLadderLines}
		\end{figure}
	\end{center}

The counts over a finite field are
$$|\mbox{PGL}(3,\F_q)|=\frac{(q^3-1)(q^3-q)(q^3-q^2)}{q-1},\quad |\P^2(\F_q)|=|\P^2(\F_q)^*|=\frac{(q^3-1)}{(q-1)}=q^2+q+1,$$
and a projective line $\P^1(\F_q)$ has $q+1$ points. Also, note that there are $|\P^1(\F_q)|=q+1$ choices of lines through a point. Now, the choice of the flag $\SA=(a,A)$ gives a count of $|\P^2(\F_q)|\cdot|\P^1(\F_q)|$. The choice of the completely transverse flag $\SC=(c,C)$ gives $q^3$, as we must have $a\not\in C$, and $c\in C$ but $c\neq A\cap C$. The line $B$ in the third transverse flag $\SB_0=(b,B)$ must contain the point $A\cap C$, and its point $b=B\cap ac$ is uniquely determined by the choice of such $B$. Since $B$ must be distinct from $A$ and $C$, we get $q-1$ choices for the line $B$. This yields a total count of

$$\frac{((1+q+q^2)(1+q))\cdot(q^3)\cdot(q-1)}{(1+q+q^2)(q^3-q)(q^3-q^2)}=\frac{1}{q-1},$$
for the case where the flags $\SA,\SC$ are completely transverse. Thus, $\SA,\SB_0,\SC$ can be fixed, mutually completely transverse, and a factor of $(q-1)^{-1}$ remains.

Second, let us consider the case where $\SA=\SC$. In this case, the flags $\SB_i$, $1\leq i\leq n$, will not all be equal. We proceed with the same systematic analysis as before. The initial choice is $\SB_1=(b_1,B_1)$, and this determines the flags $(b_1,ab_1),(AB_1,B_1)$ in the left and right adjacent regions of $\SB_1$. In turn, this determines the line in $\SB_2$ to be $B_1\sse\P^2$. The point in $\SB_2$ remains undetermined at this stage, and this is a choice of $b_2\in \SB_2$, with a count of $q$, since $b_2\in B_1$ and $b_2\neq A\cap B_1$. This is depicted in Figure \ref{fig:NLadderProof2}.  The choice of the point $b_2\in\SB_2$ readily determines the point in the flag $\SB_3$, whose line is undetermined. There are exactly $q$ choices for a line $B_2\sse\P^2$ in $\SB_3$, as it must contain $b_2$ and be different from $B_1$. This is an iterative process, where the count of choices that determine the flag $\SB_i$, $2\leq i\leq n$ is exactly $q$, either because of the choice of a point or a line. The flag configuration is depicted in Figure \ref{fig:NLadderProof3}.

\begin{center}
	\begin{figure}[h!]
		\centering
		\includegraphics[scale=0.75]{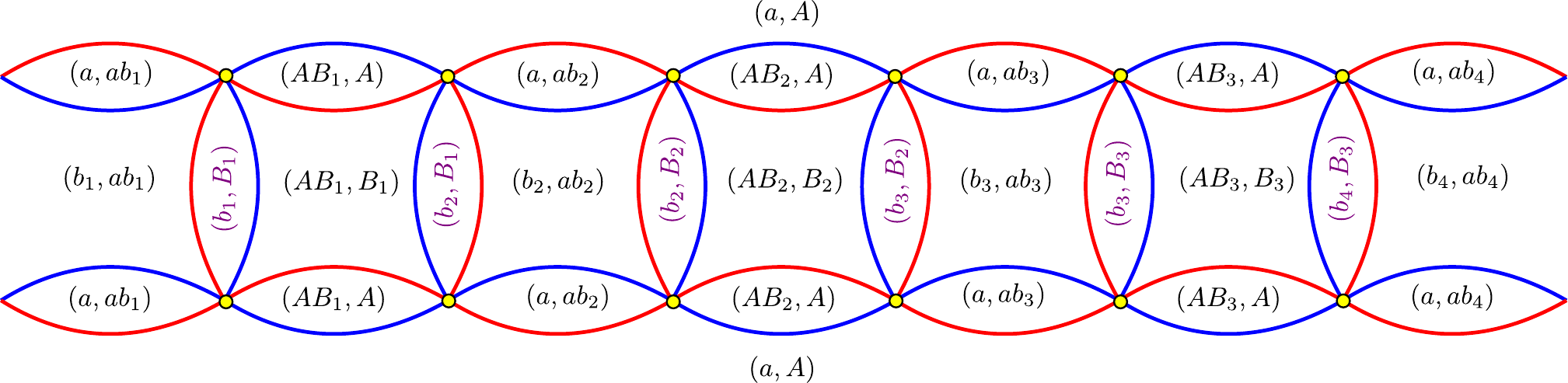}
		\caption{Flag configuration at a point of the flag moduli space $\SM(\SL_n)$ in the case $\SA=\SC$. For these configurations, the sequence of flags $(b_i,B_i)$ are part of the choice that determine the points in the flag moduli.}
		\label{fig:NLadderProof3}
	\end{figure}
\end{center}

At this stage of the case $\SA=\SC$, we need to impose the cyclic condition $\SB_0=\SB_{2n}$ given by the ladder graph. This is not automatic, and it will actually reduce the naive count of $q^{2n}$ for the choices of $\SB_i$, $0\leq i\leq 2n-1$. Let us use the $\mbox{PGL}(3,\F_q)$ symmetry to fix the flags $\SA=\SC$ and $\SB_0$.  We will now use affine coordinates, so the flag $\SA$ will be understood as a line $a\sse\F_q^3$ and a plane $A\sse\F_q^3$. Thus, we assume that the line $a\sse\F_q^3$ in $\cA$ is spanned by $\begin{pmatrix}1\\0\\0\\\end{pmatrix}$ and the
plane $A\sse\C^3$ is the kernel of the covector $(0,0,1)$, and the flag $\cB_0$ is given by the pair $\begin{pmatrix}0\\0\\1\end{pmatrix}$, $(1,0,0)$. Note that this flag configuration has a residual isotropy group isomorphic to $(\bF_q^\times)^2$, and we will divide our count for fixed $\cA,\cB_0$ by the isotropy factor of $(q-1)^2$.

Let us parametrize the remaining degrees of freedom for flags $\cB_i$, $1\leq i\leq\SB_{2n-1}$ by the choice of coordinates $x_i\in\F_q$ and $a_i\in\F_q$, respectively used for each line $b_i$ and plane $B_i$, $1\leq i\leq n$. By labeling lines and planes by their normalized vectors and covectors, we obtain the description:
$$\begin{array}{rclrcl}
\cB_0:\quad b_0 &=& \begin{pmatrix}0\\0\\1\end{pmatrix}&\quad B_0 &=& \left(1, 0, 0\right)\\
\cB_1:\quad b_1 &=& \begin{pmatrix}0\\x_1\\1\end{pmatrix}&\quad B_0 &=& \left(1, 0, 0\right)\\
\cB_2:\quad b_1 &=& \begin{pmatrix}0\\x_1\\1\end{pmatrix}&\quad B_1 &=& \left(1, a_1, -a_1 x_1\right)\\
\cB_3:\quad b_2 &=& \begin{pmatrix}-a_1 x_2\\x_1 + x_2\\1\end{pmatrix}&\quad B_1 &=& \left(1, a_1 , -a_1 x_1\right)\\
\cB_4:\quad b_2 &=& \begin{pmatrix}-a_1 x_2\\x_1 + x_2\\1\end{pmatrix}&\quad B_2 &=& \left(1, a_1 + a_2, -a_1x_1 - a_2(x_1+x_2) \right)\\
\cB_5:\quad b_3 &=& \begin{pmatrix}-a_1 x_2-(a_1+a_2)x_3\\x_1 + x_2+x_3\\1\end{pmatrix}&\quad B_2 &=& \left(1, a_1 + a_2, -a_1x_1 - a_2(x_1+x_2)\right)\\
\cB_6:\quad b_3 &=& \begin{pmatrix}-a_1 x_2-(a_1+a_2)x_3\\x_1 + x_2+x_3\\1\end{pmatrix}&\quad B_3 &=&
\left(1, a_1 + a_2 + a_3, -a_1x_1 \right.\\
&&&&&\left. - a_2(x_1+x_2)-a_3(x_1+x_2+x_3) \right)\\
&\vdots&&&\vdots&\\
\end{array}
$$

$$\cB_{2k}:\; b_{k} = \begin{pmatrix}-\sum_{i=2}^k \left(\sum_{j=1}^{i-1}a_j\right) x_i \\ \sum_{j=1}^k x_j \\1\end{pmatrix}, \quad B_{k} =
\left(1, \sum_{j=1}^k a_j, -\sum_{i=1}^k a_i\left(\sum_{j=1}^i x_j\right)\right)$$

Since the dot product $B_k\cdot b_k=0$ for all $1\leq k\leq n$, the 2-planes $B_k$ contain the points $b_k$, as required. Define the new variables
$$\alpha_i = \sum_{j=1}^i a_j,\quad y_i = x_{i+1},\quad X = \sum_{j=1}^n x_j,$$
and the vectors $\alpha=(\alpha_1,\alpha_2,\ldots,\alpha_{n-1})$, $y=(y_1,\ldots,y_{n-1})$. This is an allowed change of variables, as it is a triangular and invertible transformation. The equation $\SB_0=\SB_{2n}$ gives four equalities. Two of the equalities are $\alpha_n=0$, $X=0$. The third equation reads
$$\alpha\cdot y=0,\quad\mbox{i.e. }\sum_{i=1}^{n-1} \alpha_i y_i=0.$$
The fourth equation, imposed by the vanishing of the third coordinate of $B_{2n}$ is dependent on the first three equations, as $b_{2n}\in B_{2n}$. We are now in position to count solutions of this system over $\F_q$:

\begin{itemize}
	\item[(i)] Suppose that the vector $\alpha\in(\F_q)^{n-1}$ is non-vanishing.  There are $(q^{n-1}-1)$ such possibilities for $\alpha.$ Then the equation $\alpha\cdot y=0$ imposes exactly one linear relation among the $y_i$ variables, $1\leq i\leq n-1$. This yields a choice of $q^{n-2}$ possibilities for the vector $y$. The contribution in this case is thus $(q^{n-1}-1)q^{n-2}$.\\

	\item[(ii)] Suppose that instead $\alpha=0$ is the zero vector.  Then the equation $\alpha\cdot y=0$ is vacuous. The choice of an arbitrary vector $y\in\F_q^{n-1}$ completes the count with a factor of $q^{n-1}$.
\end{itemize}

In conclusion, the case $\SA=\SC$ yields a total count of

$$\frac{(q^{n-1}-1)q^{n-2}+q^{n-1}}{(q-1)^2}.$$

Finally, adding together the two cases for the relative position of the two flags $\SA,\SC$, we obtain a finite field count of

$$|\SM(\SL_n)(\F_q)|=\frac{1}{(q-1)}+\frac{(q^{n-1}-1)q^{n-2}+q^{n-1}}{(q-1)^2}=\frac{q^{2n-3}-q^{n-2}+q^{n-1}+q-1}{(q-1)^2}.$$

\end{proof}

Note also that the proof of Theorem \ref{thm:finitecounts} shows that the moduli space of $n$-gons $\SM_n$ \cite{NGons1,NGons2} admits an embedding into our flag moduli space $\SM(\SL_n)(\C)$. In the next section, we will consider $N$-graphs $G\sse\D^2$ with {\it non-empty} boundary $\dd G\neq\varnothing$, which feature prominently in our study of Lagrangian fillings through $N$-graphs $G$.


\section{Microlocal Monodromies and Lagrangian Fillings}\label{sec:app2}

This section explains how to use $N$-graphs $G$ in order to study 2-dimensional exact Lagrangian cobordisms between 1-dimensional Legendrian links in $(\S^3,\xi_\st)$ -- in particular, the study of their exact Lagrangian fillings. Briefly, the Legendrian mutations we developed in Section \ref{sec:moves} will be used to construct Lagrangian fillings, and we use microlocal monodromies -- and the connection to cluster algebras -- to distinguish them. The proof of Theorem \ref{thm:ThurstonLinksIntro}, using these two steps to build infinitely many distinct Lagrangian fillings for a class of Legendrian knots, is also given here.


\subsection{Exact Lagrangian Cobordisms}\label{ssec:ExactLagrangianFillings} This manuscript has heretofore focused on the study of Legendrian surfaces in an ambient 5-dimensional contact manifold. In fact, the theory of $N$-graphs and Legendrian weaves that we have developed is also useful for studying exact Lagrangian fillings of 1-dimensional Legendrian links $\La\sse(\S^3,\xi_\st)$ and, more generally, exact Lagrangian cobordisms between such Legendrian links. This is also the context in which applications to both Spectral Networks and Soergel Calculus should arise.

There are two advantages to studying exact Lagrangian fillings $L$ of $\dd L\sse(\S^3,\xi_\st)$ from the perspective of $N$-graphs. First, the manipulation of their Hamiltonian isotopy class $L\sse(\D^4,\omega_\st)$ becomes combinatorial, as do operations such as Polterovich surgery (see Theorem \ref{thm:Legsurgeries}). Second, the computation of cluster coordinates for the augmentation variety $\Aug(\La)$ associated to the Legendrian link $\dd L=\La\sse(\S^3,\xi_\st)$ is accessible.

\begin{remark}
The cluster structures in the coordinate rings of $\Aug(\La)$ have proven to be an effective method for proving new results for Legendrian knots in the 3-sphere \cite{STWZ,CasalsHonghao}. We do not know how to prove these cited results using Floer-theoretic methods (such as the Legendrian DGA \cite{Chekanov02,Etnyre05}), nor is there currently a Floer-theoretic description\footnote{As far as we know, this remains an open question even if the exact Lagrangian filling is given by a pinching sequence \cite{EkholmHondaKalman16,YuPan,YuPan2}.} for the cluster coordinates induced by an exact Lagrangian filling $L\sse(\D^4,\omega_\st)$.\hfill$\Box$
\end{remark}

In this section we present the context in which Legendrian weaves $\La(G)$ provide exact Lagrangian cobordisms. This is a viewpoint that we will use extensively in the reminder of the article, including Section \ref{sec:app3} and Appendix \ref{ssec:SoergelCalculus}.

\subsubsection{The geometric setup} Let $(\R^5,\xi_\st)$ have coordinates $(x,y,z,s,t)\in\R^5$, contact 1-form $\alpha_\st=e^s(dz-y_1dx_1)-dt$, and let $\pi:(\R^5,\xi_\st)\lr(\R^4,\la_\st)$ be the projection $\pi(x,y,z,s,t)=(x,y,z,s)$. Consider the contact 3-planes $(\R^3_l,\xi_\st):=\{t=1,s=l\}\sse\R^5$ and choose two Legendrians $\La_1\sse(\R^3_1,\xi_\st)$ and $\La_2=(\R^3_2,\xi_\st)$. Suppose that $\La\sse(\R^5,\xi_\st)$ is a Legendrian surface with isotropic boundaries $\dd\La=\La_1\sqcup\La_2$, and $\La_1=\La\cap(\R^3_1,\xi_\st)$, $\La_2=\La\cap(\R^3_2,\xi_\st)$.

The crucial geometric fact is that the projection $\pi(\La)\sse(\R^4,\la_\st)$ is an immersed exact Lagrangian, whose immersion points are in bijection with the Reeb chords of $\La\sse(\R^5,\alpha_\st)$. In particular, if the Legendrian surface $\La\sse(\R^5,\xi_\st)$ has {\it no} Reeb chords, then the Lagrangian image $\pi(\La)\sse(\R^4,\la_\st)$ is an embedded exact Lagrangian with boundary $\La_1\sqcup\La_2$. It is readily verified that $\pi(\La)$ is an exact Lagrangian cobordism from $\La_1$ to $\La_2$ (and {\it not} viceversa). The particular case of $\La_1=\varnothing$ yields exact Lagrangian fillings of $\La_2$.

In line with the constructions in this article, the Legendrians $\La_1,\La_2\sse(\R^3,\xi_\st)$ that we study arise from positive braids -- see \cite[Section 2]{CasalsHonghao} -- and thus can be described as satellites of the standard Legendrian unknot $\La_{\st}\sse(\R^3,\xi_\st)$. The description in the paragraph above is then modified as follows. Consider $(J^1(\S^1\times[1,2]),\xi_\st)$, two Legendrian links
$$\La_1\sse(J^1(\S^1\times\{1\})),\quad \La_2\sse(J^1(\S^1\times\{2\})),$$
and a Legendrian surface $\La\sse(J^1(\S^1\times[1,2]),\xi_\st)$ such that
$$\La\cap(J^1(\S^1\times\{1\}))=\La_1,\quad \La\cap(J^1(\S^1\times\{2\}))=\La_2.$$

Now, suppose that the surface $\La$ has {\it no} Reeb chords, then the Lagrangian projection $\pi(\La)\sse (J^1\S^1\times\R,\la_\st)$ in the symplectization of $(J^1\S^1,\xi_\st)$ provides an exact Lagrangian cobordisms from $\La_1$ to $\La_2$. The case in which $\La_1=\varnothing$ can be compactified to $(J^1\D^2,\xi_\st)$ in the $(J^1(\S^1\times\{1\}),\xi_\st)$ end, which symplectically corresponds to adding a standard symplectic 4-disk $(\D^4,\xi_\st)$ in the concave end of the symplectization, i.e. as an exact symplectic filling of $(\S^3,\xi_\st)$. Diagrammatically, this implies that we can describe exact Lagrangian fillings of a positive Legendrian braid $\La_2=\La(\beta)\sse(\S^3,\xi_\st)$ in $(D^4,\omega_\st)$ by drawing $N$-graphs in $\D^2$ whose free edges meet the boundary according to a positive braid word $\beta$. Here $\La(\beta)$ denotes the standard satellite of the Legendrian in $(J^1\S^1,\xi_\st)$ whose front in $\S^1\times\R$ is given by the positive braid (word) $\beta$.

In short, exact Lagrangian fillings between Legendrian links can be studied via the spatial wavefronts of their Legendrian lifts to the contactization, and the techniques we have developed for Legendrian surfaces can be applied. In particular, we can use our diagrammatic $N$-graph calculus to study and distinguish exact Lagrangian cobordisms.

\subsubsection{Free $N$-Graphs}\label{sssec:FreeNGraphs} Let $\mathscr{G}_\beta$ be the set of $N$-graphs on a 2-disk $\D^2$ with boundary braid word $\beta$. As stated above, in order to construct {\it embedded} exact Lagrangian fillings $L\sse(\D^4,\omega_\st)$ for $\La(\beta)\sse(\S^3,\xi_\st)$ as $N$-graphs $G\sse\D^2$ in $\mathscr{G}_\beta$, we must have that the Legendrian weave $\La(G)\sse(\R^5,\xi_\st)$ has no Reeb chords. Let us introduce the following:

\begin{definition}\label{def:FreeNGraph}
An $N$-graph $G\sse\D^2$ is said to be free if its associated Legendrian front $\Sigma(G)$ can be woven with no Reeb chords.\hfill$\Box$
\end{definition}

In this section many of the $N$-graphs $G\sse\D^2$ can be checked to be free by direct inspection.

\begin{ex}\label{ex:free2graphs} Let $G\sse\D^2$ be a $2$-graph such that $(\D^2\setminus G)/(\dd\D^2\cap (\D^2\setminus G))$ is simply-connected. Then $G$ is free if and only if $G$ has no faces contained in the interior of $\D^2$. Figure \ref{fig:ReebChordsExamples} shows four examples of 2-graphs.

\begin{center}
	\begin{figure}[h!]
		\centering
		\includegraphics[scale=1]{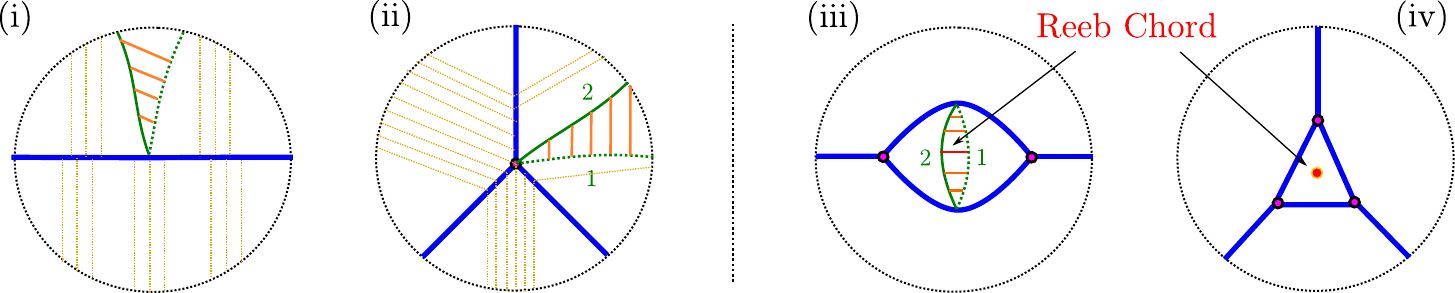}
		\caption{Two free 2-graphs $(i)$ and $(ii)$, shown on the Left. Two 2-graphs, $(iii)$ and $(iv)$, whose woven front must have a Reeb chord (Right). Each of the fronts associated to the non-free two 2-graph can be woven with exactly one Reeb chord, as indicated. In both cases, the green lines depict the two sheets of a woven front and the orange segments indicate the distance between these sheets. On the left, these length of the distance grows as we approach the boundary, whereas for the 2-graph $(iii)$ there must be a maximum for this distance, forcing a Reeb chord.}
		\label{fig:ReebChordsExamples}
	\end{figure}
\end{center}
\end{ex}

The two 2-graphs $(i),(ii)$ on Figure \ref{fig:ReebChordsExamples} (Left) are free. For that, consider a smooth 1-dimensional foliation of $\D^2\setminus G$ whose leaves are open intervals and such that the closure of each leave intersects $\dd\D^2$. The radial-like yellow foliations depicted in Figure \ref{fig:ReebChordsExamples} (Left) suffice. Then choose a woven front for such 2-graphs such that the differences between the heights of the two sheets of the front strictly increase along each of the leaves of this foliation, being $0$ at $G$ and having positive value at $\dd\D^2$. These woven fronts do not have Reeb chords, as the functions giving the differences of heights between the sheets do not have critical points. In contrast, such foliations do not exist for the two 2-graphs $(iii),(iv)$ on Figure \ref{fig:ReebChordsExamples} (Right), as $\D^2\setminus G$ contains a region whose closure is contained in the interior of $\D^2$. It can be shown that any front woven with respect to $(iii)$ or $(iv)$ must have a Reeb chord and there exists a woven front with a minimal number of Reeb chords, one per each interior face of $G$.\hfill$\Box$

From the perspective of Lagrangian fillings, the 2-graph $(i)$ in Figure \ref{fig:ReebChordsExamples} is an embedded (exact) Lagrangian filling for the 2-component standard unlink, which is the union of two disjoint Lagrangian disks $\D^2\cup\D^2$. The 2-graph $(ii)$ yields the embedded Lagrangian filling for the standard unknot, which is the standard flat Lagrangian disk $\D^2\sse\D^4$. This stands in contrast with the {\it immersed} Lagrangian fillings represented by $(iii)$ and $(iv)$. The 2-graph $(iii)$ is an immersed exact Lagrangian annulus with boundary the 2-component standard unlink, and $(iv)$ is an immersed exact Lagrangian once-punctured 2-torus filling the standard Legendrian unknot. In general, the following criterion is useful:

\begin{lemma}\label{lem:free} Let $G\sse\D^2$ be a free $N$-graph. Then the $N$-graph $\mu(G)\sse\D^2$, obtained from $G$ by performing a Legendrian mutation at any $\sf I$-cycle or $\sf Y$-tree of $G$, is also free.
\end{lemma}

\begin{proof} Consider the 2-graph mutation at a monochromatic $i$-edge of an $N$-graph $G$. Let $\Op(e)$ be a neighborhood of a monochromatic edge $e$ in a free $N$-graph. The 2-graph mutation along the 1-cycle $\gamma_e$ can then be performed by the exchange in Figure \ref{fig:ReebChordsExamples2}, which builds on Figure \ref{fig:LegendrianMutation} (Left). Since both 2-graphs $G$ and $\mu_e(G)$ in the exchange coincide in a neighborhood of the boundary, we can force that the front woven with respect to $\mu_e(G)$ coincides identically -- not just up to homotopy of Legendrian fronts -- with the given front $\Sigma(G)$ woven with respect to $G$. Let us choose a 1-dimensional foliation in $\D^2$ with respect to $G$, as in Example \ref{ex:free2graphs}, such that the difference between the heights of any pair of sheets in the woven front strictly increase (or decreases) as we move along the sheets of the foliations away from $G$. (This foliation exists because $G$ is free.) We have depicted such a foliation for $G$ in Figure \ref{fig:ReebChordsExamples2}.
	
\begin{center}
	\begin{figure}[h!]
		\centering
		\includegraphics[scale=0.65]{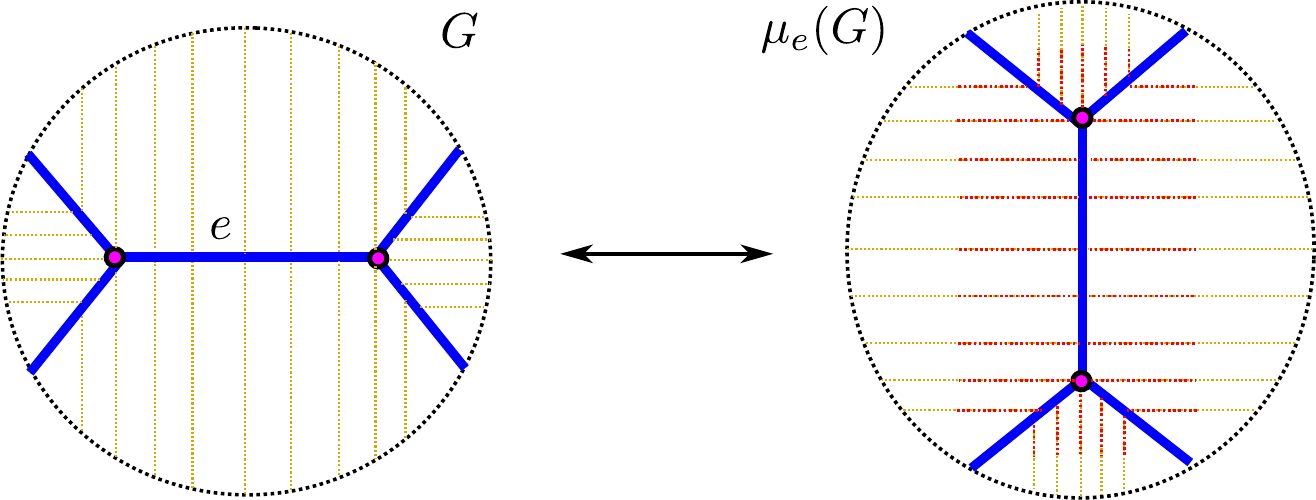}
		\caption{Mutation for an $N$-graph $G$ along a monochromatic $i$-edge $e$. The mutated graph $\mu_e(G)$ admits a woven front $\Sigma(\mu_e(G))$ which coincides with any front $\Sigma(G)$ woven with respect to $G$ near the boundary of the neighborhood $\Op(e)$. The yellow foliation near the boundary fixes the difference between the $i$th and $(i+1)$th sheets in both fronts $\Sigma(G)$ and $\Sigma(\mu_e(G))$. This foliation is extended to the interior in two different ways, yellow or red, depending on the graph being $G$ or $\mu_e(G)$.}
		\label{fig:ReebChordsExamples2}
	\end{figure}
\end{center}

In order to guarantee that $\mu_e(G)$ is free, we construct a front $\Sigma(\mu_e(G))$ woven with respect to $\mu_e(G)$ as follows: this new front is identical to that of $G$ near the boundary of the neighborhood of the monochromatic edge, and the $j$-th sheets for $\Sigma(\mu_e(G))$ coincide with those of $\Sigma(G)$ except for the sheets corresponding to $j=i,i+1$. The $i$th and $(i+1)$th sheets of $\Sigma(\mu_e(G))$ are woven according to $\mu_e(G)$ such that the difference in heights between the $i$th and the $(i+1)$th sheets increases (or decreases) strictly along the 1-dimensional red foliation as we move away from $\mu_e(G)$ as shown in Figure \ref{fig:ReebChordsExamples2} (Right). Since the red foliation is drawn to coincide with the yellow foliation at the boundary of the neighborhood $\Op(e)$, this is consistent with the sheets coinciding in that neighborhood. Given that the leaves of the 1-dimensional red foliation are intervals with a free end, it is possibly to build such a front, meeting the condition that the difference of heights between $i$th and $(i+1)$th strictly increases (or decreases). In addition, we can draw the front $\Sigma(\mu_e(G))$ such that the slopes of each sheet are arbitrarily close to the slopes of $\Sigma(G)$. This guarantees that $\mu_e(G)$ is free as required.

For a general $N$-graph mutation along a $\sf I$- or $\sf Y$-cycle, it suffices to observe that Subsections \ref{ssec:legmutation} and \ref{ssec:legmutation_local} show that such mutations are given by a composition of Legendrian Reidemeister moves, as presented in Subsection \ref{ssec:ReidemeisterMoves}, and mutations along monochromatic edges. Legendrian Reidemeister moves are local, relative to the boundary, and can be performed without ever introducing Reeb chords. Thus an $N$-graph mutation $\mu(G)$ of a free $G$ is free if the statement holds for 2-graph mutations, which we have already proven above.
\end{proof}

Lemma \ref{lem:free} allows us to perform Legendrian mutations to the $N$-graph and obtain potentially new embedded exact Lagrangian fillings. Examples of this are now illustrated. We will implicitly apply Lemma \ref{lem:free} in Subsection \ref{ssec:QuiverMutationsGeometric}, in order to realize cluster mutations as $N$-graph mutations of embedded exact Lagrangian fillings.

\subsubsection{Explicit Examples of Lagrangian Fillings} For the case of free $2$-graphs on a disk $\D^2$, this immediately yields that the max-tb Legendrian $(2,n)$-torus positive link $\La(2,n)$ has at least a Catalan $C_n$ number worth of exact Lagrangian fillings \cite{EkholmHondaKalman16,YuPan,STZ_ConstrSheaves,TreumannZaslow}. This is because $C_n$ counts binary trees, which are equivalent to free 2-graphs. These exact Lagrangian fillings are distinguished, up to Hamiltonian isotopy, through the use of cluster coordinates -- see Subsection \ref{sssec:ClusterCoord}. Now, the ability to increase $N\in\N$ greatly expands\footnote{This is particularly relevant for the study of exact Lagrangian fillings, as it is expected that any $\La(\beta)$ with $\beta\in\mbox{Br}^+_{2}$ has only {\it finitely} many exact Lagrangian fillings, and we will show in Theorem \ref{thm:ThurstonLinks} that this is {\it not} the case already for $N=3$.} the class of Legendrian links for which their Lagrangian fillings can be studied with $N$-graph calculus, including all Legendrian positive braids $\La(\beta)$, $\beta\in\mbox{Br}^+_{N}$ for any $N\in\N$.\\

{\it \underline{Example 1}}: Recently, the first examples of Legendrian links with infinitely many exact Lagrangian fillings were described in the article \cite{CasalsHonghao}. We exhibit them here in terms of 3-graphs. For any $(p,q)\in\N^2$, the max-tb Legendrian $(p,q)$-torus positive link $\La(p,q)\sse(\S^3,\xi_\st)$ is the satellite of the braid $\Delta(\sigma_1\sigma_2\cdot\ldots\cdot\sigma_{p-1})^q\Delta$ along the standard Legendrian unknot, where $\Delta=\Delta_p\in S_p$ is the $p$-stranded half-twist. Let us now illustrate how to diagrammatically visualize these infinitely many Lagrangian for the Legendrian link $\La(3,6)$.

\begin{remark}
Similar $p$-graphs can be drawn for $\La(p,q)$ for all $(p,q)\in\N^2$ and they produce infinitely many Lagrangian fillings if $p\geq3,q\geq6$ or $(p,q)=(4,4),(4,5)$. Alternatively, infinitely many exact Lagrangian fillings for $\La(p,q)\sse(\S^3,\xi_\st)$, $p\geq4,q\geq7$ can also be readily constructed from those of $\La(3,6)$ \cite[Corollary 1.5]{CasalsHonghao}.\hfill$\Box$
\end{remark}

Consider the braid word $\beta=(\sigma_1\sigma_2)^9=\Delta(\sigma_1\sigma_2)^6\Delta$ in the 1-jet space $(J^1\S^1,\xi_\st)$, where $\Delta=\Delta_3\in S_3$ is the 3-stranded half-twist $\Delta=\sigma_1\sigma_2\sigma_1$. This braid $\beta$ can be depicted as a set of points in the circle $\S^1$ labeled with two colors, corresponding to $\sigma_1,\sigma_2$. Figure \ref{fig:T36LagrangianConcordance} shows this braid $\beta$ in two circles, the inner circle $\S^1\times\{1\}$ and outer circle $\S^1\times\{2\}$ in the annulus $\S^1\times[1,2]$. These two marked circles are labeled by $\La(3,6)$, as the Legendrian link associated to the marking $\beta$ is the $(3,6)$-Legendrian link $\La(3,6)$.

\begin{center}
	\begin{figure}[h!]
		\centering
		\includegraphics[scale=0.45]{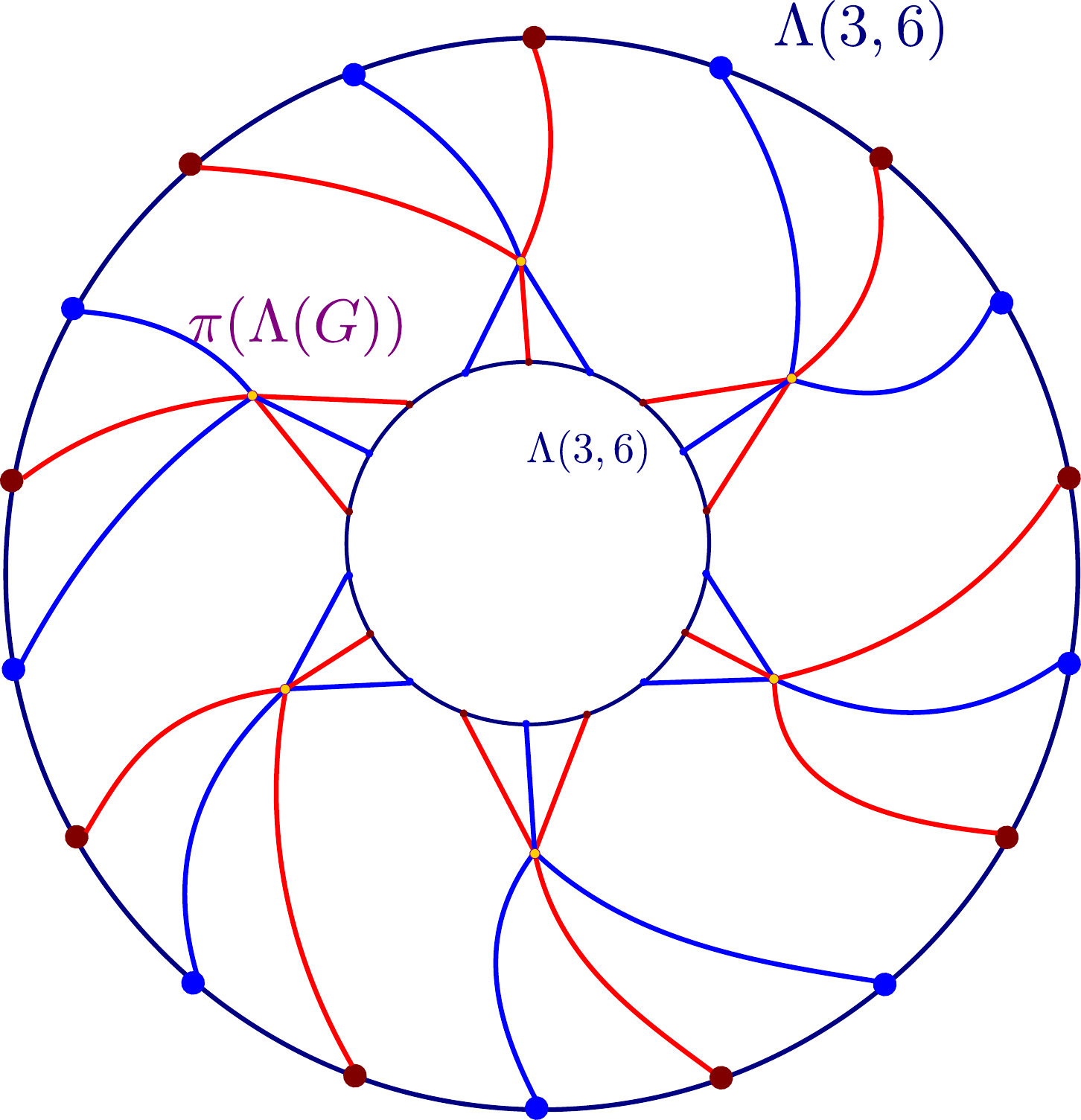}
		\caption{Legendrian weave whose Lagrangian projection defines an infinite order element in the fundamental group of the space of Legendrian links isotopic to $\La(3,6)$. In particular, this Lagrangian concordance has infinite order in the Lagrangian concordance monoid. Infinitely many Lagrangian fillings for $\La(3,6)$, and all torus links $\La(n,m)$, $n\geq3,m\geq6$, are obtained by concatenating this 3-graph.}
		\label{fig:T36LagrangianConcordance}
	\end{figure}
\end{center}

The 3-graph $G\sse \S^1\times[1,2]$ depicted in Figure \ref{fig:T36LagrangianConcordance} describes a Legendrian surface $\La(G)\sse (J^1\S^1\times[1,2],\xi_\st)$ with boundary $\La(3,6)\sqcup\La(3,6)$. By increasing the slope in the radial direction, the Legendrian surface $\La(G)$ can be assumed to have no Reeb chords, and thus $\pi(\La(G))$ is an exact Lagrangian cobordism from $\La(3,6)$ to itself. Since the graph $G$ has no trivalent vertices, $\La(G)$ has the topology of $\La(3,6)\times[1,2]$ and it is in fact an exact Lagrangian concordance. The remarkable property of the 3-graph $G$, and its Lagrangian projection $\pi(\La(G))$, is stated in the following:

\begin{thm}[\cite{CasalsHonghao}] The $3$-graph exact Lagrangian concordance in Figure \ref{fig:T36LagrangianConcordance} has infinite order. In particular, for any fixed exact Lagrangian filling of $\La(3,6)$, iterated concatenation of this $3$-graph yields infinitely many Lagrangian fillings of the Legendrian link $\La(3,6)\sse(\S^3,\xi_\st)$.\hfill$\Box$
\end{thm}

In fact, it is possible to describe the entire faithful modular PSL($2,\Z$)-representation in \cite{CasalsHonghao} with the diagrammatics of 3-graphs. Similarly, the diagrammatics of 4-graphs give explicit spatial wavefronts for the $M_{0,4}$-worth of the (Legendrian lift of the) Lagrangian fillings for the Legendrian link $\La(4,4)\sse(\S^3,\xi_\st)$. The non-triviality, and infinite order, of this Lagrangian concordance is detected by studying its action on the cluster structure of the coordinate ring of the moduli space of isomorphism classes of simple objects in $\Sh_{\Lambda(3,6)}(\R^2)$.\\

{\it \underline{Example 2}}: Let us address the following question. Given a positive braid $\beta$, and the Legendrian link $\La=\La(\Delta\beta\Delta)$, how do we diagrammatically produce an $N$-graph which represents an embedded exact Lagrangian filling for $\La\sse(\S^3,\xi_\st)$ ?

Let us begin with a simple example, with $\beta=\Delta^2=(\sigma_1\sigma_2)^3$ the full-twist, which is smoothly the $(3,3)$-torus link. The game is to draw $\tau_i$-edges along the boundary $\dd\D^2$ of a (planar) 2-disk $\D^2$ according to the braid word $\beta$ and complete these edges to an $N$-graph $G$ inside $\D^2$. The only rule is that the Legendrian weave $\La(G)$ should not have Reeb chords, or else it would yield an {\it immersed} Lagrangian filling, and thus we require $G$ to be free.

Consider the free 3-graph $G_1$ in Figure \ref{fig:T33LagrangianFillings} (upper Left). This represents an embedded exact Lagrangian filling $L_1$ of the max-tb Legendrian $(3,3)$-torus link $\La(3,3)=\La(\Delta\beta\Delta)=\La(\Delta^4)$. We can now apply the Legendrian mutation moves in Theorem \ref{thm:LegMutations} in order to produce another Lagrangian filling $L_2$ which is {\it not} Hamiltonian isotopic to the exact Lagrangian filling $L_2$. (Note that $L_1$ and $L_2$ are smoothly isotopic relative to their boundaries, and $L_2$ will be also embedded thanks to Lemma \ref{lem:free}.) In Figure \ref{fig:T33LagrangianFillings} we perform a Lagrangian disk surgery on $L_1$ along a Lagrangian 2-disk which bounds the 1-cycle in $H_1(L_1,\Z)$ graphically given by the $\sf Y$-cycle in surrounded by the dashed green curve.

\begin{center}
	\begin{figure}[h!]
		\centering
		\includegraphics[scale=0.8]{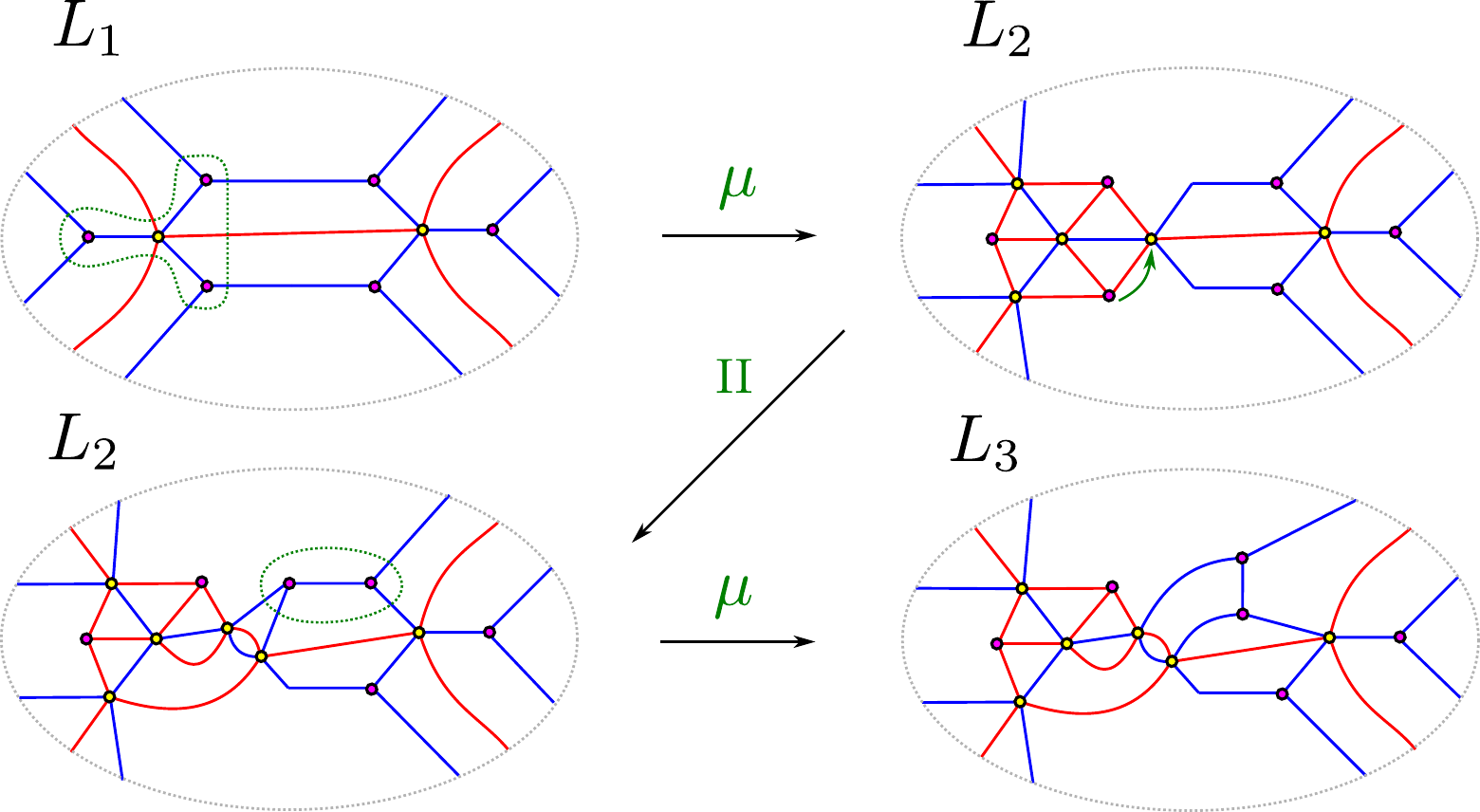}
		\caption{Four 3-graphs representing exact embedded Lagrangian fillings for the maximal-tb $(3,3)$-torus link $\La(3,3)$.}
		\label{fig:T33LagrangianFillings}
	\end{figure}
\end{center}

At this stage we can manipulate $L_2$ with Theorem \ref{thm:surfaceReidemeister}, in this case Figure \ref{fig:T33LagrangianFillings} (upper right) to \ref{fig:T33LagrangianFillings} (bottom left) shows how to apply Move II to push-through a hexagonal vertex through a trivalent vertex (as indicated by the green arrow). This is an interesting move because it makes a new $1$-cycle for $L_2$ readily visible, as represented by the blue monochromatic edge in \ref{fig:T33LagrangianFillings} (bottom left) surrounded by a dashed green curve. We can perform Lagrangian surgery at this monochromatic edge, as in Theorem \ref{thm:LegMutations}, to obtain another exact Lagrangian filling $L_3$, also embedded by Lemma \ref{lem:free}. It is immediate that $L_1$ and $L_3$ are not {\it not} Hamiltonian isotopic to $L_2$, as the cluster coordinates associated to these 3-graphs, as explained in Subsection \ref{sssec:ClusterCoord}, show that $L_1$ and $L_3$ are not Hamiltonian isotopic. In conclusion, the 3-graphs in Figure \ref{fig:T33LagrangianFillings} represent three distinct embedded exact Lagrangian fillings for $\La(3,3)$.


{\it \underline{Example 3}}: Let us illustrate what a {\it generic} $3$-graph diagram like for a positive braid $\beta\in\mbox{Br}^+_3$. The pictures in the case of $\beta\in\mbox{Br}^+_N$, $N\geq3$ are alike, with as many as $(N-1)$-colors instead. Let us consider a random braid
$$\beta=(\sigma_1\sigma_2\sigma_1)\sigma_2^2\sigma_1^2\sigma_2^2\sigma_1^3\sigma_2^3(\sigma_1\sigma_2\sigma_1),$$
which has no particular significance to us. To obtain exact Lagrangian fillings, we draw blue and red edges around a circle $\S^1\sse\R^2$, according to \color{blue}$\sigma_1$ \color{black} or \color{red}$\sigma_2$\color{black}, and construct 3-graphs with no Reeb chords and these boundary constraints. Figure \ref{fig:ThreeLagrangianFillings} shows four free $3$-graphs $G_i$, $i\in[1,4]$, such that the Lagrangian projections $\pi(\iota(\La(G_i)))\sse(\D^4,\omega_\st)$ are embedded exact Lagrangian fillings which are distinct up to Hamiltonian isotopy for $i\neq j$, $i,j\in[1,4]$.

\begin{remark}
From our experience drawing 3-graphs, the pictures in Figure \ref{fig:ThreeLagrangianFillings} accurately represent the generic appearance of exact Lagrangian fillings described by free 3-graphs. We presently do not know any example of a Lagrangian filling for a positive braid which does {\it not} arise as an $N$-graph, for some $N\in\N$.\hfill$\Box$
\end{remark}

\begin{center}
	\begin{figure}[h!]
		\centering
		\includegraphics[scale=0.8]{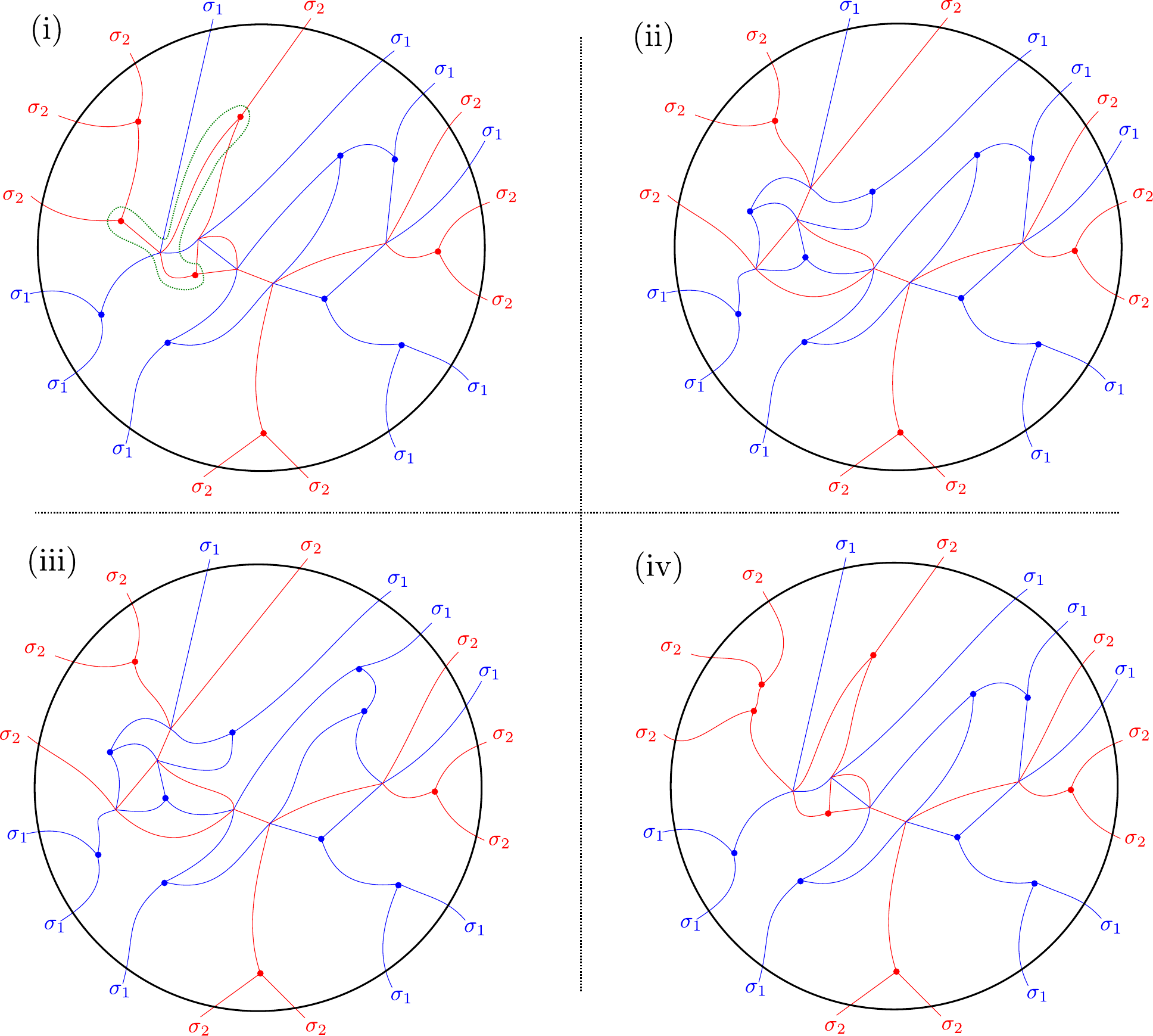}
		\caption{Four exact embedded Lagrangian fillings for the braid $\beta$ in Example 3. Their satellites in $(\R^4,\omega_\st)$ are smoothly isotopic relative to their boundaries, but {\it not} Hamiltonian isotopic.)}
		\label{fig:ThreeLagrangianFillings}
	\end{figure}
\end{center}

\begin{remark}
There exists a technique for producing many such free $N$-graphs $G$, filling $\beta$-boundary conditions at a circle and thus representing embedded exact Lagrangian fillings. This is ongoing work by the first author, which in particular proves that {\it any} Legendrian link $\La(\beta)$ arising from a positive braid $\beta\in\mbox{Br}^+_{N}$ admits an embedded Lagrangian filling whose Legendrian lift is a Legendrian weave. In precise terms, it can be proven that for each triangulation of a $|\beta|$-gon, one can assign a free $N$-graph which represented an embedded Lagrangian filling of $\beta$, where $|\beta|$ is the length of the positive braid $\beta$.\hfill$\Box$
\end{remark}


\subsection{Microlocal monodromies and cluster structures}\label{ssec:Microlocal}

In this section, we demonstrate how notions of cluster theory are borne out with $N$-graphs. This is an important ingredient in showing that microlocal monodromies can be used to distinguish exact Lagrangian fillings, as we do in Section \ref{ssec:QuiverMutationsGeometric} and as has been mentioned previously.

To orient the discussion, we recall that the cluster structures on the Fock-Goncharov moduli spaces of framed local systems described in \cite{FockGoncharov_ModuliLocSys} were given a sheaf-theoretic description in
\cite{STW,STWZ}.  In these works, the spectral surface associated to a bipartite graph, as defined in \cite[Section 2.2]{Goncharov_IdealWebs}, is described symplectically as an exact Lagrangian filling of the zigzag Legendrian curve. In the case of bipartite graphs associated to an $N$-triangulation, as in \cite[Section 1]{Goncharov_IdealWebs}, the zigzag curves isotope to concentric circles around the vertices of the triangulation, and the singular support of such a configuration translates to the data of a local system with a monodromy-invariant flag at each vertex. Sheaf quantization \cite{GKS_Quantization} then implies that local systems on the exact filling embed as a cluster chart of objects, the chart being provided by the bipartite graph (and its dual quiver), and the cluster coordinates given by microlocal monodromies.  The intersection form in $H_1(L,\Z)$, or its negative, corresponds to the skew-symmetric bilinear form in cluster theory. For us, the crucial point is that we can represent all these Lagrangian fillings by $N$-graphs, as in the diagrammatics of Subsection \ref{ssec:ExactLagrangianFillings}, and the cluster coordinates can be read directly from the $N$-graph, as we will now explain.

\begin{remark} In \cite{TreumannZaslow}, the case of Legendrian surfaces defined by trivalent 2-graphs was studied, giving a sheaf-theoretic description of the constructions in \cite{DimofteGabellaGoncharov}. In this setting, the microlocal monodromy functor $\mu mon$ induces, at the level of moduli of objects, a morphism from the sheaf moduli space to the cluster chart defined by the triangulation dual to the 2-graph. The image is a (holomorphic) Lagrangian in a (holomorphic) symplectic leaf, as in \cite{DimofteGabellaGoncharov},
	in a manner compatible with quantization of algebra of functions.\footnote{In work in progress with Linhui Shen, the second author will develop the relation to cluster theory more systematically, and prove Lagrangianicity of the moduli space.} Furthermore, in that work, the potential describing the local exact structure of the Lagrangian was interpreted as a generator of BPS states or disk invariants, following the analysis of Aganagic-Vafa \cite{AV1,AV2}. Here we generalize some of the constructions to $N$-graphs, $N\geq2$.\hfill$\Box$
\end{remark}

In this article, the Legendrian surfaces are described by $N$-graphs, a more complex construction, but we will now explain how the basic features should persist. That is, the microlocal monodromy functor allows us to read cluster coordinates for the moduli spaces of isomorphism classes of simple objects in $Sh_\La(\R^2)$, equivalently augmentation varieties, directly from $N$-graphs with boundary $\Lambda$.  Examples of these constructions are provided below.

\subsubsection{Microlocal monodromies as cluster coordinates}\label{sssec:ClusterCoord}

By definition, microlocal monodromy is a functor 
$$\mu mon:Sh_\Lambda\to Loc(\Lambda)$$
from the category $Sh_\La$ of sheaves  microsupported on the Legendrian surface $\Lambda$, as defined in Subsection \ref{ssec:constsheaves}, to the category of local systems on $\Lambda$ \cite{STZ_ConstrSheaves}.  This functor carries microlocal rank-one sheaves $F\in Sh^1_\La$, i.e.~simple sheaves, to rank one local systems on the surface $\La$.  Since it is locally defined, the monodromy of the local system $\mu mon(F)$ around a loop $\gamma \in H_1(\Lambda)$ can be evaluated by restricting the constructible sheaf $F$ to an annular tubular neighborhood of $\gamma$.  Below, these annuli are depicted as thin purple loops.  In short, the calculation for Legendrian weaves can be done using the microlocal monodromy functor $\mu mon$ as it is used for knots, as described in \cite{STZ_ConstrSheaves}.

The main point in these computations is that the stalk $\mu mon(F)\vert_\lambda$ at a point $\la\in\La$ is the cone of the restriction map corresponding to $\lambda$, and for flags this is the inclusion of subspaces, whence cones become cokernels. The transversality of adjacent flags ensures that these cokernels propogate as a local system. Let us now perform these calculations for 1-cycle $\gamma\in H_1(\La,\Z)$, starting at the $\sf I$-cycle represented by a monochromatic edge.

Let us consider a monochromatic edge with label $\tau_i$, as depicted in Figure \ref{fig:monoedge}.
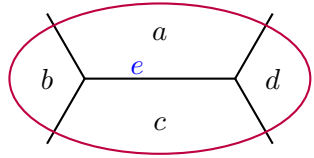
\begin{figure}[H]
	\begin{tikzpicture}
	\pgfmathsetmacro{\A}{.866}
	\draw[thick] (-1/2,\A) -- (0,0) -- (-1/2,-\A);
	\draw[thick] (0,0) -- (2,0);
	\draw[thick] (2.5,\A) -- (2,0) -- (2.5,-\A);
	\node at (1,.6) {$a$};
	\node at (-.5,0) {$b$};
	\node at (1,-.6) {$c$};
	\node at (2.5,0) {$d$};
	\node[blue] at (.7,.15) {$e$};
	\draw[purple,thick] (1,0) ellipse (2 and 1);
	\end{tikzpicture}
	\caption{Neighborhood of a monochromatic edge $e$ with the data determining a constructible sheaf $F$. As we show, the microlocal monodromy $\mu mon(F)$ along the 1-cycle $\gamma(e)$ is given by the cross-ratio $\langle a, b, c, d\rangle$.}
	\label{fig:monoedge}
\end{figure}
Near such a monochromatic edge, a sheaf object in a simply connected face is specified by the data of a quadruple of flags.  Each of these flags has the same subspaces $\SF^j$ in each region for $j\neq i$, and for $j = i$ we additionally require the data in each region of a line $l$ in the two-dimensional space $V := \SF^{i+1}/\SF^{i-1}.$  This is the data of four lines $a, b, c, d\sse V$, as specified in Figure \ref{fig:monoedge}.  Restricted to the purple oval shown, we have a cylindrical braid of type $\beta=\sigma_i^4$, where $\sigma_i$ is the lift of the transposition $\tau_i$ from the Coxeter group $S_N$ to the braid group $\mbox{Br}_{N}$.  Given the prescribed transversality imposed by the flag moduli of an $N$-graph, we further know that the cyclic chain of inequalities $a\neq b \neq c \neq d \neq a$ holds. We thus have the chain of isomorphisms of cokernels
$$a \cong V/b \cong c \cong V/d \cong a,$$
which computes the microlocal monodromy.  In this case, the isomorphism that we obtain is the cross ratio
$$\langle a, b, c, d\rangle = \frac{a\wedge b}{b\wedge c}\cdot\frac{c\wedge d}{d\wedge a}$$
of the four lines $a,b,c,d,$ and it is equal to the cluster coordinate associated to 1-cycle $\gamma$ as prescribed in \cite[Section 9]{FockGoncharov_ModuliLocSys}.

Let us now consider the cluster coordinate associated to a $\sf Y$-cycle, which is a new type of $1$-cycle, as it only appears for $N\geq3$. Figure \ref{fig:triple} depicts a $\sf Y$-cycle, drawn as a purple circle, along with the data determining a constructible sheaf in a neighborhood of this 1-cycle.
\begin{figure}[H]
	\begin{tikzpicture}
	\pgfmathsetmacro{\A}{1.732}
	\pgfmathsetmacro{\a}{.866}
	\draw[blue,thick] (1-1,\A+1/\A)--(1,\A)--(1+1,\A+1/\A);
	\draw[blue,thick] (0-1,0+1/\A)--(0,0)--(0+1,0+1/\A);
	\draw[blue,thick] (2-1,0+1/\A)--(2,0)--(2+1,0+1/\A);
	\draw[blue,thick] (2-\a,0+.5)--(2,0)--(2+\a,0+.5);
	\draw[blue,thick] (0,0)--(0,-2/\A);
	\draw[blue,thick] (2,0)--(2,-2/\A);
	\draw[blue,thick] (1,\A)--(1,\A-2/\A);
	\draw[red,thick] (1,\A-2/\A)--(1,\A-2/\A-\A);
	\draw[red,thick] (0+1,0+1/\A)--(0+1+1.5,0+1/\A+\a);
	\draw[red,thick] (0+1,0+1/\A)--(0+1-1.5,0+1/\A+\a);
	\node[blue] at (0,0) {$\bullet$};
	\node[blue] at (2,0) {$\bullet$};
	\node[blue] at (1,\A) {$\bullet$};
	\node at (3,-.4) {$(a,A)$};
	\node at (-1,-.4) {$(c,C)$};
	\node at (1,\A+.8) {$(b,B)$};
	\node[rotate=30] at (3.3,1.3) {$(a,ab)$};
	\node[rotate=30] at (2.6,2) {$(b,ab)$};
	\node[rotate=-30] at (-1,1.2) {$(c,bc)$};
	\node[rotate=-30] at (-.6,2) {$(b,bc)$};
	\node[rotate=-90] at (.5,-1.4) {$(c,ac)$};
	\node[rotate=-90] at (1.5,-1.4) {$(a,ac)$};
	\draw[purple,thick] (1,1/\A) ellipse (1.4 and 1.4);
	\end{tikzpicture}
	\caption{Neighborhood of a $\sf Y$-cycle with the data determining a constructible sheaf $F$. As we compute, the microlocal monodromy $\mu mon(F)$ along the associated 1-cycle $\gamma$ is given by the triple ratio of the three transverse flags.}
	\label{fig:triple}
\end{figure}
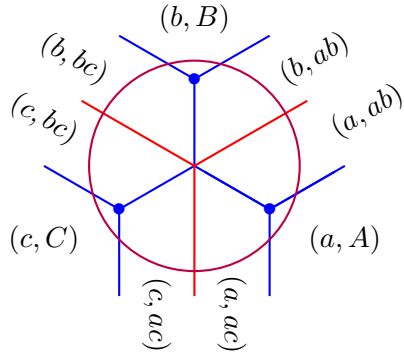
Following the notation in Section \ref{sec:app}, we denote by $ab$ the unique plane containing the two lines $a$ and $b$, while $AB$ denotes the intersection of the planes $A$ and $B$. The braid associated to the $\sf Y$-cycle $\gamma$, as drawn by the purple circle in Figure \ref{fig:triple}, is given by $\beta=(\sigma_i\sigma_{i+1}\sigma_i)^3$, where $\sigma_i$ corresponds to the crossing coming from a $\tau_i$-edge. By considering the three-dimensional vector space $V := \SF^{i+2}/\SF^{i-1}$, a given flag is specified by a line and a plane in $V$.  Since the word $\sigma_i\sigma_{i+1}\sigma_i$ represents the half-twist $\Delta$ for flags on $V$, and $\tau_i\tau_{i+1}\tau_i$ is the Coxeter element in $S_3$, the complete data specifying a constructible sheaf near the $\sf Y$-cycle is given by three transverse flags $(a,A),(b,B),(c,C)$ in $V$. In this notation, the line is written in lower case and the covector defining the the plane in upper case, thus $(a,A)$ determines a flag. Now, the microlocal monodromy functor $\mu mon$ along $\gamma$ is computed as the composition of the isomorphisms
$$a \cong V/B \cong c \cong V/A \cong b \cong V/C \cong a.$$ 
Let $v_a \in a,v_b\in b,v_c\in c,v_d \in d$ be non-zero vectors defining the corresponding one-dimensional lines. Then the parallel transport from $a$ to $c$ in this basis is given by the quotient $B(a)/B(c)$, where $B(a)$ is the pairing between the vector $v_a$ and the covector $B$. Iterating these isomorphisms, we conclude that the microlocal monodromy along the ${\sf Y}$-cycle is given by
$$\langle(a,A),(b,B),(c,C)\rangle:=\frac{B(a)C(b)A(c)}{B(c)C(a)A(b)}.$$
This expression is precisely the triple product of transverse flags as defined in \cite{FockGoncharov_ModuliLocSys}, and thus we have shown that the microlocal monodromy along a $\sf Y$-cycle determines a cluster coordinate.

\subsubsection{Legendrian Mutations are cluster transformations}\label{sssec:ClusterTransf}

The coordinate transformations upon Legendrian mutations can also be computed, as we will demonstrate in an example. The conclusion is that Legendrian mutations induce cluster transformations. The case of a monochromatic edge follows from the analysis in \cite{TreumannZaslow,STWZ}, and we now study the mutation at a $\sf Y$-cycle. To do so, consider the local geometry shown in Figure \ref{fig:ymutation-before}. We want to compute how the cluster coordinate associated to the unique monochromatic (blue) edge -- as in Subsection \ref{sssec:ClusterCoord} -- changes as we perform a Legendrian mutation along the $\sf Y$-cycle specified by the unique hexagonal vertex.

\begin{figure}[H]
	\begin{tikzpicture}
	\pgfmathsetmacro{\A}{1.732}
	\pgfmathsetmacro{\a}{.866}
	\draw[blue,thick] (1-1,\A+1/\A)--(1,\A)--(1+\A,\A+1);
	\draw[blue,thick] (0-1,0+1/\A)--(0,0)--(0+1,0+1/\A);
	\draw[blue,thick] (2-1,0+1/\A)--(2,0)--(2+1,0+1/\A);
	\draw[blue,thick] (2-\a,0+.5)--(2,0)--(2+\a,0+.5);
	\draw[blue,thick] (0,0)--(0,-2/\A);
	\draw[blue,thick] (2,0)--(2,-2/\A);
	\draw[blue,thick] (1,\A)--(1,\A-2/\A);
	
	\draw[red,thick] (1,\A-2/\A)--(1,\A-2/\A-\A);
	\draw[red,thick] (0+1,0+1/\A)--(0+1+2.5,0+1/\A+2.5/\A);
	\draw[red,thick] (0+1,0+1/\A)--(0+1-1.5,0+1/\A+\a);
	\node[blue] at (0,0) {$\bullet$};
	\node[blue] at (2,0) {$\bullet$};
	\node[blue] at (1,\A) {$\bullet$};
	\node[blue] at (1+\A,\A+1) {$\bullet$};
	\draw[blue,thick] (1.5+\A,\A+1)--(1+\A,\A+1)--(1+\A,\A+1.5);
	
	\node at (3,-.4) {$(a,A)$};
	\node at (-1,-.4) {$(c,C)$};
	\node at (1,\A+.8) {$(b,B)$};
	\node[rotate=30] at (3.3,1.3) {$(a,ab)$};
	\node[rotate=30] at (2.6,2) {$(b,ab)$};
	\node[rotate=-30] at (-1,1.2) {$(c,bc)$};
	\node[rotate=-30] at (-.6,2) {$(b,bc)$};
	\node[rotate=-90] at (.5,-1.4) {$(c,ac)$};
	\node[rotate=-90] at (1.5,-1.4) {$(a,ac)$};
	\node[rotate=-0] at (3.5,3.2) {$(b,B')$};
	\end{tikzpicture}
	\caption{The geometric setup before performing a Legendrian mutation at the $\sf Y$-cycle, where the cluster coordinate associated to the monochromatic edge is given by the cross-ratio $\langle B,bc,ab,B'\rangle$.}
	\label{fig:ymutation-before}
\end{figure}
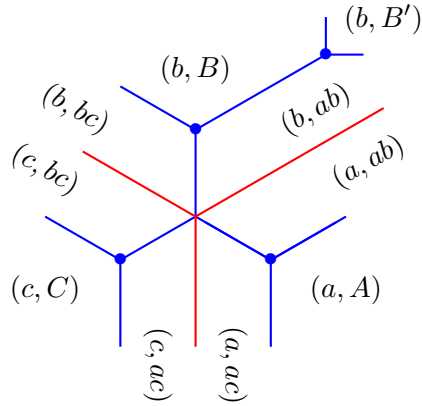
The monochromatic blue edge has monodromy equal to the cross ratio $z:=\langle B, bc, ab, B'\rangle$ of the four planes in the projective line of planes containing $b$. (This can be computed directly or by intersecting the four lines with any transverse line -- see Subsection \ref{sssec:ClusterCoord}.) Now, after Legendrian mutation at the ${\sf Y}$-cycle, the resulting 3-graph is shown Figure \ref{fig:ymutation-after}.
\begin{figure}[H]
	\begin{tikzpicture}[scale=.7]
	\pgfmathsetmacro{\A}{1.732}
	\pgfmathsetmacro{\a}{.866}
	\draw[red,thick] (1-1,\A+1/\A)--(1,\A)--(1+1,\A+1/\A);
	\draw[red,thick] (0-1,0+1/\A)--(0,0)--(0+1,0+1/\A);
	\draw[red,thick] (2-1,0+1/\A)--(2,0)--(2+1,0+1/\A);
	\draw[red,thick] (2-\a,0+.5)--(2,0)--(2+\a,0+.5);
	\draw[red,thick] (0,0)--(0,-2/\A);
	\draw[red,thick] (2,0)--(2,-2/\A);
	\draw[red,thick] (1,\A)--(1,\A-2/\A);
	\draw[blue,thick] (1,\A-2/\A)--(1,\A-2/\A-\A);
	\draw[blue,thick] (0+1,0+1/\A)--(0+1+1.5,0+1/\A+\a);
	\draw[blue,thick] (0+1,0+1/\A)--(0+1-1.5,0+1/\A+\a);
	
	\draw[red,thick] (0,-2/\A)--(1,-\A)--(2,-2/\A);
	\draw[red,thick] (1,-\A)--(1,-\A-3/\A);
	\draw[blue,thick] (0,-4/\A)--(1,-\A)--(2,-4/\A);
	\draw[blue,thick] (1,-2/\A)--(1,-\A);
	
	\draw[red,thick] (-1-1/2,\A+.5/\A)--(-1-1.5,\A+1.5/\A);
	
	\draw[red,thick] (3,1/\A)--(3,\A)--(2,4/\A);
	\draw[red,thick] (3,\A)--(3+2.5,\A+.8);
	\draw[blue,thick] (3,\A)--(3-.5,\A-.5/\A);
	\draw[blue,thick] (3+1,1/\A+1/\A)--(3,\A);
	
	\draw[blue,thick](3,\A)--(5,\A+2)--(5,\A+3);
	\draw[blue,thick] (5,\A+2)--(6,\A+2);
	\node[blue] at (5,\A+2) {$\bullet$};

	\draw[red,thick] (2-3,1/\A)--(2-3,\A)--(2-2,4/\A);
	\draw[red,thick] (2-3,\A)--(2-3-.5,\A+.5/\A);
	\draw[blue,thick] (2-3,\A)--(2-3+.5,\A-.5/\A);
	\draw[blue,thick] (2-3-1,1/\A+1/\A)--(2-3,\A)--(2-2-1,4/\A+1/\A);
	\node[blue] at (0,0) {$\bullet$};

	\node[red] at (0,0) {$\bullet$};
	\node[red] at (2,0) {$\bullet$};
	\node[red] at (1,\A) {$\bullet$};
	\node[rotate=16] at (4.7,1.8) {$(a,ab)$};
	\node[rotate=16] at (4.9,2.8) {$(b,ab)$};
	\node[rotate=30] at (2.2,.7) {${}_{(AB,A)}$};
	\node[rotate=30] at (1.8,1.5) {${}_{(AB,B)}$};
	\node[rotate=-30] at (-2.5,2) {$(c,bc)$};
	\node[rotate=-30] at (-2.1,3) {$(b,bc)$};
	\node[rotate=-30] at (.2,1.5) {${}_{(BC,B)}$};
	\node[rotate=-30] at (-.2,.7) {${}_{(BC,C)}$};
	\node[rotate=-90] at (.5,-3) {$(c,ac)$};
	\node[rotate=-90] at (1.5,-3) {$(a,ac)$};
	\node[rotate=-90] at (.5,-.6) {${}_{(AC,C)}$};
	\node[rotate=-90] at (1.5,-.6) {${}_{(AC,A)}$};

	\node[rotate=-0] at (6,\A+2.5) {$(b,B')$};
	\node at (3,-.4) {$(a,A)$};
	\node at (-1,-.4) {$(c,C)$};
	\node at (1,\A+.8) {$(b,B)$};
	\end{tikzpicture}
	\caption{The result of applying a Legendrian mutation to Figure \ref{fig:ymutation-before} along the {\sf Y}-cycle, along with the data of a constructible sheaf.}
	\label{fig:ymutation-after}
\end{figure}
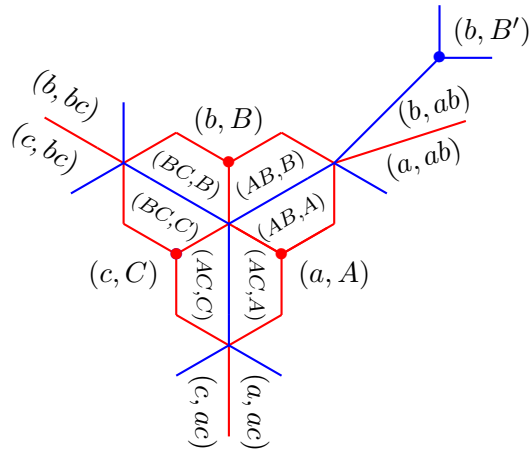
The 1-cycle determined by the blue monochromatic edge in Figure \ref{fig:ymutation-before} becomes a (bichromatic edge) 1-cycle contained in the 3-graph shown in Figure \ref{fig:ymutation-after-local}, which is itself a piece of Figure \ref{fig:ymutation-after}, in its upper-right corner\footnote{The trivalent blue vertex in Figure \ref{fig:ymutation-after-local} is the unique trivalent blue vertex in Figure \ref{fig:ymutation-after}. The trivalent red vertex in Figure \ref{fig:ymutation-after-local} is the rightmost trivalent red vertex in Figure \ref{fig:ymutation-after}. }:
\begin{figure}[H]
	\begin{tikzpicture}[scale=1.5]
	\pgfmathsetmacro{\A}{1.732}
	\pgfmathsetmacro{\a}{.866}
	\draw[red,thick] (-1,0)--(0,0);
	\draw[blue,thick] (0,0)--(1,0);
	\draw[red,thick] (-1,0)--(-1-.5,\a);
	\draw[red,thick] (-1,0)--(-1-.5,-\a);
	\draw[red,thick] (0,0)--(.5,\a);
	\draw[red,thick] (0,0)--(.5,-\a);
	\draw[blue,thick] (0,0)--(-.5,\a);
	\draw[blue,thick] (0,0)--(-.5,-\a);
	\draw[blue,thick] (1,0)--(1+.5,\a);
	\draw[blue,thick] (1,0)--(1+.5,-\a);
	\node[red] at (-1,0) {$\bullet$};
	\node[blue] at (1,0) {$\bullet$};
	
	\node[rotate=-0] at (1.5,0) {$(b,B')$};
	\node[rotate=-0] at (.67,.3) {$(b,B)$};
	\node[rotate=-0] at (0,.8) {${}_{(AB,B)}$};
	\node[rotate=-0] at (-.67,.3) {${(AB,A)}$};
	\node[rotate=-0] at (-1.7,0) {$(AC,A)$};
	\node[rotate=-0] at (-.67,-.3) {${(a,A)}$};
	\node[rotate=-0] at (0,-.8) {${}^{(a,ab)}$};
	\node[rotate=-0] at (.67,-.3) {$(b,ab)$};

	\end{tikzpicture}
	\caption{Local geometry near the 1-cycle after mutation.}
	\label{fig:ymutation-after-local}
\end{figure}
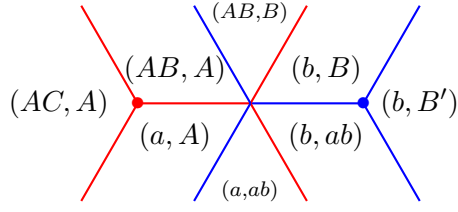
By applying Move II, we can push the red trivalent vertex in Figure \ref{fig:ymutation-after-local} through the hexavalent vertex. This allows us to represent the 1-cycle as a monochromatic edge again, as shown in Figure \ref{fig:ymutation-after-local-MoveII}:
\begin{center}
	\begin{figure}[H]
		\begin{tikzpicture}[scale=1.5]
		\pgfmathsetmacro{\A}{1.732}
		\pgfmathsetmacro{\a}{.866}
		\draw[blue,thick] (-1,0)--(0,0);
		\draw[blue,thick] (0,0)--(1,0);
		\draw[blue,thick] (-1,0)--(-1-.5,\a);
		\draw[blue,thick] (-1,0)--(-1-.5,-\a);
		\draw[blue,thick] (1,0)--(1+.5,\a);
		\draw[blue,thick] (1,0)--(1+.5,-\a);
		\node[blue] at (-1,0) {$\bullet$};
		\node[blue] at (1,0) {$\bullet$};
		
		\node[rotate=-0] at (1.5,0) {$(b,B')$};
		\node[rotate=-0] at (0,.5) {$(b,B)$};
		
		\node[rotate=-0] at (-1.7,0) {$(b, bAC)$};
		
		\node[rotate=-0] at (0,-.5) {$(b,ab)$};
		
		\end{tikzpicture}
		\caption{The constructible sheaf near the 1-cycle after Legendrian mutation and Move II. The new coordinate is thus the cross-ratio $\langle B,bAC,ab,B'\rangle$.}
		\label{fig:ymutation-after-local-MoveII}
	\end{figure}
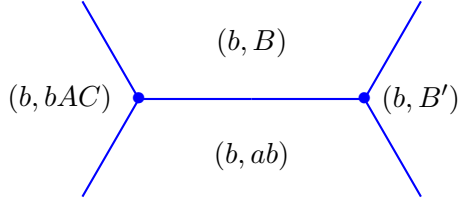
\end{center}

The required conclusion, stating that the new cross-ratio $z'=\langle B,bAC,ab,B'\rangle$ is obtained by a cluster transformation, follows from this:

\begin{lemma}\label{lem:clustertrans}
	Let $x=\langle(a,A), (b,B), (c,C)\rangle$ be the triple ratio of flags and $z=\langle B,bc,ab,B'\rangle$ the cross-ratio of lines. Denote by $z'=\langle B,bAC,ab,B'\rangle$ the new microlocal monodromy. Then
	$$z'=z(1+x).$$
\end{lemma}

\begin{proof}
	By $\PGL_3$ invariance, we may assume that
	$$a = \begin{pmatrix}1\\0\\0\end{pmatrix}, A = (0,0,1),
	b = \begin{pmatrix}0\\0\\1\end{pmatrix}, B = (1,0,0), c = \begin{pmatrix}1\\-1\\1\end{pmatrix}, C = (1,1+x,x).$$
	Since the cross-ratio $z$ is prescribed, we find that $B' = (z,1,0)$, and similarly
	$$AC = \begin{pmatrix}1+x\\ -1\\ 0\end{pmatrix}.$$
	This implies that $bAC = (1,1+x,0)$, and thus $z'=z(1+x)$.
\end{proof}

Note that $z' = z(1+x)$ in Lemma \ref{lem:clustertrans} is the transformation expected for a cluster-X transformation.\footnote{The rule for a cluster-X transformation upon mutating at loop $k$ is that the monodromy $z_i$ transforms to $1/z_k$ if $i=k$ and otherwise
	$z_i' = z_i(1 + z_k^{-{\rm sgn}\epsilon_{ik}})^{-\epsilon_{ik}},$ where $\epsilon_{i,k}$ is the skew-symmetric cluster form.  We get agreement on the nose if
	we make this form the \emph{negative} of the intersection pairing.} This concludes that a Legendrian mutation at the ${\sf Y}$-cycle induces a {\it cluster} transformation for the microlocal monodromy coordinate at the monochromatic blue edge in Figure \ref{fig:ymutation-before}. The computation is analogous if we choose a different blue monochromatic edge to be added near the ${\sf Y}$-cycle. In particular, if we had chosen instead the blue edge attaching at the lower-right of the {\sf Y}-cycle and pointing upward, and again called its monodromy $z$, then we would have $A' = (0, z, 1)$ and would obtain
$$z' = \langle ab, aBC, A, A'\rangle = z\left(1+\frac{1}{x}\right)^{-1},$$
in agreement with the cluster transformation.\footnote{We remark that the case $x = -1$ is not a generic configuration of flags, since in this case $c \in ab,$ and thus not in the domain of the birational cluster map.}

\begin{ex}{\bf Flip of a $N$-triangulation}. Let $(C,\tau)$ be a punctured surface $C$, $\tau$ an ideal triangulation and $\tau'$ an ideal triangulation obtained from $\tau$ by a flip. Denote by $t_N$, resp. $t'_N$, the $N$-triangulation refinement of $\tau$, resp. $\tau'$. It is an exercise \cite[Prosition 1.1]{Goncharov_IdealWebs} to show that the Legendrian weave $\La(G(t'_N))$ differs from $\La(G(t_N))$ by a sequence of ${N+1 \choose 3}$ 2-graph mutations, i.e. $\La(G(t'_N))$ can be obtained from $\La(G(t_N))$ by performing ${N+1 \choose 3}$ Legendrian mutations along 1-cycles represented by monochromatic edges.
	
	For instance, \cite[Figure 9]{Goncharov_IdealWebs} translates into four monochromatic edge mutations for a flip in a $N=3$ triangulation, as we have depicted in Figure \ref{fig:Flip3Triangulation}. We can see how to perform the corresponding moves for $3$-triangulations with $3$-graphs. Indeed, referring to the notation in Figure \ref{fig:3triangles}, perform a monochromatic edge mutation at $z$ and $w$, then perform Move III, a flop of the two trivalent and two hexagonal vertices in the center, and proceed with a mutation at the remaining two monochromatic edges. In conclusion, the constructions of this paper can therefore be used to give a geometric understanding of the intermediate quivers arising when flipping $N$-triangulations.
	
	\begin{center}
		\begin{figure}[h!]
			\centering
			\includegraphics[scale=0.7]{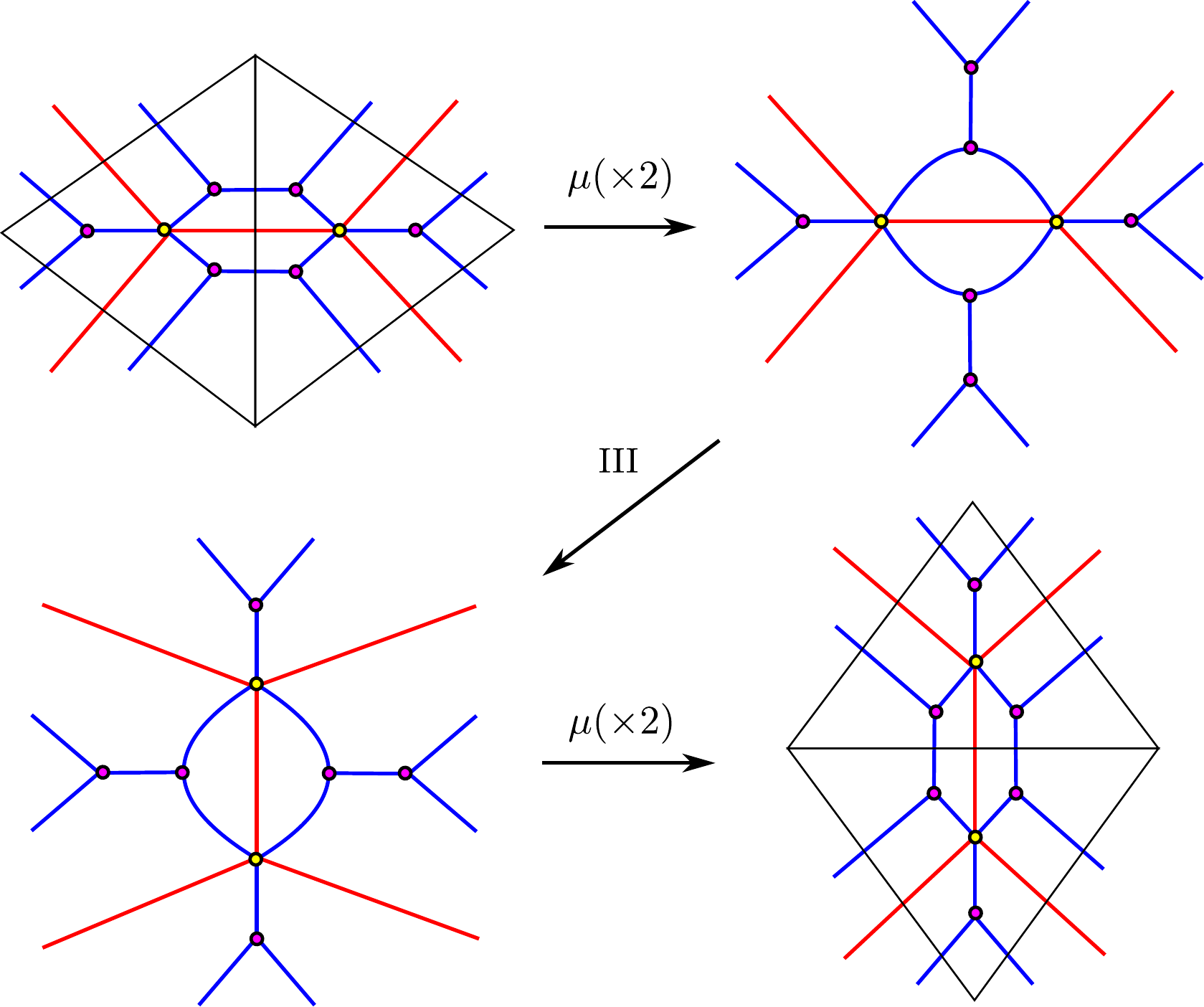}
			\caption{Flip in a 3-triangulation realized as four monochromatic edge mutations. In general, the Legendrian weaves associated to two $N$-triangulations which differ by a flip of the underlying (1-)triangulation differ by a sequence of ${N+1 \choose 3}$ such 2-graph mutations. Note the Move III flop isotopy in-between the two mutation pairs.}
			\label{fig:Flip3Triangulation}
		\end{figure}
	\end{center}
\end{ex}


\subsection{N-graph Realization of Quiver Mutations}\label{ssec:QuiverMutationsGeometric} In this subsection we explain how to use $N$-graphs in order to construct infinitely many Lagrangian fillings for certain Legendrian links in the standard contact 3-sphere $(\S^3,\xi_\st)$. These Lagrangian fillings are distinguished by the microlocal monodromies/cluster coordinates in Subsection \ref{ssec:Microlocal}.

Let $\beta\in\mbox{Br}^+_N$ be a positive braid, with a fixed braid word $w(\beta)$. Consider an $N$-graph $G\sse\D^2$ on the 2-disk such that the labels of the edges of $G$ near $\dd\D^2$, read cyclically, form the word $w(\beta)$. Following Subsection \ref{ssec:ExactLagrangianFillings}, the Lagrangian projection $L(G)=\pi(\iota(\La(G)))\sse(\R^4,\omega_\st)$ of $\iota(\La(G))\sse(\R^5,\xi_\st)$ is an exact Lagrangian filling of the Legendrian link $\La(\beta)\sse(\S^3,\xi_\st)$ associated to the positive braid $\beta$. All the $N$-graphs $G\sse\D^2$ which feature in this subsection will be free, and thus the Lagrangian projections are embedded, equivalently $\La(G)$ has no Reeb chords.

Now consider a free $N$-graph $G\sse\D^2$, $b_1(G):=\mbox{rk}(H_1(\La(G),\Z))$ and a basis
$$B=\{[\gamma_1],\ldots,[\gamma_{b_1(G)}]\}$$
for $H_1(\La(G),\Z)\cong\Z^{b_1(G)}$, equivalently a basis for the first homology group of its Lagrangian projection $L(G)$. For a choice of basis $B$, we denote by $Q(B)$ the intersection quiver of the 1-cycles $\gamma_i$, $i\in[1,b_1(G)]$. The vertices $v_i$ of the quiver $Q(B)$ are in bijection with elements of the homology basis $B$, and the number of arrows between two distinct vertices $v_i,v_j$ is given by the geometric intersection number $|\gamma_i\cap\gamma_j|$. The direction of each arrow is given by the sign of each geometric intersection, and there are no loops, i.e. no edges from a vertex $v_i$ to itself. The quiver obtained by mutation of a quiver $Q$ at the vertex $v_i$ will be denoted $\mu_i(Q)$.

We will study the realization of quiver mutations, algebraic in nature, as Legendrian mutations of free $N$-graphs, which are geometric. Suppose that there exists a subset $B_\mu$ of classes $[\gamma_i]$, $i\in[1,k]$, for some $k\leq b_1(G)$, such that $\gamma_i\in B_\mu$ is represented by a 3-graph cycle with {\it no} multiplicity. That is, each 1-cycle $\gamma_i$ is represented by either a $\sf Y$-cycle, a tree, a monochromatic edge $\sf I$-cycle or a long edge. Let $\{x_1,\ldots,x_k\}$ be the cluster coordinates associated to $\{\gamma_1,\ldots,\gamma_k\}$ via microlocal monodromies, as in Subsection \ref{sssec:ClusterCoord}.

\begin{remark}
In general, this set of cluster coordinates $\{x_1,\ldots,x_k\}$ is only a partial subset of the entire cluster seed $\{x_1,\ldots,x_{b_1(G)}\}$ for $H_1(L(G),\Z)$. The ability to work with a subset is an advantage that allows for our methods to be applied in more generality. From the viewpoint of cluster algebras, the vertices of $Q(B)$ which are {\it not} in $Q(B_\mu)$ are to be considered as frozen vertices, and the variables $\{x_{k+1},\ldots,x_{b_1(G)}\}$ as frozen coordinates.\hfill$\Box$
\end{remark}

By Subsection \ref{ssec:legmutation}, and Lemma \ref{lem:free}, we can perform a Legendrian mutation at $\gamma_i\in B_\mu$ and obtain a free $N$-graph $\mu_i(G)$. The intersection quiver $Q(\mu_i(B))$ associated to the mutated basis $\mu_i(B)$ is the mutated quiver $\mu_i(Q(B))$. The 1-cycle in the mutated graph $\mu_i(G)$ corresponding to $\gamma_j\in B$, under mutation at $\gamma_i$, is denoted by $\mu_i(\gamma_j)$. By Subsection \ref{sssec:ClusterTransf}, the cluster coordinate associated to $\mu_i(\gamma_j)$ is given by the $j$-th coordinate in the cluster transformation of $\{x_1,\ldots,x_{b_1(G)}\}$ at $x_i$. Therefore, the exact Lagrangian filling represented by the free $N$-graph $\mu_i(G)$ has intersection quiver $\mu_i(Q(B))$ and cluster coordinates obtained by mutation of the cluster seed $\{x_1,\ldots,x_{b_1(G)}\}$ for $L(G)$ at $x_i$. In conclusion, if the 1-cycles are represented by trees, performing {\it one} quiver (or cluster seed) mutation as a Legendrian mutation is possible, and the microlocal monodromies after the Legendrian mutation accurately reflect cluster mutation.

\begin{remark}
The challenging aspect of the {\it geometric} side is that iterating this procedure is not necessarily possible, or at least readily accessible. This aspect is {\it not} reflected in the algebra of quiver mutations (or cluster coordinate mutations) since, by definition, two opposite edges between vertices are canceled\footnote{In previous attempts to geometrically iterate Lagrangian mutations, such as \cite[Section 2]{STW}, this obstruction manifests itself as embedded curves becoming immersed upon performing Dehn twists, a problem which presently has no known solution.}.\hfill$\Box$
\end{remark}

The technology of 3-graphs and their mutations, as developed in Subsection \ref{ssec:legmutation}, allows us to iterate Legendrian mutations in an abundance of cases, including arbitrarily high genus. We will illustrate explicit cases in which an infinite sequence of quiver mutations can be realized as an infinite sequence of $N$-graph mutations. These cases can be inserted in (infinitely many) other examples, and the first consequence is the production of new families of Legendrian links with {\it infinitely many} exact Lagrangian fillings:

\begin{thm}\label{thm:ThurstonLinks} Let $\La_{s,t}=\La(\beta_{s,t})\sse(\S^3,\xi_\st)$ be the Legendrian link given by the standard satellite of the positive braid
	$$\beta_{s,t}=(\sigma_1^3\sigma_2)(\sigma_1^3\sigma_2^2)^s\sigma_1^3\sigma_2(\sigma_2^2\sigma_1^3)^t(\sigma_2\sigma_1^3)(\sigma_2^{t+1}\sigma_1^2\sigma_2^{s+2}),\qquad s,t\in\N,s,t\geq1.$$
	Then $\La_{s,t}\sse(\S^3,\xi_\st)$ admits infinitely many embedded exact Lagrangian fillings in $(\D^4,\la_\st)$ realized as $3$-graphs $G_{s,t}\sse\D^2$ and their Legendrian mutations.
\end{thm}

\begin{proof}
The argument is uniform for all $s,t\in\N$ and all the difficulties, and their solutions, are already present for the simplest case.\footnote{We thank Dylan Thurston for useful discussions on quivers and their mutations. In particular, for providing the infinite sequence of mutations that we use in this proof.} Let us thus assume $s=t=1$ for now. First, we need to construct a free 3-graph $G=G_{1,1}$ which represents a Lagrangian filling for the Legendrian link $\La(\beta)$ associated to $\beta=\beta_{1,1}$. This 3-graph is shown in Figure \ref{fig:ThurstonQuiver}:

\begin{center}
	\begin{figure}[h!]
		\centering
		\includegraphics[scale=0.8]{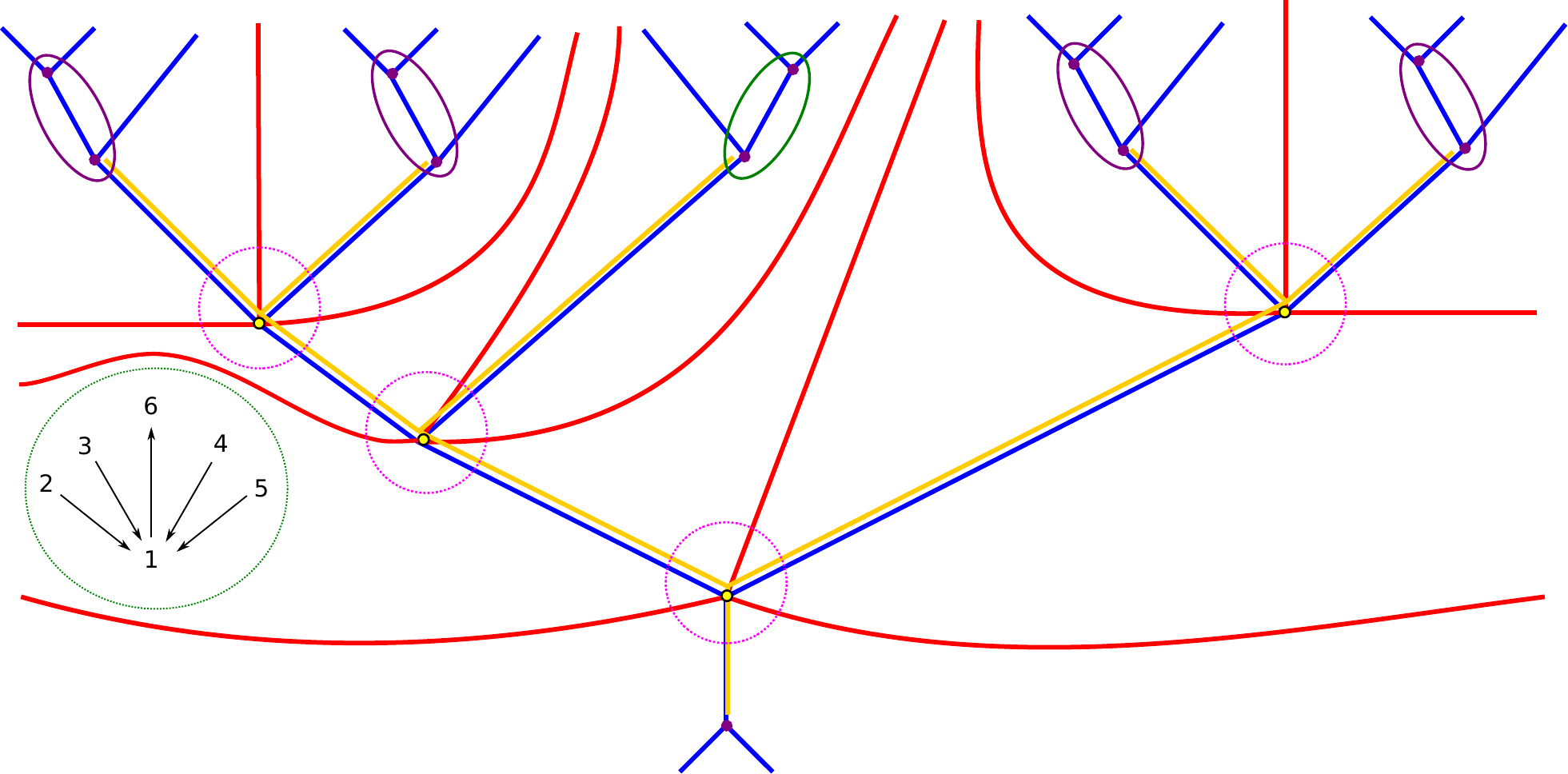}
		\caption{The 3-graph $G$ and the initial Quiver $Q$.}
		\label{fig:ThurstonQuiver}
	\end{figure}
\end{center}

The exact Lagrangian $L(G)$ associated to $G$ is a genus-4 surface with two boundary components, and thus $b_1(G)=9$. Let us consider the subset $B_\mu=\{\gamma_1,\gamma_2,\gamma_3,\gamma_4,\gamma_5\}$ given by the following 1-cycles: $\gamma_1$ is represented by the yellow 1-cycle in Figure \ref{fig:ThurstonQuiver}, which is a tree of $\sf Y$-pieces, and $\gamma_2,\gamma_3,\gamma_4,\gamma_5$ are represented by monochromatic edges, in purple in Figure \ref{fig:ThurstonQuiver}. In addition, we consider the 1-cycle $\gamma_6$ represented by the monochromatic edge, in green. The intersection quiver $Q=Q(B_\mu\cup\{\gamma_6\})$ is given by the quiver drawn in Figure \ref{fig:ThurstonQuiver}. The quiver $Q$ is of infinite mutation type, as it is associated\footnote{Precisely, the quiver $Q$ corresponds to the rank 6 paracompact hyperbolic Coxeter group $\overline{L}_5=[3^{[1,1,1,1,1]}]$.} to hyperbolic Coxeter diagram \cite[Table 1]{Lawson}. In fact, we claim that the sequence of quiver mutations $\mu_{s_n}$, where
$$s_n= \begin{cases}
1 &\mbox{if } n \equiv 1 \\
2 & \mbox{if } n \equiv 2\\
3 & \mbox{if } n \equiv 3\\
4 & \mbox{if } n \equiv 4\\
5 & \mbox{if } n \equiv 5
\end{cases} \pmod{5},$$
is an infinite sequence of quiver mutations. Indeed, each time we apply the sequence of mutations $\mu_5\mu_4\mu_3\mu_2\mu_1$, the number of arrows from the vertex $v_1$ to $v_6$ increases by two, and the number of arrows from $v_i$ to $v_6$, for $i\in[2,5]$ increases by one. In particular, at the $k$th iteration there are $2k+1$ arrows from $v_1$ to $v_6$ and $k$ arrows from $v_i$ to $v_6$, $i\in[2,5]$. Now we reach the core of the issue, which is realizing this infinite sequence of quiver mutations as Legendrian mutations of 3-graphs. For that, we observe the following two properties:

Let us perform a Legendrian mutation along the $\sf Y$-tree which represents the 1-cycle $\gamma_1$. The resulting 3-graph, which is free by Lemma \ref{lem:free}, is depicted in Figure \ref{fig:ThurstonQuiver2}:
\begin{center}
	\begin{figure}[h!]
		\centering
		\includegraphics[scale=0.8]{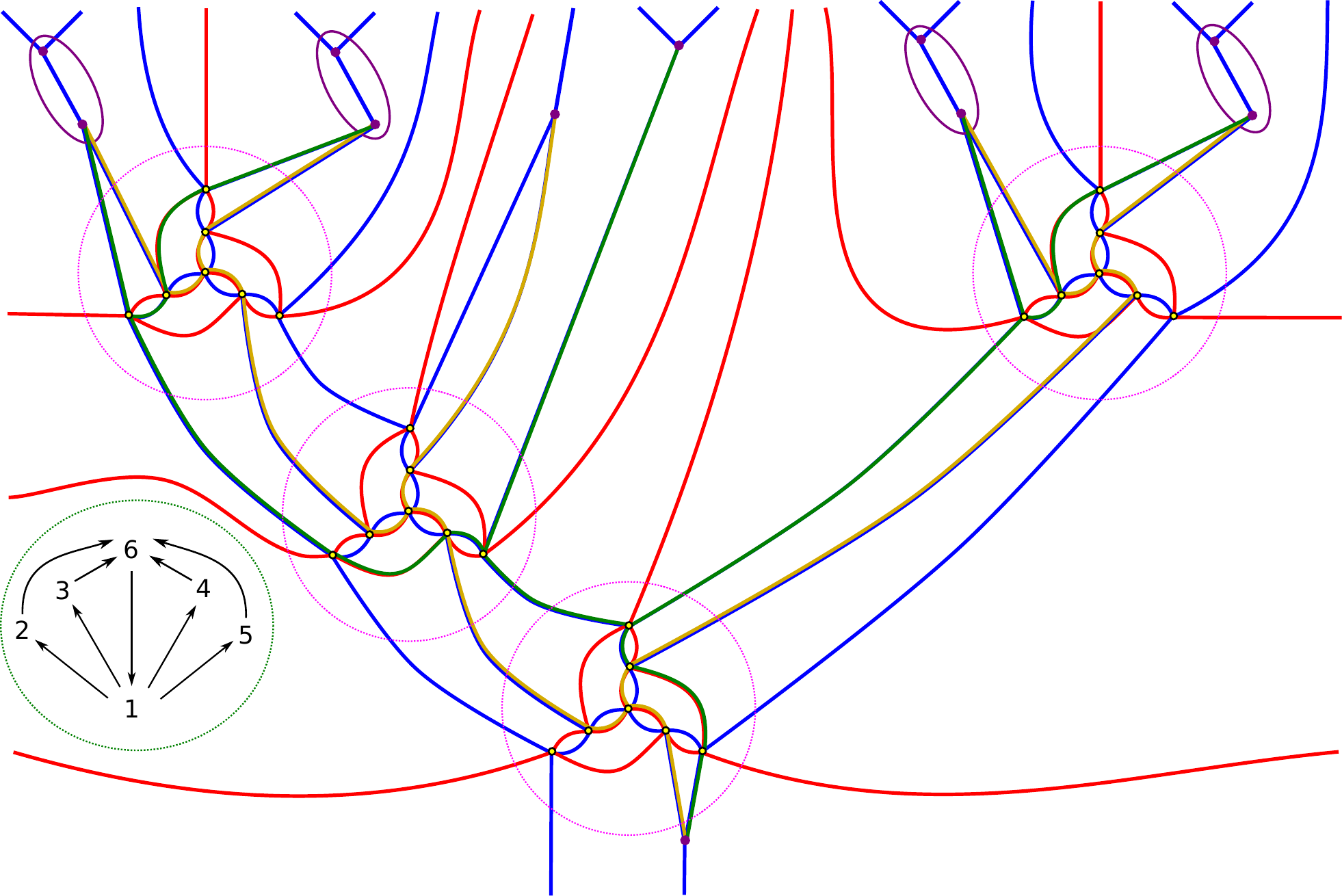}
		\caption{Mutated 3-graph $\mu_1(G)$ at the 1-cycle $\gamma_1$ (yellow) corresponding to vertex 1 in $Q$, and its associated intersection quiver $\mu_1(Q)$.}
		\label{fig:ThurstonQuiver2}
	\end{figure}
\end{center}
Now, upon this Legendrian mutation at $\gamma_1$ the 1-cycles $\gamma_2,\gamma_3,\gamma_4,\gamma_5$ are {\it still} represented by monochromatic edges. These new 1-cycles $\mu_1(\gamma_2),\mu_1(\gamma_3),\mu_1(\gamma_4),\mu_1(\gamma_5)$ are circled in purple in Figure \ref{fig:ThurstonQuiver2}. The figure also displays the mutated quiver $\mu_i(Q(B_\mu\cup\{\gamma_6\}))$ and the cycle $\gamma_6$ in green. Similarly, upon this Legendrian mutation, the 1-cycle $\mu_1(\gamma_1)$ is {\it still} represented by an embedded $\sf Y$-tree, as depicted in yellow in Figure \ref{fig:ThurstonQuiver2}.

These properties hold true as we now perform Legendrian mutations at the monochromatic edges $\gamma_2,\gamma_3,\gamma_4,\gamma_5$. The free 3-graph resulting from these four mutations is drawn in Figure \ref{fig:ThurstonQuiver3}:

\begin{center}
	\begin{figure}[h!]
		\centering
		\includegraphics[scale=0.8]{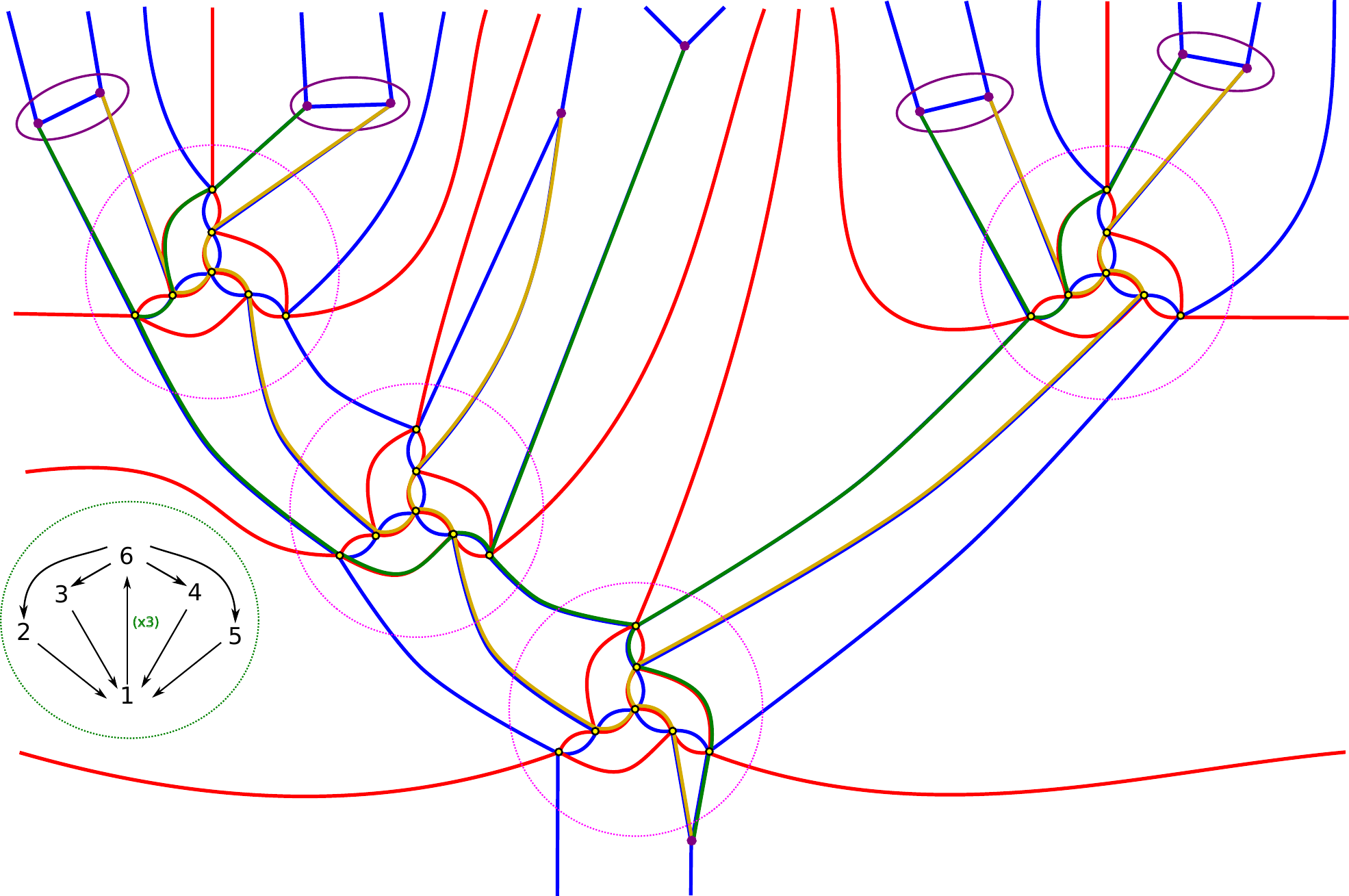}
		\caption{Mutated 3-graph $\mu_5\mu_4\mu_3\mu_2\mu_1(G)$ and its associated intersection quiver $\mu_5\mu_4\mu_3\mu_2\mu_1(Q)$.}
		\label{fig:ThurstonQuiver3}
	\end{figure}
\end{center}

The claim is that we can iterate the sequence of mutations $\mu_5\mu_4\mu_3\mu_2\mu_1$ geometrically as Legendrian mutation of the free 3-graph, and these two properties hold. That is, at any stage in the sequence of mutations $\mu_{s_n}$ we have that

\begin{itemize}
	\item[(i)] the 1-cycles $\gamma_2,\gamma_3,\gamma_4,\gamma_5$ are represented by monochromatic edges,
	\item[(ii)] the 1-cycle $\gamma_1$ is represented by an embedded $\sf Y$-tree, with no multiplicities.
\end{itemize}

In fact, the $\sf Y$-tree representing $\gamma_1$ always has exactly four $\sf Y$-pieces. These four pieces have been surrounded by a dashed pink circle in Figures \ref{fig:ThurstonQuiver} through \ref{fig:ThurstonQuiver4}. The two items above can be readily verified, as follows. The behavior of the mutated $3$-graph near each of the the monochromatic edges is as depicted in Figure \ref{fig:ThurstonQuiverNearMonochromatic}:

\begin{center}
	\begin{figure}[h!]
		\centering
		\includegraphics[scale=0.8]{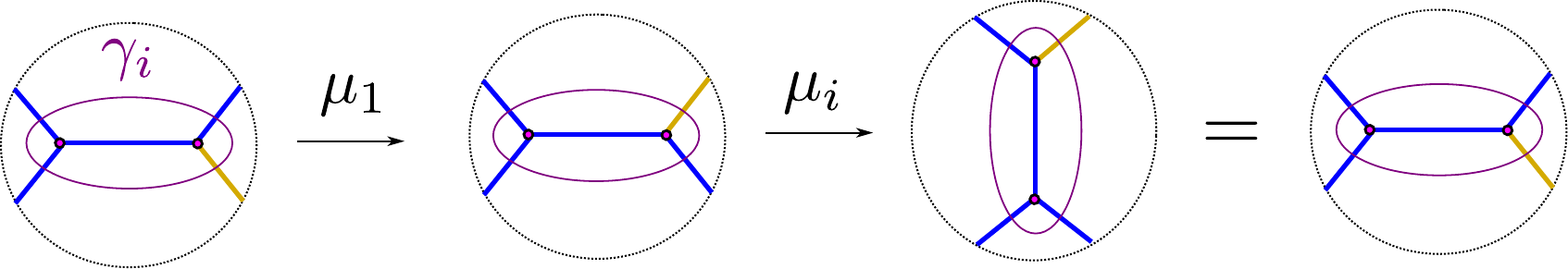}
		\caption{The effect of the sequence of mutations $\mu_5\mu_4\mu_3\mu_2\mu_1$ near the monochromatic edges $\gamma_i$, $i\in[2,5]$. The only mutations from these five that affect $\gamma_i$ are $\mu_1$ and $\mu_i$.}
		\label{fig:ThurstonQuiverNearMonochromatic}
	\end{figure}
\end{center}

It thus follows that $\gamma_i$, $i\in[2,5]$, remains a monochromatic edge upon {\it any} iteration of the $3$-graph mutation $\mu_5\mu_4\mu_3\mu_2\mu_1$. Similarly, according to the Legendrian mutation rules of Subsection \ref{ssec:legmutation}, each $\sf Y$-piece of the $\sf Y$-tree representing $\gamma_1$ itself mutates to a $\sf Y$-piece, and mutating at $\gamma_2,\gamma_3,\gamma_4,\gamma_5$ preserves this property. Thus the pattern persists upon any iteration. The two properties $(i)$ and $(ii)$ now allow us to perform the sequence of mutations $\mu_{s_n}$ up to {\it any} point in the sequence. For instance, the sequence of mutations $\mu_1\mu_5\mu_4\mu_3\mu_2\mu_1$ applied to $G$ lead to the 3-graph in Figure \ref{fig:ThurstonQuiver4}:

\begin{center}
	\begin{figure}[h!]
		\centering
		\includegraphics[scale=0.85]{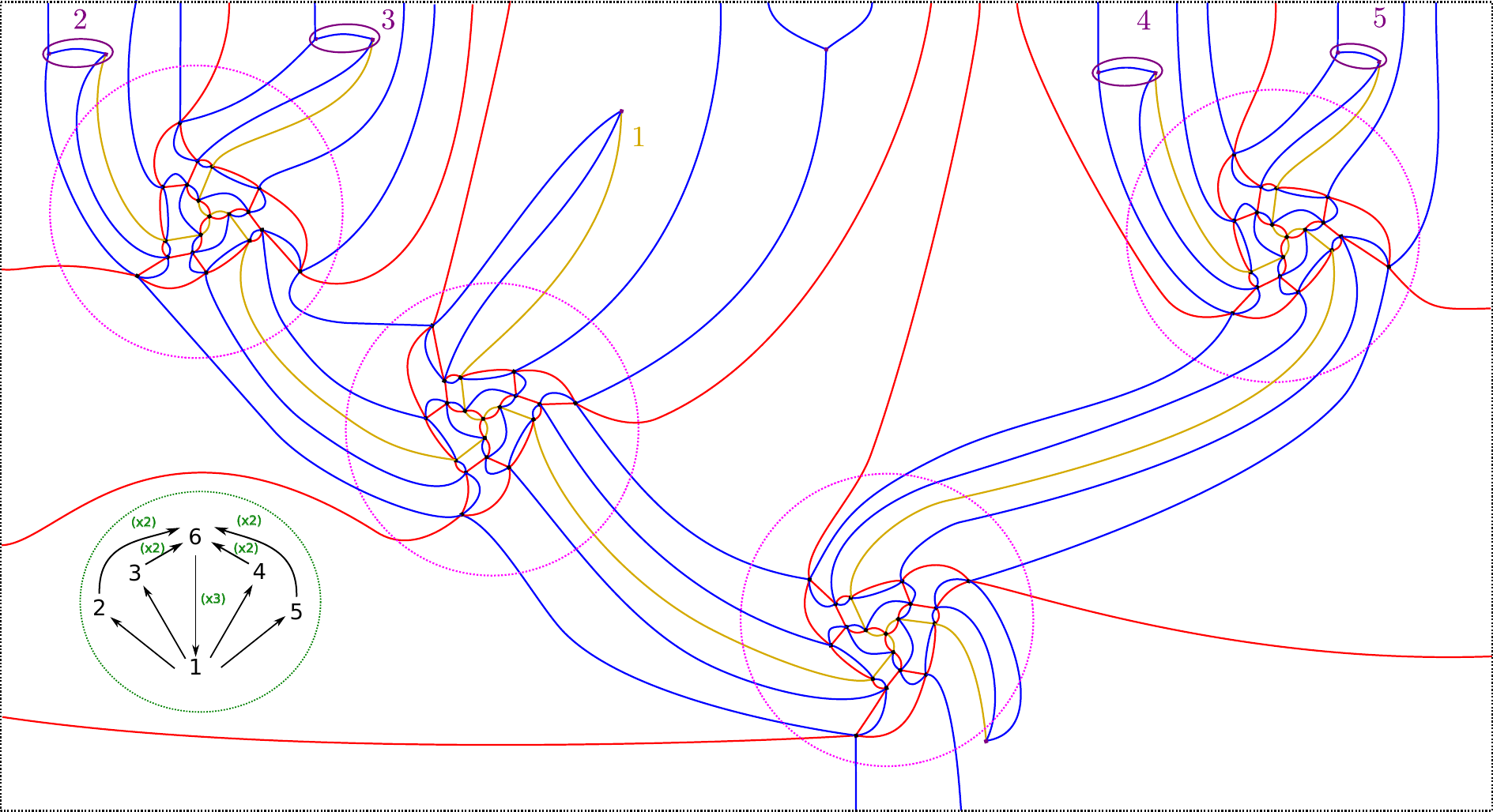}
		\caption{Mutated 3-graph $\mu_1\mu_5\mu_4\mu_3\mu_2\mu_1(G)$ and its associated intersection quiver $\mu_1\mu_5\mu_4\mu_3\mu_2\mu_1(Q)$.}
		\label{fig:ThurstonQuiver4}
	\end{figure}
\end{center}

In order to pairwise distinguish the exact Lagrangian fillings associated to the sequence of $3$-graphs $(\mu_{s_n}\mu_{s_{n-1}}\cdots\mu_{1})(G)$, up to Hamiltonian isotopy, we use the microlocal monodromies $\{x_1,x_2,x_3,x_4,x_5,x_6\}$ along the 1-cycles $\gamma_i$, $i\in[1,6]$, and their mutations. By Subsection \ref{sssec:ClusterTransf}, the cluster seed $\{x_1,x_2,x_3,x_4,x_5,x_6\}$ associated to the quiver $Q$ mutates to the cluster seed associated to $(\mu_{s_n}\mu_{s_{n-1}}\cdots\mu_{1})(Q)$ upon performing the Legendrian 3-graph mutations $(\mu_{s_n}\mu_{s_{n-1}}\cdots\mu_{1})(G)$. Since the quivers $(\mu_{s_n}\mu_{s_{n-1}}\cdots\mu_{1})(Q)$ are distinct, and so are the associated cluster seeds, it follows that the associated Lagrangian fillings are distinct. This concludes the proof for $s=t=1$.

The general case $s,t\in\N$ is proven with the same argument. Indeed, the free 3-graph $G_{1,1}$ in Figure \ref{fig:ThurstonQuiver} generalizes to a 3-graph whose boundary is $\beta_{s,t}$, just by adding $s$ copies of the leftmost pattern in $G_{1,1}$, to the left, and $t$ copies of the rightmost pattern in $G_{1,1}$, to the right. In this general case, it is still true that $\gamma_1$ is represented\footnote{In this case the $\sf Y$-tree has ($s+t+2$) $\sf Y$-pieces, $s+1$ to the left and $t+1$ to the right of the base root.} by a $\sf Y$-tree and the remaining $\{\gamma_2,\gamma_3,\ldots,\gamma_s,\gamma_{s+1},\ldots,\gamma_{s+t+3}\}$ cycles are represented by monochromatic edges. The argument is then identical, with the infinite sequence of mutations given by
$$s_n= i,\qquad \mbox{if }n \equiv i\pmod{s+t+3},\quad 1\leq i\leq s+t+3.$$
The reader can directly verify that this is an infinite sequence of mutations, as the multiplicity of the arrows to the cycle $\gamma_{s+t+4}$ -- generalizing the green cycle $\gamma_6$ in Figure \ref{fig:ThurstonQuiver} -- increases as we apply the mutations $\mu_{s+t+3}\mu_{s+t+2}\cdots\mu_{2}\mu_1$.
\end{proof}

\begin{remark}
For $s=t=1$, note that the sequence $\mu_{s_n}$ never mutates at the 1-cycle $\gamma_6$, i.e. at the sixth vertex $v_6$ in $Q$. It is nevertheless crucial to include $\gamma_6$ in the quiver as well as the cluster variable $x_6$, with its subsequent mutations. Note that the 1-cycle $\gamma_6$ is initially represented by an embedded curve in the 3-graph, but this curve develops immersed points as we iterate the sequence of mutations $\mu_{s_n}$ according to Subsection \ref{ssec:legmutation}. This still allows us to define the cluster coordinate associated to it {\it but} we would not be able to mutate along such a 1-cycle just with the rules developed in Subsection \ref{ssec:legmutation}. (This is just a side remark, since the argument for Theorem \ref{thm:ThurstonLinks} does {\it not} require mutating at $v_6$.)\hfill$\Box$
\end{remark}

The Legendrian links in Theorem \ref{thm:ThurstonLinks} are relatively simple. For instance, the Legendrian knot associated to $\beta_{1,1}$ is genus-4 two-component link. One of the components is an unknot and the other is the $(2,7)$-torus knot $7_1$. Note that $\La(\beta_{1,1})$ is (smoothly) distinct from the $(3,6)$-torus link that the first author studied in \cite{CasalsHonghao}, which also has genus-4. Thus, not only does Theorem \ref{thm:ThurstonLinks} bring a {\it new} method to construct infinitely many Lagrangian fillings, but it in fact provides new Legendrian links with infinitely many Lagrangian fillings.

\begin{remark} Note that the $\overline{L}_5$ quiver that we used in Theorem \ref{thm:ThurstonLinks} appears as a subquiver of the intersection quiver for several other positive braids. Following L. Lewark's positive braid table\footnote{Lukas Lewark's Positive Knots Table: \href{http://lewark.de/lukas/braids.html}{Braids and Trees} at ``{\it http://lewark.de/lukas/braids.html} ''.} each of the following positive genus-6 braids, $14n_{5644}$, $15n_{118169}$, $16n_{144958}$, $16n_{149517}$, $16n_{173894}$, $16n_{175324}$ and $16n_{339638}$, to name a few, contain $\overline{L}_5$ in their intersection quiver. We believe that an argument similar to Theorem \ref{thm:ThurstonLinks} should prove that the maximal-tb representative of each of these links has infinitely many exact Lagrangian fillings.\hfill$\Box$
\end{remark}

Finally, the contrast between Theorem \ref{thm:ThurstonLinks} and \cite[Corollary 1.5]{CasalsHonghao} is interesting. The former constructs an infinite family of Lagrangian fillings for a Legendrian link by directly using Legendrian mutations, which are themselves distinguished by their effect -- as cluster mutations -- on the microlocal monodromies. The latter result \cite{CasalsHonghao} is entirely about constructing infinite order Lagrangian concordances, coming from Legendrian loops of positive braids, and the infinite family of Lagrangian fillings is a byproduct of such construction. In particular, $N$-graph calculus should apply to much more general Legendrian links, and does not require knowing about the existence of an infinite order element in their Lagrangian concordance monoid.


\section{Moduli Space for $N$-triangles and Non-Abelianization}\label{sec:app3}

In this final section, we focus on $N$-graphs associated to $N$-triangulations, as introduced in Section \ref{sec:constr}. This class of $N$-graphs $G$ yields Legendrian weaves $\La(G)$ whose Lagrangian projections are related to the Goncharov-Kenyon conjugate Lagrangian surfaces \cite{GoncharovKenyon13,STWZ}. These Lagrangian surfaces have also appeared in the context of Gaiotto-Moore-Neitzke's spectral networks \cite{GMN_SpecNet13,GMN_Cluster}. In particular, we prove Theorem \ref{thm:FlagModuli_Ntriangle_Intro}, which computes the flag moduli space $\SM(G)$ for $G$ any $N$-triangle $t_N$, matching the algebraic results in \cite[Section 8]{GMN_SpecNetSnakes14} and \cite[Section 9]{FockGoncharov_ModuliLocSys}.

\subsection{Flag moduli space of the $N$-triangle}\label{ssec:FlagModuli_Ntriangle} Let us compute the flag moduli space associated to the $N$-graph $G(t_N)$ of an $N$-triangle $t_N$, as we defined in Section \ref{ssec:Ngraph_Ntriangle} (see Figure \ref{fig:NTriangles}). The result reads as follows:

\begin{thm}\label{thm:FlagModuli_Ntriangle}
	Let $G(t_N)$ be the $N$-graph associated to an $N$-triangle $t_N$. The flag moduli space of $G(t_N)$ is a ${N-1\choose 2}$-dimensional complex torus, i.e.
	$$\SM(t_N,G(t_N);k)\cong(k^*)^{{N-1\choose 2}}.$$
\end{thm}

This rest of this subsection is devoted to the proof of Theorem \ref{thm:FlagModuli_Ntriangle}. The statement of Theorem \ref{thm:FlagModuli_Ntriangle} is an instance of how incidence geometry problems connect to the contact topology of Legendrian surfaces. Indeed, although our proof is entirely within projective geometry, the conclusion from Theorem \ref{thm:FlagModuli_Ntriangle} ought to be read as the fact that the moduli space $\SM(t_N,G(t_N);k)$ is parametrized by the toric coordinates provided by the holonomies $\mbox{Hom}(H_1(\La(G(t_N),\Z)),k^*)$. For $k=\C$, this complex torus should be related to the complex torus appearing in Fock-Goncharov \cite{FockGoncharov_ModuliLocSys,FockGoncharovII} in their study of cluster varieties, see \cite[Theorem 8.3]{Kuwagaki20}.

Theorem \ref{thm:FlagModuli_Ntriangle} can also be interpreted as follows. The triangle $t_N$ is topologically a disk $\D^2$ with boundary a circle $\dd\D^2=\S^1$.  The Legendrian weave
	$$\Lambda(G(t_N)) \subset (J^1(\D^2),\xi_\st)$$
	has a Lagrangian projection $L:= \pi(\La(G(t_N)))$, which is an exact Lagrangian submanifold, where $\pi:J^1(\D^2)\lr T^*\D^2$ is the projection along the standard (vertical) Reeb flow.  The Lagrangian $L$ has boundary in $T^*\D^2\vert_{S^1} \cong J^1(S^1),$ and it is checked that $\partial L$ is the cylindrical Legendrian braid $\Delta^3,$ where $\Delta$ is the half-twist positive braid corresponding to a longest word in the Weyl group, i.e.~the Garside element.  Since $G(t_N)$ is free, $L$ is an embedded exact Lagrangian filling of $\partial L$. Now, by looking at the boundary circle $\S^1$ and considering the moduli space \`a la \cite{STZ_ConstrSheaves}, we conclude that the moduli space of Lagrangian fillings should carry a cluster structure: the flag moduli space $\SM(t_N,G(t_N);k)$ is {\it one} such chart.
	
	In fact, by an argument akin to Lemma \ref{lem:hexagonal}, the flags at two vertices of the triangle $t_N$ determine the flags along the edge they bound, and therefore the flags along the boundary circle $\d\D^2=\S^1$ must be determined by the flags at the vertices, themselves three mutually completely transverse flags in the flag variety $\cB$.  This space of triples of mutually transverse flags is one of the Richardson varieties\footnote{We thank Ian Le for many discussions on the Richardson variety.} $\cR$.  Now, by the $\PGL_N$ action, two totally transverse flags can be put in standard position $B$, $B^-$, with residual symmetry the Cartan $H$ of diagonal matrices up to scale.  Then the moduli space $\cR/H$ is a cluster variety and the exact Lagrangian filling $L$ provides a cluster chart via its moduli of local systems
	$$Loc(L) \cong \mbox{Hom}(H_1(\La(G(t_N),\Z)),k^*) \cong (k^*)^{{N-1\choose 2}}.$$
	This torus can be checked to agree with that of \ref{thm:FlagModuli_Ntriangle}. Note also that, following Section \ref{sec:app2} many other cluster charts and exact Lagrangian fillings can be found by performing $N$-graph mutations. Let us now prove our result:

\begin{proof}[Proof of Theorem \ref{thm:FlagModuli_Ntriangle}]
	Let us argue by induction on $N$, where the base case $N=2$ follows from the fact that $\PGL_2$ acts transitively on triples of distinct points. Let us assume that $$\SM(t_N,G(t_N);k)\cong(k^*)^{{N-1\choose 2}}$$
	for the $N$-graph of an $N$-triangle. Consider an $(N+1)$-triangle with one side being an arbitrary fixed preferred base, and thus the row associated to this base contains $2N-3$ triangles. It is combinatorially apparent that the complement of this row in $t_{N+1}$ is in fact an $N$-triangle $t_N$, and thus we can construct $t_{N+1}$ by adding such base row to $t_{N}$. This combinatorial splitting is translated into a containment of an $N$-graph $G(t_N)$ within $G(t_{N+1})$. Let us describe such splitting in the $(N+1)$-graph by providing its construction starting from the $N$-graph $G(t_N)$.
	
	Start with the $N$-graph $G(t_N)$ -- see Section \ref{ssec:Ngraph_Ntriangle} -- and consider the ${N \choose 2}$ edges intersecting the base side of $t_N$. The edges are depicted vertically and the base side horizontally -- see Figure \ref{fig:NTriangles}. These are $\tau_i$-edges, $i=1,\ldots,N-1$, with exactly $(N-1-k)$ $\tau_k$-edges. The $(N+1)$-graph $G(t_{N+1})$ can be described in the following $N$ stages:
	
	\begin{itemize}
		\item[1.] First, insert an $(N-1,N)$-hexagonal point in the unique $\tau_{N-1}$ edge in the base side of $G(t_N)$. The $\tau_N$-edge aligned with the previously existing $\tau_{N-1}$-edge is continued down vertically. The remaining two $\tau_N$-edges are extended horizontally, respectively to the left and to the right, and the remaining two $\tau_{N-1}$ edges are continued down diagonally, in south-east and south-west direction respectively.\\
		
		\item[2.] Second, continue down the $\tau_{i}$-edges, $i=1,\ldots,N$, until the two $\tau_{N-1}$-edges intersect with the two originally existing $\tau{N-2}$-edges. In the moment of collision, insert two $(N-2,N-1)$-hexagonal vertices at the intersection point matching the two incoming $\tau_{N-1}$ and $\tau{N-2}$ trajectories. We extend the two $\tau_{N-1}$-edges adjacent to the incoming $\tau_{N-2}$-edges horizontally to the left and to the right.
		
		The remaining two pairs of three edges, each with two $\tau_{N-2}$-edges and a $\tau_{N-2}$-edge, are continued down, with the $\tau_{N-1}$-edges continued vertically and the $\tau_{N-2}$-edges continued diagonally in the south-east or south-west directions, accordingly.\\
		
		\item[3.] Iteratively, we proceed as follows in the $l$th stage, $2\leq l\leq N-1$. We continue down the $\tau_{i}$-edges, $i=1,\ldots,N$, and at this stage the only edges being continued diagonally down are $\tau_{N-l+1}$-edges. There are $2l-2$ of such edges, which can be gathered in two groups, internal and external.\\
		
		\noindent By definition, there are two external edges, which are the leftmost and rightmost $\tau_{N-l+1}$-edges, respectively continuing south-west and south-east. For these two external edges, we insert two $(N-l,N-l+1)$-hexagonal and describe the $N$-graph as described in Stage 2. The $2l-4$ internal $\tau_{N-l+1}$-edges, which continue down diagonally, ought to intersect with $\tau_{N-l}$-edges, which continue down vertically.\\
		
		\noindent For the $2(l-2)$ internal edges, there are $(l-2)$ such intersections, since an intersection occurs for each pair. For each such an intersection, insert a $(N-l,N-l+1)$-hexagonal vertex, and continue the outgoing three edges down as described by the local model for the hexagonal vertex. Hence, for each of these hexagonal vertices, the outgoing $\tau_{N-l+1}$-edge continues vertically down whereas the two $\tau_{N-l}$-edges continue down diagonally. Thus, at the $l$th stage we have inserted exactly $l$ $(N-l,N-l+1)$-hexagonal points.\\
		
		\item[4.] In the $N$th stage, all $\tau_i$-edges, $i\leq 2\leq N$ continue down vertically and we are left with $2(N-1)$ $\tau_1$-edges continuing diagonally. In line with the previous stages, there are two external $\tau_1$-edges and $2(N-2)$ internal edges. Insert two $\tau_1$-trivalent vertices at the end of the two external $\tau_1$-edges. The internal edges will meet in consecutive pairs at $N-2$ intersection points. In this final stage we insert a $\tau_1$-vertex in each of these intersection points, and continue the remaining $\tau_1$-edge vertically down.\\
	\end{itemize}
	
	
	Let us now compute the flag moduli space $\SM(t_N,G(t_{N+1});k)$ using this inductive construction of $G(t_{N+1})$. A crucial fact to be used is Lemma \ref{lem:hexagonal}, i.e. at a hexagonal vertex, four consecutive flags uniquely determine the remaining two flags. Let us assume that we have chosen a point in $\SM(t_N,G(t_{N});k)$ and we thus have the data of a flag $\SF$ in $\P^{N}$ for each open region in $\D^2\setminus G(t_{N})$. This data needs to be considered in the moduli space of flags, given that $G(t_{N+1})$ is an $(N+1)$-graph, and thus we fix an embedding $i_0:\P^N\lr\P^{N+1}$ and the corresponding inclusion $\PGL_N\sse\PGL_{N+1}$. Let us then start the construction $G(t_{N+1})$ from $G(t_{N})$ by stages, as described above, and prove the statement in Theorem \ref{thm:FlagModuli_Ntriangle}.
	
	In the first stage, the flag data at the inserted $(N-1,N)$-hexagonal point in the $\tau_{N-1}$-edge is uniquely determined by a choice of a codimension-2 projective subspace $H^2$ in $\P^{N}$, transverse to $i_0(\P^N)$. Note that the intersection of $H^2$ and $i_0(\P^N)$ is uniquely determined by the flag data coming from $\SM(t_N,G(t_{N});k)$. We claim that this choice in the first stage can be absorbed by the symmetry group $\PGL_{N+1}$.
	
	In order to understand the symmetry group, it is convenient to represent an element in $\PGL_{N+1}$ via the projective matrix
	
	\renewcommand{\arraystretch}{2}$$\newcommand*{\temp}{\multicolumn{1}{r|}{}}\left[\begin{array}{cccccc}
	a_{1,1} & a_{1,2} & \ldots & a_{1,N} &\temp & a_{1,N+1}\\
	a_{2,1} & a_{2,2} & \ldots & a_{2,N} &\temp & a_{2,N+1}\\
	\vdots & \vdots & \ddots & \vdots &\temp & \vdots\\
	a_{N,1} & a_{N,2} & \ldots & a_{N,N} &\temp & a_{N,N+1}\\
	\cline{1-6}
	a_{N+1,1} & a_{N+1,2} & \ldots & a_{N+1,N} &\temp & a_{N+1,N+1}\\\end{array}\right],$$
	
	where the subgroup $\PGL_N\sse\PGL_{N+1}$ is defined by
	$$\PGL_N=\{A\in\PGL_{N+1}:a_{N+1,N+1}=1,a_{i,N+1}=a_{N+1,i}=0,1\leq i\leq N\}.$$
	
	In these coordinates, we can assume that the subgroup $K\sse\PGL_{N+1}$ fixing our fixed hyperplane $i_0(\P^n)\sse\P^{N-1}$ is cut out by the equations
	$$K:=\{A\in\PGL_{N+1}:a_{N+1,i}=0,1\leq i\leq N\}.$$
	
	As a result, the remaining $\PGL_{N+1}$-symmetries (once the flag moduli space $\SM(t_N,G(t_{N});k)$ is fixed, and thus the symmetries of $\PGL_N$ have been used) consist of projective transformations of the form
	
	\renewcommand{\arraystretch}{2}$$\newcommand*{\temp}{\multicolumn{1}{r|}{}}\left[\begin{array}{cccccc}
	a_{1,1} & a_{1,2} & \ldots & a_{1,N} &\temp & c_1\\
	a_{2,1} & a_{2,2} & \ldots & a_{2,N} &\temp & c_2\\
	\vdots & \vdots & \ddots & \vdots &\temp & \vdots\\
	a_{N,1} & a_{N,2} & \ldots & a_{N,N} &\temp & c_N\\
	\cline{1-6}
	0 & 0 & \ldots & 0 &\temp & c_{N+1}\\
	\end{array}\right],$$
	where $a_{i,j}$ are fixed, $1\leq i,j\leq N$, $c_i\in k$, $1\leq i\leq N$, and $c_{N+1}\in k^*$ are free. Then, in this coordinate system, we can assume that the choice of the codimension-2 projective subspace $H^2$ uses the gauge provided by $c_1,c_2\in k$.
	
	In the second stage, two $(N-2,N-1)$-hexagonal vertices are inserted. For each of them, the flag data is fixed by induction in three out of the six regions near the hexagonal vertex. Hence, there is a choice of a codimension-3 projective subspace $H^3$ in each of these two vertices. Let us fix one of these choices by using the free coordinate $c_3\in k$ and notice that the other choice has an a priori moduli of $k$. Nevertheless, the $\tau_{N-3}$-edge that interacts with the $\tau_{N-2}$ edges in the third stage forces that moduli to be $k^*$, since the two newly chosen flags must be $\tau_{N-3}$-transverse. Thus in the second stage we have used the symmetry provided by $c_3\in k$ and we are left with a $k^*$ contribution to the flag moduli.
	
	In the $l$th stage, $3\leq l\leq N-1$, we proceed inductively as follows. We partition the $(N-l,N-l+1)$-hexagonal vertices inserted in this stage into two groups: external, containing two of them, and internal, containing $(l-2)$ of them. By definition, the two external $(N-l,N-l+1)$-hexagonal vertices are the leftmost and rightmost vertices. Each of these two external vertices have flags fixed in three out of the six regions, by the process in the $(l-1)$st stage. Thus, as in the second stage, there is exactly one choice of flag at each of these $(N-l,N-l+1)$-hexagonal vertices which determines each of their respective neighborhoods. This corresponds to a choice of codimension-$l$ projective subspace $H^{l+1}$ in accordance with the incidence conditions imposed by the given flags. Proceeding as in the second stage, we fix one of these choices with the free variable $c_{l+1}$ and the remaining choice contributes $k^*$ to the flag moduli.
	
	The $(l-2)$ internal $(N-l,N-l+1)$-hexagonal vertices have flags fixed in four out of the six regions, given the process in the $(l-1)$st stage. By Lemma \ref{lem:hexagonal}, these hexagonal vertices are uniquely determined in their neighborhoods. Thus, although the $k^*$ contribution of one of the external hexagonal vertices interacts with an internal vertex, no contributions to the flag moduli space come directly from the internal vertices.
	
	The argument then develops iteratively in the above manner until the $(N-1)$st stage is completed. The last $N$th stage consists of the insertion of $N$ $\tau_1$-trivalent vertices. Following the same pattern as before, only the two external trivalent vertices contribute to the flag moduli, since each of the internal trivalent vertices have their three surrounding flags determined at the $(N-1)$st stage. In this last stage, the variables $c_i$, $1\leq i\leq N$ have been fixed and the only remaining degree of free symmetry is $c_{N+1}\in k^*$. Let us use such symmetry to fix the choice in one of the two external trivalent vertices, and thus the contributions of this last stage to the flag moduli space is the $k^*$ choice of the remaining point coming from the remaining external trivalent vertex.
	
	The conclusion in the statement Theorem \ref{thm:FlagModuli_Ntriangle} now follows by gathering the contributions of the flag moduli space at each stage. Indeed, the first stage has no contribution, whereas each of the $(N-1)$ stages, from the second to the last $N$th stage, has a $k^*$ flag moduli space contribution. By the inductive hypothesis, the desired flag moduli space is
	
	$$\SM(t_{N+1},G(t_{N+1});k)\cong \SM(t_N,G(t_N);k)\times(k^*)^{N-1}\cong(k^*)^{{N-1\choose 2}}\times(k^*)^{N-1}\cong(k^*)^{{N\choose 2}},$$
	
	which corresponds to the statement, as required.
\end{proof}

\begin{remark}
	Note that the inductive combinatorial description of $G(t_{N+1})$ in terms of $G(t_N)$ used in the proof of Theorem \ref{thm:FlagModuli_Ntriangle} can be used to provide a third alternative definition of the local $N$-graph $G(t_N)$, in addition to the descriptions introduced in Subsection \ref{ssec:Ngraph_Ntriangle}.\hfill$\Box$
\end{remark}

\subsubsection{Tetrahedral Triangulations at $N=3$ and $N=4$} Let us study the Legendrian weaves $\La(G(\tau))$ and flag moduli space $\SM(\S^2,G(\tau))$ associated to 3- and 4-graphs $G(\tau)$ for the tetrahedral 3- and 4- triangulations $\tau$ of the 2-sphere $\S^2$. The case $N=2$ has been discussed in Subsection \ref{ssec:FirstComputations} above, where $\La(G)\cong\bT^2_c$ is the Legendrian Clifford Torus and $\SM(G)\cong\P^1\setminus\{0,1,\infty\}$. Let us denote the pair of pants $\P^1\setminus\{0,1,\infty\}$ by $\bH$.

Let us consider the 3-graph $G^{(3)}=G(\tau^{(3)})\sse\S^2$ associated to the tetrahedral 3-triangulation $\tau^{(3)}$ of the 2-sphere $\S^2$, according to the construction in Section \ref{sec:constr}. We want to compute its flag moduli space $\SM(G)$. This will be done directly by using the $N$-graph calculus computations in Section \ref{sec:moves}. Indeed, it is proven in Subsection \ref{ssec:RealLifeExample} that in this case the (satellite of the) Legendrian weave $\La(G)$ is Legendrian isotopic to the four-fold connected sum of the Clifford torus $\bT_c^2$. Hence, we obtain that $\SM(\S^2,G^{(3)})\cong\bH^4$. From the description in Theorem \ref{thm:FlagModuli_Ntriangle}, we are also giving a contact geometric proof of the following

\begin{cor}[\cite{FockGoncharov_ModuliLocSys}]
	The moduli of four generic flags in $\C^3$ is isomorphic to $\bH^4$.\hfill$\Box$
\end{cor}

The same argument, using $N$-graph calculus also allows us to study the flag moduli space $\SM(\S^2,G^{(4)})$, where $G^{(4)}=G(\tau^{(4)})\sse\S^2$ is the 4-graph associated to the tetrahedral 4-triangulation $\tau^{(4)}$ of the 2-sphere $\S^2$. It is left as an exercise for the reader to use Theorem \ref{thm:FlagModuli_Ntriangle} and conclude that $\SM(\S^2,G^{(4)})$ is isomorphic, as an algebraic variety, to
$$\SM(\S^2,G^{(4)})=\{(z_1,w_1,\ldots,z_5,w_5)\in(\C^*)^9:(1-\kappa)w_iz_i-z_i+1=0,1\leq i\leq5\}\times(\bH)^4,$$
where $\kappa=1-z_1z_2z_3z_4z_5\in\C^*$. The exercise is solved in \cite[Section 6.3.2]{DimofteGabellaGoncharov} in the language of the $3d$ $N=2$ superconformal field theory $T_4[\Delta,\Pi]$.


\subsection{A Computation of the Non-Abelianization Map}\label{ssec:Nonabelianization} We conclude the main body of the manuscript by exploring the relationship between Legendrian weaves and the works \cite{FockGoncharov_ModuliLocSys,AV1,AV2,Palesi15} in some explicit examples. In particular, we present a case in which the non-Abelianization map featured in \cite{GMN_SpecNet13,GMN_SpecNetSnakes14} can be realized by the microlocal monodromies associated to constructible sheaves microlocally supported along Legendrian weaves.

The context is described as follows. Let $(C,\tau_N)$ be a polygon endowed with an ideal $N$-triangulation $\tau_N$, and  choose a wavefront for $\La(G(\tau_N))$ with no Reeb chords, such that the Lagrangian projection is a smooth exact Lagrangian $L$ embedded in the cotangent bundle $(T^*C,\la_\st)$. This Lagrangian projection $L$ has a sheaf quantization \cite{Sheaves3} to a rank-$N$ sheaf on $C$ with no singular support, i.e. a local system in $C$. Now, of course, all local systems on polygons are trivial, {\it but} the crucial point is that the Lagrangian covering gives a preferred\footnote{In the Floer-theoretic languange of the Fukaya category, the basis elements are the intersections of the exact Lagrangian with the cotangent fibers.} basis for the fibers of the local system, which can undergo changes \`a la handle-slides in the Morse context -- see \cite{GKS_Quantization}. Here, the Lagrangian covering is given by the restriction $\pi|_L:L\lr C$ of the projection $\pi:T^*C\lr C$ onto the zero section. Now, the $N$-graphs and the microlocal monodromies, as discussed in Section \ref{ssec:Microlocal}, precisely encode these changes. In our context, the non-Abelianization map is the construction that recovers the constructible sheaf from its microlocal monodromy.

We illustrate this in the following example. Figure \ref{fig:3triangles} shows the 3-graph $G$ associated to two adjacent 3-triangles. Suppose that we are given a local system on $\La(G)$. Denote by $x,y$ the two monodromies of the corresponding Legendrian weave around the two {\sf Y}-cycles, and $z,w$ the two microlocal monodromies along the two $\sf I$-cycles represented by the two (red) monochromatic edges.  
\begin{figure}[H]
	\vskip-.2in
	\includegraphics[scale=.6]{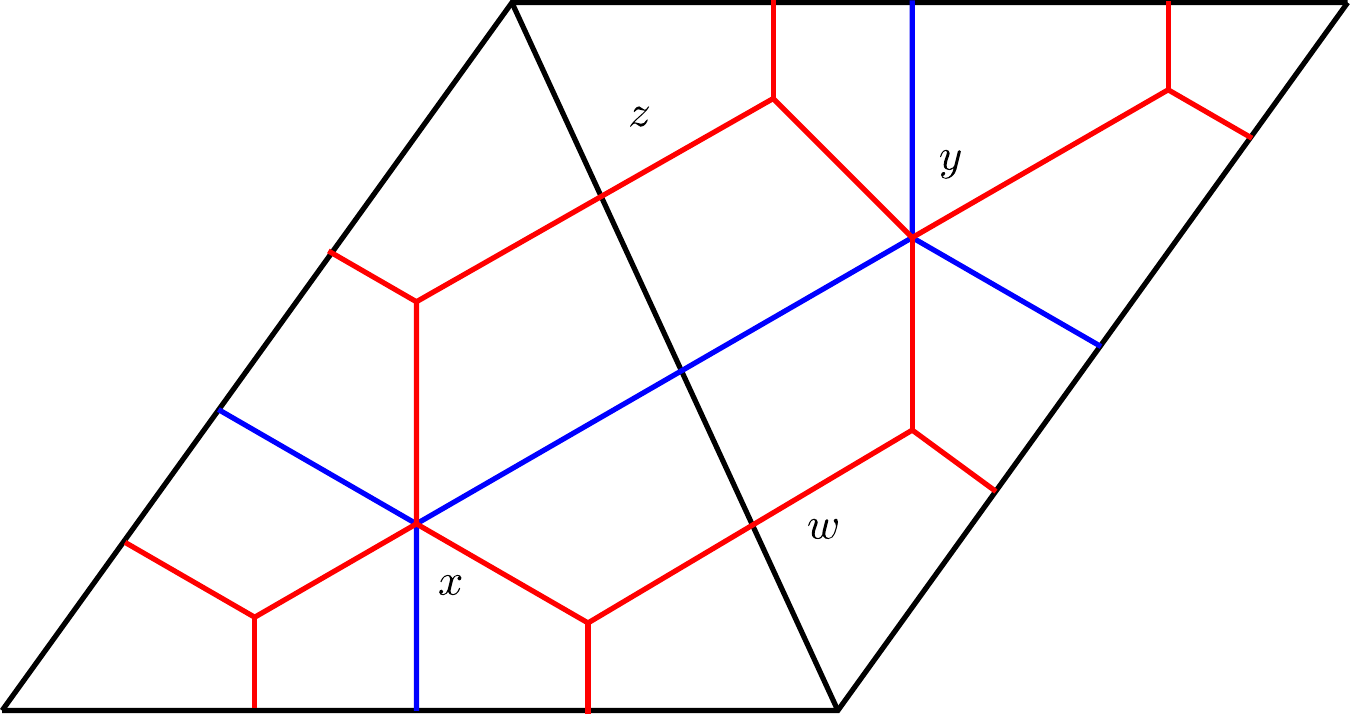}
	\vskip -.2in
	\caption{The 3-graph associated to an adjacent pair of $3$-triangles. The monodromies along the four 1-cycles are labeled $x,y,z,w$.}
	\label{fig:3triangles}
\end{figure}

The Legendrian weave $\La(G)$ is a thrice-punctured genus-one surface and these four 1-cycles are a basis for $H_1(\La(G),\Z)$. We would like to reconstruct the flag data, specifying a constructible sheaf, from the monodromies $x,y,z,w$ of the local system.  Indeed, this will realize the Non-Abelianization map \cite{GMN_SpecNet13} from rank-one local systems on the (spectral, or conjugate) Lagrangian -- parametrized by monodromies -- to decorated rank-two local systems on the base surface.  Since the base surface here is contractible, the only degrees of freedom are the choices of flags at vertices.  The map is computed as follows.

Let $(a,A), (b,B), (c,C)$ be the flags at the vertices of the left triangle, and let $(d,D)$ be the remaining flag.  We would like a birational map from the monodromies $(x,y,z,w)$ to the choice of flags.  
By using the $\PGL_3$-action, we may assume $(a,A), (b,B), (c,C)$ are as in Subsection \ref{sssec:ClusterTransf} above, with triple product $x$. Then the flag $(c,D)$ is determined by the cross ratios $z$ and $w$, and the triple product $y$. For instance, $z$ is the cross ratio $\langle b,BC,AB,BD\rangle$ while we find $w = \langle a,AD,AB,AC\rangle.$  These determine $D$, whence the triple product $y$ fixes $d \in D$. Direct computation shows
$$d = \begin{pmatrix}-\frac{x}{q}(1+x)\\x(1+y)\\-py(1+x)\end{pmatrix},\quad D = (pq,p(1+x),x).$$
This thus recovers \cite{FockGoncharov_ModuliLocSys,Palesi15} from the perspective of $N$-graphs.

\subsection*{Example: Tetrahedron with 3-Triangulation.} Let us conclude this subsection by analyzing the genus-4 Legendrian weave $\La$ in Example \ref{ssec:RealLifeExample} from the microlocal perspective. We also compute, following \cite{TreumannZaslow}, the primitive which characterizes (a discrete cover of) $\SM(G)$ as an exact Lagrangian subvariety. Following \cite{AV1,AV2}, this primitive -- the superpotential of an effective 4d theory -- is interpreted as a generating function of BPS numbers, and should have integrality properties. We check this for this example.

Consider the tetrahedron with its unique 3-triangulation, as in Example \ref{ssec:RealLifeExample}, which gives rise to a 3-graph $G$. An object in the category of simple constructible sheaves $Sh^1_{\La(G)}(\S^2\times \bR,\Lambda)_0$ microlocally supported along $\La(G)$ is defined by a four-tuple of transverse flags in $V \cong \bC^3$, placed at the vertices of the tetrahedron, as in Figure \ref{fig:triple}.

Note that there are $4\cdot 3 = 12$ total nodes, and the Legendrian surface indeed has genus $g = 4$. We therefore have $2g = 8$ cluster variables, specified by the monodromies around each of the eight loops, which themselves are a basis for $H_1(\La(G),\Z)\cong\Z^8$. Four of the monodromies are the triple ratios along the faces.  
Let us label the faces by the three unordered vertices it contains, e.g. we write $x_{123}$ for the monodromy of the loop deteremined by the minimal triangle
at the center of the face (123):  it is the triple ratio of the three flags at vertices 1, 2 and 3.
There are $4\times 3 = 12$ more edge monodromies, but we will find $4\cdot 2 = 8$ relations
among all these 16 total, giving 8 independent monodromies as expected for a genus-$4$ surface.  Let us compute the edge monodromies.

First, following \cite{TreumannZaslow}, for each edge $e$ we define a corresponding coordinate $x_e$ to be the \emph{negative} of the cross ratio.\footnote{We believe the sign appears due to the fact that we should be considering \emph{twisted} local systems, i.e.~lifts to the circle bundle of the surface that have monodromy $-1$ over the circle fibers, as in \cite{FockGoncharov_ModuliLocSys} and \cite[Section 10]{GMN_SpecNet13}.}
Now there are two relations for each vertex:  first, the product of the edge coordinates around the encircling triangular face is unity; second, the product of edge and ${\sf Y}$-monodromies encircling the vertex at a greater distance is unity.  There are thus $8$ independent coordinates, and we can take two from each of the triangles surrounding the four vertices. Let us then write
$$x_{12} = -\frac{w_1\wedge v_4}{v_4\wedge v_2}\frac{v_2\wedge v_3}{v_3\wedge w_1}$$
for the coordinate associated to the edge of the triangle encircling vertex $1$ and traversing the one-simplex of the triangulation between vertices $1$ and $2$,
where $v_i$, $i\in[1,4]$, are generators of lines and $w_i$ are generators for planes (thought of as anti-symmetric two-vectors) --- and likewise for the other edges.
Then the relation for the triangle encircling vertex $1$ is $x_{12}x_{13}x_{14}=1$, and likewise for the other vertices.
Recall that we have similarly denoted by $x_{123}$ the inverse of the coordinate corresponding to the {\sf Y}-cycle in the face containing vertices $1$, $2$ and $3$ --- and likewise for the coordinates $x_{ijk}$, $i,j,k\in[1,4]$. For the first vertex we have the unital relation
$$x_{123} x_{21} x_{142} x_{41} x_{134} x_{31} = 1,$$
and likewise for the other three vertices. This expresses the flag moduli in terms of generators, given by $x_{ij},x_{ijk}$, and relations, as above.

Let us verify that these coordinates define a (holomorphic) Lagrangian embedding of flag moduli space $\SM(G)$ associated to the genus-4 Legendrian $\La(G)$ into the moduli space of framed local systems for $C=\S^2$. The symplectic 2-form is computed from the intersection form to be
$$\omega = -d\log x_{12} \wedge d\log x_{13}+d\log x_{23}\wedge d\log x_{24} -d\log x_{34}\wedge d\log x_{31}+d\log x_{41}\wedge d\log x_{42}.$$
We can directly compute the following four relations
$$x_{12} = \frac{-1}{1+x_{13}},\qquad
x_{24} = \frac{-1}{1+x_{23}},\qquad
x_{34} = \frac{-1}{1+x_{31}},\qquad
x_{42} = \frac{-1}{1+x_{41}},$$
which readily imply that the embedding of the flag moduli space $\cM(G)$ in each of the cluster charts for the moduli space of framed local systems is Lagrangian. This holomorphic Lagrangian $\cM(G)$ is in fact {\it exact} and we can compute a primitive function $W$ for the restriction of the Liouville 1-form $\la_\st$. This would allow us to write $\cM(G)$ as the graph $\Gamma_{dW}$ of the 1-form $dW$ in this chart. This primitive encodes the BPS states associated to some Lagrangian filling, given by the Lagrangian projection of $\La(G)$, determined by a phase and a framing (implicit here) as in \cite[Section 4.8]{TreumannZaslow} -- see also \cite{AV1,AV2}. For that, we define the variables
$$U_1 = -x_{13},\quad V_1 = -x_{12},\quad U_2 = -x_{23},\quad V_2 = -x_{24},$$
$$U_3 = -x_{31},\quad V_3 = -x_{34},\quad U_4 = -x_{41},\quad V_4 = -x_{42}.$$
Also, recall that if we have $U + V^{-1} = 1$ with $U = e^u$ and $V=e^v$, then we can write
$$v = -\log(1-U) = \partial_u {\rm Li}_2(U).$$
Hence, since we have $U_i + V_i^{-1} = 1$ for all $i$, with symplectic 2-form $\omega = \sum_{i} du_i \wedge dv_i$, we conclude that $\cM(G) = \Gamma_{dW}$ where
$$W(U_1,U_2,U_3,U_4) = \sum_{i=1}^4 {\rm Li}_2(U_i).$$
This computation for the BPS potential is in line with the results in \cite[Section 5]{TreumannZaslow}.

Finally, let us review how {\it geometric} methods, as developed in Section \ref{sec:moves}, would lead to this result. Instead of the algebraic computation above, we could have directly used the diagrammatic calculus, as in Example \ref{ssec:RealLifeExample}, and deduced that our Legendrian weave $\La(G)\cong\#_{i=1}^4\bT^2_{c}$ is the Legendrian connected sum of four Clifford 2-tori $\bT^2_c$. Since the generating function of BPS numbers for $\bT^2_c$ is given by one dilogarithm ${\rm Li}_2(U)$, by direct computation, and the potential $W$ is additive under connected sum, we could have directly deduced that $W(U_1,U_2,U_3,U_4) = \sum_{i=1}^4 {\rm Li}_2(U_i)$. This concludes that our algebraic computation above is consistent with the contact topology of the underlying Legendrian weave.


\appendix

\section{Soergel Calculus and Legendrian Weaves}\label{ssec:SoergelCalculus}

In this appendix, we provide a construction and a concise speculation regarding the symplectic geometrization of Soergel Calculus via Legendrian weaves. The following discussion owes a good deal to B. Elias and E. Gorsky, as explained in the introduction, to whom we are very grateful. Soergel calculus, as developed by B. Elias, M. Khovanov and G. Williamson \cite{Soergel1,Soergel3}, provides a {\it diagrammatic} presentation of the category of Soergel bimodules, which itself categorifies the Hecke algebra.  The similarities between Elias' diagrammatic calculus and our Legendrian weaves are apparent. Legendrian weaves can be understood as a {\it geometric} approach to the study of the algebra of certain complexes of Soergel bimodules.  We explain this below.

\begin{remark}
Soergel bimodules are essential to categorifications of knot invariants \cite{Rouquier06,Soergel07,Khovanov07,Soergel1}.  The link between these and moduli spaces of sheaves for Legendrian braid closures was described in \cite[Section 6]{STZ_ConstrSheaves}.  From this perspective, it is not unnatural to seek a connection between planar Soergel structures and planar structure defined by Legendrian weaves, the two-dimensional version of braids.\hfill$\Box$
\end{remark}

The category of Soergel bimodules is the Karoubi completion of the subcategory of Bott-Samelson bimodules, arising as the equivariant cohomology of a closed Bott-Samelson variety, and thus it suffices to understand the relation to this latter class of bimodules. The key connection between the present work and Soergel bimodules is that a subclass of Legendrian weaves yields exact Lagrangian cobordisms between Legendrian links, which are themselves represented as positive braids. The moduli space of microlocal constructible sheaves supported on a {\it singular} compactification of a positive braid is a closed Bott-Samelson variety, and our Legendrian weaves, understood as Lagrangian cobordisms -- and singularly compactified -- induce morphisms between these closed Bott-Samelson varieties.

Thus, we are able to geometrize the diagrammatics of Soergel calculus by considering the $D_4^-$-singularity for the trivalent vertices in \cite{Soergel1,Soergel3}, the $A_3$-swallowtail singularity for the univalent vertex and the $A_1^3$-singularity for their hexagonal vertices. (The Soergel calculus we geometrize corresponds to the $m=2$ Coxeter exponent.)\\

\begin{remark}
Exact Lagrangian cobordisms are {\it directed}, due to the convexity directionality in symplectic topology. The dissonance arises from the fact that, as of today, Soergel calculus only considers closed Bott-Samelson varieties, whereas the moduli space of microlocal sheaves supported on a positive braid is an {\it open} Bott-Samelson variety. Thus, the Soergel calculus is geometrized by {\it singular} compactifications of our Legendrian weaves, and our Legendrian weave calculus, without compactification, should naturally induce a Soergel calculus for {\it open} Bott-Samelson varieties.\hfill$\Box$
\end{remark}

For instance, the $A_3$-Zamolodchikov relation from Soergel calculus corresponds to the $A_1^4$-Reidemeister move in Legendrian weave calculus, as depicted in Figure \ref{fig:A3Zamolodzhikov}.

\begin{center}
	\begin{figure}[h!]
		\centering
		\includegraphics[scale=0.55]{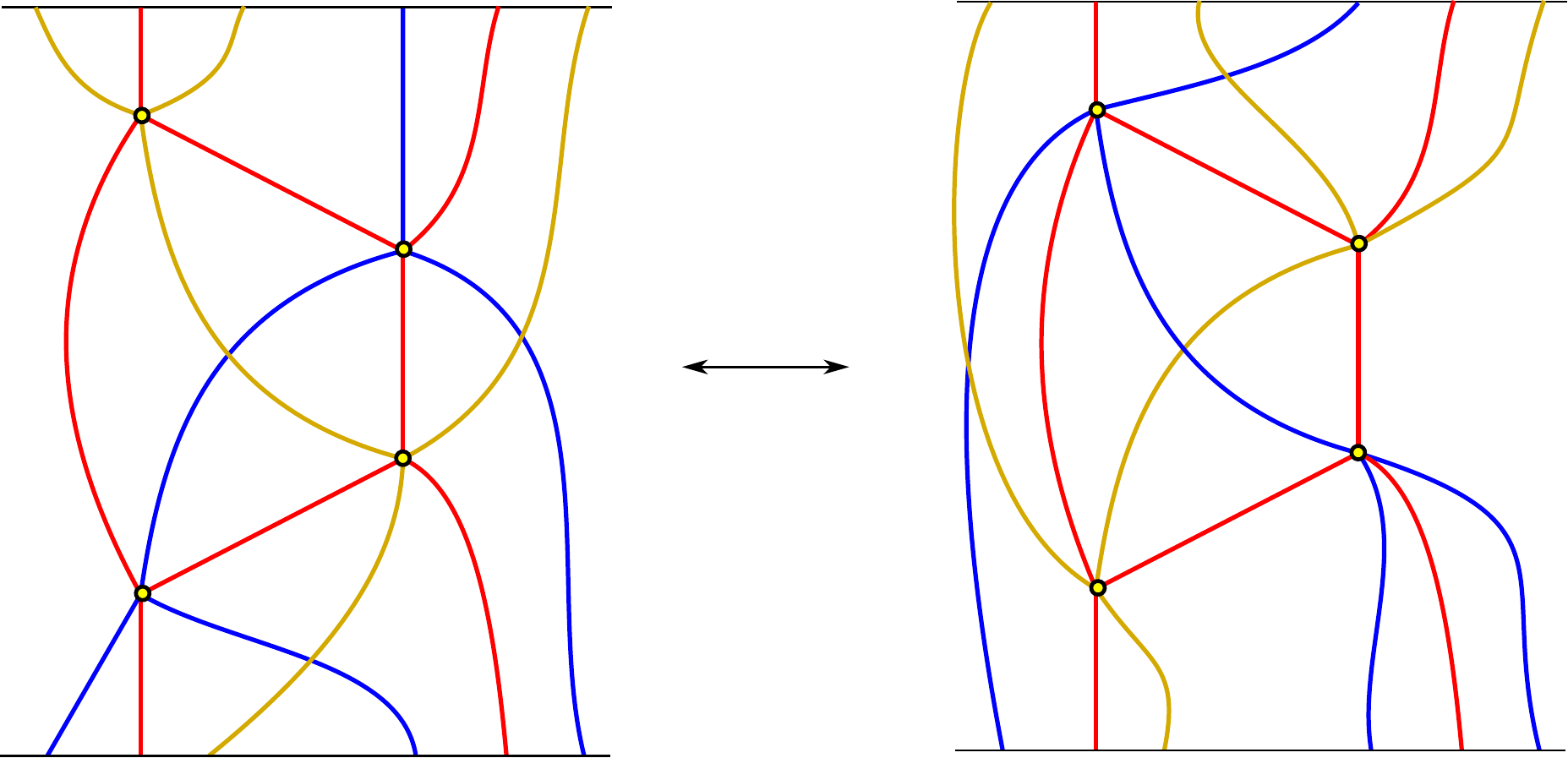}
		\caption{Contact Isotopy among Legendrian weaves, relative to the boundaries. The lack of Reeb chords allows us to interpret these as exact Lagrangian cobordisms between the positive braids $\sigma_{i+1}\sigma_{i}\sigma_{i-1}\sigma_{i+1}\sigma_{i}\sigma_{i+1}$ and $\sigma_{i-1}\sigma_{i}\sigma_{i-1}\sigma_{i+1}\sigma_{i}\sigma_{i-1}$. The fact that these Lagrangian cobordisms are Hamiltonian isotopic implies that the morphism induced between the associated Bott-Samelson bimodules must coincide.\hfill$\Box$}
		\label{fig:A3Zamolodzhikov}
	\end{figure}
\end{center}

Now, let us consider two positive braids $\beta_1,\beta_2\in\mbox{Br}^+_n$, $n\in\N$, and their associated Legendrian (long) links $\La(\beta_1),\La(\beta_2)\sse(J^1[0,1],\xi_\st)$ \cite[Section 2]{CasalsHonghao}. A Legendrian weave $\La\sse(J^1([0,1]\times[1,2]),\xi_\st)$ with no Reeb chords and boundaries $\La(\beta_1)$ at $[0,1]\times\{1\}$, and $\La(\beta_2)$ at $[0,1]\times\{2\}$, yields an embedded and exact Lagrangian cobordism $L(\Sigma)$ from $\La(\beta_1)$ to $\La(\beta_2)$ in the symplectization of $(J^1[0,1],\xi_\st)$, as in Section \ref{sec:app2}. In particular, each trivalent vertex $\Sigma(G_{tri})$ and hexagonal vertex $\Sigma(G_{hex})$ yield the following exact Lagrangian cobordism:

\begin{itemize}
	\item[(i)] The Lagrangian projection $L(G_{tri})$ of the Legendrian weave $\La(G_{tri})$ is a Lagrangian cobordism from the Legendrian tangle $\La(\beta_1)$ given by one crossing in two strands $\beta_1=\sigma_i$, to the Legendrian tangle $\La(\beta_2)$ given by two crossing in two strands $\beta_1=\sigma^2_i$, where $i\in\N$ is labeling the transposition $\tau_i$ of the edges of $G_{tri}$. Smoothly, this is a saddle cobordism obtained by an index-1 handle attachment to the Lagrangian cone $\La(\beta_1)\times[0,\varepsilon]$ in the symplectization, for $\varepsilon\in\R^+$ small.\\
	
	\item[(ii)] The Lagrangian projection $L(G_{hex})$ of the Legendrian weave $\La(G_{hex})$ is a Lagrangian concordance from the Legendrian tangle $\La(\beta_1)$ given by three crossings in three strands $\beta_1=\sigma_i\sigma_{i+1}\sigma_i$, to the Legendrian tangle $\La(\beta_2)$ given by $\beta_2=\sigma_{i+1}\sigma_{i}\sigma_{i+1}$, where $i\in\N$ is labeling the transpositions $\tau_i,\tau_{i+1}$ in the edges of $G_{hex}$. Smoothly, this is a Lagrangian surface obtained by graphing a Reidemeister three move.\\
\end{itemize}

For simplicity, let us suppose that the relative homology $H_1(L,\dd_-L;\Z)$, which we denote by $H_1(L)$, is a free $\Z$-module and the surface $L$ is spin, as is verified for the two local cobordisms above. An exact Lagrangian cobordism $L\sse (J^1[0,1],\xi_\st)\times[1,2]$ from $\La(\beta_1)$ to $\La(\beta_2)$ yields an algebraic map
$$\Phi_L:\widehat{\SM}(\La(\beta_1))\lr\SM(\La(\beta_2)),$$
where $\widehat{\SM(\La(\beta_1))}$ is an algebraic $(\C^*)^{b_1(L)}$-bundle over $\SM(\La(\beta_1))$, and $\SM(\La(\beta))$ denotes the moduli space of microlocal rank-1 objects in the dg-category of of microlocal sheaves in $\S^1\times$ microlocally supported on $\La(\beta)$, as described in \cite[Section 3]{CasalsHonghao}, \cite{STWZ,STZ_ConstrSheaves}.

\begin{remark}
In the Floer-theoretic context, the map $\Phi_L$ is obtained by applying the contravariant functor $Hom(\cdot,k)$ in the category of dg-algebras to the morphism $$\Phi_L^{Fl}:\SA(\La(\beta_2))\lr\SA(\La(\beta_1))\otimes_\Z\Z[H_1(L)]$$
of the Legendrian Contact dg-algebras $\SA(\La(\beta))$ associated to Legendrian links $\La(\beta)$. The Floer theoretic map $\Phi_L^{Fl}$ is described in \cite{EkholmHondaKalman16,YuPan}, and it is a count of holomorphic strips whose boundary homology classes are encoded in $\Z[H_1(L)]$. To ease the geometry, we have tensored by the flag moduli space map $\Phi_L$ above $\C[H_1(L)]\cong\Z[H_1(L)]\otimes_\Z\C$ to base change the $\mbox{Spec}(\Z[H_1(L)])$-bundle to a complex variety $\widehat{\SM}(\La(\beta_1))$.\hfill$\Box$
\end{remark}

The relation to Soergel calculus now arises because the moduli space of simple microlocal sheaves $\SM(\La(\beta_1))$ is (explicitly) isomorphic to the open Bott-Samelson variety associated to $\beta$, also known as the Brou\'e-Michel variety of $\beta$ \cite{STZ_ConstrSheaves,OBS,CasalsHonghao}. Let $R=H^*(\SB)$ denote the cohomology of the complete flag variety for $\GL(N,\C)$, $N\in\N$, and $B_{s_i}$ the Bott-Samelson Soergel ($R\otimes R$)-bimodule associated to a permutation $s_i\in S_N$ in the Weyl group $S_N$ of $\GL(N,\C)$. The Rouquier complex $T_i:=[B_{s_i}\lr R]$ will be denoted by $T_i$, for all $i\in\N$. Consider a braid
$$\beta=\prod_{j=1}^{l}\sigma_{i_j},\quad 1\leq i_j\leq k-1,$$
where $\sigma_{i_1}$ is the leftmost crossing in the front diagram of the Legendrian braid, and the crossings are read from left to right. Then the (singular) compactly supported cohomology of algebraic variety $\SM(\La(\beta))$ is described by the tensor product
$$T_\beta=T_{i_1}\otimes_RT_{i_2}\otimes_R\cdots \otimes_RT_{i_l},$$
of Rouquier complexes.

\begin{remark} Should the reader be interested in the closure of the Legendrian $\La(\beta)\sse(J^1S^1,\xi_\st)$, instead of the long link $(J^1[0,1],\xi_\st)$, the cohomology of the corresponding moduli space $\SM(\La(\beta))$ is obtained by applying Hochschild homology to the above complex $T_\beta$. In particular, $H^*(\SM(\La(\beta)))$ coincides with the triply-graded homology of the knot associated to $\beta$, equivalently, Khovanov-Rozansky link homology -- see \cite[Theorem 6.14]{STZ_ConstrSheaves}.\hfill$\Box$
\end{remark}

In conclusion, the geometric map $\Phi_L:\widehat{\SM}(\La(\beta_1))\lr\SM(\La(\beta_2))$ functorially induces
$$H^*_c(\Phi_L):H^*_c(\widehat{\SM}(\La(\beta_1)))\lr H^*_c(\SM(\La(\beta_2))),$$
which is a map of (products of) Rouquier complexes $\widehat{T}_{\beta_1}\lr T_{\beta_2}$, where $\widehat{T}_{\beta_1}$ is the compactly supported cohomology of $\widehat{\SM}(\La(\beta_1))$, which contains the information of the compactly supported cohomology $T_{\beta_1}$ of the open Bott-Samelson variety for $\beta_1$.

Now, applying this to the two Lagrangian cobordisms associated to the trivalent vertices $G_{tri}$ and the hexagonal vertices $G_{hex}$, we obtain the following two maps:

\begin{itemize}
	\item[(i)] The map $\Phi_{L(G_{tri})}:T_{s_i}\otimes H^*(\S^1)\lr T_{s_i}\otimes_R T_{s_i}$, where $i$ labels the $\tau_i$-edges of $G_{tri}$, and we have identified the fiber bundle $\widehat{\SM}(\La(\beta_1))\cong(\La(\beta_1))\times\C^*$ with the Cartesian product, as in this case the bundle is topologically trivial. The fact that there is {\it one} copy of $\C^*=\S^1\times\R$ corresponds to the fact that the Lagrangian cobordism $L(G_{tri})$ has a unique index 1 critical point and its cocore carries the data $\C^*$.\\
	
	\item[(ii)] The map $\Phi_{L(G_{hex})}:T_{s_i}\otimes T_{s_{i+1}}\otimes T_{s_i} \lr T_{s_{i+1}}\otimes T_{s_i}\otimes T_{s_{i+1}}$, where in this case $\widehat{\SM}(\La(\beta_1))\cong\SM(\La(\beta_1))$ as the Lagrangian $L(G_{hex})$ is a cylinder and $H_1(L)\cong\{0\}$ is trivial.\\
\end{itemize}

In conclusion, the above discussion can summarized according to the following tenet:

\begin{principle} Let $T_{\beta_1},T_{\beta_2}$ be the Rouquier complexes associated to positive braids $\beta_1,\beta_2$ and $\Psi:T_{\beta_1}\lr T_{\beta_2}$ the morphism given by a graph $G_\Psi$ with only (upwards) trivalent and hexagonal morphisms in (open) Soergel calculus. Then the Lagrangian projection of the Legendrian weave $\La(G_\Psi)$ yields an embedded exact Lagrangian cobordism $L$ from $\La(\beta_1)$ to $\La(\beta_2)$ and a geometric map
$$\Phi_L:\widehat{\SM}(\La(\beta_1))\lr\La(\beta_2)$$
such that $H^*_c(\Phi_L)=\Psi$.\hfill$\Box$
\end{principle}
The difference between the principle above and a theorem lies on the correct definition of {\it open} Soergel calculus, of which we are not aware at this stage. That said, since the trivalent and the hexagonal vertices are two of the main building blocks for closed Soergel calculus, the above construction provides a potential symplectic geometrization of {\it open} Soergel calculus, associated to Rouquier complexes, instead of Soergel bimodules. In particular, in the context of {\it open} Bott-Samelson varieties, the Lagrangian cobordisms above indicate the need for additional data from $H_1(L)=\Z^{|V|}$ in specifying a morphism, where $|V|$ is the number of trivalent vertices. The development of open Soergel calculus, the computations establishing that our geometric maps induce the expected algebraic maps, as well as the Lagrangian description of the univalent vertex, will be the subject of upcoming and more algebraic work.


\bibliographystyle{alpha}
\bibliography{LegendrianWeaves_CasalsZaslow}

\newcommand{\etalchar}[1]{$^{#1}$}
\def\cprime{$'$}
\begin{thebibliography}{EENS13b}

\bibitem[AB27]{AlexanderBriggs}
J.~W. Alexander and G.~B. Briggs.
\newblock On types of knotted curves.
\newblock {\em Ann. of Math. (2)}, 28(1-4):562--586, 1926/27.

\bibitem[Ad75]{ArnoldICM74}
Vladimir~Igorevich Arnol\cprime~d.
\newblock Critical points of smooth functions.
\newblock {\em Proceedings of the {I}nternational {C}ongress of
  {M}athematicians ({V}ancouver, {B}. {C}., 1974), {V}ol. 1}, pages 19--39,
  1975.

\bibitem[Ad76]{Arnold76SurgeryI}
Vladimir~Igorevich Arnol\cprime~d.
\newblock Wave front evolution and equivariant {M}orse lemma.
\newblock {\em Comm. Pure Appl. Math.}, 29(6):557--582, 1976.

\bibitem[Ad79]{Arnold79SurgeryII}
Vladimir~Igorevich Arnol\cprime~d.
\newblock Indexes of singular points of {$1$}-forms on manifolds with boundary,
  convolutions of invariants of groups generated by reflections, and singular
  projections of smooth surfaces.
\newblock {\em Uspekhi Mat. Nauk}, 34(2(206)):3--38, 1979.

\bibitem[Ad90]{ArnoldSing}
Vladimir~Igorevich Arnol\cprime~d.
\newblock {\em Singularities of caustics and wave fronts}, volume~62 of {\em
  Mathematics and its Applications (Soviet Series)}.
\newblock Kluwer Academic Publishers Group, Dordrecht, 1990.

\bibitem[AdG01]{ArnoldGivental01}
V.~I. Arnol\cprime~d and A.~B. Givental\cprime.
\newblock Symplectic geometry [ {MR}0842908 (88b:58044)].
\newblock In {\em Dynamical systems, {IV}}, volume~4 of {\em Encyclopaedia
  Math. Sci.}, pages 1--138. Springer, Berlin, 2001.

\bibitem[Aur07]{Auroux_Wallcrossing2}
Denis Auroux.
\newblock Mirror symmetry and {$T$}-duality in the complement of an
  anticanonical divisor.
\newblock {\em J. G\"{o}kova Geom. Topol. GGT}, 1:51--91, 2007.

\bibitem[Aur09]{Auroux_Wallcrossing}
Denis Auroux.
\newblock Special {L}agrangian fibrations, wall-crossing, and mirror symmetry.
\newblock In {\em Surveys in differential geometry. {V}ol. {XIII}. {G}eometry,
  analysis, and algebraic geometry: forty years of the {J}ournal of
  {D}ifferential {G}eometry}, volume~13 of {\em Surv. Differ. Geom.}, pages
  1--47. Int. Press, Somerville, MA, 2009.

\bibitem[AV00]{AV1}
Mina Aganagic and Cumrun Vafa.
\newblock {Mirror Symmetry, D-Branes and Counting Holomorphic Discs}.
\newblock 2000.

\bibitem[AV12]{AV2}
Mina Aganagic and Cumrun Vafa.
\newblock {Large N Duality, Mirror Symmetry, and a Q-deformed A-polynomial for
  Knots}.
\newblock 2012.

\bibitem[Ben83]{Bennequin83}
Daniel Bennequin.
\newblock Entrelacements et \'equations de {P}faff.
\newblock In {\em Third Schnepfenried geometry conference, Vol. 1
  (Schnepfenried, 1982)}, volume 107 of {\em Ast\'erisque}, pages 87--161. Soc.
  Math. France, Paris, 1983.

\bibitem[Ben86]{BennequinBourbaki}
Daniel Bennequin.
\newblock Caustique mystique (d'apr\`es {A}rnol\cprime d et al.).
\newblock Number 133-134, pages 19--56. 1986.
\newblock Seminar Bourbaki, Vol. 1984/85.

\bibitem[Bir74]{Birman74Braids}
Joan~S. Birman.
\newblock {\em Braids, links, and mapping class groups}.
\newblock Princeton University Press, Princeton, N.J.; University of Tokyo
  Press, Tokyo, 1974.
\newblock Annals of Mathematics Studies, No. 82.

\bibitem[BM08]{Graph1}
J.~A. Bondy and U.~S.~R. Murty.
\newblock {\em Graph theory}, volume 244 of {\em Graduate Texts in
  Mathematics}.
\newblock Springer, New York, 2008.

\bibitem[BS15]{BridgelandSmith_QuadDiff}
Tom Bridgeland and Ivan Smith.
\newblock Quadratic differentials as stability conditions.
\newblock {\em Publ. Math. Inst. Hautes \'{E}tudes Sci.}, 121:155--278, 2015.

\bibitem[BST15]{BourgeoisSabloffTraynor15}
Fr\'ed\'eric Bourgeois, Joshua~M. Sabloff, and Lisa Traynor.
\newblock Lagrangian cobordisms via generating families: construction and
  geography.
\newblock {\em Algebr. Geom. Topol.}, 15(4):2439--2477, 2015.

\bibitem[Cas20]{CasalsSkel}
Roger Casals.
\newblock {Lagrangian Skeleta and Plane Curve Singularities}.
\newblock {\em J. Fixed Point Theory and Applications}, (Viterbo 60), 2020.

\bibitem[CE12]{CieliebakEliashberg12}
Kai Cieliebak and Yakov Eliashberg.
\newblock {\em From {S}tein to {W}einstein and back}, volume~59 of {\em
  American Mathematical Society Colloquium Publications}.
\newblock American Mathematical Society, Providence, RI, 2012.
\newblock Symplectic geometry of affine complex manifolds.

\bibitem[CE14]{CieliebakEliashberg14}
Kai Cieliebak and Yakov Eliashberg.
\newblock Stein structures: existence and flexibility.
\newblock In {\em Contact and symplectic topology}, volume~26 of {\em Bolyai
  Soc. Math. Stud.}, pages 357--388. J\'anos Bolyai Math. Soc., Budapest, 2014.

\bibitem[CG20]{CasalsHonghao}
Roger Casals and Honghao Gao.
\newblock {Infinitely many Lagrangian fillings}.
\newblock {\em ArXiv e-prints}, 2001.01334, 2020.

\bibitem[CGGS]{CGGS}
Roger Casals, Eugene Gorsky, Mikhail Gorsky, and Jos\'e Simental.
\newblock {\em arXiv}, 2012.06931.

\bibitem[Che02]{Chekanov02}
Yuri Chekanov.
\newblock Differential algebra of {L}egendrian links.
\newblock {\em Invent. Math.}, 150(3):441--483, 2002.

\bibitem[CM18]{CasalsMurphy2}
Roger Casals and Emmy Murphy.
\newblock Differential algebra of cubic planar graphs.
\newblock {\em Adv. Math.}, 338:401--446, 2018.

\bibitem[CM19]{CasalsMurphy}
Roger Casals and Emmy Murphy.
\newblock Legendrian fronts for affine varieties.
\newblock {\em Duke Math. J.}, 168(2):225--323, 2019.

\bibitem[CMP19]{CasalsMurphyPresas}
Roger Casals, Emmy Murphy, and Francisco Presas.
\newblock Geometric criteria for overtwistedness.
\newblock {\em J. Amer. Math. Soc.}, 32(2):563--604, 2019.

\bibitem[CN20]{CasalsNg}
Roger Casals and Lenhard Ng.
\newblock {Braid Loops with infinite monodromy on the Legendrian contact DGA}.
\newblock {\em Arxiv e-prints}, 2020.

\bibitem[Con84]{AlternatingHurwitz}
Marston D.~E. Conder.
\newblock Some results on quotients of triangle groups.
\newblock {\em Bull. Austral. Math. Soc.}, 30(1):73--90, 1984.

\bibitem[DGG16]{DimofteGabellaGoncharov}
Tudor Dimofte, Maxime Gabella, and Alexander~B. Goncharov.
\newblock K-decompositions and 3d gauge theories.
\newblock {\em J. High Energy Phys.}, (11):151, front matter+144, 2016.

\bibitem[Die17]{Graph2}
Reinhard Diestel.
\newblock {\em Graph theory}, volume 173 of {\em Graduate Texts in
  Mathematics}.
\newblock Springer, Berlin, fifth edition, 2017.

\bibitem[DR11]{Rizell_TwistedSurgery}
Georgios Dimitroglou~Rizell.
\newblock Knotted {L}egendrian surfaces with few {R}eeb chords.
\newblock {\em Algebr. Geom. Topol.}, 11(5):2903--2936, 2011.

\bibitem[DR16]{Rizell_Surgeries}
Georgios Dimitroglou~Rizell.
\newblock Legendrian ambient surgery and {L}egendrian contact homology.
\newblock {\em J. Symplectic Geom.}, 14(3):811--901, 2016.

\bibitem[EENS13a]{EkholmEtnyreNgSullivan13a}
Tobias Ekholm, John Etnyre, Lenhard Ng, and Michael Sullivan.
\newblock Filtrations on the knot contact homology of transverse knots.
\newblock {\em Math. Ann.}, 355(4):1561--1591, 2013.

\bibitem[EENS13b]{EkholmEtnyreNgSullivan13}
Tobias Ekholm, John~B. Etnyre, Lenhard Ng, and Michael~G. Sullivan.
\newblock Knot contact homology.
\newblock {\em Geom. Topol.}, 17(2):975--1112, 2013.

\bibitem[EES05a]{EkholmEtnyreSullivan05b}
Tobias Ekholm, John Etnyre, and Michael Sullivan.
\newblock The contact homology of {L}egendrian submanifolds in {${\Bbb R}\sp
  {2n+1}$}.
\newblock {\em J. Differential Geom.}, 71(2):177--305, 2005.

\bibitem[EES05b]{EkholmEtnyreSullivan05a}
Tobias Ekholm, John Etnyre, and Michael Sullivan.
\newblock Non-isotopic {L}egendrian submanifolds in {$\Bbb R\sp {2n+1}$}.
\newblock {\em J. Differential Geom.}, 71(1):85--128, 2005.

\bibitem[EGH00]{EliashbergGiventalHofer00}
Y.~Eliashberg, A.~Givental, and H.~Hofer.
\newblock Introduction to symplectic field theory.
\newblock {\em Geom. Funct. Anal.}, (Special Volume, Part II):560--673, 2000.
\newblock GAFA 2000 (Tel Aviv, 1999).

\bibitem[EHK16]{EkholmHondaKalman16}
Tobias Ekholm, Ko~Honda, and Tam\'as K\'alm\'an.
\newblock Legendrian knots and exact {L}agrangian cobordisms.
\newblock {\em J. Eur. Math. Soc. (JEMS)}, 18(11):2627--2689, 2016.

\bibitem[EK10]{Soergel1}
Ben Elias and Mikhail Khovanov.
\newblock Diagrammatics for {S}oergel categories.
\newblock {\em Int. J. Math. Math. Sci.}, pages Art. ID 978635, 58, 2010.

\bibitem[Eli90]{Eliashberg90a}
Yakov Eliashberg.
\newblock Topological characterization of {S}tein manifolds of dimension
  {$>2$}.
\newblock {\em Internat. J. Math.}, 1(1):29--46, 1990.

\bibitem[Eli93]{Eliashberg93}
Yakov Eliashberg.
\newblock Legendrian and transversal knots in tight contact {$3$}-manifolds.
\newblock In {\em Topological methods in modern mathematics (Stony Brook, NY,
  1991)}, pages 171--193. Publish or Perish, Houston, TX, 1993.

\bibitem[EM02]{EliashbergMishachev02}
Y.~Eliashberg and N.~Mishachev.
\newblock {\em Introduction to the {$h$}-principle}, volume~48 of {\em Graduate
  Studies in Mathematics}.
\newblock American Mathematical Society, Providence, RI, 2002.

\bibitem[ENS18]{Conormal2}
Tobias Ekholm, Lenhard Ng, and Vivek Shende.
\newblock A complete knot invariant from contact homology.
\newblock {\em Invent. Math.}, 211(3):1149--1200, 2018.

\bibitem[Etn05]{Etnyre05}
John~B. Etnyre.
\newblock Legendrian and transversal knots.
\newblock In {\em Handbook of knot theory}, pages 105--185. Elsevier B. V.,
  Amsterdam, 2005.

\bibitem[EV18]{Satellite1}
John Etnyre and Vera V\'{e}rtesi.
\newblock Legendrian satellites.
\newblock {\em Int. Math. Res. Not. IMRN}, (23):7241--7304, 2018.

\bibitem[EW16]{Soergel3}
Ben Elias and Geordie Williamson.
\newblock Soergel calculus.
\newblock {\em Represent. Theory}, 20:295--374, 2016.

\bibitem[FG06a]{FockGoncharovII}
V.~V. Fock and A.~B. Goncharov.
\newblock Cluster x-varieties, amalgamation, and {P}oisson-{L}ie groups.
\newblock In {\em Algebraic geometry and number theory}, volume 253 of {\em
  Progr. Math.}, pages 27--68. Birkh\"{a}user Boston, Boston, MA, 2006.

\bibitem[FG06b]{FockGoncharov_ModuliLocSys}
Vladimir Fock and Alexander Goncharov.
\newblock Moduli spaces of local systems and higher {T}eichm\"{u}ller theory.
\newblock {\em Publ. Math. Inst. Hautes \'{E}tudes Sci.}, (103):1--211, 2006.

\bibitem[FOOO09]{FOOO_Anomaly}
Kenji Fukaya, Yong-Geun Oh, Hiroshi Ohta, and Kaoru Ono.
\newblock {\em Lagrangian intersection {F}loer theory: anomaly and obstruction.
  {P}art {II}}, volume~46 of {\em AMS/IP Studies in Advanced Mathematics}.
\newblock American Mathematical Society, Providence, RI; International Press,
  Somerville, MA, 2009.

\bibitem[FZ02]{FominZelevinsky_ClusterI}
Sergey Fomin and Andrei Zelevinsky.
\newblock Cluster algebras. {I}. {F}oundations.
\newblock {\em J. Amer. Math. Soc.}, 15(2):497--529, 2002.

\bibitem[Gei08]{Geiges08}
Hansj{\"o}rg Geiges.
\newblock {\em An introduction to contact topology}, volume 109 of {\em
  Cambridge Studies in Advanced Mathematics}.
\newblock Cambridge University Press, Cambridge, 2008.

\bibitem[GK13]{GoncharovKenyon13}
Alexander~B. Goncharov and Richard Kenyon.
\newblock Dimers and cluster integrable systems.
\newblock {\em Ann. Sci. \'{E}c. Norm. Sup\'{e}r. (4)}, 46(5):747--813, 2013.

\bibitem[GKS12]{GKS_Quantization}
St\'{e}phane Guillermou, Masaki Kashiwara, and Pierre Schapira.
\newblock Sheaf quantization of {H}amiltonian isotopies and applications to
  nondisplaceability problems.
\newblock {\em Duke Math. J.}, 161(2):201--245, 2012.

\bibitem[GLPY17]{Gabella_BPSGraphs}
Maxime Gabella, Pietro Longhi, Chan~Y. Park, and Masahito Yamazaki.
\newblock B{PS} graphs: from spectral networks to {BPS} quivers.
\newblock {\em J. High Energy Phys.}, (7):032, front matter+47, 2017.

\bibitem[GMN10]{GMN_Wallcrossing}
Davide Gaiotto, Gregory~W. Moore, and Andrew Neitzke.
\newblock Four-dimensional wall-crossing via three-dimensional field theory.
\newblock {\em Comm. Math. Phys.}, 299(1):163--224, 2010.

\bibitem[GMN13]{GMN_SpecNet13}
Davide Gaiotto, Gregory~W. Moore, and Andrew Neitzke.
\newblock Spectral networks.
\newblock {\em Ann. Henri Poincar\'{e}}, 14(7):1643--1731, 2013.

\bibitem[GMN14]{GMN_SpecNetSnakes14}
Davide Gaiotto, Gregory~W. Moore, and Andrew Neitzke.
\newblock Spectral networks and snakes.
\newblock {\em Ann. Henri Poincar\'{e}}, 15(1):61--141, 2014.

\bibitem[Gom98]{Gompf98}
Robert~E. Gompf.
\newblock Handlebody construction of {S}tein surfaces.
\newblock {\em Ann. of Math. (2)}, 148(2):619--693, 1998.

\bibitem[Gon17]{Goncharov_IdealWebs}
A.~B. Goncharov.
\newblock Ideal webs, moduli spaces of local systems, and 3d {C}alabi-{Y}au
  categories.
\newblock In {\em Algebra, geometry, and physics in the 21st century}, volume
  324 of {\em Progr. Math.}, pages 31--97. Birkh\"{a}user/Springer, Cham, 2017.

\bibitem[GPS19a]{GPS3}
Sheel Ganatra, John Pardon, and Vivek Shende.
\newblock { Microlocal Morse theory of wrapped Fukaya categories}.
\newblock 2019.

\bibitem[GPS19b]{GPS1}
Sheel Ganatra, John Pardon, and Vivek Shende.
\newblock {Covariantly functorial wrapped Floer theory on Liouville sectors}.
\newblock {\em Publ. Math. Inst. Hautes \'Etudes Sci. (to appear)}, 2019.

\bibitem[GPS19c]{GPS2}
Sheel Ganatra, John Pardon, and Vivek Shende.
\newblock {Sectorial descent for wrapped Fukaya categories}.
\newblock 2019.

\bibitem[Gro86]{Gromov86}
Mikhael Gromov.
\newblock {\em Partial differential relations}, volume~9 of {\em Ergebnisse der
  Mathematik und ihrer Grenzgebiete (3)}.
\newblock Springer-Verlag, Berlin, 1986.

\bibitem[GS14]{Sheaves2}
St\'{e}phane Guillermou and Pierre Schapira.
\newblock Microlocal theory of sheaves and {T}amarkin's non displaceability
  theorem.
\newblock In {\em Homological mirror symmetry and tropical geometry}, volume~15
  of {\em Lect. Notes Unione Mat. Ital.}, pages 43--85. Springer, Cham, 2014.

\bibitem[GS18]{GoncharovLinhui_DT}
Alexander Goncharov and Linhui Shen.
\newblock Donaldson-{T}homas transformations of moduli spaces of {G}-local
  systems.
\newblock {\em Adv. Math.}, 327:225--348, 2018.

\bibitem[GSW20a]{GSW}
Honghao Gao, Linhui Shen, and Daping Weng.
\newblock {Augmentations, Fillings, and Clusters}.
\newblock {\em ArXiv e-prints}, 2008.10793, 2020.

\bibitem[GSW20b]{GSW2}
Honghao Gao, Linhui Shen, and Daping Weng.
\newblock {Positive Braid Links with Infinitely Many Fillings}.
\newblock {\em ArXiv e-prints}, 2009.00499, 2020.

\bibitem[Hur92]{Hurwitz}
A.~Hurwitz.
\newblock Ueber algebraische {G}ebilde mit eindeutigen {T}ransformationen in
  sich.
\newblock {\em Math. Ann.}, 41(3):403--442, 1892.

\bibitem[K\'05]{Kalman}
Tam\'{a}s K\'{a}lm\'{a}n.
\newblock Contact homology and one parameter families of {L}egendrian knots.
\newblock {\em Geom. Topol.}, 9:2013--2078, 2005.

\bibitem[Kel94]{Keller_DerivingDG}
Bernhard Keller.
\newblock Deriving {DG} categories.
\newblock {\em Ann. Sci. \'{E}cole Norm. Sup. (4)}, 27(1):63--102, 1994.

\bibitem[Kho07]{Khovanov07}
Mikhail Khovanov.
\newblock Triply-graded link homology and {H}ochschild homology of {S}oergel
  bimodules.
\newblock {\em Internat. J. Math.}, 18(8):869--885, 2007.

\bibitem[KS85]{KashiwaraSchapira}
Masaki Kashiwara and Pierre Schapira.
\newblock Microlocal study of sheaves.
\newblock {\em Ast\'{e}risque}, (128):235, 1985.
\newblock Corrections to this article can be found in Ast\'{e}risque No. 130,
  p. 209.

\bibitem[KS90]{KashiwaraSchapira_Book}
Masaki Kashiwara and Pierre Schapira.
\newblock {\em Sheaves on manifolds}, volume 292 of {\em Grundlehren der
  Mathematischen Wissenschaften}.
\newblock Springer-Verlag, Berlin, 1990.
\newblock With a chapter in French by Christian Houzel.

\bibitem[KS10]{KontsevichSoibelman_MotivicDT}
Maxim Kontsevich and Yan Soibelman.
\newblock Motivic {D}onaldson-{T}homas invariants: summary of results.
\newblock In {\em Mirror symmetry and tropical geometry}, volume 527 of {\em
  Contemp. Math.}, pages 55--89. Amer. Math. Soc., Providence, RI, 2010.

\bibitem[KS14]{KontsevichSoibelman_Wallcrossing}
Maxim Kontsevich and Yan Soibelman.
\newblock Wall-crossing structures in {D}onaldson-{T}homas invariants,
  integrable systems and mirror symmetry.
\newblock In {\em Homological mirror symmetry and tropical geometry}, volume~15
  of {\em Lect. Notes Unione Mat. Ital.}, pages 197--308. Springer, Cham, 2014.

\bibitem[Kuw20]{Kuwagaki20}
Tatsuki Kuwagaki.
\newblock {Sheaf quantization from exact WKB analysis}.
\newblock 2020.

\bibitem[Law17]{Lawson}
John~W. Lawson.
\newblock Minimal mutation-infinite quivers.
\newblock {\em Exp. Math.}, 26(3):308--323, 2017.

\bibitem[LO08]{Artin1}
Yves Laszlo and Martin Olsson.
\newblock The six operations for sheaves on {A}rtin stacks. {I}. {F}inite
  coefficients.
\newblock {\em Publ. Math. Inst. Hautes \'{E}tudes Sci.}, (107):109--168, 2008.

\bibitem[LO09]{Artin2}
Yves Laszlo and Martin Olsson.
\newblock Perverse {$t$}-structure on {A}rtin stacks.
\newblock {\em Math. Z.}, 261(4):737--748, 2009.

\bibitem[LO10]{OrlovLunts_DGCat}
Valery~A. Lunts and Dmitri~O. Orlov.
\newblock Uniqueness of enhancement for triangulated categories.
\newblock {\em J. Amer. Math. Soc.}, 23(3):853--908, 2010.

\bibitem[LT99]{Hurwitz2}
A.~Lucchini and M.~C. Tamburini.
\newblock Classical groups of large rank as {H}urwitz groups.
\newblock {\em J. Algebra}, 219(2):531--546, 1999.

\bibitem[Mar35]{Markov35Moves}
A.A. Markov.
\newblock \"uber die freie aquivalenz der geschlossen zopfe.
\newblock {\em Recueil Math. Moscou}, 1:73--78, 1935.

\bibitem[MGOT12]{NGons1}
Sophie Morier-Genoud, Valentin Ovsienko, and Serge Tabachnikov.
\newblock 2-frieze patterns and the cluster structure of the space of polygons.
\newblock {\em Annales de l'Institut Fourier}, 62(3):937--987, 2012.

\bibitem[Mur12]{Loose}
Emmy Murphy.
\newblock {Loose Legendrian Embeddings in High Dimensional Contact Manifolds}.
\newblock 2012.

\bibitem[Nad09]{Nadler_MicrolocalBrane}
David Nadler.
\newblock Microlocal branes are constructible sheaves.
\newblock {\em Selecta Math. (N.S.)}, 15(4):563--619, 2009.

\bibitem[Nad17a]{Nadler17}
David Nadler.
\newblock Arboreal singularities.
\newblock {\em Geom. Topol.}, 21(2):1231--1274, 2017.

\bibitem[Nad17b]{Nadler_LG1}
David Nadler.
\newblock A combinatorial calculation of the {L}andau-{G}inzburg model
  {$M=\Bbb{C}^3$}, {$W=z_1z_2z_3$}.
\newblock {\em Selecta Math. (N.S.)}, 23(1):519--532, 2017.

\bibitem[Nei14]{GMN_Cluster}
Andrew Neitzke.
\newblock Cluster-like coordinates in supersymmetric quantum field theory.
\newblock {\em Proc. Natl. Acad. Sci. USA}, 111(27):9717--9724, 2014.

\bibitem[Ng03]{Ng03}
Lenhard~L. Ng.
\newblock Computable {L}egendrian invariants.
\newblock {\em Topology}, 42(1):55--82, 2003.

\bibitem[Ng11]{Ng10}
Lenhard Ng.
\newblock Combinatorial knot contact homology and transverse knots.
\newblock {\em Adv. Math.}, 227(6):2189--2219, 2011.

\bibitem[NR13]{Satellite2}
Lenhard Ng and Daniel Rutherford.
\newblock Satellites of {L}egendrian knots and representations of the
  {C}hekanov-{E}liashberg algebra.
\newblock {\em Algebr. Geom. Topol.}, 13(5):3047--3097, 2013.

\bibitem[NRS{\etalchar{+}}15]{NRSSZ}
Lenhard Ng, Dan Rutherford, Vivek Shende, Steven Sivek, and Eric Zaslow.
\newblock {Augmentations are Sheaves}.
\newblock {\em To appear in Geom. Top.}, 2015.

\bibitem[NZ09]{Sheaves3}
David Nadler and Eric Zaslow.
\newblock Constructible sheaves and the {F}ukaya category.
\newblock {\em J. Amer. Math. Soc.}, 22(1):233--286, 2009.

\bibitem[OST13]{NGons2}
Valentin Ovsienko, Richard~Evan Schwartz, and Serge Tabachnikov.
\newblock Liouville–arnold integrability of the pentagram map on closed
  polygons.
\newblock {\em Duke Math. J.}, 162(12):2149--2196, 2013.

\bibitem[Pal15]{Palesi15}
Frederic Palesi.
\newblock {Introduction to positive representations and Fock-Goncharov
  Coordinates}.
\newblock 2015.

\bibitem[Pan17a]{YuPan2}
Yu~Pan.
\newblock {\em Augmentations and {E}xact {L}agrangian {C}obordisms}.
\newblock ProQuest LLC, Ann Arbor, MI, 2017.
\newblock Thesis (Ph.D.)--Duke University.

\bibitem[Pan17b]{YuPan}
Yu~Pan.
\newblock Exact {L}agrangian fillings of {L}egendrian {$(2,n)$} torus links.
\newblock {\em Pacific J. Math.}, 289(2):417--441, 2017.

\bibitem[Pol91]{Polterovich_Surgery}
L.~Polterovich.
\newblock The surgery of {L}agrange submanifolds.
\newblock {\em Geom. Funct. Anal.}, 1(2):198--210, 1991.

\bibitem[PS97]{PrasolovSossinsky}
V.~V. Prasolov and A.~B. Sossinsky.
\newblock {\em Knots, links, braids and 3-manifolds}, volume 154 of {\em
  Translations of Mathematical Monographs}.
\newblock American Mathematical Society, Providence, RI, 1997.

\bibitem[Rei27]{Reidemeister}
Kurt Reidemeister.
\newblock Elementare {B}egr\"{u}ndung der {K}notentheorie.
\newblock {\em Abh. Math. Sem. Univ. Hamburg}, 5(1):24--32, 1927.

\bibitem[Rol76]{Rolfsen76}
Dale Rolfsen.
\newblock {\em Knots and links}.
\newblock Publish or Perish Inc., Berkeley, Calif., 1976.
\newblock Mathematics Lecture Series, No. 7.

\bibitem[Ros98]{Roseman95}
Dennis Roseman.
\newblock Reidemeister-type moves for surfaces in four-dimensional space.
\newblock In {\em Knot theory ({W}arsaw, 1995)}, volume~42 of {\em Banach
  Center Publ.}, pages 347--380. Polish Acad. Sci. Inst. Math., Warsaw, 1998.

\bibitem[Rou06]{Rouquier06}
Rapha\"{e}l Rouquier.
\newblock Categorification of $sl_2$ and braid groups.
\newblock In {\em Trends in representation theory of algebras and related
  topics}, volume 406 of {\em Contemp. Math.}, pages 137--167. Amer. Math.
  Soc., Providence, RI, 2006.

\bibitem[RS19a]{RutherfordSullivan1}
Daniel Rutherford and Michael Sullivan.
\newblock Cellular {L}egendrian contact homology for surfaces, part {II}.
\newblock {\em Internat. J. Math.}, 30(7):1950036, 135, 2019.

\bibitem[RS19b]{RutherfordSullivan2}
Daniel Rutherford and Michael Sullivan.
\newblock Cellular {L}egendrian contact homology for surfaces, part {III}.
\newblock {\em Internat. J. Math.}, 30(7):1950037, 111, 2019.

\bibitem[RSTZ14]{RSTZ14}
Helge Ruddat, Nicol\`o Sibilla, David Treumann, and Eric Zaslow.
\newblock Skeleta of affine hypersurfaces.
\newblock {\em Geom. Topol.}, 18(3):1343--1395, 2014.

\bibitem[Sch53]{Schubert53_Satellite}
Horst Schubert.
\newblock Knoten und {V}ollringe.
\newblock {\em Acta Math.}, 90:131--286, 1953.

\bibitem[She19]{Conormal1}
Vivek Shende.
\newblock The conormal torus is a complete knot invariant.
\newblock {\em Forum Math. Pi}, 7:e6, 16, 2019.

\bibitem[Siv11]{Sivek_Bordered}
Steven Sivek.
\newblock A bordered {C}hekanov-{E}liashberg algebra.
\newblock {\em J. Topol.}, 4(1):73--104, 2011.

\bibitem[Smi15]{Smith_QuiverFuk}
Ivan Smith.
\newblock Quiver algebras as {F}ukaya categories.
\newblock {\em Geom. Topol.}, 19(5):2557--2617, 2015.

\bibitem[Soe07]{Soergel07}
Wolfgang Soergel.
\newblock Kazhdan-{L}usztig-{P}olynome und unzerlegbare {B}imoduln \"{u}ber
  {P}olynomringen.
\newblock {\em J. Inst. Math. Jussieu}, 6(3):501--525, 2007.

\bibitem[SS16]{SabloffSullivan16}
Joshua~M. Sabloff and Michael~G. Sullivan.
\newblock Families of {L}egendrian submanifolds via generating families.
\newblock {\em Quantum Topol.}, 7(4):639--668, 2016.

\bibitem[Sta18]{Starkston}
Laura Starkston.
\newblock Arboreal singularities in {W}einstein skeleta.
\newblock {\em Selecta Math. (N.S.)}, 24(5):4105--4140, 2018.

\bibitem[STW16]{STW}
Vivek Shende, David Treumann, and Harold Williams.
\newblock {On the combinatorics of exact Lagrangian surfaces}.
\newblock 2016.

\bibitem[STWZ19]{STWZ}
Vivek Shende, David Treumann, Harold Williams, and Eric Zaslow.
\newblock Cluster varieties from {L}egendrian knots.
\newblock {\em Duke Math. J.}, 168(15):2801--2871, 2019.

\bibitem[STZ17]{STZ_ConstrSheaves}
Vivek Shende, David Treumann, and Eric Zaslow.
\newblock Legendrian knots and constructible sheaves.
\newblock {\em Invent. Math.}, 207(3):1031--1133, 2017.

\bibitem[Tab05]{DGDerived2}
Goncalo Tabuada.
\newblock Une structure de cat\'{e}gorie de mod\`eles de {Q}uillen sur la
  cat\'{e}gorie des dg-cat\'{e}gories.
\newblock {\em C. R. Math. Acad. Sci. Paris}, 340(1):15--19, 2005.

\bibitem[Toe07]{DGDerived1}
Bertrand Toen.
\newblock The homotopy theory of dg-categories and derived morita theory.
\newblock {\em Invent. Math.}, 167(3):615--667, 2007.

\bibitem[Tri19]{OBS}
Minh-Tam~Q. Trinh.
\newblock {Annular Homology of Artin Braids I}.
\newblock 2019.

\bibitem[TZ18]{TreumannZaslow}
David Treumann and Eric Zaslow.
\newblock Cubic planar graphs and {L}egendrian surface theory.
\newblock {\em Adv. Theor. Math. Phys.}, 22(5):1289--1345, 2018.

\bibitem[Via14]{Vianna_Thesis}
Renato Vianna.
\newblock On exotic {L}agrangian tori in {$\Bbb{CP}^2$}.
\newblock {\em Geom. Topol.}, 18(4):2419--2476, 2014.

\bibitem[Wei71]{WeinsteinNeigh71}
Alan Weinstein.
\newblock Symplectic manifolds and their {L}agrangian submanifolds.
\newblock {\em Advances in Math.}, 6:329--346 (1971), 1971.

\end{thebibliography}

\end{document}